\documentclass[reqno]{amsart}

\usepackage{comment}
\numberwithin{subsection}{section}
\numberwithin{equation}{section}
\counterwithin*{footnote}{section}

\usepackage{a4wide}

\newcommand{\cB}{\mathcal{B}}
\newcommand{\cC}{\mathcal{C}}
\usepackage{cite,mathtools,mathrsfs,amscd,amsmath,amssymb}
\usepackage{yhmath}
\usepackage{hyperref}                   
\hypersetup{colorlinks,
    linkcolor=blue,%
    citecolor=blue}
\usepackage{graphicx}
\newcommand{\norm}[1]{\left\lVert#1\right\rVert}

\usepackage{tikz}
\usetikzlibrary{matrix}

\newtheorem{theorem}{\bf Theorem}[section]
\newtheorem{lemma}[theorem]{\bf Lemma}
\newtheorem{fact}[theorem]{\bf Fact}
\newtheorem{corollary}[theorem]{\bf Corollary}
\newtheorem{proposition}[theorem]{\bf Proposition}
\newtheorem{example}[theorem]{\bf Example}

\newtheorem{notation}[theorem]{\bf Notation}

\newtheorem{conjecture}[theorem]{\bf Conjecture}

\theoremstyle{definition}
\newtheorem{definition}[theorem]{\bf Definition}

\newtheorem{remark}[theorem]{\bf Remark}

\newcommand{\cM}{\mathcal{M}}

 \usepackage{float}

\usepackage{exscale}
\usepackage{fix-cm}
 
\begin{document}
\title[Arazy-type decomposition theorem and commutators on $\mathcal C_1$]{Arazy-type decomposition theorem for bounded linear operators and commutators on the trace class}

\author[J. Huang]{Jinghao Huang}
\address{Institute for  Advanced Study in  Mathematics of HIT, Harbin Institute of Technology, Harbin, 150001, China}
\email{jinghao.huang@hit.edu.cn}
 
\author[F. Sukochev]{Fedor Sukochev}
\address{School of Mathematics and Statistics, University of NSW, Sydney,  2052, Australia}
\email{f.sukochev@unsw.edu.au}

\author[Z. Yu]{Zhizheng Yu}
\address{Institute for  Advanced Study in  Mathematics of HIT, Harbin Institute of Technology, Harbin, 150001, China}
\email{zhizheng.yu@hit.edu.cn}

\thanks{J. Huang was supported the NNSF of China (No.12031004, 12301160 and 12471134). F. Sukochev was supported by the ARC}


\begin{abstract}  
The classical Arazy's decomposition theorem provides a powerful tool in the study of  sequences in (and isomorphisms on) a separable operator ideal $\cC_E$ of the algebra $\cB(H)$ of all bounded linear operators on the separable infinite-dimensional Hilbert space $H$. 
In this paper, we extend and strengthen  Arazy's decomposition theorem to the setting  of general bounded linear operators on  a separable (quasi-Banach)  operator ideal  $\cC_E$ of $\cB(H)$. 
Several applications are given to the study of   $\cC_E$-strictly singular operators, largest proper ideals in the algebra $\cB(\cC_E)$ of all  bounded linear operators on $\cC_E$ and complementably homogeneous Banach spaces among others. 
Our versions of  decomposition theorems supply tools for a noncommutative generalization of deep commutator theorems for operators on $\ell_p$ and $L_p$, $1\le p <\infty $, due to Brown and Pearcy, Apostol, and  Dosev, Johnson and Schechtman. 
We are able to characterize commutators on   the Schatten--von Neumann class $\cC_p$, $1\le p<\infty $. 
For the crucial case, $p=1$, we establish  that any operator $T\in\mathcal B(\cC_1)$ is a commutator if and only if $T$ is not of the form $\lambda I+K$  for some  $\lambda\ne 0$ and  $\cC_1$-strictly singular operator $K $. 
\end{abstract}

\maketitle

\tableofcontents

\section[Introduction]{Introduction}

\subsection{Commutators on Banach spaces}

This paper takes as its starting point a well-known result of Wintner~\cite{Wintner}:
\begin{theorem}\label{win}
The identity in a unital Banach algebra is not a commutator. 
\end{theorem}
As an immediate consequence, no element of the form $\lambda I +K$ is a commutator in a Banach algebra $\mathcal{A}$, where $K$ belongs to a norm-closed proper ideal $\mathcal I$ of $\mathcal A$ and $\lambda\ne 0$ is a scalar. 

We denote by $\cB(X)$ the Banach algebra of all bounded linear operators on a Banach space $X$. 
When $\dim(X)=n<\infty$, $\cB(X)$ has no nontrivial ideal  but    Shoda  \cite{Shoda} proved that for any $n$ by $n$ matrix $A$ is a commutator if and only if its trace is zero.
On the other hand, the situation in $\dim(X)=\infty$ is quite different: $\cB(X)$ possesses at least one nontrivial closed ideal, such as $\mathcal K(X)$, the ideal of compact operators on $X$.
Brown and Pearcy \cite{Brown Pearcy} proved that an operator $T\in \cB(\ell_2)$ is not a commutator if and only if 
$T$ can be represented as $\lambda I + K$ for some compact operators $K$ with $0\ne \lambda \in \mathbb{C}$. 
Later, Apostol\cite{Apostol_lp,Apostol_c0} obtained similar characterizations for $\ell_p$, $1<p<\infty$, and $c_0$. 
The remaining case for $\ell_1$ was resolved $35$ years later by Dosev\cite{Dosev_l1}, see also \cite{Dosev Johnson} for results concerning commutators on various Banach spaces.
Note that the situation for commutators on $\ell_\infty$ changes dramatically. 
Precisely, Apostol showed that if $T\in \cB(\ell_\infty)$ is compact, then it is a commutator\cite{Apostol_lp}. 
This result was later extended by Dosev and Johnson\cite{Dosev Johnson}, who showed that 
  {\it $T\in \cB(\ell_\infty)$ is not a commutator if and only if $T$ is of the from $\lambda I+S$ for some strictly singular operator\footnote{An operator $T:X\to X$ is called strictly singular if the restriction of $T$ to any infinite-dimensional subspace of $X$ is not an isomorphism\cite[Definition 2.c.2]{LT}. } $S$ and $\lambda\ne 0$.} Indeed, Dosev et al.  \cite{Dosev_l1,Dosev Johnson} refined key ideas from earlier works and obtained a series of important tools (e.g., \cite[Corollary 12, Theorem 16]{Dosev_l1} and \cite[Theorem 3.2, Corollary 5.4]{Dosev Johnson}). These tools enable the consideration of commutator representations on Banach spaces that extend beyond classical sequence spaces such as $\ell_p$ (for $1 < p < \infty$) and $c_0$.

We say that a Banach space $X$ is a Wintner space if 
every
non-commutator in $\cB(X)$  is of the form $\lambda I+ K$, where $\lambda  \ne  0$ and $K$ lies
in a proper ideal of $\cB(X)$. 
It is natural to consider the following wild conjecture \cite{Chen Johnson Bentuo Zheng}: 
\begin{quote}\label{wintner}
    {\it every infinite-dimensional Banach space is a Wintner space.}
\end{quote}
However, 
this conjecture has been disproved by Dosev, Johnson and Schechtman~\cite{Dosev Johnson Schechtman} with the help of the example constructed by Tarbard\cite{Tarbard}, which fails the  Pe{\l}czy\'nski decomposition\footnote{A Banach space $X$ is said to satisfy the Pe{\l}czy\'nski decomposition if $X$ is linearly isomorphic to $(\sum X)_0:=\big\{(x_k)_{k=1}^\infty\in X^{\mathbb  Z^+}:\left\Vert (x_k)_{k=1}^\infty\right\Vert:=\norm{( \norm{x_k})_{k=1}^\infty }_{c_0}<\infty \big\}$ or to 
$(\sum X)_p:=\big\{(x_k)_{k=1}^\infty\in X^{\mathbb  Z^+}:\Vert (x_k)_{k=1}^\infty\Vert:=(\sum_{k=1}^\infty\Vert x_k\Vert^p)^{1/p}<\infty\big\}$, 
for some  $p\in   [1,\infty)$.}. They also recognized that the Pe{\l}czy\'nski decomposition plays a crucial role in the study of Wintner spaces, 
 which leads to the following conjecture.
\begin{conjecture}\label{conjecture}
    Let $X$ be a Banach space such that $X\approx \left(\sum X \right)_p$, $1\le p\le \infty$, or $p=0$. 
    Assume that $B(X)$ has a largest nontrivial ideal $\cM$. Then, $T\in \cB(X)$ is not a commutator if and only if $T=\lambda I+K$ for some $K\in \cM$ and $\lambda\ne 0$\cite{Dosev Johnson} (see also \cite{Dosev Johnson Schechtman,CY, Chen Johnson Bentuo Zheng, Bentuo Zheng}). 
\end{conjecture}

 Dosev, Johnson, and Schechtman showed that $T\in \cB(L_p(0,1))$, $1\le p<\infty $, is not a commutator if and only if it can be written as   $T= \lambda I+K$, where $K$ is $L_p(0,1)$-strictly singular and $\lambda\ne 0$\cite{Dosev Johnson Schechtman}. In particular, this shows that Conjecture~\ref{conjecture} holds for $L_p(0,1)$, $1\le p<\infty$. 
Similar results for $$X=\left( \sum \ell_q \right)_p, ~1\le q,p<\infty $$
were established in 
 \cite{Chen Johnson Bentuo Zheng, Bentuo Zheng}.

  A noteworthy class of Banach spaces having  the  Pe{\l}czy\'nski decomposition is the Schatten--von Neumann $p$-class $\cC_p$, $1\leq p<\infty$\cite{Arazy2}.   In \cite{CY}, a characterization of commutators in the algebra $\cB(\cC_p)$, $1<p<\infty$, was given.  
In particular, Conjecture~\ref{conjecture} holds for $\cC_p$ when $1<p<\infty $ and the lower triangle part $\mathcal T_p$ when $1\le p<\infty $. 
However, the proof in \cite{CY} relies on the existence of linear isomorphisms between $\cC_p$ and   $\mathcal T_p$ when $1<p<\infty$ \cite[Theorem 4.7]{Arazy1}, which fails in the setting when $p=1$ (that is,  $\cC_1$ is not isomorphic to $\mathcal{T}_1$).
The main motivation of the present paper is to establish a representation theorem for  commutators on  $\cC_1$, which was left untreated in \cite{CY}. In particular, this shows that Conjecture \ref{conjecture} holds for the trace class.

\begin{theorem}\label{main theorem}
The algebra  of all $\cC_1$-strictly singular operators (see Definition~\ref{stri-sing} below) on $\cC_1$ is the largest nontrivial ideal in $\cB(\cC_1)$. 
    An operator $T\in \cB(\cC_1)$ is a commutator if and only if {$T-\lambda I$ is not a $\cC_1$-strictly singular operator for all scalars $\lambda \ne 0$}.
\end{theorem}

The study of the largest nontrivial ideal in $\cB(X)$ for a Banach space $X$ is of interest in its own right;
see, e.g., \cite{Leung1,Leung2,Lin Sari Zheng,Kania Laustsen1,Kania Laustsen2,BKL2020}.

 \subsection{Arazy's decomposition theorem and its applications}
Recall the so-called Calkin Correspondence  between (not necessarily separable) (quasi-)Banach symmetric sequence spaces $E$ and (quasi-)Banach operator ideals $\mathcal C_E$ in $\cB(H) $,
the $*$-algebra of all bounded linear operators on a separable Hilbert space $H$.
Precisely, 
$$E:= \{ x\in \ell_\infty :\mu(x)=s(a) \mbox{ for some } a\in \mathcal C_E\},~\norm{x}_E=\norm{a}_{\mathcal C_E}, $$
and 
$$\mathcal C_E:= \{ a\in \cB(H) : s(a) \in E\},  ~\norm{a}_{\mathcal C_E}= \norm{s(a)}_E,$$
is a bijective correspondence $(E,\norm{\cdot}_E) \Leftrightarrow (\mathcal C_E,\norm{\cdot}_{\mathcal C_E})$ between a (quasi-)Banach symmetric sequence space and a (quasi-)Banach symmetric operator ideal $ \mathcal C_E$ (or $\mathcal C_E(H)$) in $\cB(H)$\cite{LSZ, KS,Schatten,S14}, where $\mu(x)$ (respectively, $s(a)$) stands for the decreasing rearrangement of $|x|$, $x \in \ell_\infty$ (respectively, the singular values of $a\in \cB(H)$).

The classical theorem of  Arazy\cite{Arazy4} asserts that each basic sequence in a separable operator ideal $\cC_E$ contains  a subsequence which embeds into $\ell_2\oplus E$. 
The idea used in proving this deep result can be traced back to earlier work due to Friedman\cite{Friedman}, Holub\cite{Holub}, 
Arazy and Lindentrauss\cite{Arazy Lindenstrauss}, and Arazy\cite{Arazy2} concerning specific operator ideals in $\cB(H)$. 
Arazy's decomposition theorem serves as a useful tool in the study of geometry of operator ideals  (see, e.g., \cite{AHS,Arazy83,Arazy81b,HSS} and references therein). 
Furthermore, a decomposition theorem for linear isomorphisms between separable operator ideals was also established in \cite{Arazy4}.

In the study of operator ideals $\cC_E$ of $\cB(H)$, 
 it is natural to fix two orthonormal bases $\{\xi_{i}\}_{i=1}^\infty$ and $\{\eta_i\}_{i=1}^\infty$ in $H$, and consider the matrix representations of elements in $\mathcal C_E$ with respect to these bases via the correspondence
\[
x\in \mathcal C_E\to \big((x\eta_j|\xi_i)\big)_{i,j=1}^\infty.
\] 
 However, since
we shall frequently change the basis of a Hilbert space in the present paper, 
 we introduce another terminology (see Section \ref{blocksequences}). 
 For the study of problems considered in the present paper and
 for the readability and self-containedness, we recast several results due to  Arazy \cite{Arazy4} (for separable  quasi-Banach operator ideals)
 in terms of our terminology (see Appendix \ref{appendix}).

In the present paper, we extend Arazy's decomposition theorem to the setting of bounded linear operators on separable quasi-Banach operator ideals (see,  e.g.,  the main results of 
Section~\ref{sec:decomposition CE operator}: Theorems \ref{thma'}, \ref{thmb} and \ref{thmc}). 
Informally speaking, 
let $E$ and $F$ be two separable quasi-Banach symmetric sequence spaces with $c_0,\ell_2\not \hookrightarrow E$, and let  $T$ be  a bounded linear operator from   the lower triangular part $\mathcal T_F$ of $\mathcal C_F$ (see Notation \ref{tenotation} below)   to $\mathcal C_E$.
There exist two closed subspaces $K$ and $L$ of $H$ such that  the lower triangular part $\mathcal T_F(L,K)$ of $\mathcal C_F(L,K)$  is contained in $\mathcal T_F$, and the restriction $T|_{\mathcal T_F(L,K)}$ can be extended to an operator $\tilde{T}$ from $\mathcal C_F(L,K)$ to $\mathcal C_E$ such that $\tilde{T}$ admits a ``nice" representation (in the sense of small perturbations).
Note that the boundedness of   the triangular projection on $\cC_E$ (e.g. when $E$ is a Banach symmetric sequence space such that  $\cC_E\approx \mathcal{T}_E\oplus \mathcal{T}_E$)  allows one to consider perturbations on the lower triangular part and upper triangular part separately. 
However, when  the triangular projection on $\cC_E$ is unbounded,
establishing a `global' result for $\cC_E$
is always problematic, and its proof is of interest in its own right. 
On the other hand,  since  one cannot guarantee the boundedness of diagonal projections (see \eqref{diag-proj}) 
and triangular projections (see \eqref{tria-proj})
in general quasi-Banach operator ideals, 
 we have to avoid the usage of  diagonal/triangular projections in some proofs below (see, e.g., Section \ref{subsection:operator}), which are different from the proofs in \cite{Arazy4}. 
These results are particularly useful  in the subsequent study of isomorphic embedding between quasi-Banach operator ideals, see \cite{GHSY} (see \cite{Arazy4,HSS} for the case of Banach operator ideals). 

Below, we present a few 
  important consequences of this result, and important ingredients in the study of commutators on $\mathcal C_1$, which are of independent interest:
  \begin{itemize}
      \item  {\bf 
      Isomorphisms between $\mathcal T_E$ and $\mathcal C_E$.} 
      Suppose that $E$ is a separable Banach symmetric sequence space.  
It is well-known that  
$\mathcal T_E$ is not isomorphic to $\mathcal C_E$ if and only if $\mathcal T_E$ is not complemented in $\mathcal C_E$ (see, e.g., \cite{Arazy1}). In the setting when $E$ separable (quasi-)Banach space which does not  contain copies $c_0$ and $\ell_2$, we   extend  this result by showing  that $\mathcal T_E$ is not isomorphic to $\mathcal C_E$ if and only if $\mathcal T_E$ is not isomorphic to a complemented subspace of $\mathcal C_E$ (see Theorem \ref{teicce} below). 
In particular,  the lower triangle part $\mathcal T_1$ of $\mathcal{C}_1$ is not isomorphic to a complemented subspace of $\mathcal C_1$, 
which strengthens a classical result due to Kwapien and Pelczynski\cite{KwapienPelczynski}, and Arazy   \cite{Arazy1}. 
 This   serves as a crucial step towards  proving that $\mathcal T_1\oplus \mathcal C_1$ is a  Wintner space (see Theorem \ref{t1+c1}).
\item  {\bf Local representation of operators on $\cC_E$}.
Up to  small perturbations, we obtain a local presentation of bounded  operators from $\mathcal T_E$ (respectively, $\mathcal C_E$) to $\mathcal C_E$ when $c_0\not\hookrightarrow E$ (respectively, $c_0,\ell_2\not\hookrightarrow E$),  see Theorems \ref{thma'}, \ref{thmb} and \ref{thmc}.
In particular,  for any bounded operator $T$ from $\mathcal C_p$ to $\mathcal C_p$,
$0<p\leq1$, there exist two increasing sequences $\{i_k\}_{k=1}^\infty$ and $\{j_k\}_{k=1}^\infty$ of positive integers, and  two operators $a,b\in\mathcal C_p$ with $\norm{a}_{\mathcal C_p}=1$, such that---up to  small perturbations and via an appropriate tensor product representation---one has 
\[
T(a\otimes e_{i_{2k+1},i_{2l}})=b\otimes e_{j_{2k+1},j_{2l}}\;\;\;\;\;\;\;\;{\rm for\;every\;}k,l\geq1,
\]
where $\{e_{i,j}\}_{i,j=1}^\infty$ denotes the standard matrix units in $\mathcal C_p$, see Theorem~\ref{p=1'} below  (see also Theorems \ref{1p2}, \ref{2<p} and \ref{pneq2} for related results with respect to  $\mathcal T_p$, $0<p<\infty$).
Similar results for  the special case when $T$ is an isomorphic embedding of $\mathcal C_E$ were established earlier by Arazy \cite{Arazy4}. 
 \item {\bf  $\cC_E$-strictly singular operators and $\mathcal{T}_E$-strictly singular operators}. 
 It is well-known that $\mathcal C_E $ is not isomorphic to its lower triangular part $\mathcal T_E$ when $E$ has trivial Boyd indices (in particular,  $\mathcal C_1 $ is not isomorphic to its lower triangular part $\mathcal T_1$, 
 see, e.g., \cite{Arazy1}), which is an  obstacle in studying the structure of commutators on $\mathcal C_1 $. 
One   consequence of our 
representation theorem  for  operators on $\cC_E$  is given in Theorem \ref{c1s} below, which states that,  whenever  $E$   does not contain copies of $c_0$ and $\ell_2$,   
 an operator $T$ on $\mathcal C_E $ is $\mathcal C_E$-strictly singular if and only if  it  is  $\mathcal T_E$-strictly singular (see Theorem~\ref{cs=ts}).
This together with Theorem \ref{main theorem} implies that 
an operator $T\in \cB(\cC_1)$ is a commutator if and only if the operator $T-\lambda I$ is not $\mathcal T_1$-strictly singular  for every  scalar $\lambda \ne 0$.

 \item  {\bf A characterization of $\cC_1$-strictly singular operators. }
 Let $1\leq p<\infty$. It is well-known that $\mathcal K(\ell_p)$, the set of all  compact operators on $\ell_p$, is the unique nontrivial closed ideal, hence it is automatically the largest proper ideal, in $\cB(\ell_p)$ (see \cite{FGM1960}), which coincides with the 
 ideal of strictly singular operators (see e.g. \cite{Kato} and \cite[Theorem 6.2]{Whitley}).
Such Banach spaces ($\ell_p$, $1\le p<\infty$) are said to have   the  Kato property (i.e., every strictly singular
endomorphism on them  is necessarily compact\cite{HST}).
 However, 
 $\mathcal K(L_p)$ is not the largest ideal in $\mathcal B(L_p)$, which is  the  ideal consisting  of all $L_p$-strictly singular operators (see, e.g., \cite[Theorem 1.3]{Dosev Johnson Schechtman}).
 Dosev, Johnson, and Schechtman showed
 for $1\le p<2$, the following holds:
 {\it Given any 
   $L_p$-strictly singular operator $T$ and any subspace $X$ of $L_p$ which is isomorphic to $L_p$, 
 there exists a further subspace $Y\subset X$, also isomorphic to $L_p$,  such that the restriction  $T|_Y$ is compact and has a norm which is  small enough (see \cite[Theorem 1.3 and Lemma 3.7]{Dosev Johnson Schechtman}).}
 This property  serves as the key tool for proving  that every  $L_p$-strictly singular operators is a commutator  in $B(L_p)$. 
In the present paper, we show that for $0<p<\infty$, all $\mathcal C_p$-strictly singular (resp. $\mathcal T_p$-strictly singular) operators constitute the largest proper ideal in $\mathcal B(\mathcal C_p)$ (resp. $\mathcal B(\mathcal T_p)$) (see Remark \ref{c1srema} and Section~\ref{Sec:largest proper ideal}). Moreover,
\begin{itemize}
    \item [(1)]
Let $0<p\leq1$. Then,  $T\in\mathcal B(\mathcal C_p)$ is $\cC_p$-strictly singular if and only if $T$ is $\mathcal T_p$-strictly singular (see Theorem \ref{c1s} below).
    \item [(2)]
Let $0<p<\infty$. Then $T\in\mathcal B(\mathcal C_p)$ is $\mathcal C_p$-strictly singular if and only if for every subspace $X$ of $\mathcal C_p$ with $X\approx\mathcal C_p$ and for any $\varepsilon>0$, there exists a subspace $Y\subset X$ such that $Y\approx \mathcal C_p$ and $\Vert T|_Y\Vert<\varepsilon$ \footnote{Note that, in contrast with the result of  Doesv et. al.,   $T|_Y$ may not be compact, see Remark~\ref{c1srema}.} (see Corollary \ref{cps0<p<} below).
    \item [(3)]
Let $0<p<\infty$. Then $T\in\mathcal B(\mathcal T_p)$ is $\mathcal T_p$-strictly singular if and only if
for every subspace $X$ of $\mathcal T_p$ with $X\approx\mathcal T_p$ and for any $\varepsilon>0$, there exists a subspace $Y\subset X$ such that $Y\approx \mathcal T_p$ and $\Vert T|_Y\Vert<\varepsilon$ (see Theorem \ref{tps} below).
\end{itemize}

 \item {\bf Complementably homogeneous Banach spaces.} The notion of a complementably homogeneous Banach space  plays a crucial role  in the study of commutators on a Banach space. 
 It is shown in \cite{Arazy4,Arazy2} that every Banach symmetric operator ideal $\cC_E$   is complementably homogeneous (see Definition~\ref{ch} below) when  $E$ does not contain copies of  $c_0$ and $\ell_2$. 
 Moreover,  it is known that the lower triangular part $\mathcal{T}_p$ of  the Schatten $p$-class, $1\leq p<\infty $,  is $2^+$-complementably homogeneous (see Definition~\ref{+ch}). 
We improve these results by showing that every quasi-Banach symmetric operator ideal $\cC_E$ is complementably homogeneous (see Theorem \ref{ce-hc}) when $E$ does not contain copies of $c_0$ and $\ell_2$. Also, 
 in Section \ref{C1homogeneous},  $\mathcal{T}_p$ ($1\le p <\infty $) and $\mathcal{C}_1$ are shown to be  $1^+$-complementably homogeneous.
It is worth noting that there is no   criterion for a symmetric sequence space to be complementably homogeneous.

\end{itemize}

\section{Preliminaries}

 We denote by $[x_n]_{n=1}^\infty$ the closure of ${\rm span}(\{x_n\}_{n=1}^\infty)$ in a quasi-Banach space ${X}$. A basic sequence $\{x_n\}_{n=1}^\infty$ in some quasi-Banach space is said to be \emph{equivalent to} another basic sequence $\{y_n\}_{n=1}^\infty$ in another quasi-Banach space (denoted by $\{x_n\}_{n=1}^\infty\sim \{y_n\}_{n=1}^\infty$) provided that there exists $1\leq\lambda<\infty$ such that for every scalar sequence  $\{t_n\}_{n=1}^\infty$,
\[
{\lambda}^{-1}\Big\Vert\sum_n t_ny_n\Big\Vert\leq\left\Vert\sum_n t_nx_n\right\Vert\leq\lambda\left\Vert\sum_n t_ny_n\right\Vert.
\]
If $\lambda=1$, then the basic sequences $\{x_n\}_{n=1}^\infty$ and $\{y_n\}_{n=1}^\infty$ are said to be \emph{isometrically equivalent} (denoted by $\{x_n\}_{n=1}^\infty\simeq \{y_n\}_{n=1}^\infty$).

Let $\lambda\geq1$ be a positive number. We say that a subspace $Y$ of the (quasi-)Banach space $X$ is \emph{$\lambda$-complemented} in $X$ if there exists a projection $P$ from $X$ onto $Y$ with $\Vert P\Vert\leq\lambda$\cite[p.131]{Pietsch}. For two quasi-Banach spaces $X$ and $Y$, we said \emph{$X$ is $\lambda$-isomorphic to $Y$} provided that  there exists an linear isomorphism $T$ from $X$ onto $Y$ such that $\Vert T^{-1}\Vert\cdot\Vert T\Vert\leq\lambda$. The notation  ${X}\approx{Y}$ denotes  that ${X}$ is linearly isomorphic to ${Y}$, and ${X}\cong{Y}$ means that ${X}$ is linearly isometric to ${Y}$.
 If $X$ is linearly isomorphic (respectively, is not linearly isomorphic) to a subspace of $Y$, then we write $X\hookrightarrow Y$  (respectively, $X\not \hookrightarrow Y$). 
 We denote by $\mathcal B(X,Y)$ the space of bounded linear operators from $X$ to $Y$. Throughout this paper, an ``operator" means a bounded linear operator, and an ``isomorphism" means a linear isomorphism.
 Throughout this paper, 
 if $Z$ is a subspace of  $Y$ and  $T:X\to  Y$ and $S:X\to Z$ are two operators satisfying   $T(x)=S(x)$ for all $x\in X$, then we identify $T$ and $S$.

\begin{definition}
A quasi-norm on vector space $X$ (over real numbers field $
\mathbb R$ or complex numbers field $\mathbb C$) is a function $x\in X\mapsto\Vert x\Vert\in[0,\infty)$ which satisfies
\begin{itemize}
    \item [(i)]
    $\Vert x\Vert>0$ for every $x\in X\setminus\{0\}$,
    \item [(ii)]
    $\Vert\alpha x\Vert=\vert\alpha\vert\Vert x\Vert$ for every scalar $\alpha$ and $x\in X$,
    \item [(iii)]
    there is a $\kappa(\geq 1)$ such that $\Vert x+y\Vert\leq\kappa(\Vert x\Vert+\Vert y\Vert)$ for every
$x,y\in X$. The smallest possible $\kappa$ is called modulus of concavity of $\Vert\cdot\Vert$.
\end{itemize}
Set  $\mathscr U=\{U_n:U_n=\{x\in X:\Vert x\Vert<1/n\}\;{\rm and\;}n\in\mathbb Z^+\}$. 
$\mathscr U$ induces on $X$ a metrizable vector topology $\tau$ such that $\mathscr U$ is a local base for $\tau$. We said $X$ is \emph{quasi-Banach space} if $(X,\tau)$ is complete.
\end{definition}

\begin{definition}\cite[p.7]{KPR}
Let $r(\leq1)$ be a positive number. A quasi-norm on vector space $X$ is said to be \emph{$r$-subadditive} if
\[
\norm{x+y}^r\leq \norm{x}^r+\norm{y}^r\;\;\;\;\;\;\mbox{for any }x,y\in X.
\]
\end{definition}

\begin{remark}\label{rem-lim-norm}
In general, a quasi-norm $\Vert\cdot\Vert$ on vector space $X$ is not continuous. For a sequence $\{x_n\}_{n=1}^\infty$ converging to $x\in X$, we only have
\begin{equation}\label{lim-norm}
    \kappa^{-1}\Vert x\Vert\leq\varliminf_{n\to\infty}\Vert x_n\Vert\leq\varlimsup_{n\to\infty}\Vert x_n\Vert\leq\kappa\Vert x\Vert,
\end{equation}
where $\kappa$ is the modulus of concavity of $\left\Vert\cdot\right\Vert$. However, an $r$-subadditive  quasi-norm is continuous.
\end{remark}

\begin{theorem}\cite[Lemma 1.1 and Theorem 1.3]{KPR}\label{qbs-norm}
  Let $\left\Vert\cdot \right\Vert$ be a quasi-norm on $X$ with modulus $\leq2^{\frac{1}{r}-1}$. Then for any $x_1,\dots,x_n\in X$, we have
      \begin{equation}\label{qt-eq}
      \Vert x_1+\cdots+x_n\Vert\leq4^{\frac{1}{r}}(\Vert x_1\Vert^r+\cdots+\Vert x_n\Vert^r)^{\frac{1}{r}}.
  \end{equation}
  In particular, there is a $r$-subadditive quasi-norm define by
  \begin{equation}
      \vert\vert\vert x\vert\vert\vert=\inf\left\{\left({\sum}_{i=1}^n\Vert x_i\Vert^r\right)^{\frac{1}{r}}:x={\sum}_{i=1}^nx_i\right\}\;\;\;\;\;\;\mbox{for any }x\in X,
  \end{equation}
  which $\left\vert\left\vert\left\vert\cdot\right\vert\right\vert\right\vert\leq \left\Vert\cdot \right\Vert\leq4^{\frac{1}{r}}\left\vert\left\vert\left\vert\cdot\right\vert\right\vert\right\vert $.
\end{theorem}

\subsection{Quasi-Banach symmetric operator ideals in $\cB(H)$}\label{prel-qbsoi}

We denote by $\{e_k\}_{k=1}^\infty$ the standard unit vectors of a  sequence space over the real or complex field (defined by $e_k(i)=\delta_{k,i}$).
Throughout this paper, the notations $H$, $H'$ and $H''$  always denote the infinite-dimensional separable Hilbert spaces, i.e. $H\cong\ell_2$, over the real or complex field. Let $(\cdot|\cdot)$ be the inner product of $H$. 
For any  $a\in\mathcal B(H)$ and $b\in\mathcal B(H')$, we  define $a\otimes b\in\mathcal B(H\otimes_{_2}H')$ by
\[
(a\otimes b)(\xi\otimes\eta)=a(\xi)\otimes b(\eta),\;\;\;\;\;\;\xi\in H\;{\rm and}\;\eta\in H'.
\]
where $H\otimes_{_2}H'$ is  Hilbert–Schmidt tensor product of $H$ and $H'$.

For any index set $A$,
the space 
$\ell_2(A)$ consists of functions $f$ from   $A$ into the corresponding real or complex field, such that $\sum_{a\in A} \vert f(a)\vert^2 < \infty$. In this case, $e_i^H = \chi_{\{i\}}$ is the characteristic function of the set $\{i\}$, where $i \in A$.

\begin{definition}\label{hilertbasis}
Let $\{e^H_i\}_{i\in \mathbb{A}}$ stand for the natural basis of the Hilbert space $\ell_2(\mathbb{A})$ over an index set $\mathbb{A}$.
\end{definition}

Let $P(\cB(H))$ be the lattice of all projections in $\cB(H)$.
For any operator $a$ in $\cB(H)$, the sequence $s(a)$ of singular values of $a$ is given by  
$$ s_k(a)= \inf \{ \norm{a({\bf 1} - p )}_\infty: p\in P(\cB(H)),~{\rm Tr (p)} <  k     \},   $$
where ${\rm Tr}$ stands for the standard trace of $\mathcal B(H)$. 

Let $E$ be a quasi-Banach symmetric sequence space, i.e., $E$ is a subspace of $\ell_\infty$ and constitutes a quasi-Banach space under the equipped quasi-norm $\norm{\cdot}_E$, and satisfying that 
$y\in E, \mu(x)\le \mu(y) \Rightarrow  x\in E,~\norm{x}_E\le \norm{y}_E,
$
where by $\mu(z)$   we denote the usual decreasing rearrangement of the sequence $|z|$.
Without loss of generality, we always assume that $\norm{e_k}_E=1$ for all $k\ge 1$. 
We can define the associated Schatten ideal $\mathcal C_E\subset\cB(H)$ by setting 
$$a\in \mathcal C_E  \Leftrightarrow  s(a)\in E$$
with the quasi-norm
$$\norm{a}_{\mathcal C_E} =\norm{s(a)}_E,$$
see \cite{LSZ,KS,LPSZ,S14,HLS}. 
The space $\mathcal C_E$ is called the 
quasi-Banach symmetric operator ideal associated with $E$.
When $E=\ell_p$, $0<p<\infty$, or $c_0$, we write $\mathcal C_{\ell_p}$ as $\mathcal C_p$, $\mathcal C_{c_0}$ as $\mathcal C_0$.
For detailed information on symmetric operator ideals, we refer to \cite{Gohberg Krein,DPS,LSZ,S14,Simon}.

Below, we collect some facts concerning quasi-Banach symmetric operator ideals. Throughout this paper, we denote by $e_{\xi,\eta}$ the operator $(\cdot|\eta)\xi$,  $\xi,\eta\in H$. 
\begin{proposition}\label{ie}
Let $E$ be a quasi-Banach symmetric sequence space, $\{\xi_k\}_{k=1}^\infty$, $\{\eta_k\}_{k=1}^\infty$, $\{\xi'_k\}_{k=1}^\infty$ and $\{\eta'_k\}_{k=1}^\infty$ be four orthonormal sequences in $H$, and let $a,b\in \mathcal C_E$. Then
\[
\big\{a\otimes e_{\xi_{2i+1},\eta_{2j}}+b\otimes e_{\xi_{2j},\eta_{2i+1}}\big\}_{i,j=1}^\infty\simeq\left\{
\left(
\begin{smallmatrix}
0&b\\a&0
\end{smallmatrix}\right)\otimes e_{\xi'_i,\eta'_j}
\right\}_{i,j=1}^\infty.
\]
\end{proposition}

\begin{proof}
For any finitely non-zero sequence $\{\alpha_{i,j}\}_{i,j=1}^\infty$ of scalars, since $\sum_{i,j=1}^\infty \alpha_{i,j}(a\otimes e_{\xi_{2i+1},\eta_{2j}})$ and $\sum_{i,j=1}^\infty \alpha_{i,j}(b\otimes e_{\xi_{2j},\eta_{2i+1}})$ have disjoint left and right supports, it follows that

\begin{align*} 
&s\left(\sum_{i,j=1}^\infty\alpha_{i,j}\big(a\otimes e_{\xi_{2i+1},\eta_{2j}}+b\otimes e_{\xi_{2j},\eta_{2i+1}}\big)\right)\\
=&\mu\left(s\Bigg( \sum_{i,j=1}^\infty\alpha_{i,j}\big(a\otimes e_{\xi_{2i+1},\eta_{2j}}\big)\Bigg)\oplus s\Bigg(\sum_{i,j=1}^\infty\alpha_{i,j}\big(b\otimes e_{\xi_{2j},\eta_{2i+1}}\big)\Bigg)\right),
\end{align*}
where ``$\oplus$" means   the disjoint union of two decreasing positive sequences. Consequently,
\begin{align}\label{singular value of mapping}
&s\left(\sum_{i,j=1}^\infty\alpha_{i,j}\big(a\otimes e_{\xi_{2i+1},\eta_{2j}}+b\otimes e_{\xi_{2j},\eta_{2i+1}}\big)\right)\nonumber\\
=&\mu\left( 
s\Bigg(a\otimes  \bigg(\sum_{i,j=1}^\infty\alpha_{i,j}e_{\xi_{2i+1},\eta_{2j}}\bigg)\Bigg)\oplus s\Bigg(b\otimes  \bigg(\sum_{i,j=1}^\infty\alpha_{i,j}e_{\xi_{2j},\eta_{2i+1}}\bigg)\Bigg)  \right)\nonumber\\
=&\mu\left( 
s\Bigg(a\otimes  \bigg(\sum_{i,j=1}^\infty\alpha_{i,j}e_{\xi'_i,\eta'_j}\bigg)\Bigg)\oplus s\Bigg(b\otimes  \bigg(\sum_{i,j=1}^\infty\alpha_{i,j}e_{\xi'_i,\eta'_j}\bigg)\Bigg)\right)\nonumber\\
=&s\left(\left(
\begin{smallmatrix}
0&b\\a&0
\end{smallmatrix}\right)\otimes  \bigg(\sum_{i,j=1}^\infty\alpha_{i,j}e_{\xi'_i,\eta'_j}\bigg) \right)\nonumber\\
=&s\left(\sum_{i,j=1}^\infty\alpha_{i,j}\left(\left(
\begin{smallmatrix}
0&b\\a&0
\end{smallmatrix}\right)\otimes e_{\xi'_i,\eta'_j}\right)\right).
\end{align}

Hence, $\sum_{i,j=1}^\infty\alpha_{i,j}\big(a\otimes e_{\xi_{2i+1},\eta_{2j}}+b\otimes e_{\xi_{2j},\eta_{2i+1}}\big) \mapsto  \sum_{i,j=1}^\infty\alpha_{i,j}\left(\left(
\begin{smallmatrix}
0&b\\a&0
\end{smallmatrix}\right)\otimes e_{\xi'_i,\eta'_j}\right)
$
extends to an  isomorphic  mapping $S$ from  $  
\left[a\otimes e_{\xi_{2i+1},\eta_{2j}}+b\otimes e_{\xi_{2j},\eta_{2i+1}}\right]_{i,j=1}^\infty$ onto $\left[
\left(
\begin{smallmatrix}
0&b\\a&0
\end{smallmatrix}\right)\otimes e_{\xi'_i,\eta'_j}
\right]_{i,j=1}^\infty.
$ 
It follows that \eqref{singular value of mapping} holds for every element in  $  
\left[a\otimes e_{\xi_{2i+1},\eta_{2j}}+b\otimes e_{\xi_{2j},\eta_{2i+1}}\right]_{i,j=1}^\infty$, and thus $S$ is an isometry\footnote{
For general quasi-Banach spaces, 
since the quasi-norm is not necessarily continuous, 
 a mapping  is an isometry on a dense subspace of a quasi-Banach space does not  guarantee that   
  it extends to  an isometry on the whole space. 
  }.  
\end{proof}

Let $\{K_k\}_{k=1}^n$ and $\{L_k\}_{k=1}^n$ be two sequences of mutually orthogonal closed subspaces of $H$, where $n\in\mathbb N^+\cup\{\infty\}$. If $n<\infty$, then for any $x\in \mathcal C_E$, we have 
\begin{align}\label{diag-proj}
\left\Vert\sum_{k=1}^np_{_{K_k}}xp_{_{L_k}}\right\Vert_{\mathcal C_E}=&\left\Vert \frac{1}{2^n }\sum_{\{\delta_k\}_{k=1}^n=\{-1,1\}^n}\sum_{1\leq i,j\leq n} \delta_i\delta_jp_{_{K_i}}xp_{_{L_j}}\right\Vert_{\mathcal C_E}  \\
=&  \frac{1}{2^n } \left\Vert\sum_{\{\delta_k\}_{k=1}^n=\{-1,1\}^n} 
\left({\sum}_{i=1}^n\delta_ip_{_{K_i}}\right)x\left({\sum}_{j=1}^n\delta_jp_{_{L_j}}\right)\right\Vert_{\mathcal C_E}\nonumber\\
\leq&\frac{1}{2^n } \kappa\sum_{\delta_1=\pm1}\kappa\sum_{\delta_2=\pm1}\dots\kappa\sum_{\delta_n=\pm1}\left\Vert
\left({\sum}_{i=1}^n\delta_ip_{_{K_i}}\right)x\left({\sum}_{j=1}^n\delta_jp_{_{L_j}}\right)\right\Vert_{\mathcal C_E}\nonumber\\
\leq&\kappa^n \left\Vert x\right\Vert_{\mathcal C_E}\nonumber,
\end{align}
where $\kappa$ is the modulus of $\left\Vert\cdot\right\Vert_{\mathcal C_E}$.
If $E$ is a Banach symmetric sequence space, then we have
\begin{align}\label{projection inequality}
\left\Vert\sum_{k=1}^np_{_{K_k}}xp_{_{L_k}}\right\Vert_{\mathcal C_E}\leq \left\Vert x \right\Vert_{\mathcal C_E}.
\end{align}
This is a generalization of  \cite[Lemma 5.9.1]{DPS} (see also \cite{CKS92},  \cite[Theorem 3.1 and Corollary 3.2]{BS21} for a more general result). 
If $E$ is separable, then 
for any $x\in\mathcal C_E$,   the series $\sum_{i=1}^\infty p_{_{K_i}}xp_{_{L_i}}$ is (unconditionally) convergent in $\mathcal C_E$, and 
\begin{equation}\label{bs-diag-proj}
\left\Vert\sum_{k=1}^\infty p_{_{K_k}}xp_{_{L_k}}\right\Vert_{\mathcal C_E}\leq \left\Vert x \right\Vert_{\mathcal C_E},
\end{equation}
see e.g. \cite[Corollary 3.4]{CKS92}. 
Therefore, 
for any $1\le n\le \infty $, we may define a contractive projection $P_{\{K_k,L_k\}_{k=1}^n}$ by 
\begin{align}\label{diga-proj}
P_{\{K_k,L_k\}_{k=1}^n}:x\longmapsto\sum_{k=1}^np_{_{K_k}}xp_{_{L_k}}.
\end{align}
When $n=1$, we simply denote $P_{\{K_k,L_k\}_{k=1}^n}$ by $P_{K_1,L_1}$.

If $E$ is not a Banach space,
then one can not guarantee that  $P_{\{K_k,L_k\}_{k=1}^\infty}$ is a contractive projection, see e.g. \cite{BS21}. Indeed, 
  $P_{\{K_k,L_k\}_{k=1}^\infty}$ is even  unbounded on  $\cC_p$, $0<p<1$. For any two orthonormal bases $\{\xi_i\}_{i=1}^\infty$ and $\{\eta_i\}_{i=1}^\infty$ in $H$, we have
\[
\bigg\Vert\sum_{1\leq i,j\leq n}e_{\xi_i,\eta_j}\bigg\Vert_p=n
\]
and
\[
\bigg\Vert\sum_{i=1}^ ne_{\xi_i,\eta_i}\bigg\Vert_p=n^{1/p}.
\]
It follows that 
\begin{equation}\label{diag-proj-unb}
\frac{\Vert\sum_{i=1}^ ne_{\xi_i,\eta_i}\Vert_p}{\Vert\sum_{1\leq i,j\leq n}e_{\xi_i,\eta_j}\Vert_p}\to\infty\;\;\;\;{\rm as}\;n\to\infty,
\end{equation}
and thus $P_{\{[\xi_k],[\eta_k]\}_{k=1}^\infty}$ is unbounded on $\mathcal C_p$. 
However, the projection $P_{\{[\xi_k],[\eta_k]\}_{k=1}^\infty}$ is bounded on the 
lower triangular part of $\cC_p$ (with respect to  bases $\{
\xi_i\}_{i=1}^\infty$ and $\{\eta_i\}_{i=1}^\infty$) by the Weyl inequality\cite[Theorem 3.3]{Hiai}.
More generally, since every quasi-Banach symmetric operator ideal $\cC_E$ has an equivalent quasi-norm which is monotone with respect to Weyl submajorization (or logarithmic submajorization), see, e.g., \cite[Proposition 3.2]{Kalton98} or \cite[Theorem 7]{SZ},
it follows from the Weyl inequality
that  $P_{\{[\xi_k],[\eta_k]\}_{k=1}^\infty}$ is bounded on    the 
lower triangular part of $\cC_E$ (with respect to  bases $\{
\xi_i\}_{i=1}^\infty$ and $\{\eta_i\}_{i=1}^\infty$).

\begin{notation}\label{tenotation}
Let $\{\xi_i\}_{i=1}^\infty$ and $\{\eta_i\}_{i=1}^\infty$ be two orthonormal sequences in $H$. We denote by $\mathcal T_{E,\{\xi_i,\eta_i\}_{i=1}^\infty}$ the subspace of $\mathcal C_E$ given by
\begin{equation}\label{ts}
\mathcal T_{E,\{\xi_i,\eta_i\}_{i=1}^\infty}=[e_{\xi_i,\eta_j}]_{1\leq j\leq i<\infty},
\end{equation}
where   $e_{\xi,\eta}=(\cdot|\eta)\xi$ and $[e_{\xi_i,\eta_j}]_{1\leq j\leq i<\infty}$ stands for the closure of the linear span of $\{ e_{\xi_i,\eta_j}\} _{1\leq j\leq i<\infty}$ in $\mathcal{C}_{E}$.
Note that  $\mathcal T_{E,\{\xi_i,\eta_i\}_{i=1}^\infty}$ is isometric to 
 $\mathcal T_{E,\{\xi'_i,\eta'_i\}_{i=1}^\infty}$  for any two  orthonormal sequences 
  $\{\xi_i'\}_{i=1}^\infty$ and $\{\eta_i'\}_{i=1}^\infty$ in $H$.
  Therefore, we denote by $\mathcal T_E$ a quasi-Banach space which is isometrically isomorphic to any of these spaces.  
\end{notation}
An  important projection is the \emph{(lower) triangular projection $T^{E,\{\xi_i,\eta_i\}_{i=1}^\infty}$} on span$\big(\{e_{\xi_i,\eta_j}\}_{i,j=1}^\infty\big)$ defined by
\begin{equation}\label{tria-proj}
T^{E,\{\xi_k,\eta_k\}_{k=1}^\infty}\big(e_{\xi_i,\eta_j}\big)=\left\{
\begin{array}{rcl}
e_{\xi_i,\eta_j}, &  & {1\leq j\leq i<\infty}\\
0\;\;\;, &  & \text{otherwise,}\\
\end{array}
\right.
\end{equation}

The following theorem can be found in \cite[Theorem 4.7]{Arazy1}. 
\begin{theorem}\label{tria}
Suppose that $E$ is a separable Banach symmetric sequence space, and $\{\xi_i\}_{i=1}^\infty$ and $\{\eta_i\}_{i=1}^\infty$ are two orthonormal sequences in $H$. The following are equivalent:
\begin{itemize}
\item [(a)]
$T^{E,\{\xi_i,\eta_i\}_{i=1}^\infty}$ is bounded on the span$\big(\{e_{\xi_i,\eta_j}\}_{i,j=1}^\infty\big)$;
\item [(b)]
$\mathcal T_{E,\{\xi_i,\eta_i\}_{k=1}^\infty}$ is complemented in $\mathcal C_E$;
\item [(c)]
$\mathcal T_E\approx \mathcal C_E$.
\end{itemize}
\end{theorem}

\subsection{Perturbations of bounded operators on quasi-Banach spaces}

\begin{definition}
Let $X$ be a quasi-Banach space. A sequence $\{X_k\}_{k=1}^\infty$ of nonzero closed subspaces of $X$ is called a \emph{Schauder decomposition} of $X$ if every $x\in X$ has a unique representation of the form $x=\sum_{k=1}^\infty x_k$ with $x_k\in X_k$ for every $k$. If $\dim(X_k)=1$ for every $k$, and $x_k\in X_k$ such that $X_k={\rm span}\{x_k\}$ for every $k$, then the sequence $\{x_k\}_{k=1}^\infty$ is called a \emph{Schauder basis} of $X$.
\end{definition}

Every Schauder decomposition $\{X_k\}_{k=1}^\infty$ of a quasi-Banach space $X$ determines a sequence of projections $\{Q_n\}_{n=1}^\infty$ on $X$ by putting $Q_n:\sum_{k=1}^\infty x_k\mapsto x_n$. It is known that these projections are uniformly bounded (see e.g. Theorem \ref{qbssd}). We  denote 
\begin{equation}
K_c:=\sup_n\Vert Q_n\Vert<\infty.
\end{equation}

For example, 
if 
$\{K_i\}_i^\infty$ and $\{L_i\}_i^\infty$ are two sequences of mutually orthogonal closed subspaces of $H$,  
\[
X_k=\Big\{x\in \mathcal C_E:p_{[\bigcup_{i=1}^kK_i]}xp_{[\bigcup_{i=1}^kL_i]}-p_{[\bigcup_{i=1}^{k-1}K_i]}xp_{[\bigcup_{i=1}^{k-1}L_i]}=x\Big\}\;\;\;\;\;\;{\rm for\;every}\;k\geq2
\]
and
\[
X_1=\Big\{x\in\mathcal C_E:p_{_{K_1}}xp_{_{L_1}}=x\Big\},
\]
then $\{X_k\}_{k=1}^\infty$ is   a Schauder decomposition of $[\bigcup_{k=1}^\infty X_k]$.

\begin{lemma}\label{sdper}
Suppose that $X$ is a quasi-Banach space with a Schauder decomposition $\{X_k\}_{k=1}^\infty$, and $Y$ is a quasi-Banach space whose quasi-norm is $r$-subadditive. If $T,S\in\mathcal{B}(X,Y)$, then
\[
\Vert T-S\Vert\leq K_c \bigg(\sum_{k=1}^\infty\big\Vert T|_{X_k}-S|_{X_k}\big\Vert^r\bigg)^{\frac{1}{r}}.   
\]
\end{lemma}

\begin{definition}\label{per-oper}
Let $X$ and $Y$ be two quasi-Banach spaces, and $\{X_k\}_{k=1}^\infty$ be a Schauder decomposition of $X$. Let $T,S\in\mathcal B(X,Y)$, and $\{\varepsilon_k\}_{k=1}^\infty$ be a sequence of nonnegative numbers with $\lim_{k\to\infty}\varepsilon_k=0$. If
\[
\big\Vert T|_{X_k}-S|_{X_k}\big\Vert\leq\varepsilon_k
\]
for every $k$, then   \emph{$S$ is said to be a perturbation of $T$ associated with $\{X_k\}_{k=1}^\infty$ for $\{\varepsilon_k\}_{k=1}^\infty$}.
\end{definition}

\begin{theorem}\label{2d}
	Suppose that $X$ is a quasi-Banach space with a Schauder decomposition $\{X_k\}_{k=1}^\infty$, and $Y$ is a quasi-Banach space whose quasi-norm is $r$-subadditive. If $T\in\mathcal{B}(X,Y)$, $R_k\in\mathcal{B}(X_k,Y)$, $k=1,2,...$ such that
	\begin{equation}\label{2.4}
	K_c \bigg(\sum_{k=1}^\infty\big\Vert T|_{X_k}-R_k\big\Vert^r\bigg)^{\frac{1}{r}}\leq\varepsilon,
	\end{equation}
for some $\varepsilon>0$, then 
  there exists $T_0\in\mathcal{B}(X,Y)$ such that
\begin{equation}\label{2.5} T_0|_{X_k}=R_k\;\;{\rm and\;}\; \Vert T-T_0\Vert\leq\varepsilon.\end{equation}
\end{theorem}
\begin{proof}
	For each $k\in\mathbb N^+$, let $Q_k$ be the natural projection from $X=[  \bigcup _{i=1}^\infty X_i ]$ onto $X_k$ for every $k$.
It suffices to observe that 
$$
\sum_{k=1}^\infty \Vert (T|_{X_k}-R_k)Q_k(x)\Vert^r\stackrel{\eqref{2.4}}{\leq}\varepsilon\Vert x\Vert^r<\infty\;\;\;\;\;\;\mbox{for every }x\in X.
$$
Since $T(x) = T \big(\sum_{k=1}^\infty Q_k(x)\big)=\sum_{k=1}^\infty TQ_k(x) $ for every $x\in X$ and $TQ_k = T|_{X_k} Q_k $, it follows  that the series
\[
\sum_{k=1}^\infty R_kQ_k(x)=\sum_{k=1}^\infty TQ_k(x)-\sum_{k=1}^\infty (T|_{X_k}-R_k)Q_k(x)
\]
is convergent. Hence, the operator
\[
T_0:=\sum_{k=1}^\infty R_kQ_k=T-\sum_{k=1}^\infty (T|_{X_k}-R_k)Q_k
\]
is strongly convergent. 
By \cite[2.8 Theorem]{Rudin3}, $T_0$ is a bounded from $X$ to $Y$, and satisfies \eqref{2.5}. By Lemma \ref{sdper}, we complete the proof.
\end{proof}

Let $\{X_k\}_{k=1}^\infty$ be a Schauder decomposition of the quasi-Banach space $X$.
It is readily verified  that 
if $Z_k$ is a nonzero closed subspace of $X_k$ for every $k$, then   $\{Z_k\}_{k=1}^\infty$ is a Schauder decomposition of $[\bigcup_k^\infty Z_k]$.
Throughout this paper, we always denote by $I_X$ the identity operator on a quasi-Banach space $X$.

\begin{lemma}\label{2d''}
Suppose that
\begin{itemize}
    \item [(1)] $X$ is a quasi-Banach space with a Schauder decomposition $\{X_k\}_{k=1}^\infty$,
        \item [(2)] $Z_k$ is a nonzero complemented subspace of $X_k$ with the projection operator $Q_k:X_k\to Z_k $ for every $k$, and $\sup_{k}\Vert Q_k\Vert\leq M<\infty$,
    \item [(3)] $Y$ is a quasi-Banach space whose quasi-norm is $r$-subadditive, and
    \item [(4)] $T\in\mathcal B(X,Y)$ and $S\in\mathcal B([\cup_{k=1}^\infty Z_k],Y)$ such that $T|_{[\bigcup_k^\infty Z_k]}$ is a perturbation of $S$ associated with $\{Z_k\}_{k=1}^\infty$ for $\{\varepsilon_k\}_{k=1}^\infty$, where $\sum_{k=1}^\infty{\varepsilon_k}^r<\infty$.
\end{itemize}
Then $S$  extends to an operator $\tilde{S}\in\mathcal B(X,Y)$ {\color{red}} such that $\tilde{S}(I_{X_k}-Q_k)=T(I_{X_k}-Q_k)$ for every $k$, and $T$ is the perturbation of $\tilde{S}$ associated with $\{X_k\}_{k=1}^\infty$ for $\{M\varepsilon_k\}_{k=1}^\infty$.
\end{lemma}

\begin{proof}
Put $R_k=S|_{Z_k}Q_k+T(I_{X_k}-Q_k)$ for every $k$. We have
\[
\big\Vert T|_{X_k}-R_k\big\Vert=\big\Vert TQ_k+T(I_{X_k}-Q_k)-R_k\big\Vert=\big\Vert (T|_{Z_k}-S|_{Z_K})Q_k\big\Vert\leq M\varepsilon_k.  
\]
By Theorem \ref{2d},
there exists a bounded operator $\tilde{S}$ such that $\tilde{S}|_{X_k} = R_k$ for every $k$. 
Moreover,
\begin{align*}
\tilde{S}(I_{X_k}-Q_k)&=\tilde{S}|_{X_k}(I_{X_k}-Q_k)=R_k(I_{X_k}-Q_k)\\
&=S|_{Z_k}Q_k(I_{X_k}-Q_k)+T(I_{X_k}-Q_k)(I_{X_k}-Q_k)=T(I_{X_k}-Q_k)
\end{align*}
for every $k$.
\end{proof}

When $Y\subseteq X$,  we denote by $ \iota_{Y\hookrightarrow X}$  the natural embedding from $Y$ to $X$.  

\begin{lemma}\label{comp}
Let $X$ be a quasi-Banach space whose quasi-norm is $r$-subadditive. Suppose that $Y$ is a $\lambda$-complemented ($\lambda \ge 1$)  subspace of $X$ and there is an operator $T:Y\to X$ satisfying $\Vert T-\iota_{Y\hookrightarrow X}\Vert\leq\varepsilon<\lambda^{-1}$. Then, $T$ is an isomorphic embedding and $T(Y)$ is $\lambda(1+(\varepsilon\lambda)^r)^{\frac{1}{r}}(1-(\varepsilon\lambda)^r)^{-\frac{1}{r}}$-complemented in $X$.
\end{lemma}

\begin{proof}
Since $\Vert T-\iota_{Y\hookrightarrow X}\Vert\leq\varepsilon<\lambda^{-1}\leq1$, it follows that
$$\norm{Tx}^r \ge \norm{x}^r - \norm{(T- \iota_{Y\hookrightarrow X}) x }^r \ge (1- \varepsilon^r)\norm{x}^r,~\forall x\in Y.$$
In other words, 
$T$ is an isomorphic embedding from $Y$ into $X$.
Let $P$ be the projection from $X$ onto $Y$ with $\Vert P\Vert\leq\lambda$. Define 
\begin{align}\label{defS}
S=TP+I_{X}-P.
\end{align}
We have 
\[
\Vert S-I_{X} \Vert=\Vert(T-\iota_{Y\hookrightarrow X})P\Vert\leq\varepsilon\lambda<1,
\]
which implies that $S$ is an invertible operator with $\left\Vert S \right\Vert \leq(\Vert I_X\Vert^r+\Vert S-I_X\Vert^r)^{\frac{1}{r}}\leq(1+(\varepsilon\lambda)^r)^{\frac{1}{r}}$ and $\left\Vert S^{-1}\right\Vert=\left\Vert \sum_{k=0}^\infty(I_X-S)^k\right\Vert\leq
\left(\sum_{k=0}^\infty\Vert(I_X-S)^k\Vert^r\right)^{\frac{1}{r}}
\le \left(\sum_{k=0}^\infty (\varepsilon \lambda )^{rk} \right)^{\frac{1}{r}}
= (1-(\varepsilon\lambda)^r)^{-\frac{1}{r}}$. Noting that 
\[
(SPS^{-1})^2=SPS^{-1}
\]
and \[ 
 {\rm im}(SPS^{-1})={\rm im}(SP)\stackrel{\eqref{defS}}{=} {\rm im}((TP+(I_X-P))P)  = {\rm im}(TP)={\rm im}(T),
\]
where im($T$) denotes the range of the operator $T$, 
 we obtain that  $SPS^{-1}$ is the projection from $X$ onto $T(Y)$ with $\Vert SPS^{-1}\Vert\leq\lambda(1+(\varepsilon\lambda)^r)^{\frac{1}{r}}(1-(\varepsilon\lambda)^r)^{-\frac{1}{r}}$.
\end{proof}

\section{Block sequences in $\mathcal B(H)$}\label{blocksequences}

Let $\{K_i\}_i$ be a finite or infinite sequence of mutually orthogonal closed subspaces of $H$. Suppose that $\big\{u_i:K_i\to H'\big\}_i$ is a sequence of isometries. 
For each $i$,
we define an isometry from $K_i$ into $H'\otimes_{_2}\ell_2$ by
\begin{equation}\label{ui}
\wideparen{u_i}:\xi\in K_i\mapsto u_i(\xi)\otimes e_i\in H'\otimes_{_2}\ell_2,
\end{equation}
where $\{e_i\}_{i=1}^\infty$ stands for the natural basis of $\ell_2$.
 Then there is a unique isometry $\wideparen{u}$ from $[\bigcup_iK_i]$ into $H'\otimes_{_2}\ell_2$ given by
\begin{equation}\label{u}
\wideparen{u}:\sum_i\xi_i\mapsto\sum_iu_i(\xi_i)\otimes e_i,
\end{equation}
where $\xi_i\in K_i$ for every $i$, such that $\wideparen{u}|_{K_i}=\wideparen{u_i}$ for every $i$.

Let $\{K_i\}_i$ and $\{L_j\}_j$ be two  finite or infinite sequences of mutually orthogonal closed subspaces of $H$. Let $G$ be a nonempty subset of $\mathbb N\times\mathbb N$.

\begin{definition}
A sequence 
$\{x_{i,j}\}_{(i,j)\in G}$   in $\mathcal B(H)$ is called a \emph{block sequence associated with $\{K_i\}_i$ and $\{L_j\}_j$} if $x_{i,j}=p_{_{K_i}}x_{i,j}p_{_{L_j}}$ for all $(i,j)\in G$, and is denoted by
\[
\{x_{i,j}\}_{(i,j)\in G}\stackrel{\mbox{\tiny b.s}}{\boxplus}\{K_i\}_i\otimes\{L_j\}_j.
\]
\end{definition}
In particular, when $\{K_i\}_i$ \big(resp. $\{L_j\}_j$\big) contains only one element $K'$ \big(resp. $L'$\big), we may write $\{x_{i,j}\}_{(i,j)\in G}$ as $\{x_j\}_j$  \big(resp. $\{x_i\}_i$\big) and
\[
\{x_j\}_j\stackrel{\mbox{\tiny b.s}}{\boxplus}K'\otimes\{L_j\}_j\;\;\;\;\big({\rm resp.}\;\{x_i\}_i\stackrel{\mbox{\tiny b.s}}{\boxplus}\{K_i\}_i\otimes L'\big).
\]

\begin{definition}
Let $\{x_{i,j}\}_{(i,j)\in G}$ be a block sequence associated with $\{K_i\}_i$ and $\{L_j\}_j$ in $\mathcal B(H)$, and $\big\{u_i:K_i\to H'\big\}_i$ and $\big\{v_j:L_j\to H'\big\}_j$ be the two sequences of isometries. The sequence $\{x_{i,j}\}_{(i,j)\in G}$ is said to be  \emph{generated by the  system $\{u_i:K_i\to H'\}_i\otimes\{v_j:L_j\to, H'\}_j$ of isometries  and an operator} if there exist an operator $a\in\mathcal B(H')$ such that
\[
u_ix_{i,j}{v_j}^\ast=a
\]
for every $(i,j)\in G$, and is denoted by
\[
\{x_{i,j}\}_{(i,j)\in G}\stackrel{\mbox{\tiny b.s}}{\boxplus}\{u_i:K_i\to H'\}_i\otimes\{v_j:L_j\to H'\}_j\curvearrowleft a.
\] 
If one only   emphasizes the existence of such an isometries system, then the sequence $\{x_{i,j}\}_{(i,j)\in G}$ is said to be  \emph{generated by an operator} $a\in \cB(H')$, and is denoted by
\[
\{x_{i,j}\}_{(i,j)\in G}\stackrel{\mbox{\tiny b.s}}{\boxplus}\{K_i\}_i\otimes\{L_j\}_j\curvearrowleft a
\]
or
\[
\{x_{i,j}\}_{(i,j)\in G}\stackrel{\mbox{\tiny b.s}}{\boxplus}\{K_i\}_i\otimes\{L_j\}_j\curvearrowleft\;
\]
if the  choice of $a$ is irrelevant. 
\end{definition}

\begin{example}
    For the case $G=\{(i,j)\in\mathbb N\times\mathbb N:j\leq i\}$, the following figure is helpful in understanding  the role  of $a\in \cB(H')$ in the  representation of the form   
    \[
\{x_{i,j}\}_{(i,j)\in G}\stackrel{\mbox{\tiny b.s}}{\boxplus}\{K_i\}_i\otimes\{L_j\}_j\curvearrowleft a.
\]

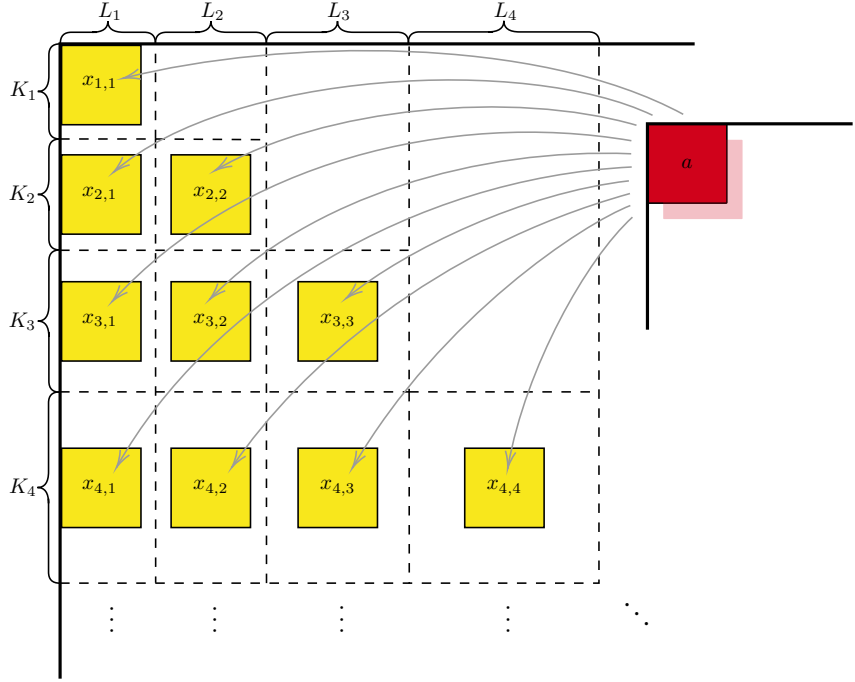
\begin{center}

\scalebox{0.8}{

\tikzset{every picture/.style={line width=0.75pt}} 

\begin{tikzpicture}[x=0.75pt,y=0.75pt,yscale=-1,xscale=1]

\draw  [line width=1.5]  (480.22,50.06) -- (80.17,50.06) -- (80.17,450.55) ;
\draw  [dash pattern={on 4.5pt off 4.5pt}] (80.17,110.06) -- (140.33,110.06) -- (140.33,50.06) ;
\draw  [dash pattern={on 4.5pt off 4.5pt}] (80.17,180.27) -- (210.21,180.27) -- (210.21,50.06) ;
\draw  [dash pattern={on 4.5pt off 4.5pt}] (80.17,269.61) -- (300.21,269.61) -- (300.21,50.06) ;
\draw  [dash pattern={on 4.5pt off 4.5pt}] (80.17,390.28) -- (419.89,390.28) -- (419.89,50.06) ;
\draw  [dash pattern={on 4.5pt off 4.5pt}]  (140.33,110.06) -- (140.42,390.08) ;
\draw  [dash pattern={on 4.5pt off 4.5pt}]  (210.21,180.27) -- (210.3,390.29) ;
\draw  [dash pattern={on 4.5pt off 4.5pt}]  (300.21,269.61) -- (300.3,390.62) ;
\draw   (80.11,50.06) .. controls (75.44,50.06) and (73.11,52.39) .. (73.11,57.06) -- (73.11,70.06) .. controls (73.11,76.73) and (70.78,80.06) .. (66.11,80.06) .. controls (70.78,80.06) and (73.11,83.39) .. (73.11,90.06)(73.11,87.06) -- (73.11,103.06) .. controls (73.11,107.73) and (75.44,110.06) .. (80.11,110.06) ;
\draw   (80.11,110.06) .. controls (75.44,110.06) and (73.11,112.39) .. (73.11,117.06) -- (73.11,135.06) .. controls (73.11,141.73) and (70.78,145.06) .. (66.11,145.06) .. controls (70.78,145.06) and (73.11,148.39) .. (73.11,155.06)(73.11,152.06) -- (73.11,173.06) .. controls (73.11,177.73) and (75.44,180.06) .. (80.11,180.06) ;
\draw   (80.11,180.06) .. controls (75.44,180.06) and (73.11,182.39) .. (73.11,187.06) -- (73.11,214.95) .. controls (73.11,221.62) and (70.78,224.95) .. (66.11,224.95) .. controls (70.78,224.95) and (73.11,228.28) .. (73.11,234.95)(73.11,231.95) -- (73.11,262.84) .. controls (73.11,267.51) and (75.44,269.84) .. (80.11,269.84) ;
\draw   (79.89,269.62) .. controls (75.22,269.62) and (72.89,271.95) .. (72.89,276.62) -- (72.89,319.95) .. controls (72.89,326.62) and (70.56,329.95) .. (65.89,329.95) .. controls (70.56,329.95) and (72.89,333.28) .. (72.89,339.95)(72.89,336.95) -- (72.89,383.28) .. controls (72.89,387.95) and (75.22,390.28) .. (79.89,390.28) ;
\draw   (140.27,49.73) .. controls (140.27,45.06) and (137.94,42.73) .. (133.27,42.73) -- (120.3,42.73) .. controls (113.63,42.73) and (110.3,40.4) .. (110.3,35.73) .. controls (110.3,40.4) and (106.97,42.73) .. (100.3,42.73)(103.3,42.73) -- (87.33,42.73) .. controls (82.66,42.73) and (80.33,45.06) .. (80.33,49.73) ;
\draw   (210.2,49.73) .. controls (210.2,45.06) and (207.87,42.73) .. (203.2,42.73) -- (185.24,42.73) .. controls (178.57,42.73) and (175.24,40.4) .. (175.24,35.73) .. controls (175.24,40.4) and (171.91,42.73) .. (165.24,42.73)(168.24,42.73) -- (147.27,42.73) .. controls (142.6,42.73) and (140.27,45.06) .. (140.27,49.73) ;
\draw   (299.89,49.73) .. controls (299.89,45.06) and (297.56,42.73) .. (292.89,42.73) -- (265.05,42.73) .. controls (258.38,42.73) and (255.05,40.4) .. (255.05,35.73) .. controls (255.05,40.4) and (251.72,42.73) .. (245.05,42.73)(248.05,42.73) -- (217.2,42.73) .. controls (212.53,42.73) and (210.2,45.06) .. (210.2,49.73) ;
\draw   (420.21,49.51) .. controls (420.21,44.84) and (417.88,42.51) .. (413.21,42.51) -- (369.94,42.51) .. controls (363.27,42.51) and (359.94,40.18) .. (359.94,35.51) .. controls (359.94,40.18) and (356.61,42.51) .. (349.94,42.51)(352.94,42.51) -- (306.67,42.51) .. controls (302,42.51) and (299.67,44.84) .. (299.67,49.51) ;
\draw  [fill={rgb, 255:red, 248; green, 231; blue, 28 }  ,fill opacity=1 ] (81.11,51.01) -- (131.16,51.01) -- (131.16,101.06) -- (81.11,101.06) -- cycle ;
\draw  [dash pattern={on 4.5pt off 4.5pt}]  (140.33,110.06) -- (210.33,110.51) ;
\draw  [dash pattern={on 4.5pt off 4.5pt}]  (210.21,180.27) -- (299.89,180.28) ;
\draw  [dash pattern={on 4.5pt off 4.5pt}]  (300.21,269.61) -- (419.89,269.62) ;
\draw  [fill={rgb, 255:red, 248; green, 231; blue, 28 }  ,fill opacity=1 ] (81.11,120.06) -- (131.16,120.06) -- (131.16,170.11) -- (81.11,170.11) -- cycle ;
\draw  [fill={rgb, 255:red, 248; green, 231; blue, 28 }  ,fill opacity=1 ] (81.11,200.11) -- (131.16,200.11) -- (131.16,250.16) -- (81.11,250.16) -- cycle ;
\draw  [fill={rgb, 255:red, 248; green, 231; blue, 28 }  ,fill opacity=1 ] (81.11,305.16) -- (131.16,305.16) -- (131.16,355.21) -- (81.11,355.21) -- cycle ;
\draw  [fill={rgb, 255:red, 248; green, 231; blue, 28 }  ,fill opacity=1 ] (150.16,120.06) -- (200.21,120.06) -- (200.21,170.11) -- (150.16,170.11) -- cycle ;
\draw  [fill={rgb, 255:red, 248; green, 231; blue, 28 }  ,fill opacity=1 ] (150.16,200.11) -- (200.21,200.11) -- (200.21,250.16) -- (150.16,250.16) -- cycle ;
\draw  [fill={rgb, 255:red, 248; green, 231; blue, 28 }  ,fill opacity=1 ] (230.21,200.11) -- (280.26,200.11) -- (280.26,250.16) -- (230.21,250.16) -- cycle ;
\draw  [fill={rgb, 255:red, 248; green, 231; blue, 28 }  ,fill opacity=1 ] (150.16,305.16) -- (200.21,305.16) -- (200.21,355.21) -- (150.16,355.21) -- cycle ;
\draw  [fill={rgb, 255:red, 248; green, 231; blue, 28 }  ,fill opacity=1 ] (230.21,305.16) -- (280.26,305.16) -- (280.26,355.21) -- (230.21,355.21) -- cycle ;
\draw  [fill={rgb, 255:red, 248; green, 231; blue, 28 }  ,fill opacity=1 ] (335.26,305.16) -- (385.31,305.16) -- (385.31,355.21) -- (335.26,355.21) -- cycle ;
\draw  [fill={rgb, 255:red, 208; green, 2; blue, 27 }  ,fill opacity=1 ] (450.46,100.99) .. controls (450.46,100.71) and (450.69,100.49) .. (450.96,100.49) -- (500.01,100.49) .. controls (500.29,100.49) and (500.51,100.71) .. (500.51,100.99) -- (500.51,150.04) .. controls (500.51,150.32) and (500.29,150.54) .. (500.01,150.54) -- (450.96,150.54) .. controls (450.69,150.54) and (450.46,150.32) .. (450.46,150.04) -- cycle ;
\draw  [line width=1.5]  (579.81,100.49) -- (450.46,100.49) -- (450.46,230.37) ;
\draw  [draw opacity=0][fill={rgb, 255:red, 208; green, 2; blue, 27 }  ,fill opacity=0.25 ] (460.38,111.06) .. controls (460.38,110.78) and (460.61,110.56) .. (460.88,110.56) -- (509.93,110.56) .. controls (510.21,110.56) and (510.43,110.78) .. (510.43,111.06) -- (510.43,160.11) .. controls (510.43,160.38) and (510.21,160.61) .. (509.93,160.61) -- (460.88,160.61) .. controls (460.61,160.61) and (460.38,160.38) .. (460.38,160.11) -- cycle ;
\draw [color={rgb, 255:red, 155; green, 155; blue, 155 }  ,draw opacity=1 ]   (473,94.49) .. controls (380.23,44.12) and (219.38,46.84) .. (121.14,71.45) ;
\draw [shift={(119.67,71.82)}, rotate = 345.73] [color={rgb, 255:red, 155; green, 155; blue, 155 }  ,draw opacity=1 ][line width=0.75]    (10.93,-3.29) .. controls (6.95,-1.4) and (3.31,-0.3) .. (0,0) .. controls (3.31,0.3) and (6.95,1.4) .. (10.93,3.29)   ;
\draw [color={rgb, 255:red, 155; green, 155; blue, 155 }  ,draw opacity=1 ]   (443.54,101.4) .. controls (366.06,78.84) and (245.55,89.95) .. (178.67,129.46) ;
\draw [shift={(177.67,130.06)}, rotate = 329.04] [color={rgb, 255:red, 155; green, 155; blue, 155 }  ,draw opacity=1 ][line width=0.75]    (10.93,-3.29) .. controls (6.95,-1.4) and (3.31,-0.3) .. (0,0) .. controls (3.31,0.3) and (6.95,1.4) .. (10.93,3.29)   ;
\draw [color={rgb, 255:red, 155; green, 155; blue, 155 }  ,draw opacity=1 ]   (453.66,95.1) .. controls (374.07,56.92) and (187.54,67.29) .. (114.76,130.2) ;
\draw [shift={(113.67,131.16)}, rotate = 318.47] [color={rgb, 255:red, 155; green, 155; blue, 155 }  ,draw opacity=1 ][line width=0.75]    (10.93,-3.29) .. controls (6.95,-1.4) and (3.31,-0.3) .. (0,0) .. controls (3.31,0.3) and (6.95,1.4) .. (10.93,3.29)   ;
\draw [color={rgb, 255:red, 155; green, 155; blue, 155 }  ,draw opacity=1 ]   (440.54,127.73) .. controls (301.7,134.69) and (173.62,214.59) .. (116.52,315.87) ;
\draw [shift={(115.67,317.4)}, rotate = 299.05] [color={rgb, 255:red, 155; green, 155; blue, 155 }  ,draw opacity=1 ][line width=0.75]    (10.93,-3.29) .. controls (6.95,-1.4) and (3.31,-0.3) .. (0,0) .. controls (3.31,0.3) and (6.95,1.4) .. (10.93,3.29)   ;
\draw [color={rgb, 255:red, 155; green, 155; blue, 155 }  ,draw opacity=1 ]   (439.43,144.48) .. controls (316.05,178.97) and (231.06,254.87) .. (188.96,316.47) ;
\draw [shift={(188.33,317.4)}, rotate = 304.14] [color={rgb, 255:red, 155; green, 155; blue, 155 }  ,draw opacity=1 ][line width=0.75]    (10.93,-3.29) .. controls (6.95,-1.4) and (3.31,-0.3) .. (0,0) .. controls (3.31,0.3) and (6.95,1.4) .. (10.93,3.29)   ;
\draw [color={rgb, 255:red, 155; green, 155; blue, 155 }  ,draw opacity=1 ]   (440.33,111.16) .. controls (324.91,91.49) and (192.33,125.72) .. (114.05,211.43) ;
\draw [shift={(112.87,212.73)}, rotate = 312.03] [color={rgb, 255:red, 155; green, 155; blue, 155 }  ,draw opacity=1 ][line width=0.75]    (10.93,-3.29) .. controls (6.95,-1.4) and (3.31,-0.3) .. (0,0) .. controls (3.31,0.3) and (6.95,1.4) .. (10.93,3.29)   ;
\draw [color={rgb, 255:red, 155; green, 155; blue, 155 }  ,draw opacity=1 ]   (441,159.4) .. controls (411.96,185.8) and (373.77,254.25) .. (362.66,316.19) ;
\draw [shift={(362.33,318.06)}, rotate = 279.7] [color={rgb, 255:red, 155; green, 155; blue, 155 }  ,draw opacity=1 ][line width=0.75]    (10.93,-3.29) .. controls (6.95,-1.4) and (3.31,-0.3) .. (0,0) .. controls (3.31,0.3) and (6.95,1.4) .. (10.93,3.29)   ;
\draw [color={rgb, 255:red, 155; green, 155; blue, 155 }  ,draw opacity=1 ]   (439.67,152.06) .. controls (410.48,163.34) and (320.57,221.9) .. (265.17,316.4) ;
\draw [shift={(264.33,317.82)}, rotate = 300.13] [color={rgb, 255:red, 155; green, 155; blue, 155 }  ,draw opacity=1 ][line width=0.75]    (10.93,-3.29) .. controls (6.95,-1.4) and (3.31,-0.3) .. (0,0) .. controls (3.31,0.3) and (6.95,1.4) .. (10.93,3.29)   ;
\draw [color={rgb, 255:red, 155; green, 155; blue, 155 }  ,draw opacity=1 ]   (440.21,119.4) .. controls (343.49,115.42) and (238.06,146.42) .. (175.28,210.43) ;
\draw [shift={(174.33,211.4)}, rotate = 314.1] [color={rgb, 255:red, 155; green, 155; blue, 155 }  ,draw opacity=1 ][line width=0.75]    (10.93,-3.29) .. controls (6.95,-1.4) and (3.31,-0.3) .. (0,0) .. controls (3.31,0.3) and (6.95,1.4) .. (10.93,3.29)   ;
\draw [color={rgb, 255:red, 155; green, 155; blue, 155 }  ,draw opacity=1 ]   (438.76,136.48) .. controls (372.63,143.98) and (298.39,179.75) .. (261.44,209.82) ;
\draw [shift={(260.33,210.73)}, rotate = 320.39] [color={rgb, 255:red, 155; green, 155; blue, 155 }  ,draw opacity=1 ][line width=0.75]    (10.93,-3.29) .. controls (6.95,-1.4) and (3.31,-0.3) .. (0,0) .. controls (3.31,0.3) and (6.95,1.4) .. (10.93,3.29)   ;

\draw (57.08,79.67) node    {$K_{1}$};
\draw (56.41,145) node    {$K_{2}$};
\draw (56.41,225.33) node    {$K_{3}$};
\draw (56.75,330) node    {$K_{4}$};
\draw (111.75,29.67) node    {$L_{1}$};
\draw (176.41,29.33) node    {$L_{2}$};
\draw (256.08,29.33) node    {$L_{3}$};
\draw (361.08,29.67) node    {$L_{4}$};
\draw (105.14,75.03) node    {$x_{1,1}$};
\draw (105.14,145.08) node    {$x_{2,1}$};
\draw (175.19,145.08) node    {$x_{2,2}$};
\draw (105.14,330.18) node    {$x_{4,1}$};
\draw (105.14,225.13) node    {$x_{3,1}$};
\draw (175.19,225.13) node    {$x_{3,2}$};
\draw (175.19,330.18) node    {$x_{4,2}$};
\draw (255.24,225.13) node    {$x_{3,3}$};
\draw (255.24,330.18) node    {$x_{4,3}$};
\draw (360.29,330.18) node    {$x_{4,4}$};
\draw (475.49,125.52) node    {$a$};
\draw (444.59,411.33) node  [font=\LARGE,rotate=-45]  {$\cdots $};
\draw (112.26,414.67) node  [font=\Large,rotate=-90]  {$\cdots $};
\draw (177.26,414.67) node  [font=\Large,rotate=-90]  {$\cdots $};
\draw (257.26,414.33) node  [font=\Large,rotate=-90]  {$\cdots $};
\draw (362.26,414.33) node  [font=\Large,rotate=-90]  {$\cdots $};

\end{tikzpicture}
}

\end{center}
\vspace{0.1cm}
 In the figure above, the yellow squares correspond to the restriction of the operator $x_{i,j}$ to its support, while the red square correspond to the restriction of the operator $a$ to its support.

\end{example}

\begin{remark}\label{rem}
\noindent
\begin{itemize}
\item [(1)]
Suppose that $\{x_{i,j}\}_{(i,j)\in G}$ is a sequence in $\mathcal B(H)$, and $\{\xi'_i\}_{i=1}^\infty$ and $\{\eta'_i\}_{i=1}^\infty$ are two sequences of orthonormal vectors of $H'$. It is easy to see that
\[
\big\{x_{i,j}\otimes e_{\xi'_i,\eta'_j}\big\}_{(i,j)\in G}\stackrel{\mbox{\tiny b.s}}{\boxplus}\{H\otimes_{_2}\xi'_i\}_i\otimes\{H\otimes_{_2}\eta'_j\}_j.
\]
Suppose that $b\in\mathcal B(H')$ and
\[
\{x_{i,j}\}_{(i,j)\in G}\stackrel{\mbox{\tiny b.s}}{\boxplus}\{K_i\}_i\otimes\{L_j\}_j.
\]
We have 
\[
\{x_{i,j}\otimes b\}_{(i,j)\in G}\stackrel{\mbox{\tiny b.s}}{\boxplus}\{K_i\otimes_{_2} H'\}_i\otimes\{L_j\otimes_{_2} H'\}_j
\]
and
\[
\{b\otimes x_{i,j}\}_{(i,j)\in G}\stackrel{\mbox{\tiny b.s}}{\boxplus}\{H'\otimes_{_{2}} K_i\}_i\otimes\{H'\otimes_{_{2}} L_j\}_j.
\]

\item [(2)]
Assume that
\[
\{x_{i,j}\}_{(i,j)\in G}\stackrel{\mbox{\tiny b.s}}{\boxplus}\{K_i\}_i\otimes\{L_j\}_j.
\]
Suppose that $\{u_i\}_i$ is a sequence of isometries from $\{K_i\}_i$ into $H'$ and $\{v_j\}_j$ is a sequence of isometries from $\{L_j\}_j$ into $H'$,
and $\wideparen{u}$ is the  isometry  given in \eqref{u}, and similarly, $\wideparen{v_j}$ and $\wideparen{v}$ are given by
\begin{equation}\label{vj}
\wideparen{v_j}:\eta\in L_j\mapsto v_j(\eta)\otimes e_j\in H'\otimes_{_2}\ell_2,
\end{equation}
and
\[
\wideparen{v}:\sum_j\eta_j\mapsto\sum_jv_j(\eta_j)\otimes e_j,
\]
where $\eta_j\in L_i$ for every $j$, such that $\wideparen{v}|_{L_j}=\wideparen{v_j}$ for every $j$.
There exists
a linear embedding of  span$\big(\{x_{i,j}\}_{(i,j)\in G}\big)$ into $\mathcal B(H'\otimes_{_2}\ell_2)$ given by
\[
x\in[x_{i,j}]_{(i,j)\in G}\mapsto \wideparen{u}x{\wideparen{v}}^\ast.
\]
Note that $\wideparen{u}\ell(x)=\wideparen{u}$ and $r(x){\wideparen{v}}^\ast={\wideparen{v}}^\ast$, where $\ell(x)$ (resp. $r(x)$) stands for 
  the left (resp. right) support of $x$.

For any sequence $\{\alpha_{i,j}\}_{(i,j)\in G}$ of scalars with finitely many nonzero elements, we have
\begin{align*}
\wideparen{u}\bigg(\sum_{(i,j)\in G}\alpha_{i,j}x_{i,j}\bigg){\wideparen{v}}^\ast&\quad =\quad \sum_{(i,j)\in G}\alpha_{i,j}\wideparen{u}x_{i,j}{\wideparen{v}}^\ast\\
&\quad =\quad\sum_{(i,j)\in G}\alpha_{i,j}\wideparen{u_i}x_{i,j}{\wideparen{v_j}}^\ast\\
&\stackrel{\eqref{ui},\eqref{vj}}{=}\sum_{(i,j)\in G}\alpha_{i,j}(u_ix_{i,j}{v_j}^\ast)\otimes(\cdot|e_j)e_i.
\end{align*}

In particular, if
\[
\{x_{i,j}\}_{(i,j)\in G}\stackrel{\mbox{\tiny b.s}}{\boxplus}\{u_i:K_i\to H'\}_i\otimes\{v_j:L_j\to H'\}_j\curvearrowleft a,
\]
then
\[
\wideparen{u}\bigg(\sum_{(i,j)\in G}\alpha_{i,j}x_{i,j}\bigg){\wideparen{v}}^\ast=\sum_{(i,j)\in G}\alpha_{i,j}\;a\otimes(\cdot|e_j)e_i=a\otimes c,
\]
where $c=\sum_{(i,j)\in G}\alpha_{i,j}(\cdot|e_j)e_i\in\mathcal B(\ell_2)$. Since $c$ is finite-rank, it follows that
there exist two orthonormal sequences $\{f_l\}_l$ and $\{g_l\}_l$ in $\ell_2$ such that $c=\sum_ls_l(c)(\cdot|g_l)f_l$. Then
\[
\wideparen{u}\bigg(\sum_{(i,j)\in G}\alpha_{i,j}x_{i,j}\bigg){\wideparen{v}}^\ast=a\otimes c=\sum_ls_l(c)\cdot a\otimes(\cdot|g_l)f_l.
\]
Consequently, if $a$ is a compact operator, then there are two orthonormal sequences $\{\xi'_k\}_k$ and $\{\eta'_k\}_k$ in $H'$ such that $a=\sum_ks_k(a)(\cdot|\eta'_k)\xi'_k$. Hence,
\[
\wideparen{u}\bigg(\sum_{(i,j)\in G}\alpha_{i,j}x_{i,j}\bigg){\wideparen{v}}^\ast=a\otimes c=\sum_{k,l} s_k(a)s_l(c)\cdot\big( \cdot |\eta_k\otimes g_l\big)\xi_k\otimes f_l.
\]
This  implies that the singular value sequence of $\sum_{(i,j)\in G}\alpha_{i,j}x_{i,j}$ is the decreasing rearrangement of 
$\big\{s_k(a)s_l(c)\big\}_{k,l}$.
\item [(3)]
 Assume that 
\[
\{x_{i,j}\}_{(i,j)\in G}\stackrel{\mbox{\tiny b.s}}{\boxplus}\{K_i\}_i\otimes\{L_j\}_j\curvearrowleft a.
\]
If there are other sequences  $\{K'_i\}_i$ and $\{L'_j\}_j$ of mutually orthogonal closed subspaces of $H$ such that
\[
\{x_{i,j}\}_{(i,j)\in G}\stackrel{\mbox{\tiny b.s}}{\boxplus}\{K'_i\}_i\otimes\{L'_j\}_j,
\]
then
\[
\{x_{i,j}\}_{(i,j)\in G}\stackrel{\mbox{\tiny b.s}}{\boxplus}\{K'_i\}_i\otimes\{L'_j\}_j\curvearrowleft a.
\]
Indeed, note that $\ell(x_{i,j})\subset K_i\cap K'_i$ and $r(x_{i,j})\subset L_j\cap L'_j$ for every $(i,j)\in G$, and thus, we have  
\[
\{x_{i,j}\}_{(i,j)\in G}\stackrel{\mbox{\tiny b.s}}{\boxplus}\{K_i\cap K'_i\}_i\otimes\{L_j\cap L'_j\}_j\curvearrowleft a.
\]
One sees that  
$\{x_{i,j}\}_{(i,j)\in G}$ is the block sequence which is generated by  $a$ associated with $\{K'_i\}_i$ and $\{L'_j\}_j$.
\end{itemize}
\end{remark}

Below,   for a  given  block sequence generated by a single operator, we provide a criteria for its ``block subsequence'' to be generated by another operator.

\begin{theorem}\label{3a'}
Let $G$ and $U$ be two nonempty sets of $\mathbb N\times\mathbb N$ such that  there are two sequences $\{I_i\}_i$ and $\{J_j\}_j$ of nonempty disjoint finite sets of $\mathbb N$ such that
\[
I_i\times J_j\subset G\;\;\;\;\;\;\mbox{for\;every}\;(i,j)\in U.
\]
Assume that $\{x_{i',j'}\}_{(i',j')\in G}$ is a sequence in $\mathcal B(H)$ with
\[
\{x_{i',j'}\}_{(i',j')\in G}\stackrel{\mbox{\tiny{\rm b.s}}}{\boxplus}\{K_{i'}\}_{i'}\otimes\{L_{j'}\}_{j'}\curvearrowleft a \in \cB({H'}).
\]
If there is a sequence $\left\{{\sum}_{(i',j')\in I_i\times J_j}\alpha^{(i,j)}_{i',j'}(\cdot|e_{j'})e_{i'}\right\}_{(i,j)\in U}$ in $\mathcal B(\ell_2)$ with
\[
\left\{{\sum}_{(i',j')\in I_i\times J_j}\alpha^{(i,j)}_{i',j'}(\cdot|e_{j'})e_{i'}\right\}_{(i,j)\in U}\stackrel{\mbox{\tiny{\rm  b.s}}}{\boxplus}\big\{[e_{i'}]_{i'\in I_i}\big\}_i\otimes\big\{[e_{j'}]_{j'\in J_j}\big\}_j\curvearrowleft c \in \cB(H'') ,
\]
then
\[
\left\{{\sum}_{(i',j')\in I_i\times J_j}\alpha^{(i,j)}_{i',j'}x_{i',j'}\right\}_{(i,j)\in U}\stackrel{\mbox{\tiny{\rm b.s}}}{\boxplus}\big\{[\cup_{i'\in I_i}K_{i'}]\big\}_i\otimes\big\{[\cup_{j'\in J_j}L_{j'}]\big\}_j\curvearrowleft a\otimes c.
\]

\end{theorem}
\begin{proof}
Assume that  $a\in\mathcal B(H')$ and $c\in\mathcal B(H'')$, and assume that
\[
\{x_{i',j'}\}_{(i',j')\in G}\stackrel{\mbox{\tiny b.s}}{\boxplus}\{u_{i'}:K_{i'}\to H'\}_{i'}\otimes\{v_{j'}:L_{j'}\to H'\}_{j'}\curvearrowleft a
\]
and
\begin{align*}
&\left\{{\sum}_{(i',j')\in I_i\times J_j}\alpha^{(i,j)}_{i',j'}(\cdot|e_{j'})e_{i'}\right\}_{(i,j)\in U}\\
\stackrel{\mbox{\tiny b.s}}{\boxplus}&\big\{u'_i:{\sum}_{i'\in I_i}\mathbb C e_{i'}\to H''\big\}_i\otimes\big\{v'_j:\sum_{j'\in J_j}\mathbb Ce_{j'}\to H''\big\}_j\curvearrowleft c.
\end{align*}

Put
\[
K'_i=[\cup_{i'\in I_i}K_{i'}]\;\;\;\;\;\;{\rm and}\;\;\;\;\;\;L'_j=[\cup_{j'\in I_j}L_{j'}],
\]
We set
\[
u''_i:\sum_{i'\in I_i}\xi_{i'}\in K'_i=\sum_{i'\in I_i}K_{i'}\longmapsto\sum_{i'\in I_i}u_{i'}(\xi_{i'})\otimes u'_i(e_{i'})
\]
and
\[
v''_j:\sum_{j'\in I_j}\xi_{j'}\in L'_j=\sum_{j'\in J_j}L_{j'}\longmapsto\sum_{i'\in I_i} v_{j'}(\xi_{j'})\otimes v'_j(e_{j'}).
\]
Thus, 
\begin{align*}
u''_i\bigg(\sum_{(i',j')\in I_i\times J_j}\alpha^{(i,j)}_{i',j'}x_{i',j'}\bigg){v''_j}^\ast&=\sum_{(i',j')\in I_i\times J_j}\alpha^{(i,j)}_{i',j'}\big(u_{i'}x_{i',j'}{v_{j'}}^\ast\big)\otimes\big({v'_j}^\ast(\cdot)\big|e_{j'}\big)u'_i(e_{i'})\\\notag
&=a\otimes\sum_{(i',j')\in I_i\times J_j}\alpha^{(i,j)}_{i',j'}({v'_j}^\ast(\cdot)\big|e_{j'}\big)u'_i(e_{i'})\\\notag
&=a\otimes u'_i\bigg(\sum_{(i',j')\in I_i\times J_j}\alpha^{(i,j)}_{i',j'}(\cdot|e_{j'})e_{i'}\bigg){v_j'}^\ast\\\notag
&=a\otimes c.\notag
\end{align*}
Therefore,
\begin{align*}
&\left\{{\sum}_{(i',j')\in I_i\times J_j}\alpha^{(i,j)}_{i',j'}x_{i',j'}\right\}_{(i,j)\in U}\\
\stackrel{\mbox{\tiny b.s}}{\boxplus}&\big\{u''_i:[\cup_{i'\in I_i}K_{i'}]\to H'\otimes_{_2} H''\big\}_i\otimes\big\{v''_j:[\cup_{j'\in J_j}L_{j'}]\to H'\otimes_{_2} H''\big\}_j\curvearrowleft a\otimes c.
\end{align*}
\end{proof}

\begin{remark}\label{rembl}
The following consequence of 
  Theorem \ref{3a'} will be frequently used  throughout this paper.
If there are two sequences $\{(\alpha^{(i)}_{i'})_{i'\in I_i}\big\}_i$ and $\big\{(\beta^{(j)}_{j'})_{j'\in J_j}\big\}_j$ of sequences of scalars such that
\[
\norm{(\alpha^{(i)}_{i'})_{i'\in I_i}}_{\ell_2}=\alpha\;\;\;\;\mbox{and}\;\;\;\;\norm{(\beta^{(j)}_{j'})_{j'\in J_j}}_{\ell_2}=\beta\;\;\;\;\;\;\mbox{for\;every}\;(i,j)\in U.
\]
and the matrix
\[
\left(\alpha^{(i,j)}_{i',j'}\right)_{(i',j')\in I_i\times J_j}=\left(\alpha^{(i)}_{i'}\beta^{(j)}_{j'}\right)_{(i',j')\in I_i\times J_j}
\]
for every $(i,j)\in U$, then we have
\[
\left\{{\sum}_{(i',j')\in I_i\times J_j}\alpha^{(i,j)}_{i',j'}(\cdot|e_{j'})e_{i'}\right\}_{(i,j)\in U}\stackrel{\mbox{\tiny b.s}}{\boxplus}\big\{[e_{i'}]_{i'\in I_i}\big\}_i\otimes\big\{[e_{j'}]_{j'\in J_j}\big\}_j\curvearrowleft\alpha\beta(\cdot|e_1)e_1
\]
Hence,
\[
\left\{{\sum}_{(i',j')\in I_i\times J_j}\alpha^{(i,j)}_{i',j'}x_{i',j'}\right\}_{(i,j)\in U}\stackrel{\mbox{\tiny b.s}}{\boxplus}\big\{[\cup_{i'\in I_i}K_{i'}]\big\}_i\otimes\big\{[\cup_{j'\in J_j}L_{j'}]\big\}_j\curvearrowleft\alpha\beta a\otimes (\cdot|e_1)e_1,
\]
where $a$ is the operator given in   Theorem \ref{3a'} above.
\end{remark}

\begin{definition}\label{consistent}
Let $\{G_n\}_{n\geq 1}$ be a sequence (finite or infinite) of disjoint nonempty subsets of $\mathbb N\times\mathbb N$.  
For    $  \left\{  \{x_{i,j}\}_{(i,j)\in  G_n} \right\}_{n\geq 1}$,  
in $\mathcal B(H)$ such that $\{x_{i,j}\}_{(i,j)\in G_n}\stackrel{\mbox{\tiny b.s}}{\boxplus}\{K_i\}_i\otimes\{L_j\}_j\curvearrowleft$ for all $n\ge 1$, we say that 
\[
\{x_{i,j}\}_{(i,j)\in G_n}\stackrel{\mbox{\tiny b.s}}{\boxplus}\{K_i\}_i\otimes\{L_j\}_j\curvearrowleft 
\]
are \emph{consistent}  for all  $n\ge 1$, if they are generated by an operator and by a common isometries system, i.e. there exist a sequence of isometries $\{u_i:K_i\to H'\}_i$  (independent of $n$), a sequence of isometries $\{v_j:L_j\to H'\}_j$  (independent of $n$), and a sequence $\{a_n\}_{n\geq1}$ in $\mathcal B(H')$ such that for each $n$, we have 
\[
\{x_{i,j}\}_{(i,j)\in G_n}\stackrel{\mbox{\tiny b.s}}{\boxplus}\{u_i:K_i\to H'\}_i\otimes\{v_j:L_j\to H'\}_j\curvearrowleft a_n.
\]
\end{definition}

From now on, assume that  $E$ is  a separable  quasi-Banach symmetric sequence space.

The following result is a well-known fact, see e.g. \cite[Proposition 2.1]{Arazy4}. 
\begin{theorem}\label{hi} 
Let $a,b\in \mathcal C_E$ be such that $\left\Vert\left(
\begin{smallmatrix}
0&b\\a&0
\end{smallmatrix}\right)
\right\Vert_{\mathcal C_E(\ell_2 \oplus \ell_2 )}=1$\;\footnote{Let $p$ be the identity operator on the first $\ell_2$ and $q$ be the identity operator on the first $\ell_2$. We denote $x:= \left(
\begin{smallmatrix}
pxp & pxq  \\
qxp &  qxq
\end{smallmatrix}\right)$ for any $x\in B(\ell_2\oplus \ell_2)$. }. Assume that there are two sequences $\{x_k\}_{k=1}^\infty$ and $\{y_k\}_{k=1}^\infty$ in $\mathcal C_E$ satisfying
\[
\{x_k\}_{k=1}^\infty\stackrel{\mbox{{\rm\tiny b.s}}}{\boxplus}\{K_k\}_{k=1}^\infty\otimes L_0\curvearrowleft a
\]
and
\[
\{y_k\}_{k=1}^\infty\stackrel{\mbox{{\rm\tiny b.s}}}{\boxplus}K_0\otimes \{L_k\}_{k=1}^\infty\curvearrowleft b,
\]
where $\{K_k\}_{k=0}^\infty$ and $\{L_k\}_{k=0}^\infty$ are two sequences of mutually orthogonal closed subspaces of $H$. Then, we have 
\[
\{x_k+y_k\}_{k=1}^\infty\simeq\{e^{H}_k\}_{k=1}^\infty
\]
(recall that ``$\simeq$" means that  the two basic sequences are isometrically equivalent),
where $\{e^{H}_k\}_{k=1}^\infty$ is the natural basis of $\ell_2$ (see Definition \ref{hilertbasis}).
\end{theorem}

 The idea used in the proof of the following lemma can be  found in \cite{Arazy4,Arazy2,Holub,Friedman}. 

\begin{lemma}\label{3b}
Let $\{K_i\}_{i=1}^\infty$ be a sequence of mutually orthogonal closed subspaces of $H$, and $\{L_i\}_{i=1}^\infty$ be a sequence of (not necessarily mutually orthogonal) finite dimensional subspaces of $H$. Assume that $\{x_{i,j}\}_{1\leq j\leq i<\infty}$ is a bounded sequence in $\mathcal C_E$ with
\[
\{x_{i,j}\}_{i=j}^{\infty}\stackrel{\mbox{{\rm\tiny b.s}}}{\boxplus}\{K_i\}_{i=1}^\infty\otimes L_j\;\;\;\;\;\;{\rm for\;every\;}j,
\]
For an  arbitrary  sequence $\{\varepsilon_{k,j}\}_{k,j=1}^\infty$ of positive numbers\footnote{In the sequel, we often assume that $\varepsilon_{k,j} \downarrow 0$  as $k\to \infty $ for each fixed $j$
and $\varepsilon_{k,j} \downarrow  0$  as $j\to \infty $ for each fixed $k$.  },  there exist an increasing\footnote{Throughout this paper, \emph{``increasing''} always means strictly increasing.}  sequence $\{i_k\}_{k=1}^\infty$ of positive integers and a sequence $\{y_{i_k,j}\}_{1\leq j\leq i_{k-1}<\infty}$ in $\mathcal C_E$ satisfying
\begin{itemize}
\item [(1)]
$\{y_{i_k,j}\}_{j\leq i_{k-1}<\infty}\stackrel{\mbox{{\rm\tiny b.s}}}{\boxplus}\{K_{i_k}\}_{k=1}^\infty\otimes L_j\curvearrowleft\;$ for every $j$,
\item [(2)] 
$\{y_{i_k,j}\}_{j\leq i_{k-1}<\infty}\stackrel{\mbox{{\rm\tiny b.s}}}{\boxplus}\{K_{i_k}\}_{k=1}^\infty\otimes L\curvearrowleft\;$ are consistent for all $j$,
where $L=[\bigcup_{j=1}^\infty L_j]$, and
\item [(3)]
$\Vert x_{i_k,j}-y_{i_k,j}\Vert_{\mathcal C_E}\leq\varepsilon_{k,j}$ for every $k\geq2$ and $1\leq j\leq i_{k-1}$
\end{itemize}
If, in addition that,  $\{L_i\}_{i=1}^\infty$ is mutually orthogonal, then $\{y_{i_k,j}\}_{1\leq j\leq i_{k-1}<\infty}$ can be chosen to be
\[
\{y_{i_k,j}\}_{j\leq i_{k-1}<\infty}\stackrel{\mbox{{\rm\tiny b.s}}}{\boxplus}\{K_{i_k}\}_{k=1}^\infty\otimes \{L_j\}_{j=1}^\infty\curvearrowleft\;
\]
are consistent for all $j\geq1$.
\end{lemma}
\begin{proof}
Without loss of generality, we may assume that $\{\varepsilon_{i,j}\}_{i=j}^\infty$ is a decreasing sequence and $L_j\neq\{0\}$ for every $j$. Setting  $m_j: =\dim(L_1+\cdots+L_j)$, we may choose an orthonormal sequence $\{\eta_{j'}\}_{1\leq j'\leq m_j}$ with $L_1+\cdots+L_j={\rm span}\big(\{\eta_1,\dots,\eta_{m_j}\}\big)$. Noting  that $\dim({\rm im}(x_{i,j}))\leq\dim(L_j)$ for every $j$,  we may  choose a sequences of orthogonal sequences
\[
\Big\{\xi^i_1,\cdots,\xi^i_{n^i_1},\cdots,\xi^i_{n^i_{i-1}+1},\cdots,\xi^i_{n^i_i}\Big\}\subseteq K_i\;\;\;\;\;\;i=1,2,\cdots
\]
 of orthonormal vectors,
where $n^i_0=0$ and $n^i_j\leq m_j$ for every $1\leq j\leq i$, such that
\[
x_{i,j}=\sum_{i'=1}^{n^i_j}\sum_{j'=1}^{m_j}\alpha^{(i,j)}_{i',j'}(\cdot|\eta_{j'})\xi^i_{i'},
\]
where $\Big\{\alpha^{(i,j)}_{i',j'}\Big\}_{1\leq i'\leq n^i_j, 1\leq j'\leq m_j}$ is a sequence of scalars. If $  L_j$'s, $ 1\le j < \infty$, are mutually orthogonal, then we can write
\[
x_{i,j}=\sum_{i'=1}^{n^i_j}\sum_{j'=m_{j-1}+1}^{m_j}\alpha^{(i,j)}_{i',j'}(\cdot|\eta_{j'})\xi^i_{i'}
\]
for any $1\leq j\leq i$ (put $m_0=0$), i.e. $\alpha^{(i,j)}_{i',j'}=0$ for all $1\leq j'\leq m_{j-1}$. 

By the compactness of the unit ball of a  finite-dimensional Banach space, $\mathbb M_{n_i^j,m_j}$, 
and by  induction on passing to   subsequences, there exist
\begin{itemize}
\item [(1)]
increasing sequences $\{i^{(j)}_k\}_{k=1}^\infty$ of positive integers, $j=1,2, \cdots$, with $\{i^{(j+1)}_k\}_{k=1}^\infty$ is the subsequence of $\{i^{(j)}_k\}_{k=1}^\infty$ and $i^{(j)}_1<i^{(j+1)}_1$,
\item [(2)]
a nondecreasing sequence $\{n_j\}_{j=1}^\infty$ with $n_j\leq n^{i^{(j)}_k}_j$ for all $k$, and 
\item [(3)]
the sequences $\Big\{\alpha^j_{i',j',}\Big\}_{1\leq i'\leq n_j, 1\leq j'\leq m_j}$ of of scalars,
\end{itemize}
such that for every $l\ge 2 $ and $1\leq j\leq i^{(l-1)}_1$,
and for all $k\ge 1$, we have 
\[
\bigg\Vert x_{i^{(j)}_k,j}-\sum_{i'=1}^{n_j}\sum_{j'=1}^{m_j}\alpha^j_{i',j'}(\cdot|\eta_{j'})\xi^{i^{(l)}_k}_{i'}\bigg\Vert_{\mathcal C_E}\leq\varepsilon_{l-1+k,j}.
\]
The proof is complete by setting 
  $i_k=i^{(k)}_1$ and
\[
y_{i_k,j}=\sum_{i'=1}^{n_j}\sum_{j'=1}^{m_j}\alpha^j_{i',j'}(\cdot|\eta_{j'})\xi^{i_k}_{i'}
\]
for every $k\geq2$ and $1\leq j\leq i_{k-1}$.  
 \end{proof}

\begin{remark}\label{3b'}
The argument used in the proof of Theorem \ref{3b} applies to 
  upper  triangular sequences $\{x_{i,j}\}_{1\leq i\leq j<\infty}$ of elements in $\mathcal C_E$. 
\end{remark}

\begin{corollary}
\label{3a}
Let $L$ be a finite-dimensional subspace of $H$, and $\{K_k\}_{k=1}^\infty$ be a sequence of mutually orthogonal closed subspaces of $H$. Assume that $\{x_k\}_{k=1}^\infty$ is a bounded sequence in $\mathcal C_E$ with
\[
\{x_k\}_{k=1}^\infty\stackrel{\mbox{{\rm\tiny b.s}}}{\boxplus}\{K_k\}_{k=1}^\infty\otimes L.
\]
For a given sequence $\{\varepsilon_i\}_{i=1}^\infty$ of positive numbers, there exist an increasing sequence $\{k_i\}_{i=1}^\infty$ of positive integer, and a sequence $\{y_i\}_{i=1}^\infty$ in $\mathcal C_E$ with
\[
\{y_i\}_{i=1}^\infty\stackrel{\mbox{{\rm\tiny b.s}}}{\boxplus}\{K_{k_i}\}_{i=1}^\infty\otimes L\curvearrowleft\;
\]
such that $\Vert x_{k_i}-y_i\Vert_{\mathcal C_E}\leq\varepsilon_i$ for all $i$.
\end{corollary}

\section{Decompositions of operators on $\mathcal C_E$}\label{sec:decomposition CE operator}

From now on, unless otherwise stated, we always assume that
\begin{itemize}
    \item 
    $\{\xi_i\}_{i=1}^\infty$ and $\{\eta_i\}_{i=1}^\infty$ are two sequences of orthonormal vectors of the Hilbert space $H$,
    \item
    $\{K_i\}_{i=1}^\infty$ and $\{L_i\}_{i=1}^\infty$ are two sequences of mutually orthogonal finite dimensional subspaces of $H$,
    \item 
    $E$ is a separable  quasi-Banach symmetric sequence space  $E$  with $c_0\not\hookrightarrow E$ (equivalently,  $E$ has the Fatou property and order continuous norm, see Proposition \ref{KB c0} below), and
    \item
    the modulus of concavity of $\left\Vert\cdot \right\Vert_{\mathcal C_E}$ is $2^{\frac{1}{r}-1}$, where $r$ is positive number with $0<r\leq1$. In particular, for any $\sum_{k=1}^\infty x_k$ converging in $\mathcal C_E$, we have
  \begin{equation}\label{sum-infty}
      \norm{\sum_{k=1}^\infty x_k}_{\mathcal C_E}\stackrel{\eqref{qt-eq}}{\leq}4^{\frac{1}{r}}\left(\sum_{k=1}^\infty\Vert x_k\Vert^r\right)^{\frac{1}{r}}.
  \end{equation}
\end{itemize}

\subsection{Arazy-type decomposition theorem for operators on $\mathcal C_E$}\label{subsection:operator}
 
Below, $\hat{H},\hat{H}'$, $\tilde{H}$ and $\tilde{H}'$ stand for separable infinite-dimensional Hilbert spaces. Recall that $\mathcal T_{E,\{\xi_i,\eta_i\}_{i=1}^\infty}$ denotes a lower triangular subspace of $\mathcal C_E$ by \eqref{ts}, i.e.
\[
\mathcal T_{E,\{\xi_i,\eta_i\}_{i=1}^\infty}=[e_{\xi_i,\eta_j}]_{1\leq j\leq i<\infty}.
\]

Arazy established an important decomposition theorem for isomorphic embeddings between separable Banach operator ideals\cite[Theorem 4.1]{Arazy4}:

Suppose that $E$ and $F$ are two Banach symmetric sequence spaces, and $c_0\not\hookrightarrow E$. If $\mathcal T_{F,\{\xi_i,\eta_i\}_{i=1}^\infty}$ is isomorphic to a subspace of $\mathcal C_E$, then there exists a linear  isomorphism $V$ from $\mathcal T_{F,\{\xi_i,\eta_i\}_{i=1}^\infty}$ into $\mathcal C_E$ such that $V$ can be decomposed into the sum of three operators $V_1$, $V_2$ and $V_3$, i.e.
\[
V=V_1+V_2+V_3.
\]
Roughly speaking, in the representations on the tensor product of $3$ Hilbert spaces, 
 they possess the following properties, respectively:
\begin{itemize}
    \item 
    $V_1$ has diagonal form, i.e.,  there is a sequence $\{c_{i,j}\}_{1\leq j\leq i<\infty}$ such that
    \[
    V_1(e_{\xi_i,\eta_j})=c_{i,j}\otimes e_{\xi_i,\eta_i},\;\;\;\;\;\;\;\;1\leq j\leq i<\infty,
    \]
    \item
    There is an element $b\in\mathcal C_E$ such that
    \[
    V_2(e_{\xi_i,\eta_j})=b\otimes e_{\xi_{\varphi(i,j)},\eta_1},\;\;\;\;\;\;\;\;1\leq j\leq i<\infty,
    \]
    where $\varphi$ is the mapping that orders $\{(i,j)\in\mathbb N\times\mathbb N:j\leq i\}$ lexicographically into $\mathbb N$. In particular, if $b\neq0$, then $\{V_2(e_{\xi_i,\eta_j})\}_{1\leq j\leq i<\infty}$ is equivalent to an orthonormal basis of the infinite-dimensional separable  Hilbert.
    \item
    There is an element $a\in\mathcal C_E$ such that
    \[
    V_3(e_{\xi_i,\eta_j})=a\otimes e_{\xi_i,\eta_j},\;\;\;\;\;\;\;\;1\leq j\leq i<\infty.
    \] 
\end{itemize}
Moreover, there exist three projection operators $P_i$, $i=1,2,3$, on $\mathcal C_E$ such that
\[
V_i=P_iV,\;\;\;\;\;\;\;\;i=1,2,3.
\]
The following figures are helpful for understanding the concrete forms of  $V_1$, $V_2$ and $V_3$:

\noindent
\\
\\
\scalebox{0.65}{

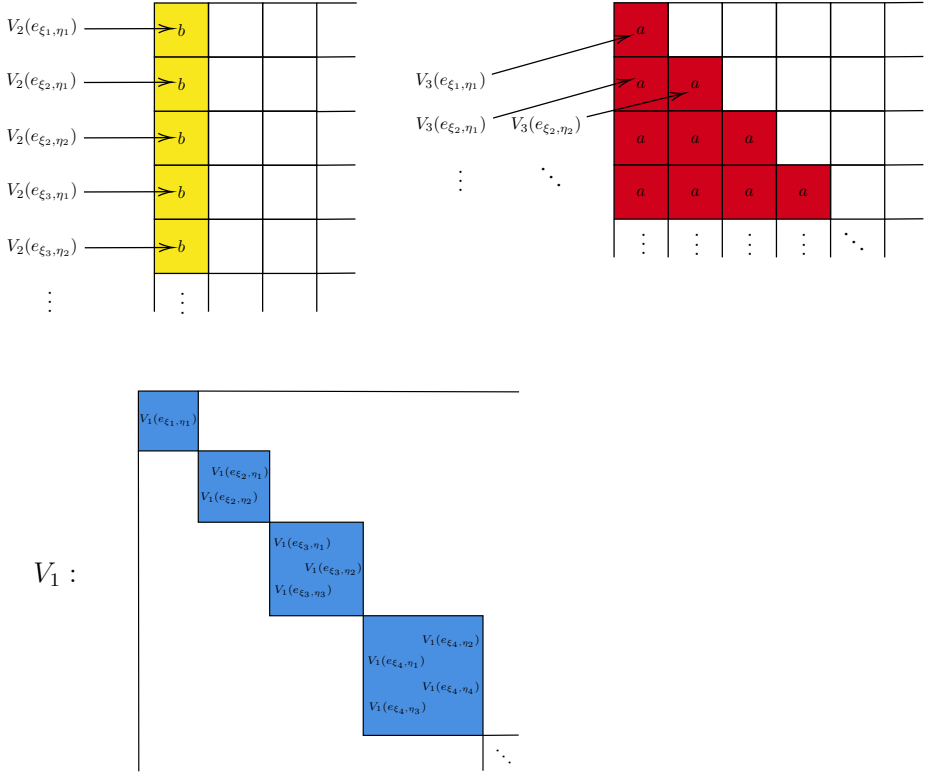
\tikzset{every picture/.style={line width=0.75pt}} 

\begin{tikzpicture}[x=0.75pt,y=0.75pt,yscale=-1,xscale=1]

\draw  [fill={rgb, 255:red, 248; green, 231; blue, 28 }  ,fill opacity=1 ] (147.33,70.11) -- (189.33,70.11) -- (189.33,112.11) -- (147.33,112.11) -- cycle ;
\draw   (231.33,70.11) -- (273.33,70.11) -- (273.33,112.11) -- (231.33,112.11) -- cycle ;
\draw   (189.33,112.11) -- (231.33,112.11) -- (231.33,154.11) -- (189.33,154.11) -- cycle ;
\draw   (189.33,70.11) -- (231.33,70.11) -- (231.33,112.11) -- (189.33,112.11) -- cycle ;
\draw  [fill={rgb, 255:red, 248; green, 231; blue, 28 }  ,fill opacity=1 ] (147.33,154.11) -- (189.33,154.11) -- (189.33,196.11) -- (147.33,196.11) -- cycle ;
\draw  [fill={rgb, 255:red, 248; green, 231; blue, 28 }  ,fill opacity=1 ] (147.33,112.11) -- (189.33,112.11) -- (189.33,154.11) -- (147.33,154.11) -- cycle ;
\draw   (231.33,196.11) -- (273.33,196.11) -- (273.33,238.11) -- (231.33,238.11) -- cycle ;
\draw   (189.33,196.11) -- (231.33,196.11) -- (231.33,238.11) -- (189.33,238.11) -- cycle ;
\draw  [fill={rgb, 255:red, 248; green, 231; blue, 28 }  ,fill opacity=1 ] (147.33,196.11) -- (189.33,196.11) -- (189.33,238.11) -- (147.33,238.11) -- cycle ;
\draw   (231.33,154.11) -- (273.33,154.11) -- (273.33,196.11) -- (231.33,196.11) -- cycle ;
\draw   (189.33,154.11) -- (231.33,154.11) -- (231.33,196.11) -- (189.33,196.11) -- cycle ;
\draw   (231.33,112.11) -- (273.33,112.11) -- (273.33,154.11) -- (231.33,154.11) -- cycle ;
\draw   (231.33,238.11) -- (273.33,238.11) -- (273.33,280.11) -- (231.33,280.11) -- cycle ;
\draw   (189.33,238.11) -- (231.33,238.11) -- (231.33,280.11) -- (189.33,280.11) -- cycle ;
\draw  [fill={rgb, 255:red, 248; green, 231; blue, 28 }  ,fill opacity=1 ] (147.33,238.11) -- (189.33,238.11) -- (189.33,280.11) -- (147.33,280.11) -- cycle ;
\draw    (147.33,280.11) -- (147.33,309.84) ;
\draw    (189.33,280.11) -- (189.33,309.84) ;
\draw    (273.33,280.11) -- (273.33,309.84) ;
\draw    (231.33,280.11) -- (231.33,309.84) ;
\draw    (304.33,70.09) -- (273.33,70.11) ;
\draw    (303.33,237.84) -- (273.33,238.11) ;
\draw    (303.33,195.84) -- (273.33,196.11) ;
\draw    (303.33,154.17) -- (273.33,154.44) ;
\draw    (303.33,111.84) -- (291.66,111.94) -- (273.33,112.11) ;
\draw  [fill={rgb, 255:red, 208; green, 2; blue, 27 }  ,fill opacity=1 ] (504,69.93) -- (546,69.93) -- (546,111.93) -- (504,111.93) -- cycle ;
\draw   (588,69.93) -- (630,69.93) -- (630,111.93) -- (588,111.93) -- cycle ;
\draw  [fill={rgb, 255:red, 208; green, 2; blue, 27 }  ,fill opacity=1 ] (546,111.93) -- (588,111.93) -- (588,153.93) -- (546,153.93) -- cycle ;
\draw   (546,69.93) -- (588,69.93) -- (588,111.93) -- (546,111.93) -- cycle ;
\draw  [fill={rgb, 255:red, 208; green, 2; blue, 27 }  ,fill opacity=1 ] (504,153.93) -- (546,153.93) -- (546,195.93) -- (504,195.93) -- cycle ;
\draw  [fill={rgb, 255:red, 208; green, 2; blue, 27 }  ,fill opacity=1 ] (504,111.93) -- (546,111.93) -- (546,153.93) -- (504,153.93) -- cycle ;
\draw  [fill={rgb, 255:red, 208; green, 2; blue, 27 }  ,fill opacity=1 ] (588,195.93) -- (630,195.93) -- (630,237.93) -- (588,237.93) -- cycle ;
\draw   (630,153.93) -- (672,153.93) -- (672,195.93) -- (630,195.93) -- cycle ;
\draw  [fill={rgb, 255:red, 208; green, 2; blue, 27 }  ,fill opacity=1 ] (546,195.93) -- (588,195.93) -- (588,237.93) -- (546,237.93) -- cycle ;
\draw  [fill={rgb, 255:red, 208; green, 2; blue, 27 }  ,fill opacity=1 ] (504,195.93) -- (546,195.93) -- (546,237.93) -- (504,237.93) -- cycle ;
\draw  [fill={rgb, 255:red, 208; green, 2; blue, 27 }  ,fill opacity=1 ] (588,153.93) -- (630,153.93) -- (630,195.93) -- (588,195.93) -- cycle ;
\draw  [fill={rgb, 255:red, 208; green, 2; blue, 27 }  ,fill opacity=1 ] (546,153.93) -- (588,153.93) -- (588,195.93) -- (546,195.93) -- cycle ;
\draw   (630,111.93) -- (672,111.93) -- (672,153.93) -- (630,153.93) -- cycle ;
\draw   (588,111.93) -- (630,111.93) -- (630,153.93) -- (588,153.93) -- cycle ;
\draw   (630,69.93) -- (672,69.93) -- (672,111.93) -- (630,111.93) -- cycle ;
\draw  [fill={rgb, 255:red, 208; green, 2; blue, 27 }  ,fill opacity=1 ] (630,195.93) -- (672,195.93) -- (672,237.93) -- (630,237.93) -- cycle ;
\draw   (672,195.93) -- (714,195.93) -- (714,237.93) -- (672,237.93) -- cycle ;
\draw   (672,153.93) -- (714,153.93) -- (714,195.93) -- (672,195.93) -- cycle ;
\draw   (672,111.93) -- (714,111.93) -- (714,153.93) -- (672,153.93) -- cycle ;
\draw   (672,69.93) -- (714,69.93) -- (714,111.93) -- (672,111.93) -- cycle ;
\draw    (504,237.93) -- (504,267.66) ;
\draw    (546,237.93) -- (546,267.66) ;
\draw    (672,237.93) -- (672,267.66) ;
\draw    (630,237.93) -- (630,267.66) ;
\draw    (588,237.93) -- (588,267.66) ;
\draw    (744,69.66) -- (714,69.93) ;
\draw    (714,237.93) -- (714,267.66) ;
\draw    (744,237.66) -- (714,237.93) ;
\draw    (744,195.66) -- (714,195.93) ;
\draw    (744,153.66) -- (714,153.93) ;
\draw    (744,111.66) -- (714,111.93) ;
\draw    (303.33,279.84) -- (273.33,280.11) ;
\draw    (93,90.17) -- (159.59,89.98) ;
\draw [shift={(161.59,89.98)}, rotate = 179.84] [color={rgb, 255:red, 0; green, 0; blue, 0 }  ][line width=0.75]    (10.93,-3.29) .. controls (6.95,-1.4) and (3.31,-0.3) .. (0,0) .. controls (3.31,0.3) and (6.95,1.4) .. (10.93,3.29)   ;
\draw    (410.13,159.47) -- (515.41,130.04) ;
\draw [shift={(517.33,129.51)}, rotate = 164.38] [color={rgb, 255:red, 0; green, 0; blue, 0 }  ][line width=0.75]    (10.93,-3.29) .. controls (6.95,-1.4) and (3.31,-0.3) .. (0,0) .. controls (3.31,0.3) and (6.95,1.4) .. (10.93,3.29)   ;
\draw    (482.8,158.47) -- (556.85,140.61) ;
\draw [shift={(558.8,140.14)}, rotate = 166.44] [color={rgb, 255:red, 0; green, 0; blue, 0 }  ][line width=0.75]    (10.93,-3.29) .. controls (6.95,-1.4) and (3.31,-0.3) .. (0,0) .. controls (3.31,0.3) and (6.95,1.4) .. (10.93,3.29)   ;
\draw    (92.67,131.84) -- (159.26,131.65) ;
\draw [shift={(161.26,131.65)}, rotate = 179.84] [color={rgb, 255:red, 0; green, 0; blue, 0 }  ][line width=0.75]    (10.93,-3.29) .. controls (6.95,-1.4) and (3.31,-0.3) .. (0,0) .. controls (3.31,0.3) and (6.95,1.4) .. (10.93,3.29)   ;
\draw    (92.67,174.84) -- (159.26,174.65) ;
\draw [shift={(161.26,174.65)}, rotate = 179.84] [color={rgb, 255:red, 0; green, 0; blue, 0 }  ][line width=0.75]    (10.93,-3.29) .. controls (6.95,-1.4) and (3.31,-0.3) .. (0,0) .. controls (3.31,0.3) and (6.95,1.4) .. (10.93,3.29)   ;
\draw    (92.67,216.84) -- (159.26,216.65) ;
\draw [shift={(161.26,216.65)}, rotate = 179.84] [color={rgb, 255:red, 0; green, 0; blue, 0 }  ][line width=0.75]    (10.93,-3.29) .. controls (6.95,-1.4) and (3.31,-0.3) .. (0,0) .. controls (3.31,0.3) and (6.95,1.4) .. (10.93,3.29)   ;
\draw    (93.33,260.17) -- (159.93,259.98) ;
\draw [shift={(161.93,259.98)}, rotate = 179.84] [color={rgb, 255:red, 0; green, 0; blue, 0 }  ][line width=0.75]    (10.93,-3.29) .. controls (6.95,-1.4) and (3.31,-0.3) .. (0,0) .. controls (3.31,0.3) and (6.95,1.4) .. (10.93,3.29)   ;
\draw    (409.13,125.47) -- (514.41,96.04) ;
\draw [shift={(516.33,95.51)}, rotate = 164.38] [color={rgb, 255:red, 0; green, 0; blue, 0 }  ][line width=0.75]    (10.93,-3.29) .. controls (6.95,-1.4) and (3.31,-0.3) .. (0,0) .. controls (3.31,0.3) and (6.95,1.4) .. (10.93,3.29)   ;

\draw (60,174.17) node    {$V_{2}( e_{\xi _{2} ,\eta _{2}})$};
\draw (168.33,91.11) node    {$b$};
\draw (168.33,133.11) node    {$b$};
\draw (168.33,175.11) node    {$b$};
\draw (168.33,217.11) node    {$b$};
\draw (168.33,259.11) node    {$b$};
\draw (168.83,304.6) node  [font=\Large,rotate=-90]  {$\cdots $};
\draw (65.83,303.6) node  [font=\Large,rotate=-90]  {$\cdots $};
\draw (525,90.93) node    {$a$};
\draw (525,132.93) node    {$a$};
\draw (525,175.93) node    {$a$};
\draw (525,216.93) node    {$a$};
\draw (567,132.93) node    {$a$};
\draw (526.5,258.42) node  [font=\Large,rotate=-90]  {$\cdots $};
\draw (384.5,209.42) node  [font=\Large,rotate=-90]  {$\cdots $};
\draw (567,175.93) node    {$a$};
\draw (567,216.93) node    {$a$};
\draw (609,216.93) node    {$a$};
\draw (609,175.93) node    {$a$};
\draw (651,216.93) node    {$a$};
\draw (568,258.93) node  [font=\Large,rotate=-90]  {$\cdots $};
\draw (691,256.93) node  [font=\LARGE,rotate=-45]  {$\cdots $};
\draw (651,258.93) node  [font=\Large,rotate=-90]  {$\cdots $};
\draw (609,258.93) node  [font=\Large,rotate=-90]  {$\cdots $};
\draw (59.67,130.84) node    {$V_{2}( e_{\xi _{2} ,\eta _{1}})$};
\draw (60.33,88.84) node    {$V_{2}( e_{\xi _{1} ,\eta _{1}})$};
\draw (60.33,216.17) node    {$V_{2}( e_{\xi _{3} ,\eta _{1}})$};
\draw (60.33,259.84) node    {$V_{2}( e_{\xi _{3} ,\eta _{2}})$};
\draw (377.67,164.51) node    {$V_{3}( e_{\xi _{2} ,\eta _{1}})$};
\draw (451.33,165.17) node    {$V_{3}( e_{\xi _{2} ,\eta _{2}})$};
\draw (377.33,130.84) node    {$V_{3}( e_{\xi _{1} ,\eta _{1}})$};
\draw (456,206.6) node  [font=\LARGE,rotate=-45]  {$\cdots $};

\end{tikzpicture}
}

\noindent
\\

\scalebox{0.6}{
\tikzset{every picture/.style={line width=0.75pt}} 

\begin{tikzpicture}[x=0.75pt,y=0.75pt,yscale=-1,xscale=1]

\draw  [fill={rgb, 255:red, 74; green, 144; blue, 226 }  ,fill opacity=1 ] (270.67,30.67) -- (321,30.67) -- (321,81) -- (270.67,81) -- cycle ;
\draw  [fill={rgb, 255:red, 74; green, 144; blue, 226 }  ,fill opacity=1 ] (321,81) -- (381,81) -- (381,141) -- (321,141) -- cycle ;
\draw  [fill={rgb, 255:red, 74; green, 144; blue, 226 }  ,fill opacity=1 ] (381,141) -- (459.67,141) -- (459.67,219.67) -- (381,219.67) -- cycle ;
\draw  [fill={rgb, 255:red, 74; green, 144; blue, 226 }  ,fill opacity=1 ] (459.67,219.67) -- (560.33,219.67) -- (560.33,320.33) -- (459.67,320.33) -- cycle ;
\draw    (321,30.67) -- (590.67,31.33) ;
\draw    (270.67,81) -- (271,350.55) ;
\draw    (590.33,320.06) -- (560.33,320.33) ;
\draw    (560.33,320.33) -- (560.33,350.06) ;

\draw (295.83,54.83) node  [font=\scriptsize]  {$V_{1}( e_{\xi _{1} ,\eta _{1}})$};
\draw (356.17,99.17) node  [font=\scriptsize]  {$V_{1}( e_{\xi _{2} ,\eta _{1}})$};
\draw (347.17,120.5) node  [font=\scriptsize]  {$V_{1}( e_{\xi _{2} ,\eta _{2}})$};
\draw (408.5,158.83) node  [font=\scriptsize]  {$V_{1}( e_{\xi _{3} ,\eta _{1}})$};
\draw (409.17,198.83) node  [font=\scriptsize]  {$V_{1}( e_{\xi _{3} ,\eta _{3}})$};
\draw (434.5,180.17) node  [font=\scriptsize]  {$V_{1}( e_{\xi _{3} ,\eta _{2}})$};
\draw (487.5,258.17) node  [font=\scriptsize]  {$V_{1}( e_{\xi _{4} ,\eta _{1}})$};
\draw (488.5,296.83) node  [font=\scriptsize]  {$V_{1}( e_{\xi _{4} ,\eta _{3}})$};
\draw (533.5,240.83) node  [font=\scriptsize]  {$V_{1}( e_{\xi _{4} ,\eta _{2}})$};
\draw (533.5,280.17) node  [font=\scriptsize]  {$V_{1}( e_{\xi _{4} ,\eta _{4}})$};
\draw (578.33,339) node  [font=\Large,rotate=-45]  {$\cdots $};
\draw (181,173.07) node [anchor=north west][inner sep=0.75pt]  [font=\Large]  {\scalebox{1.5}{$V_{1} :$}};

\end{tikzpicture}
}
\noindent
\\

Below, We extend \cite[Theorem 4.1]{Arazy4} to the setting of general bounded linear operators on  quasi-Banach operator ideals (see Theorem \ref{thma'}). In Theorem \ref{thma} below, $T_1$ is analogous to $V_1$ above, $T_2+T_3$ is analogous to $V_2$ above (see Remark \ref{t2+t3-rem}), and $T_4$ is analogous to $V_3$ above (see Lemma~\ref{t4} and Remark \ref{rem-t4}).
The following theorem   presents  a representation with more details, and 
 demonstrates how  operators admitting  ``good" decompositions  are obtained through perturbations of the original operators---this point is also key to the subsequent proofs of Theorems~\ref{c1s} and \ref{c1ch}.

\begin{theorem}\label{thma}
Suppose that $T:\mathcal T_{E,\{\xi_i,\eta_i\}_{i=1}^\infty}\to \mathcal C_E$ is a bounded operator with $p_{_{[\bigcup_{i=1}^\infty K_i]}}T(\cdot)p_{_{[\bigcup_{i=1}^\infty L_i]}}=T$.
For a given a sequence $\{\varepsilon_l\}_{l=1}^\infty$ of positive numbers,
there exist two increasing sequences $\{i_k\}_{k=1}^\infty$ and $\{n_k\}_{k=1}^\infty$ of positive integers, and $4$ operators $T_1,T_2,T_3,T_4\in\mathcal B\big(\mathcal T_{E,\{\xi_{i_{k+1}},\eta_{i_k}\}_{k=2}^\infty},\mathcal C_E\big)$ of  the forms
\begin{equation}\label{T1}
T_1(e_{\xi_{i_k},\eta_{i_l}})=p_{_{K'_k}}T(e_{\xi_{i_k},\eta_{i_l}})p_{_{L'_k}},\;\;\;\;\;\;\;\;\textcolor[RGB]{74,144,226}{ 2\leq l<k<\infty},
\end{equation}
\begin{equation}\label{T2}
T_2(e_{\xi_{i_k},\eta_{i_l}})=x_{_{(k,l),l}}+y_{_{l,(l,k)}},\;\;\;\;\;\;\;\;\textcolor[RGB]{126,211,33}{ 2\leq l<k<\infty},
\end{equation}
\begin{equation}\label{T3}
T_3(e_{\xi_{i_k},\eta_{i_l}})=\sum_{m=1}^{l-1}\big(x_{_{(k,l),m}}+y_{_{m,(l,k)}}\big),\;\;\;\;\;\;\;\;\textcolor[RGB]{248,231,28}{2\leq l<k<\infty},
\end{equation}
\begin{equation}\label{T4}
T_4(e_{\xi_{i_k},\eta_{i_l}})=\sum_{m=1}^{l-1}\big(x_{_{(k,m),l}}+y_{_{l,(m,k)}}\big),\;\;\;\;\;\;\;\;\textcolor[RGB]{208,2,27}{2\leq l<k<\infty},
\end{equation}
 such that $T_1+T_2+T_3+T_4$ is a perturbation of $T|_{\mathcal T_{E,\{\xi_{i_{k+1}},\eta_{i_k}\}_{k=2}^\infty}}$ associated with the Schauder decomposition $\big\{[e_{\xi_{i_k},\eta_{i_l}}]_{k=l+1}^\infty\big\}_{l=2}^\infty$ for $\{\varepsilon_l\}_{l=2}^\infty$, 
where

\noindent
\\
\scalebox{0.8}{
\tikzset{every picture/.style={line width=0.75pt}} 

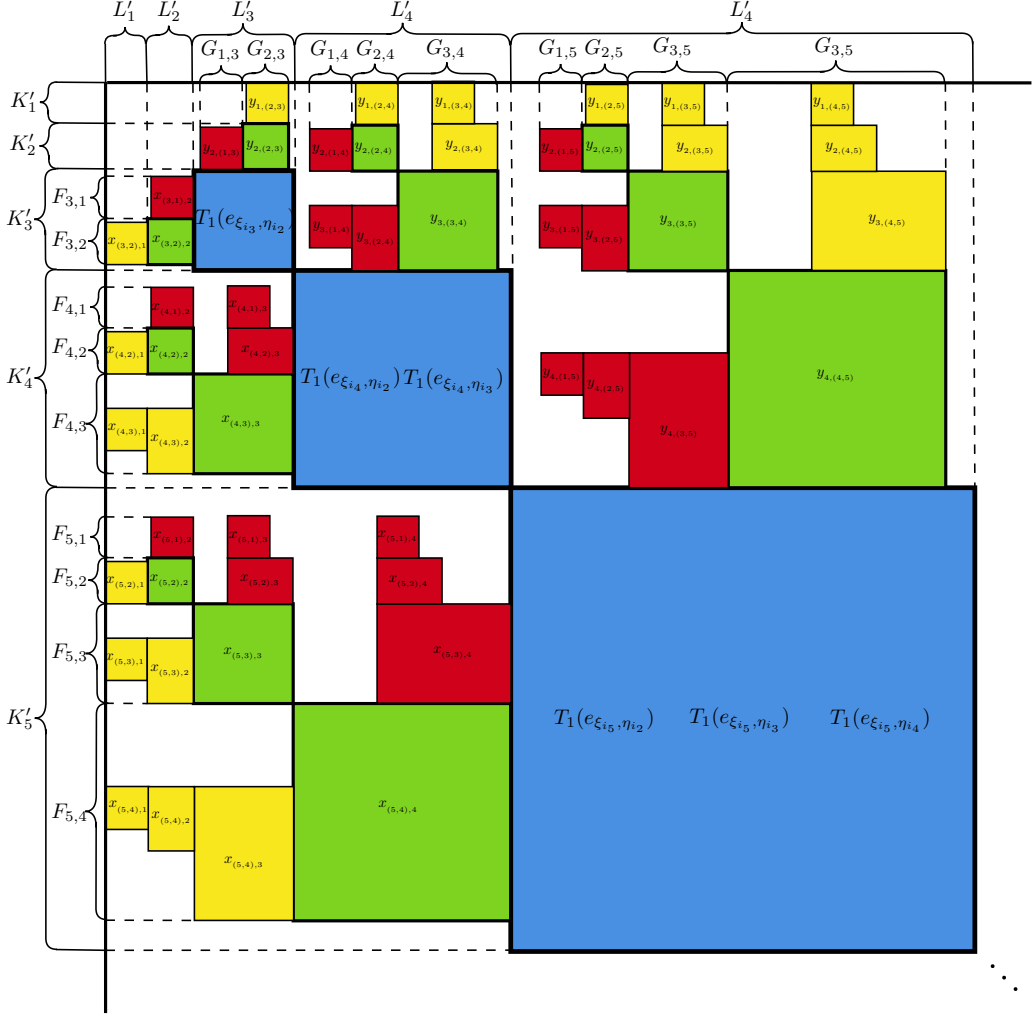
\begin{tikzpicture}[x=0.75pt,y=0.75pt,yscale=-1,xscale=1]

\draw  [fill={rgb, 255:red, 74; green, 144; blue, 226 }  ,fill opacity=1 ][line width=2.25]  (126.19,131.86) -- (189.12,131.86) -- (189.12,194.79) -- (126.19,194.79) -- cycle ;
\draw  [fill={rgb, 255:red, 74; green, 144; blue, 226 }  ,fill opacity=1 ][line width=2.25]  (189.07,194.79) -- (326.17,194.79) -- (326.17,331.88) -- (189.07,331.88) -- cycle ;
\draw  [fill={rgb, 255:red, 74; green, 144; blue, 226 }  ,fill opacity=1 ][line width=2.25]  (325.93,331.88) -- (618.52,331.88) -- (618.52,624.47) -- (325.93,624.47) -- cycle ;
\draw  [fill={rgb, 255:red, 208; green, 2; blue, 27 }  ,fill opacity=1 ] (99.17,135.44) -- (125.85,135.44) -- (125.85,162.12) -- (99.17,162.12) -- cycle ;
\draw  [fill={rgb, 255:red, 126; green, 211; blue, 33 }  ,fill opacity=1 ][line width=1.5]  (96.87,162.12) -- (125.85,162.12) -- (125.85,191.1) -- (96.87,191.1) -- cycle ;
\draw  [fill={rgb, 255:red, 248; green, 231; blue, 28 }  ,fill opacity=1 ] (70.19,164.42) -- (96.87,164.42) -- (96.87,191.1) -- (70.19,191.1) -- cycle ;
\draw  [fill={rgb, 255:red, 248; green, 231; blue, 28 }  ,fill opacity=1 ] (159.17,76.44) -- (185.85,76.44) -- (185.85,103.12) -- (159.17,103.12) -- cycle ;
\draw  [fill={rgb, 255:red, 126; green, 211; blue, 33 }  ,fill opacity=1 ][line width=1.5]  (156.87,102.12) -- (185.85,102.12) -- (185.85,131.1) -- (156.87,131.1) -- cycle ;
\draw  [fill={rgb, 255:red, 208; green, 2; blue, 27 }  ,fill opacity=1 ] (130.19,104.42) -- (156.87,104.42) -- (156.87,131.1) -- (130.19,131.1) -- cycle ;
\draw  [fill={rgb, 255:red, 126; green, 211; blue, 33 }  ,fill opacity=1 ][line width=1.5]  (254.85,132.1) -- (317.78,132.1) -- (317.78,195.03) -- (254.85,195.03) -- cycle ;
\draw  [fill={rgb, 255:red, 248; green, 231; blue, 28 }  ,fill opacity=1 ] (228.17,76.44) -- (254.85,76.44) -- (254.85,103.12) -- (228.17,103.12) -- cycle ;
\draw  [fill={rgb, 255:red, 126; green, 211; blue, 33 }  ,fill opacity=1 ][line width=1.5]  (225.87,103.12) -- (254.85,103.12) -- (254.85,132.1) -- (225.87,132.1) -- cycle ;
\draw  [fill={rgb, 255:red, 208; green, 2; blue, 27 }  ,fill opacity=1 ] (199.19,105.42) -- (225.87,105.42) -- (225.87,132.1) -- (199.19,132.1) -- cycle ;
\draw  [fill={rgb, 255:red, 126; green, 211; blue, 33 }  ,fill opacity=1 ][line width=1.5]  (463.07,194.79) -- (600.17,194.79) -- (600.17,331.88) -- (463.07,331.88) -- cycle ;
\draw  [fill={rgb, 255:red, 126; green, 211; blue, 33 }  ,fill opacity=1 ][line width=1.5]  (399.85,132.1) -- (462.78,132.1) -- (462.78,195.03) -- (399.85,195.03) -- cycle ;
\draw  [fill={rgb, 255:red, 248; green, 231; blue, 28 }  ,fill opacity=1 ] (373.17,77.44) -- (399.85,77.44) -- (399.85,104.12) -- (373.17,104.12) -- cycle ;
\draw  [fill={rgb, 255:red, 126; green, 211; blue, 33 }  ,fill opacity=1 ][line width=1.5]  (370.87,103.12) -- (399.85,103.12) -- (399.85,132.1) -- (370.87,132.1) -- cycle ;
\draw  [fill={rgb, 255:red, 208; green, 2; blue, 27 }  ,fill opacity=1 ] (344.19,105.42) -- (370.87,105.42) -- (370.87,132.1) -- (344.19,132.1) -- cycle ;
\draw  [fill={rgb, 255:red, 208; green, 2; blue, 27 }  ,fill opacity=1 ] (199.08,153.79) -- (225.77,153.79) -- (225.77,180.47) -- (199.08,180.47) -- cycle ;
\draw  [fill={rgb, 255:red, 208; green, 2; blue, 27 }  ,fill opacity=1 ] (225.77,153.79) -- (254.85,153.79) -- (254.85,195.03) -- (225.77,195.03) -- cycle ;
\draw  [fill={rgb, 255:red, 208; green, 2; blue, 27 }  ,fill opacity=1 ] (400.49,246.8) -- (463.07,246.8) -- (463.07,331.88) -- (400.49,331.88) -- cycle ;
\draw  [fill={rgb, 255:red, 208; green, 2; blue, 27 }  ,fill opacity=1 ] (344.08,153.79) -- (370.77,153.79) -- (370.77,180.47) -- (344.08,180.47) -- cycle ;
\draw  [fill={rgb, 255:red, 208; green, 2; blue, 27 }  ,fill opacity=1 ] (370.77,153.79) -- (399.85,153.79) -- (399.85,195.03) -- (370.77,195.03) -- cycle ;
\draw  [fill={rgb, 255:red, 248; green, 231; blue, 28 }  ,fill opacity=1 ] (462.61,103.21) -- (462.61,132.29) -- (421.37,132.29) -- (421.37,103.21) -- cycle ;
\draw  [fill={rgb, 255:red, 248; green, 231; blue, 28 }  ,fill opacity=1 ] (421.37,76.52) -- (448.05,76.52) -- (448.05,103.21) -- (421.37,103.21) -- cycle ;
\draw  [fill={rgb, 255:red, 208; green, 2; blue, 27 }  ,fill opacity=1 ] (345.08,246.79) -- (371.77,246.79) -- (371.77,273.47) -- (345.08,273.47) -- cycle ;
\draw  [fill={rgb, 255:red, 208; green, 2; blue, 27 }  ,fill opacity=1 ] (371.77,246.79) -- (400.85,246.79) -- (400.85,288.03) -- (371.77,288.03) -- cycle ;
\draw  [fill={rgb, 255:red, 248; green, 231; blue, 28 }  ,fill opacity=1 ] (600.17,132.2) -- (600.17,194.79) -- (515.67,194.79) -- (515.67,132.2) -- cycle ;
\draw  [fill={rgb, 255:red, 248; green, 231; blue, 28 }  ,fill opacity=1 ] (556.61,103.21) -- (556.61,132.29) -- (515.37,132.29) -- (515.37,103.21) -- cycle ;
\draw  [fill={rgb, 255:red, 248; green, 231; blue, 28 }  ,fill opacity=1 ] (515.37,76.52) -- (542.05,76.52) -- (542.05,103.21) -- (515.37,103.21) -- cycle ;
\draw  [fill={rgb, 255:red, 248; green, 231; blue, 28 }  ,fill opacity=1 ] (317.61,102.21) -- (317.61,131.29) -- (276.37,131.29) -- (276.37,102.21) -- cycle ;
\draw  [fill={rgb, 255:red, 248; green, 231; blue, 28 }  ,fill opacity=1 ] (276.37,75.52) -- (303.05,75.52) -- (303.05,102.21) -- (276.37,102.21) -- cycle ;
\draw  [fill={rgb, 255:red, 126; green, 211; blue, 33 }  ,fill opacity=1 ][line width=1.5]  (125.85,260.1) -- (188.78,260.1) -- (188.78,323.03) -- (125.85,323.03) -- cycle ;
\draw  [fill={rgb, 255:red, 208; green, 2; blue, 27 }  ,fill opacity=1 ] (99.17,205.44) -- (125.85,205.44) -- (125.85,232.12) -- (99.17,232.12) -- cycle ;
\draw  [fill={rgb, 255:red, 126; green, 211; blue, 33 }  ,fill opacity=1 ][line width=1.5]  (96.87,231.12) -- (125.85,231.12) -- (125.85,260.1) -- (96.87,260.1) -- cycle ;
\draw  [fill={rgb, 255:red, 248; green, 231; blue, 28 }  ,fill opacity=1 ] (70.19,233.42) -- (96.87,233.42) -- (96.87,260.1) -- (70.19,260.1) -- cycle ;
\draw  [fill={rgb, 255:red, 248; green, 231; blue, 28 }  ,fill opacity=1 ] (70.08,281.79) -- (96.77,281.79) -- (96.77,308.47) -- (70.08,308.47) -- cycle ;
\draw  [fill={rgb, 255:red, 248; green, 231; blue, 28 }  ,fill opacity=1 ] (96.77,281.79) -- (125.85,281.79) -- (125.85,323.03) -- (96.77,323.03) -- cycle ;
\draw  [fill={rgb, 255:red, 208; green, 2; blue, 27 }  ,fill opacity=1 ] (188.78,231.02) -- (188.78,260.1) -- (147.54,260.1) -- (147.54,231.02) -- cycle ;
\draw  [fill={rgb, 255:red, 208; green, 2; blue, 27 }  ,fill opacity=1 ] (147.37,204.52) -- (174.05,204.52) -- (174.05,231.21) -- (147.37,231.21) -- cycle ;
\draw  [fill={rgb, 255:red, 126; green, 211; blue, 33 }  ,fill opacity=1 ][line width=1.5]  (188.78,468.03) -- (325.88,468.03) -- (325.88,605.13) -- (188.78,605.13) -- cycle ;
\draw  [fill={rgb, 255:red, 126; green, 211; blue, 33 }  ,fill opacity=1 ][line width=1.5]  (125.85,405.1) -- (188.78,405.1) -- (188.78,468.03) -- (125.85,468.03) -- cycle ;
\draw  [fill={rgb, 255:red, 208; green, 2; blue, 27 }  ,fill opacity=1 ] (99.17,350.44) -- (125.85,350.44) -- (125.85,377.12) -- (99.17,377.12) -- cycle ;
\draw  [fill={rgb, 255:red, 126; green, 211; blue, 33 }  ,fill opacity=1 ][line width=1.5]  (96.87,376.12) -- (125.85,376.12) -- (125.85,405.1) -- (96.87,405.1) -- cycle ;
\draw  [fill={rgb, 255:red, 248; green, 231; blue, 28 }  ,fill opacity=1 ] (70.19,378.42) -- (96.87,378.42) -- (96.87,405.1) -- (70.19,405.1) -- cycle ;
\draw  [fill={rgb, 255:red, 248; green, 231; blue, 28 }  ,fill opacity=1 ] (126.49,520.47) -- (189.07,520.47) -- (189.07,604.88) -- (126.49,604.88) -- cycle ;
\draw  [fill={rgb, 255:red, 248; green, 231; blue, 28 }  ,fill opacity=1 ] (70.08,426.79) -- (96.77,426.79) -- (96.77,453.47) -- (70.08,453.47) -- cycle ;
\draw  [fill={rgb, 255:red, 248; green, 231; blue, 28 }  ,fill opacity=1 ] (96.77,426.79) -- (125.85,426.79) -- (125.85,468.03) -- (96.77,468.03) -- cycle ;
\draw  [fill={rgb, 255:red, 208; green, 2; blue, 27 }  ,fill opacity=1 ] (188.61,376.21) -- (188.61,405.29) -- (147.37,405.29) -- (147.37,376.21) -- cycle ;
\draw  [fill={rgb, 255:red, 208; green, 2; blue, 27 }  ,fill opacity=1 ] (147.37,349.52) -- (174.05,349.52) -- (174.05,376.21) -- (147.37,376.21) -- cycle ;
\draw  [fill={rgb, 255:red, 248; green, 231; blue, 28 }  ,fill opacity=1 ] (70.67,520.37) -- (97.77,520.37) -- (97.77,547.47) -- (70.67,547.47) -- cycle ;
\draw  [fill={rgb, 255:red, 248; green, 231; blue, 28 }  ,fill opacity=1 ] (97.41,520.47) -- (126.49,520.47) -- (126.49,561.03) -- (97.41,561.03) -- cycle ;
\draw  [fill={rgb, 255:red, 208; green, 2; blue, 27 }  ,fill opacity=1 ] (326.17,405.2) -- (326.17,467.79) -- (241.67,467.79) -- (241.67,405.2) -- cycle ;
\draw  [fill={rgb, 255:red, 208; green, 2; blue, 27 }  ,fill opacity=1 ] (282.61,376.21) -- (282.61,405.29) -- (241.67,405.29) -- (241.67,376.21) -- cycle ;
\draw  [fill={rgb, 255:red, 208; green, 2; blue, 27 }  ,fill opacity=1 ] (241.67,349.82) -- (268.05,349.82) -- (268.05,376.21) -- (241.67,376.21) -- cycle ;
\draw [line width=1.5]    (70.49,76.25) -- (70.5,663.32) ;
\draw [line width=1.5]    (71.49,76.25) -- (290.22,76.19) -- (654.19,76.45) ;
\draw  [dash pattern={on 4.5pt off 4.5pt}]  (70.12,194.62) -- (126.19,194.79) ;
\draw  [dash pattern={on 4.5pt off 4.5pt}]  (69.8,161.95) -- (96.87,162.12) ;
\draw  [dash pattern={on 4.5pt off 4.5pt}]  (70.09,135.27) -- (99.17,135.44) ;
\draw  [dash pattern={on 4.5pt off 4.5pt}]  (70.12,130.69) -- (126.19,130.86) ;
\draw  [dash pattern={on 4.5pt off 4.5pt}]  (70.8,101.95) -- (156.87,102.12) ;
\draw  [dash pattern={on 4.5pt off 4.5pt}]  (70.09,205.27) -- (99.17,205.44) ;
\draw  [dash pattern={on 4.5pt off 4.5pt}]  (69.8,230.95) -- (96.87,231.12) ;
\draw  [dash pattern={on 4.5pt off 4.5pt}]  (69.69,322.86) -- (96.77,323.03) ;
\draw  [dash pattern={on 4.5pt off 4.5pt}]  (70.09,350.27) -- (99.17,350.44) ;
\draw  [dash pattern={on 4.5pt off 4.5pt}]  (69.8,375.95) -- (96.87,376.12) ;
\draw  [dash pattern={on 4.5pt off 4.5pt}]  (69.69,467.86) -- (96.77,468.03) ;
\draw  [dash pattern={on 4.5pt off 4.5pt}]  (70.42,604.72) -- (126.49,604.88) ;
\draw  [dash pattern={on 4.5pt off 4.5pt}]  (96.59,76.45) -- (96.87,162.12) ;
\draw  [dash pattern={on 4.5pt off 4.5pt}]  (125.57,76.77) -- (125.85,135.44) ;
\draw  [dash pattern={on 4.5pt off 4.5pt}]  (189.84,75.19) -- (190.12,131.86) ;
\draw  [dash pattern={on 4.5pt off 4.5pt}]  (129.91,75.75) -- (130.19,104.42) ;
\draw  [dash pattern={on 4.5pt off 4.5pt}]  (198.91,76.75) -- (199.19,105.42) ;
\draw  [dash pattern={on 4.5pt off 4.5pt}]  (156.59,75.75) -- (156.87,104.42) ;
\draw  [dash pattern={on 4.5pt off 4.5pt}]  (225.59,76.75) -- (225.87,105.42) ;
\draw  [dash pattern={on 4.5pt off 4.5pt}]  (317.33,76.54) -- (317.61,102.21) ;
\draw  [dash pattern={on 4.5pt off 4.5pt}]  (343.91,76.75) -- (344.19,105.42) ;
\draw  [dash pattern={on 4.5pt off 4.5pt}]  (462.33,76.54) -- (462.61,103.21) ;
\draw  [dash pattern={on 4.5pt off 4.5pt}]  (370.59,76.75) -- (370.87,105.42) ;
\draw  [dash pattern={on 4.5pt off 4.5pt}]  (599.93,76.07) -- (600.17,132.2) ;
\draw   (39.74,75.92) .. controls (36.22,75.9) and (34.45,77.65) .. (34.43,81.18) -- (34.43,81.18) .. controls (34.41,86.21) and (32.64,88.72) .. (29.11,88.7) .. controls (32.64,88.72) and (34.39,91.25) .. (34.36,96.28)(34.37,94.01) -- (34.36,96.28) .. controls (34.35,99.8) and (36.1,101.57) .. (39.62,101.58) ;
\draw   (40.74,101.92) .. controls (36.81,101.9) and (34.83,103.86) .. (34.81,107.79) -- (34.81,107.79) .. controls (34.79,113.42) and (32.81,116.22) .. (28.88,116.2) .. controls (32.81,116.22) and (34.77,119.04) .. (34.75,124.66)(34.76,122.13) -- (34.75,124.66) .. controls (34.73,128.59) and (36.69,130.57) .. (40.63,130.58) ;
\draw   (71.74,161.92) .. controls (67.78,161.86) and (65.77,163.81) .. (65.71,167.78) -- (65.71,167.78) .. controls (65.63,173.44) and (63.61,176.24) .. (59.65,176.18) .. controls (63.61,176.24) and (65.55,179.1) .. (65.47,184.76)(65.51,182.21) -- (65.47,184.76) .. controls (65.42,188.72) and (67.37,190.73) .. (71.33,190.79) ;
\draw   (71.33,135.79) .. controls (67.79,135.83) and (66.04,137.62) .. (66.08,141.16) -- (66.08,141.16) .. controls (66.14,146.22) and (64.4,148.77) .. (60.86,148.81) .. controls (64.4,148.77) and (66.2,151.28) .. (66.25,156.33)(66.23,154.06) -- (66.25,156.33) .. controls (66.29,159.88) and (68.08,161.63) .. (71.63,161.58) ;
\draw   (71.33,204.79) .. controls (67.79,204.83) and (66.04,206.62) .. (66.08,210.16) -- (66.08,210.16) .. controls (66.14,215.22) and (64.4,217.77) .. (60.86,217.81) .. controls (64.4,217.77) and (66.2,220.28) .. (66.25,225.33)(66.23,223.06) -- (66.25,225.33) .. controls (66.29,228.88) and (68.08,230.63) .. (71.63,230.58) ;
\draw   (71.74,230.92) .. controls (67.78,230.86) and (65.77,232.81) .. (65.71,236.78) -- (65.71,236.78) .. controls (65.63,242.44) and (63.61,245.24) .. (59.65,245.18) .. controls (63.61,245.24) and (65.55,248.1) .. (65.47,253.76)(65.51,251.21) -- (65.47,253.76) .. controls (65.42,257.72) and (67.37,259.73) .. (71.33,259.79) ;
\draw   (71.74,259.92) .. controls (67.07,259.95) and (64.76,262.3) .. (64.79,266.97) -- (64.91,281.39) .. controls (64.96,288.06) and (62.65,291.41) .. (57.98,291.45) .. controls (62.65,291.41) and (65.01,294.72) .. (65.06,301.39)(65.04,298.39) -- (65.17,315.82) .. controls (65.2,320.49) and (67.55,322.8) .. (72.22,322.76) ;
\draw   (71.33,349.79) .. controls (67.79,349.83) and (66.04,351.62) .. (66.08,355.16) -- (66.08,355.16) .. controls (66.14,360.22) and (64.4,362.77) .. (60.86,362.81) .. controls (64.4,362.77) and (66.2,365.28) .. (66.25,370.33)(66.23,368.06) -- (66.25,370.33) .. controls (66.29,373.88) and (68.08,375.63) .. (71.63,375.58) ;
\draw   (71.74,375.92) .. controls (67.78,375.86) and (65.77,377.81) .. (65.71,381.78) -- (65.71,381.78) .. controls (65.63,387.44) and (63.61,390.24) .. (59.65,390.18) .. controls (63.61,390.24) and (65.55,393.1) .. (65.47,398.76)(65.51,396.21) -- (65.47,398.76) .. controls (65.42,402.72) and (67.37,404.73) .. (71.33,404.79) ;
\draw   (71.74,404.92) .. controls (67.07,404.95) and (64.76,407.3) .. (64.79,411.97) -- (64.91,426.39) .. controls (64.96,433.06) and (62.65,436.41) .. (57.98,436.45) .. controls (62.65,436.41) and (65.01,439.72) .. (65.06,446.39)(65.04,443.39) -- (65.17,460.82) .. controls (65.2,465.49) and (67.55,467.8) .. (72.22,467.76) ;
\draw   (71.74,467.92) .. controls (67.07,467.91) and (64.73,470.23) .. (64.72,474.9) -- (64.61,526.1) .. controls (64.6,532.77) and (62.26,536.1) .. (57.59,536.09) .. controls (62.26,536.1) and (64.58,539.43) .. (64.56,546.1)(64.57,543.1) -- (64.45,597.3) .. controls (64.44,601.97) and (66.76,604.31) .. (71.43,604.32) ;
\draw   (156.85,76.63) .. controls (156.85,72.92) and (155,71.07) .. (151.29,71.07) -- (151.29,71.07) .. controls (146,71.07) and (143.35,69.22) .. (143.35,65.52) .. controls (143.35,69.22) and (140.7,71.07) .. (135.41,71.07)(137.79,71.07) -- (135.41,71.07) .. controls (131.7,71.07) and (129.85,72.92) .. (129.85,76.63) ;
\draw   (185.85,76.63) .. controls (185.88,72.58) and (183.86,70.55) .. (179.81,70.52) -- (179.81,70.52) .. controls (174.03,70.49) and (171.15,68.45) .. (171.18,64.4) .. controls (171.15,68.45) and (168.25,70.45) .. (162.47,70.42)(165.07,70.43) -- (162.47,70.42) .. controls (158.42,70.39) and (156.39,72.4) .. (156.36,76.45) ;
\draw   (316.85,77.63) .. controls (316.86,72.96) and (314.54,70.62) .. (309.87,70.61) -- (295.62,70.57) .. controls (288.95,70.55) and (285.63,68.21) .. (285.65,63.54) .. controls (285.63,68.21) and (282.29,70.53) .. (275.62,70.51)(278.62,70.52) -- (261.38,70.47) .. controls (256.71,70.46) and (254.37,72.78) .. (254.36,77.45) ;
\draw   (225.85,76.63) .. controls (225.85,72.92) and (224,71.07) .. (220.29,71.07) -- (220.29,71.07) .. controls (215,71.07) and (212.35,69.22) .. (212.35,65.52) .. controls (212.35,69.22) and (209.7,71.07) .. (204.41,71.07)(206.79,71.07) -- (204.41,71.07) .. controls (200.7,71.07) and (198.85,72.92) .. (198.85,76.63) ;
\draw   (254.85,76.63) .. controls (254.88,72.58) and (252.86,70.55) .. (248.81,70.52) -- (248.81,70.52) .. controls (243.03,70.49) and (240.15,68.45) .. (240.18,64.4) .. controls (240.15,68.45) and (237.25,70.45) .. (231.47,70.42)(234.07,70.43) -- (231.47,70.42) .. controls (227.42,70.39) and (225.39,72.4) .. (225.36,76.45) ;
\draw   (461.85,76.63) .. controls (461.86,71.96) and (459.54,69.62) .. (454.87,69.61) -- (440.62,69.57) .. controls (433.95,69.55) and (430.63,67.21) .. (430.65,62.54) .. controls (430.63,67.21) and (427.29,69.53) .. (420.62,69.51)(423.62,69.52) -- (406.38,69.47) .. controls (401.71,69.46) and (399.37,71.78) .. (399.36,76.45) ;
\draw   (370.85,75.63) .. controls (370.85,71.92) and (369,70.07) .. (365.29,70.07) -- (365.29,70.07) .. controls (360,70.07) and (357.35,68.22) .. (357.35,64.52) .. controls (357.35,68.22) and (354.7,70.07) .. (349.41,70.07)(351.79,70.07) -- (349.41,70.07) .. controls (345.7,70.07) and (343.85,71.92) .. (343.85,75.63) ;
\draw   (399.85,76.63) .. controls (399.88,72.58) and (397.86,70.55) .. (393.81,70.52) -- (393.81,70.52) .. controls (388.03,70.49) and (385.15,68.45) .. (385.18,64.4) .. controls (385.15,68.45) and (382.25,70.45) .. (376.47,70.42)(379.07,70.43) -- (376.47,70.42) .. controls (372.42,70.39) and (370.39,72.4) .. (370.36,76.45) ;
\draw   (599.37,76.3) .. controls (599.36,71.63) and (597.03,69.3) .. (592.36,69.31) -- (540.86,69.36) .. controls (534.19,69.37) and (530.86,67.04) .. (530.85,62.37) .. controls (530.86,67.04) and (527.53,69.37) .. (520.86,69.38)(523.86,69.38) -- (469.35,69.44) .. controls (464.68,69.44) and (462.35,71.77) .. (462.36,76.44) ;
\draw    (39,76) -- (71.49,76.25) ;
\draw    (39.31,101.7) -- (70.8,101.95) ;
\draw    (39.63,130.44) -- (70.12,130.69) ;
\draw    (39.63,194.37) -- (70.12,194.62) ;
\draw   (39.33,130.8) .. controls (34.66,130.83) and (32.34,133.17) .. (32.37,137.84) -- (32.45,152.51) .. controls (32.48,159.18) and (30.17,162.52) .. (25.5,162.54) .. controls (30.17,162.52) and (32.52,165.84) .. (32.55,172.51)(32.54,169.51) -- (32.63,187.17) .. controls (32.66,191.84) and (35,194.16) .. (39.66,194.14) ;
\draw    (39.2,331.61) -- (69.69,331.86) ;
\draw   (39.67,194.14) .. controls (35,194.15) and (32.67,196.48) .. (32.68,201.15) -- (32.75,252.63) .. controls (32.76,259.3) and (30.43,262.63) .. (25.76,262.64) .. controls (30.43,262.63) and (32.77,265.96) .. (32.78,272.63)(32.77,269.63) -- (32.85,324.12) .. controls (32.86,328.79) and (35.19,331.12) .. (39.86,331.11) ;
\draw    (39.27,623.24) -- (70.42,623.72) ;
\draw   (39.86,331.47) .. controls (35.19,331.47) and (32.86,333.8) .. (32.86,338.47) -- (32.86,467.29) .. controls (32.86,473.96) and (30.53,477.29) .. (25.86,477.29) .. controls (30.53,477.29) and (32.86,480.62) .. (32.86,487.29)(32.86,484.29) -- (32.86,616.11) .. controls (32.86,620.78) and (35.19,623.11) .. (39.86,623.11) ;
\draw  [dash pattern={on 4.5pt off 4.5pt}]  (70.42,623.72) -- (325.28,623.62) ;
\draw  [dash pattern={on 4.5pt off 4.5pt}]  (71.3,331.97) -- (189.17,331.88) ;
\draw  [dash pattern={on 4.5pt off 4.5pt}]  (326,76.47) -- (327.17,194.79) ;
\draw  [dash pattern={on 4.5pt off 4.5pt}]  (617.35,76.56) -- (618.52,331.88) ;
\draw   (95.74,51.72) .. controls (95.71,48.19) and (93.93,46.45) .. (90.41,46.48) -- (90.41,46.48) .. controls (85.38,46.53) and (82.84,44.79) .. (82.81,41.27) .. controls (82.84,44.79) and (80.34,46.57) .. (75.31,46.62)(77.57,46.6) -- (75.31,46.62) .. controls (71.78,46.65) and (70.04,48.43) .. (70.07,51.95) ;
\draw   (124.76,52.07) .. controls (124.72,48.14) and (122.73,46.19) .. (118.8,46.23) -- (118.8,46.23) .. controls (113.18,46.28) and (110.35,44.34) .. (110.31,40.41) .. controls (110.35,44.34) and (107.56,46.34) .. (101.93,46.39)(104.46,46.37) -- (101.93,46.39) .. controls (98,46.43) and (96.05,48.42) .. (96.09,52.35) ;
\draw    (96.04,52.15) -- (96.59,76.45) ;
\draw    (124.97,51.64) -- (125.13,76.13) ;
\draw    (70.32,51.76) -- (70.49,76.25) ;
\draw    (189.67,50.7) -- (189.84,76.19) ;
\draw   (189.33,52.8) .. controls (189.31,48.13) and (186.97,45.81) .. (182.3,45.84) -- (167.15,45.92) .. controls (160.48,45.96) and (157.14,43.65) .. (157.11,38.98) .. controls (157.14,43.65) and (153.82,46) .. (147.15,46.03)(150.15,46.01) -- (132,46.11) .. controls (127.33,46.14) and (125.01,48.48) .. (125.04,53.15) ;
\draw    (325.83,49.98) -- (326,76.47) ;
\draw   (325.33,52.8) .. controls (325.35,48.13) and (323.03,45.79) .. (318.36,45.78) -- (267.73,45.61) .. controls (261.06,45.59) and (257.74,43.25) .. (257.76,38.58) .. controls (257.74,43.25) and (254.4,45.57) .. (247.73,45.54)(250.73,45.55) -- (197.11,45.38) .. controls (192.44,45.36) and (190.1,47.68) .. (190.09,52.35) ;
\draw    (617.18,50.07) -- (617.35,76.56) ;
\draw   (617.16,52.09) .. controls (617.15,47.42) and (614.82,45.09) .. (610.15,45.1) -- (481.62,45.21) .. controls (474.95,45.22) and (471.62,42.89) .. (471.61,38.22) .. controls (471.62,42.89) and (468.29,45.22) .. (461.62,45.23)(464.62,45.23) -- (333.08,45.34) .. controls (328.41,45.35) and (326.08,47.68) .. (326.09,52.35) ;

\draw (158.65,163.32) node    {$T_1( e_{\xi _{i_{3}} ,\eta _{i_{2}}})$};
\draw (225.07,263.31) node    {$T_1( e_{\xi _{i_{4}} ,\eta _{i_{2}}})$};
\draw (290.07,264.31) node    {$T_1( e_{\xi _{i_{4}} ,\eta _{i_{3}}})$};
\draw (385.07,478.31) node    {$T_1( e_{\xi _{i_{5}} ,\eta _{i_{2}}})$};
\draw (470.34,478.06) node    {$T_1( e_{\xi _{i_{5}} ,\eta _{i_{3}}})$};
\draw (559.07,478.31) node    {$T_1( e_{\xi _{i_{5}} ,\eta _{i_{4}}})$};
\draw (113.36,149.44) node  [font=\normalsize]  {\scalebox{0.7}{$x_{ _{(3,1),2}}$}};
\draw (637.88,642.79) node  [font=\huge,rotate=-45]  {$\cdots $};
\draw (112.36,176.61) node  [font=\normalsize]  {\scalebox{0.7}{$x_{_{( 3,2) ,2}}$}};
\draw (111.06,246.42) node  [font=\normalsize]  {\scalebox{0.7}{$x_{_{( 4,2) ,2}}$}};
\draw (156.06,290.42) node  [font=\normalsize]  {\scalebox{0.7}{$x_{_{( 4,3) ,3}}$}};
\draw (255.06,533.42) node  [font=\normalsize]  {\scalebox{0.7}{$x_{_{( 5,4) ,4}}$}};
\draw (157.32,436.57) node  [font=\normalsize]  {\scalebox{0.7}{$x_{_{( 5,3) ,3}}$}};
\draw (111.06,390.42) node  [font=\normalsize]  {\scalebox{0.7}{$x_{_{( 5,2) ,2}}$}};
\draw (160.71,217.87) node  [font=\normalsize]  {\scalebox{0.7}{$x_{_{( 4,1) ,3}}$}};
\draw (112.06,219.42) node  [font=\normalsize]  {\scalebox{0.7}{$x_{_{( 4,1) ,2}}$}};
\draw (168.56,246.21) node  [font=\normalsize]  {\scalebox{0.7}{$x_{_{( 4,2) ,3}}$}};
\draw (84.53,177.76) node  [font=\normalsize,xslant=-0.06]  {\scalebox{0.7}{$x_{_{( 3,2) ,1}}$}};
\draw (83.53,246.76) node  [font=\normalsize]  {\scalebox{0.7}{$x_{_{( 4,2) ,1}}$}};
\draw (84.43,295.13) node  [font=\normalsize]  {\scalebox{0.7}{$x_{_{( 4,3) ,1}}$}};
\draw (111.31,302.41) node  [font=\normalsize]  {\scalebox{0.7}{$x_{_{( 4,3) ,2}}$}};
\draw (113.06,363.42) node  [font=\normalsize]  {\scalebox{0.7}{$x_{_{( 5,1) ,2}}$}};
\draw (161.06,363.42) node  [font=\normalsize]  {\scalebox{0.7}{$x_{_{( 5,1) ,3}}$}};
\draw (167.99,390.75) node  [font=\normalsize]  {\scalebox{0.7}{$x_{_{( 5,2) ,3}}$}};
\draw (261.56,391.21) node  [font=\normalsize]  {\scalebox{0.7}{$x_{_{( 5,2) ,4}}$}};
\draw (254.71,362.87) node  [font=\normalsize]  {\scalebox{0.7}{$x_{_{( 5,1) ,4}}$}};
\draw (289.56,436.21) node  [font=\normalsize]  {\scalebox{0.7}{$x_{_{( 5,3) ,4}}$}};
\draw (111.31,447.41) node  [font=\normalsize]  {\scalebox{0.7}{$x_{_{( 5,3) ,2}}$}};
\draw (83.06,440.42) node  [font=\normalsize]  {\scalebox{0.7}{$x_{_{( 5,3) ,1}}$}};
\draw (83.53,391.76) node  [font=\normalsize]  {\scalebox{0.7}{$x_{_{( 5,2) ,1}}$}};
\draw (85.06,533.42) node  [font=\normalsize]  {\scalebox{0.7}{$x_{_{( 5,4) ,1}}$}};
\draw (112.06,540.42) node  [font=\normalsize]  {\scalebox{0.7}{$x_{_{( 5,4) ,2}}$}};
\draw (157.78,568.46) node  [font=\normalsize]  {\scalebox{0.7}{$x_{_{( 5,4) ,3}}$}};
\draw (143.91,119.1) node  [font=\normalsize]  {\scalebox{0.7}{$y_{_{2,( 1,3)}}$}};
\draw (171.36,116.61) node  [font=\normalsize]  {\scalebox{0.7}{$y_{_{2,( 2,3)}}$}};
\draw (239.91,118.1) node  [font=\normalsize]  {\scalebox{0.7}{$y_{_{2,( 2,4)}}$}};
\draw (287.32,163.57) node  [font=\normalsize]  {\scalebox{0.7}{$y_{_{3,( 3,4)}}$}};
\draw (385.36,117.61) node  [font=\normalsize]  {\scalebox{0.7}{$y_{_{2,( 2,5)}}$}};
\draw (432.06,164.12) node  [font=\normalsize]  {\scalebox{0.7}{$y_{_{3,( 3,5)}}$}};
\draw (531.06,261.12) node  [font=\normalsize]  {\scalebox{0.7}{$y_{_{4,( 4,5)}}$}};
\draw (172.91,90.1) node  [font=\normalsize]  {\scalebox{0.7}{$y_{_{1,( 2,3)}}$}};
\draw (241.79,89.75) node  [font=\normalsize]  {\scalebox{0.7}{$y_{_{1,( 2,4)}}$}};
\draw (290.56,89.52) node  [font=\normalsize]  {\scalebox{0.7}{$y_{_{1,( 3,4)}}$}};
\draw (434.71,89.87) node  [font=\normalsize]  {\scalebox{0.7}{$y_{_{1,( 3,5)}}$}};
\draw (386.79,89.75) node  [font=\normalsize]  {\scalebox{0.7}{$y_{_{1,( 2,5)}}$}};
\draw (528.71,89.87) node  [font=\normalsize]  {\scalebox{0.7}{$y_{_{1,( 4,5)}}$}};
\draw (536.56,118.21) node  [font=\normalsize]  {\scalebox{0.7}{$y_{_{2,( 4,5)}}$}};
\draw (563.56,163.21) node  [font=\normalsize]  {\scalebox{0.7}{$y_{_{3,( 4,5)}}$}};
\draw (442.56,118.21) node  [font=\normalsize]  {\scalebox{0.7}{$y_{_{2,( 3,5)}}$}};
\draw (357.53,118.76) node  [font=\normalsize]  {\scalebox{0.7}{$y_{_{2,( 1,5)}}$}};
\draw (358.28,167.79) node  [font=\normalsize]  {\scalebox{0.7}{$y_{_{3,( 1,5)}}$}};
\draw (385.28,173.79) node  [font=\normalsize]  {\scalebox{0.7}{$y_{_{3,( 2,5)}}$}};
\draw (358.28,260.79) node  [font=\normalsize]  {\scalebox{0.7}{$y_{_{4,( 1,5)}}$}};
\draw (386.28,267.79) node  [font=\normalsize]  {\scalebox{0.7}{$y_{_{4,( 2,5)}}$}};
\draw (432.28,296.79) node  [font=\normalsize]  {\scalebox{0.7}{$y_{_{4,( 3,5)}}$}};
\draw (240.91,175.1) node  [font=\normalsize]  {\scalebox{0.7}{$y_{_{3,( 2,4)}}$}};
\draw (213.28,167.79) node  [font=\normalsize]  {\scalebox{0.7}{$y_{_{3,( 1,4)}}$}};
\draw (212.91,119.1) node  [font=\normalsize]  {\scalebox{0.7}{$y_{_{2,( 1,4)}}$}};
\draw (296.99,116.75) node  [font=\normalsize]  {\scalebox{0.7}{$y_{_{2,( 3,4)}}$}};
\draw (47.98,149.45) node    {$F_{3,1}$};
\draw (47.98,176.45) node    {$F_{3,2}$};
\draw (47.98,217.45) node    {$F_{4,1}$};
\draw (47.98,244.45) node    {$F_{4,2}$};
\draw (47.98,291.45) node    {$F_{4,3}$};
\draw (47.98,362.45) node    {$F_{5,1}$};
\draw (47.98,389.45) node    {$F_{5,2}$};
\draw (47.98,436.45) node    {$F_{5,3}$};
\draw (47.98,537.45) node    {$F_{5,4}$};
\draw (16.5,262.45) node    {$K'_{4}$};
\draw (16.5,162.45) node    {$K'_{3}$};
\draw (18.5,115.45) node    {$K'_{2}$};
\draw (19,88.45) node    {$K'_{1}$};
\draw (16.5,477.45) node    {$K'_{5}$};
\draw (142.98,56.45) node    {$G_{1,3}$};
\draw (171.98,56.45) node    {$G_{2,3}$};
\draw (284.98,55.45) node    {$G_{3,4}$};
\draw (211.98,56.45) node    {$G_{1,4}$};
\draw (240.98,56.45) node    {$G_{2,4}$};
\draw (427.98,54.45) node    {$G_{3,5}$};
\draw (355.98,55.45) node    {$G_{1,5}$};
\draw (384.98,55.45) node    {$G_{2,5}$};
\draw (529.98,54.45) node    {$G_{3,5}$};
\draw (157.47,31) node    {$L'_{3}$};
\draw (110.47,31) node    {$L'_{2}$};
\draw (82.47,31) node    {$L'_{1}$};
\draw (257.47,31) node    {$L'_{4}$};
\draw (471.47,31) node    {$L'_{4}$};

\draw (0,678.87) node [anchor=north west][inner sep=0.75pt]    {The operator $T_1$ (respectively, $T_2$, $T_3$, $T_4$) corresponds to the blue part  (respectively, green part, yellow part, red part).};

\end{tikzpicture}

}

\begin{itemize}
\item [(1)]
$K'_k=[\bigcup_{i=n_{k-1}+1}^{n_k}K_i]$ and $L'_k=[\bigcup_{i=n_{k-1}+1}^{n_k}L_i]$ for every $k\geq1$ (put $n_0=0$),
\item [(2)]
$\{F_{k,l}\}_{1\leq l<k}$ is a finite sequence of mutually orthogonal subspaces of $K'_k$ for every $k$,
\item [(3)]
$\{G_{l,k}\}_{1\leq l<k}$ is a finite sequence of mutually orthogonal subspaces of $L'_k$ for every $k$,
\item [(4)]
$\big\{x_{_{(k,l),m}}\big\}_{
\substack{ _{1\leq l<k<\infty}  \\ _{1\leq m<k<\infty}}}
\stackrel{\mbox{{\rm\tiny b.s}}}{\boxplus}\{F_{k,l}\}_{1\leq l<k<\infty}\otimes\{L'_m\}_{m=1}^\infty$ such that
$x_{_{(2,1),1}}=0$,
\begin{equation}\label{hatc}
\big\{x_{_{(k,l),m}}\big\}_{2\leq m=l<k<\infty}
\stackrel{\mbox{{\rm\tiny b.s}}}{\boxplus}\big\{F_{k,l}\big\}_{2\leq l<k<\infty}\otimes\big\{L'_m\big\}_{m=2}^\infty\curvearrowleft\hat{c}_l\;\;\;\;
\end{equation}
  are consistent\footnote{Here, $\hat{c}_l$ is an operator on a Hilbert space $\hat{H}$, and there are two sequences of isometries $\big\{\hat{u}_{k,l}:F_{k,l}\to\hat{H}\big\}_{2\leq l<k<\infty}$ and $\big\{\hat{v}_l:L'_m\to\hat{H}\big\}_{m=2}^\infty$ such that for all $2\leq m=l<k<\infty$,
$$\hat{u}_{k,m}x_{_{(k,l),m}}\hat{v}_m=\hat{u}_{k,l}x_{_{(k,l),l}}\hat{v}_l=\hat{c}_l.$$} for all $l\geq2$, 
\begin{equation}\label{hatb}
\big\{x_{_{(k,l),m}}\big\}_{1\leq m< l<k<\infty}\stackrel{\mbox{{\rm\tiny b.s}}}{\boxplus}\big\{F_{k,l}\big\}_{2\leq l<k<\infty}\otimes\big\{L'_m\big\}_{m=1}^\infty\curvearrowleft\hat{b}_m\;\;\;\;
\end{equation}
are consistent for all $m\geq1$,
\begin{equation}\label{hata}
\big\{x_{_{(k,m),l}}\big\}_{1\leq m< l<k<\infty}
\stackrel{\mbox{{\rm\tiny b.s}}}{\boxplus}\big\{F_{k,m}\big\}_{1\leq m<k-1<\infty}\otimes\big\{L'_l\big\}_{l=2}^\infty\curvearrowleft\hat{a}_m\;\;\;\;
\end{equation}
are consistent for all $m\geq1$,
and $\Vert\hat{b}_m\Vert_{\mathcal C_E}\leq\varepsilon_m\Vert\hat{b}_1\Vert_{\mathcal C_E}$ and $\Vert\hat{a}_m\Vert_{\mathcal C_E}\leq\varepsilon_m\Vert\hat{a}_1\Vert_{\mathcal C_E}$ for every $m\geq 2$,
\item [(5)]
$\big\{y_{_{m,(l,k)}}\big\}_{
\substack{ _{1\leq m<k<\infty}  \\ _{1\leq l<k<\infty} } }
\stackrel{\mbox{{\rm\tiny b.s}}}{\boxplus}\{K'_m\}_{m=1}^\infty\otimes\{G_{l,k}\}_{1\leq l<k<\infty}$ such that $y_{_{1,(1,2)}}=0$,
\begin{equation}\label{tildec}
\big\{y_{_{m,(l,k)}}\big\}_{2\leq m=l<k<\infty}
\stackrel{\mbox{{\rm\tiny b.s}}}{\boxplus}\big\{K'_m\big\}_{m=2}^\infty\otimes\big\{G_{l,k}\big\}_{2\leq l<k<\infty}\curvearrowleft\tilde{c}_l\;\;\;\;
\end{equation}
are consistent for all $l\geq2$,
\begin{equation}\label{tildeb}
\big\{y_{_{m,(l,k)}}\big\}_{1\leq m<l<k<\infty}
\stackrel{\mbox{{\rm\tiny b.s}}}{\boxplus}\big\{K'_m\big\}_{m=1}^\infty\otimes\big\{G_{l,k}\big\}_{2\leq l<k<\infty}\curvearrowleft\tilde{b}_m\;\;\;\;
\end{equation}
are consistent for all $m\geq1$,
\begin{equation}\label{tildea}
\big\{y_{_{l,(m,k)}}\big\}_{1\leq m<l<k<\infty}
\stackrel{\mbox{{\rm\tiny b.s}}}{\boxplus}\big\{K'_l\big\}_{l=2}^\infty\otimes\big\{G_{m,k}\big\}_{1\leq m<k-1<\infty}\curvearrowleft\tilde{a}_m\;\;\;\;
\end{equation}
are consistent for all $m\geq1$,
and $\Vert\tilde{b}_m\Vert_{\mathcal C_E}\leq\varepsilon_m\Vert\tilde{b}_1\Vert_{\mathcal C_E}$ and $\Vert\tilde{a}_m\Vert_{\mathcal C_E}\leq\varepsilon_m\Vert\tilde{a}_1\Vert_{\mathcal C_E}$ for every $m\geq 2$.

\end{itemize}
\end{theorem}

\begin{remark}
Let $\{x_{i,j}\}_{1\leq j\leq i<\infty}$ be a  triangle sequence of a (quasi-)Banach space $X$ and $M$ be a positive number such that
\begin{itemize}
\item [(a)]
$\big\{[x_{i,j}]_{i=j}^\infty\big\}_{j=1}^\infty$ is a Schauder decomposition of $X$,
\item [(b)]
$\{x_{i,j}\}_{i=j}^\infty$ is $M$-equivalent to the unit basis of $\ell_2$ for every $j$, and
\item [(c)]
$\big\Vert\sum_{j=1}^i \alpha_jx_{i,j}\big\Vert\leq M\big(\sum_{j=1}^i\vert\alpha_j\vert^2\big)^{1/2}$  for every $i$ and finite sequence $\{\alpha_j\}_{j=1}^i$ of scalars.
\end{itemize}
Assume that $T$ is a  bounded operator from $X$ into $\mathcal C_E$. The conclusion of Theorem~\ref{thma} remains valid. 
\end{remark}

\begin{remark}\label{tftoce}
Let $F$ be another separable 
quasi-Banach symmetric sequence space. 
If we consider 
a bounded operator $T:\mathcal T_{F,\{\xi_i,\eta_i\}_{i=1}^\infty}\to \mathcal C_E$ such that $p_{_{[\bigcup_{i=1}^\infty K_i]}}T(\cdot)p_{_{[\bigcup_{i=1}^\infty L_i]}}=T$, then 
\begin{itemize}
\item [(1)]
The conclusion of  the above theorem remains true. 
\item [(2)]
If $\{K_i\}_{i=1}^\infty$ and $\{L_i\}_{i=1}^\infty$ are chosen such that $p_{K_i}({\rm im}(T))p_{L_j}=\{0\}$ for every $1\leq i<j<\infty$ (for example,  $K_i=[\xi_i]$, $L_i=[\eta_i]$ and im$(T)\subset[e_{\xi_i,\eta_j}]_{1\leq j\leq i<\infty}$), then those ``$y_{m,(l,k)}$" in \eqref{T2}, \eqref{T3} and \eqref{T4}  are $0$.
\end{itemize}

For convenience, we only consider the case when $E=F$ below.
\end{remark}

\begin{proof}
For convenience, we denote  $e_{i,j}: =(\cdot|\eta_j)\xi_i$ for all $i$ and $j$. Let $\{\delta_j\}_{j=1}^\infty$ and $\{\delta_{k,j}\}_{k,j=1}^{\infty}$ be fixed  sequences of small enough positive numbers. 

Since $\{ e_{i,j} \}_{i=j}^\infty$ is isometric to the standard basis of $\ell_2$ for each $j$, it follows that $\{T(e_{i,j})\}_{i=j}^\infty$ is weakly null for every $j$. By standard perturbation arguments, there are two increasing sequences $\{i_k\}_{k=1}^\infty$ and $\{n_k\}_{k=1}^\infty$ of positive integers such that
\[
\left\Vert T(e_{i_{k+1},j})-\Big(p_{_{[\bigcup_{l=1}^{n_{k+1}}K_l]}}T(e_{i_{k+1},j})p_{_{[\bigcup_{l=1}^{n_{k+1}}L_l]}}-p_{_{[\bigcup_{l=1}^{n_k}K_l]}}T(e_{i_{k+1},j})p_{_{[\bigcup_{l=1}^{n_k}L_l]}}\Big)\right\Vert_{\mathcal C_E}\leq\delta_{k,j},  
\]
for $j=1,2\dots,i_k$ and 
for every $k$.
Therefore, without loss of generality, we may assume that
\begin{align}\label{s2}
T(e_{i_{k+1},j})=p_{_{[\bigcup_{l=1}^{n_{k+1}}K_l]}}T(e_{i_{k+1},j})p_{_{[\bigcup_{l=1}^{n_{k+1}}L_l]}}-p_{_{[\bigcup_{l=1}^{n_k}K_l]}}T(e_{i_{k+1},j})p_{_{[\bigcup_{l=1}^{n_k}L_l]}},
\end{align}
$j=1,2,\dots,i_k$, for every $k$. We put $K'_k=[\bigcup_{l=n_{k-1}+1}^{n_k}K_l]$ and $L'_k=[\bigcup_{l=n_{k-1}+1}^{n_k}L_l]$ for every $k\geq1$ (put $n_0=0$).

The proof of Theorem \ref{thma} is divided into the following 3 lemmas.

\begin{lemma}\label{thma1}
Under the assumption of Theorem \ref{thma}, there exist an increasing sequence $\{i_s\}_{s=1}^\infty$ of positive integers, 
a bounded operator $T_1:[e_{i_s,i_t}]_{1\leq t<s<\infty}\to \mathcal C_E$ with $$T_1(e_{i_s,i_t})=p_{_{K'_s}}T(e_{i_s,i_t})p_{_{L'_s}}$$ for every $1\leq t<s<\infty$ and a bounded operator $S_2:[e_{i_s,i_t}]_{1\leq t<s<\infty}\to \mathcal C_E$  (see  \eqref{s1} below) such that $T_1+S_2$ is a small perturbation of $T|_{[e_{i_s,i_t}]_{1\leq t<s<\infty}}$ associated with the  Schauder decomposition $\big\{[e_{i_s,i_t}]_{s=t+1}^\infty\big\}_{t=1}^\infty$.
\end{lemma}

\begin{proof}
By induction and by Corollary \ref{3e'},   there exist
\begin{itemize}
\item [(1)]
increasing sequences $\{k^{_{(j)}}_s\}_{s=1}^\infty$ of positive integers for $j=1,2,\dots$ and $k^{_{(0)}}_s=s$, such that  $\{k^{_{(j)}}_s\}_{s=1}^\infty$ is a subsequence of $\{k^{_{(j-1)}}_s\}_{s=1}^\infty$ with $k^{_{(j)}}_1>k^{_{(j-1)}}_1$,
\item [(2)]
$R_{j}\in\mathcal B\bigg(\Big[e_{i_{k^{(j)}_s},i_{k^{_{(j-1)}}_2}}\Big]_{s=2}^\infty,\mathcal C_E\bigg)$ for every $j\geq1$,   
\item [(3)]
$x^{(j)}_s,y^{(j)}_s,z^{(j)}_s\in\mathcal C_E$ with $x^{(j)}_s=p_{_{K^{_{(j)}}_s}}x^{(j)}_sp_{_{L^{(j)}_{1}}}$, $y^{(j)}_{s}=p_{_{K^{(j)}_1}}y^{(j)}_{j}p_{_{L^{(j)}_{s}}}$ and $z^{(j)}_s=p_{_{K^{(j)}_s}}T\big(e_{i_{k^{(j)}_s},i_{k^{(j-1)}_2}}\big)p_{_{L^{(j)}_{s}}}$ for every $s\geq2$ and $j\geq1$, where $K^{(j)}_s=\Big[\bigcup_{i=n_{k^{(j)}_{s-1}}+1}^{n_{k^{(j)}_s}}K_i\Big]$ and $L^{(j)}_s=\Big[\bigcup_{i=n_{k^{(j)}_{s-1}}+1}^{n_{k^{(j)}_s}}L_i\Big]$ for every $s\geq1$ and $j\geq1$ (put $k^{(j)}_0=0$)
\end{itemize}
such that 
\begin{align}\label{Rj=}
R_{j}\big(e_{i_{k^{(j)}_s},i_{k^{(j-1)}_2}}\big)=x^{(j)}_{s}+y^{(j)}_s+z^{(j)}_s,\;\;\;\;\;\;s=2,3,\dots,
\end{align}
and $\left\Vert T|_{\big[e_{i_{k^{(j)}_s},i_{k^{(j-1)}_2}}\big]_{s=2}^\infty}-R_{j}\right\Vert\leq\delta_{j}$ for every $j\geq1$ 
\Big(note that if $\Big\{T\big(e_{i_{k^{(j-1)}_s},i_{k^{(j-1)}_2}}\big)\Big\}_{s=2}^\infty$ is semi-normalized, then,  by Corollary \ref{3e'}, for every  sequence $\{t_s\}_{s=2}^\infty$ with finitely nonzero scalars, we may choose a suitable sequence  $\{k^{_{(j)}}_s\}_{s=1}^\infty$ such that  
\begin{align*}
\left\Vert(T-R_j)\bigg(\sum_{s=2}^\infty t_se_{i_{k^{(j)}_s},i_{k^{(j-1)}_2}}\bigg)\right\Vert_{\mathcal C_E}&\leq\frac{\delta_j}{2^{\frac{1}{r}-1}(\Vert T\Vert+1)}\left\Vert T\bigg(\sum_{s=2}^\infty t_se_{i_{k^{(j)}_s},i_{k^{(j-1)}_2}}\bigg)\right\Vert_{\mathcal C_E}\\
&\leq\frac{\delta_j}{2^{\frac{1}{r}-1}}\left\Vert\sum_{s=2}^\infty t_se_{i_{k^{(j)}_s},i_{k^{(j-1)}_2}}\right\Vert_{\mathcal C_E}.
\end{align*}
Otherwise,  we may choose suitable  $\{k^{_{(j)}}_s\}_{s=1}^\infty$ such that $\Big\{T\big(e_{i_{k^{(j)}_s},i_{k^{(j-1)}_2}}\big)\Big\}_{s=2}^\infty$ goes to $0$   in $\cC_E$ (as fast as we need). Hence,  we may assume that  $x^{(j)}_s$'s and $y^{(j)}_s$'s are  $0$.\Big). 


Next, we consider the sequence $\{k^{(s)}_2\}_{s=1}^\infty$. For any $1\leq t< s$, 
there is a  unique positive integer $s'$ such that $k^{(t+1)}_{s'}=k^{(s)}_2$ (obviously, $s'\geq2$), and thus
\[
R_{t+1}\Big(e_{i_{k^{(s)}_2},i_{k^{(t)}_2}}\Big)=R_{t+1}\Big(e_{i_{k^{(t+1)}_{s'}},i_{k^{(t)}_2}}\Big)\stackrel{\eqref{Rj=}}{=}x^{(t+1)}_{s'}+y^{(t+1)}_{s'}+z^{(t+1)}_{s'},
\]
and we write
\[
x_{s,t}=x^{(t+1)}_{s'},~\qquad y_{t,s}=y^{(t+1)}_{s'}\qquad  \mbox{and}\qquad z_{s,t}=z^{(t+1)}_{s'}.
\]
Since $k^{(j)}_1>k^{(j-1)}_1$ for every $j$, it follows that $k^{(s-1)}_2\leq k^{(s)}_1\leq k^{(t+1)}_{s'-1}<k^{(t+1)}_{s'}=k^{(s)}_2\leq k^{(s+1)}_1$. For  convenience, 
without loss of generality, we still write $\{i_{k^{(s)}_2}\}_{s=1}^\infty$ as $\{i_s\}_{s=1}^\infty$ and $\{n_{i_{k^{(s+1)}_1}}\}_{s=0}^\infty$ as $\{n_s\}_{s=0}^\infty$ respectively. 
We have 
\begin{equation}\label{x-y-z-st}
x_{s,t}=p_{_{K'_s}}x_{s,t}p_{_{[\bigcup_{l=1}^{t}L'_l]}}\;\;\;\;\;y_{t,s}=p_{_{[\bigcup_{l=1}^{t}K'_l]}}y_{t,s}p_{_{L'_s}}\;\;\;\;\;z_{s,t}=p_{_{K'_s}}T\big(e_{i_s,i_t}\big)p_{_{L'_s}}.
\end{equation}
By Theorem \ref{2d}, without loss of generality, we may assume that
\begin{equation}\label{T=x+y+z}
T(e_{i_s,i_t})=x_{s,t}+y_{t,s}+z_{s,t},
\end{equation}
for every $1\leq t<s<\infty$.

By Theorem \ref{3b} and Remark \ref{3b'},  passing to a subsequence of $\{i_s\}_{s=1}^\infty$ if necessary,   we may assume, without loss of generality, that 
\begin{itemize}
\item [(1)]
there is a sequence of isometries $\big\{u_s:K'_s\to H'\big\}_{s=2}^\infty$, an isometry $v:L'\to H'$, where $L'=[\bigcup_{l=1}^\infty L'_l]$, and a sequence $\{a_t\}_{t=1}^\infty$ of in $\mathcal C_E(H')$ such that
\begin{align}\label{at}
\{x_{s,t}\}_{s=t+1}^\infty\stackrel{\mbox{{\rm\tiny b.s}}}{\boxplus}\{u_s:K'_s\to H'\}_{s=2}^\infty\otimes\;\{v: L'\to H'\}\curvearrowleft a_t\;
\end{align}
for every $t\geq1$, and
\item [(2)]
there is an isometry $u':K'\to H''$, where $K'=[\bigcup_{l=1}^\infty K'_l]$, a sequence of isometries $\big\{v'_s:L'_s\to H''\big\}_{s=2}^\infty$, and a sequence $\{b_t\}_{t=1}^\infty$ of in $\mathcal C_E(H'')$ such that
\begin{align}\label{bt}
\{y_{t,s}\}_{s=t+1}^\infty\stackrel{\mbox{{\rm\tiny b.s}}}{\boxplus}\{u':K'\to H''\}\otimes\;\{v'_s: L'_{s}\to H''\}_{s=2}^\infty\curvearrowleft b_t\;
\end{align}
for every $t\geq1$.
\end{itemize}

The positions of $x_{s,t}$, $y_{s,t}$ and $z_{s,t}$'s are demonstrated as follows.

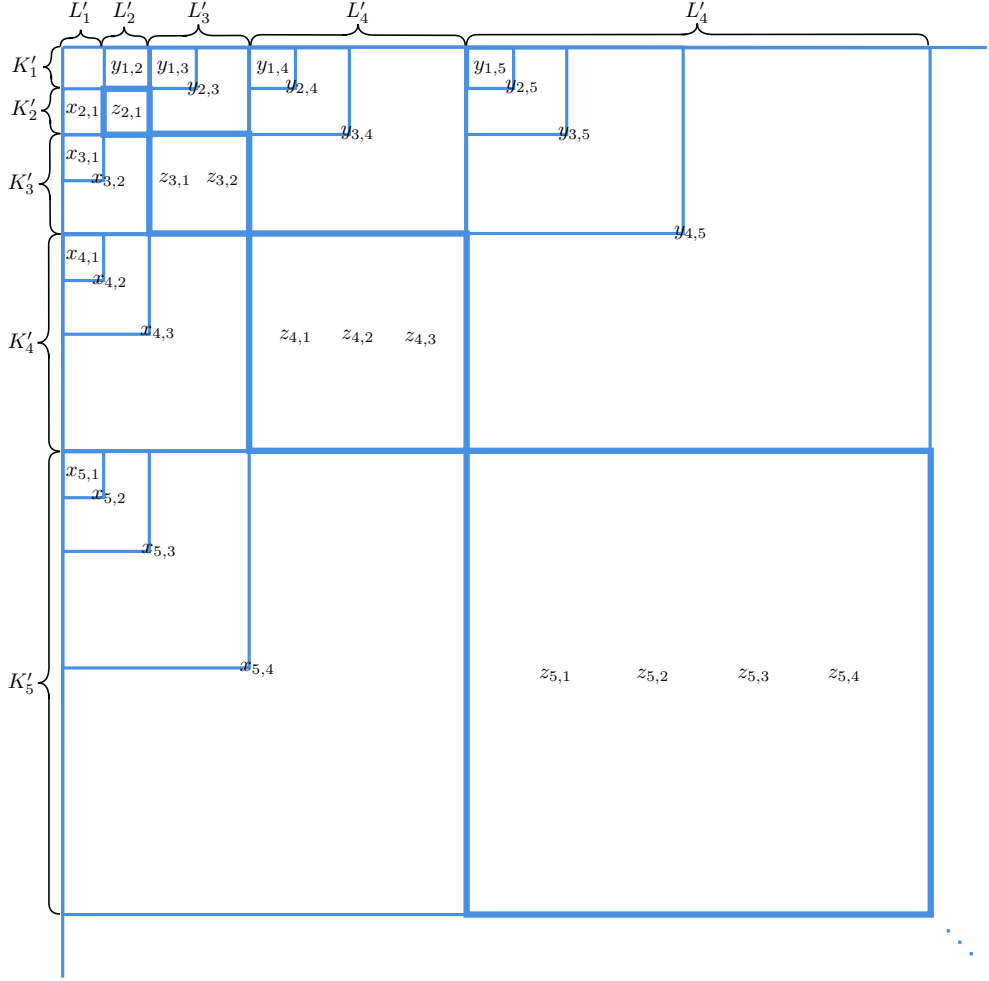
\scalebox{0.8}{

\tikzset{every picture/.style={line width=0.75pt}} 

\begin{tikzpicture}[x=0.75pt,y=0.75pt,yscale=-1,xscale=1]

\draw [color={rgb, 255:red, 74; green, 144; blue, 226 }  ,draw opacity=1 ][line width=1.5]    (51.49,76.25) -- (51.5,663.32) ;
\draw [color={rgb, 255:red, 74; green, 144; blue, 226 }  ,draw opacity=1 ][line width=1.5]    (51.49,76.25) -- (270.22,76.19) -- (634.19,76.45) ;
\draw  [color={rgb, 255:red, 0; green, 0; blue, 0 }  ,draw opacity=1 ] (49.74,75.92) .. controls (46.22,75.9) and (44.45,77.65) .. (44.43,81.18) -- (44.43,81.18) .. controls (44.41,86.21) and (42.64,88.72) .. (39.11,88.7) .. controls (42.64,88.72) and (44.39,91.25) .. (44.36,96.28)(44.37,94.01) -- (44.36,96.28) .. controls (44.35,99.8) and (46.1,101.57) .. (49.62,101.58) ;
\draw  [color={rgb, 255:red, 0; green, 0; blue, 0 }  ,draw opacity=1 ] (49.74,101.92) .. controls (45.81,101.9) and (43.83,103.86) .. (43.81,107.79) -- (43.81,107.79) .. controls (43.79,113.42) and (41.81,116.22) .. (37.88,116.2) .. controls (41.81,116.22) and (43.77,119.04) .. (43.75,124.66)(43.76,122.13) -- (43.75,124.66) .. controls (43.73,128.59) and (45.69,130.57) .. (49.63,130.58) ;
\draw  [color={rgb, 255:red, 0; green, 0; blue, 0 }  ,draw opacity=1 ] (50.33,130.8) .. controls (45.66,130.83) and (43.34,133.17) .. (43.37,137.84) -- (43.45,152.51) .. controls (43.48,159.18) and (41.17,162.52) .. (36.5,162.54) .. controls (41.17,162.52) and (43.52,165.84) .. (43.55,172.51)(43.54,169.51) -- (43.63,187.17) .. controls (43.66,191.84) and (46,194.16) .. (50.66,194.14) ;
\draw  [color={rgb, 255:red, 0; green, 0; blue, 0 }  ,draw opacity=1 ] (49.67,194.14) .. controls (45,194.15) and (42.67,196.48) .. (42.68,201.15) -- (42.75,252.63) .. controls (42.76,259.3) and (40.43,262.63) .. (35.76,262.64) .. controls (40.43,262.63) and (42.77,265.96) .. (42.78,272.63)(42.77,269.63) -- (42.85,324.12) .. controls (42.86,328.79) and (45.19,331.12) .. (49.86,331.11) ;
\draw  [color={rgb, 255:red, 0; green, 0; blue, 0 }  ,draw opacity=1 ] (49.86,331.47) .. controls (45.19,331.47) and (42.86,333.8) .. (42.86,338.47) -- (42.86,467.29) .. controls (42.86,473.96) and (40.53,477.29) .. (35.86,477.29) .. controls (40.53,477.29) and (42.86,480.62) .. (42.86,487.29)(42.86,484.29) -- (42.86,616.11) .. controls (42.86,620.78) and (45.19,623.11) .. (49.86,623.11) ;
\draw  [color={rgb, 255:red, 0; green, 0; blue, 0 }  ,draw opacity=1 ] (75.74,75.72) .. controls (75.71,72.19) and (73.93,70.45) .. (70.41,70.48) -- (70.41,70.48) .. controls (65.38,70.53) and (62.84,68.79) .. (62.81,65.27) .. controls (62.84,68.79) and (60.34,70.57) .. (55.31,70.62)(57.57,70.6) -- (55.31,70.62) .. controls (51.78,70.65) and (50.04,72.43) .. (50.07,75.95) ;
\draw  [color={rgb, 255:red, 0; green, 0; blue, 0 }  ,draw opacity=1 ] (104.76,75.07) .. controls (104.72,71.14) and (102.73,69.19) .. (98.8,69.23) -- (98.8,69.23) .. controls (93.18,69.28) and (90.35,67.34) .. (90.31,63.41) .. controls (90.35,67.34) and (87.56,69.34) .. (81.93,69.39)(84.46,69.37) -- (81.93,69.39) .. controls (78,69.43) and (76.05,71.42) .. (76.09,75.35) ;
\draw  [color={rgb, 255:red, 0; green, 0; blue, 0 }  ,draw opacity=1 ] (169.33,75.8) .. controls (169.31,71.13) and (166.97,68.81) .. (162.3,68.84) -- (147.15,68.92) .. controls (140.48,68.96) and (137.14,66.65) .. (137.11,61.98) .. controls (137.14,66.65) and (133.82,69) .. (127.15,69.03)(130.15,69.01) -- (112,69.11) .. controls (107.33,69.14) and (105.01,71.48) .. (105.04,76.15) ;
\draw  [color={rgb, 255:red, 0; green, 0; blue, 0 }  ,draw opacity=1 ] (305.33,75.8) .. controls (305.35,71.13) and (303.03,68.79) .. (298.36,68.78) -- (247.73,68.61) .. controls (241.06,68.59) and (237.74,66.25) .. (237.76,61.58) .. controls (237.74,66.25) and (234.4,68.57) .. (227.73,68.54)(230.73,68.55) -- (177.11,68.38) .. controls (172.44,68.36) and (170.1,70.68) .. (170.09,75.35) ;
\draw  [color={rgb, 255:red, 0; green, 0; blue, 0 }  ,draw opacity=1 ] (597.16,75.09) .. controls (597.15,70.42) and (594.82,68.09) .. (590.15,68.1) -- (461.62,68.21) .. controls (454.95,68.22) and (451.62,65.89) .. (451.61,61.22) .. controls (451.62,65.89) and (448.29,68.22) .. (441.62,68.23)(444.62,68.23) -- (313.08,68.34) .. controls (308.41,68.35) and (306.08,70.68) .. (306.09,75.35) ;
\draw  [color={rgb, 255:red, 74; green, 144; blue, 226 }  ,draw opacity=1 ][line width=1.5]  (51.66,102.48) -- (77.33,102.48) -- (77.33,131.48) -- (51.66,131.48) -- cycle ;
\draw  [color={rgb, 255:red, 74; green, 144; blue, 226 }  ,draw opacity=1 ][line width=1.5]  (51.66,131.48) -- (77.33,131.48) -- (77.33,160.48) -- (51.66,160.48) -- cycle ;
\draw  [color={rgb, 255:red, 74; green, 144; blue, 226 }  ,draw opacity=1 ][line width=1.5]  (51.66,194.48) -- (77.33,194.48) -- (77.33,223.48) -- (51.66,223.48) -- cycle ;
\draw  [color={rgb, 255:red, 74; green, 144; blue, 226 }  ,draw opacity=1 ][line width=1.5]  (106.78,76.63) -- (106.78,102.3) -- (77.78,102.3) -- (77.78,76.63) -- cycle ;
\draw  [color={rgb, 255:red, 74; green, 144; blue, 226 }  ,draw opacity=1 ][line width=3]  (77.33,102.48) -- (106.33,102.48) -- (106.33,131.48) -- (77.33,131.48) -- cycle ;
\draw  [color={rgb, 255:red, 74; green, 144; blue, 226 }  ,draw opacity=1 ][line width=3]  (106.14,130.86) -- (169.07,130.86) -- (169.07,193.79) -- (106.14,193.79) -- cycle ;
\draw  [color={rgb, 255:red, 74; green, 144; blue, 226 }  ,draw opacity=1 ][line width=3]  (169.07,193.79) -- (306.17,193.79) -- (306.17,330.88) -- (169.07,330.88) -- cycle ;
\draw  [color={rgb, 255:red, 74; green, 144; blue, 226 }  ,draw opacity=1 ][line width=3]  (306.17,330.88) -- (598.76,330.88) -- (598.76,623.47) -- (306.17,623.47) -- cycle ;
\draw  [color={rgb, 255:red, 74; green, 144; blue, 226 }  ,draw opacity=1 ] (51.48,130.86) -- (106.14,130.86) -- (106.14,194.34) -- (51.48,194.34) -- cycle ;
\draw  [color={rgb, 255:red, 74; green, 144; blue, 226 }  ,draw opacity=1 ][line width=1.5]  (169.26,76.63) -- (169.26,131.3) -- (105.78,131.3) -- (105.78,76.63) -- cycle ;
\draw  [color={rgb, 255:red, 74; green, 144; blue, 226 }  ,draw opacity=1 ][line width=1.5]  (135.78,76.63) -- (135.78,102.3) -- (106.78,102.3) -- (106.78,76.63) -- cycle ;
\draw  [color={rgb, 255:red, 74; green, 144; blue, 226 }  ,draw opacity=1 ][line width=1.5]  (51.48,193.86) -- (106.14,193.86) -- (106.14,257.34) -- (51.48,257.34) -- cycle ;
\draw  [color={rgb, 255:red, 74; green, 144; blue, 226 }  ,draw opacity=1 ][line width=1.5]  (51.66,194.48) -- (169.26,194.48) -- (169.26,331.02) -- (51.66,331.02) -- cycle ;
\draw  [color={rgb, 255:red, 74; green, 144; blue, 226 }  ,draw opacity=1 ][line width=1.5]  (305.61,76.19) -- (305.61,193.79) -- (169.07,193.79) -- (169.07,76.19) -- cycle ;
\draw  [color={rgb, 255:red, 74; green, 144; blue, 226 }  ,draw opacity=1 ][line width=1.5]  (51.66,331.48) -- (77.33,331.48) -- (77.33,360.48) -- (51.66,360.48) -- cycle ;
\draw  [color={rgb, 255:red, 74; green, 144; blue, 226 }  ,draw opacity=1 ][line width=1.5]  (51.48,330.86) -- (106.14,330.86) -- (106.14,394.34) -- (51.48,394.34) -- cycle ;
\draw  [color={rgb, 255:red, 74; green, 144; blue, 226 }  ,draw opacity=1 ][line width=1.5]  (51.48,331.34) -- (169.07,331.34) -- (169.07,467.88) -- (51.48,467.88) -- cycle ;
\draw  [color={rgb, 255:red, 74; green, 144; blue, 226 }  ,draw opacity=1 ][line width=1.5]  (51.48,330.88) -- (306.17,330.88) -- (306.17,623.47) -- (51.48,623.47) -- cycle ;
\draw  [color={rgb, 255:red, 74; green, 144; blue, 226 }  ,draw opacity=1 ][line width=1.5]  (232.26,76.63) -- (232.26,131.3) -- (168.78,131.3) -- (168.78,76.63) -- cycle ;
\draw  [color={rgb, 255:red, 74; green, 144; blue, 226 }  ,draw opacity=1 ][line width=1.5]  (198.26,76.63) -- (198.26,102.3) -- (169.26,102.3) -- (169.26,76.63) -- cycle ;
\draw  [color={rgb, 255:red, 74; green, 144; blue, 226 }  ,draw opacity=1 ][line width=1.5]  (442.71,76.19) -- (442.71,193.79) -- (306.17,193.79) -- (306.17,76.19) -- cycle ;
\draw  [color={rgb, 255:red, 74; green, 144; blue, 226 }  ,draw opacity=1 ][line width=1.5]  (369.26,76.63) -- (369.26,131.3) -- (305.78,131.3) -- (305.78,76.63) -- cycle ;
\draw  [color={rgb, 255:red, 74; green, 144; blue, 226 }  ,draw opacity=1 ][line width=1.5]  (335.78,76.63) -- (335.78,102.3) -- (306.78,102.3) -- (306.78,76.63) -- cycle ;
\draw  [color={rgb, 255:red, 74; green, 144; blue, 226 }  ,draw opacity=1 ][line width=1.5]  (598.37,76.63) -- (598.37,331.32) -- (305.78,331.32) -- (305.78,76.63) -- cycle ;

\draw (25.5,262.45) node    {$K'_{4}$};
\draw (25.5,162.45) node    {$K'_{3}$};
\draw (27.5,115.45) node    {$K'_{2}$};
\draw (28,88.45) node    {$K'_{1}$};
\draw (25.5,477.45) node    {$K'_{5}$};
\draw (137.47,54) node    {$L'_{3}$};
\draw (90.47,54) node    {$L'_{2}$};
\draw (62.47,54) node    {$L'_{1}$};
\draw (237.47,54) node    {$L'_{4}$};
\draw (451.47,54) node    {$L'_{4}$};
\draw (618,642.79) node  [font=\huge,color={rgb, 255:red, 74; green, 144; blue, 226 }  ,opacity=1 ,rotate=-45]  {$\cdots $};
\draw (64.5,116.98) node  [color={rgb, 255:red, 0; green, 0; blue, 0 }  ,opacity=1 ]  {$x_{2,1}$};
\draw (64.5,145.98) node  [color={rgb, 255:red, 0; green, 0; blue, 0 }  ,opacity=1 ]  {$x_{3,1}$};
\draw (80.5,161) node  [color={rgb, 255:red, 0; green, 0; blue, 0 }  ,opacity=1 ]  {$x_{3,2}$};
\draw (81.33,224.48) node  [color={rgb, 255:red, 0; green, 0; blue, 0 }  ,opacity=1 ]  {$x_{4,2}$};
\draw (64.5,208.98) node  [color={rgb, 255:red, 0; green, 0; blue, 0 }  ,opacity=1 ]  {$x_{4,1}$};
\draw (111.46,256.75) node  [color={rgb, 255:red, 0; green, 0; blue, 0 }  ,opacity=1 ]  {$x_{4,3}$};
\draw (64.5,345.98) node  [color={rgb, 255:red, 0; green, 0; blue, 0 }  ,opacity=1 ]  {$x_{5,1}$};
\draw (80.81,360.6) node  [color={rgb, 255:red, 0; green, 0; blue, 0 }  ,opacity=1 ]  {$x_{5,2}$};
\draw (112.28,394.61) node  [color={rgb, 255:red, 0; green, 0; blue, 0 }  ,opacity=1 ]  {$x_{5,3}$};
\draw (174.28,468.61) node  [color={rgb, 255:red, 0; green, 0; blue, 0 }  ,opacity=1 ]  {$x_{5,4}$};
\draw (92.28,89.47) node  [color={rgb, 255:red, 0; green, 0; blue, 0 }  ,opacity=1 ]  {$y_{1,2}$};
\draw (121.28,89.47) node  [color={rgb, 255:red, 0; green, 0; blue, 0 }  ,opacity=1 ]  {$y_{1,3}$};
\draw (184.28,89.47) node  [color={rgb, 255:red, 0; green, 0; blue, 0 }  ,opacity=1 ]  {$y_{1,4}$};
\draw (321.28,89.47) node  [color={rgb, 255:red, 0; green, 0; blue, 0 }  ,opacity=1 ]  {$y_{1,5}$};
\draw (140.52,102.97) node  [color={rgb, 255:red, 0; green, 0; blue, 0 }  ,opacity=1 ]  {$y_{2,3}$};
\draw (202.52,101.97) node  [color={rgb, 255:red, 0; green, 0; blue, 0 }  ,opacity=1 ]  {$y_{2,4}$};
\draw (237.34,130.99) node  [color={rgb, 255:red, 0; green, 0; blue, 0 }  ,opacity=1 ]  {$y_{3,4}$};
\draw (341.28,101.47) node  [color={rgb, 255:red, 0; green, 0; blue, 0 }  ,opacity=1 ]  {$y_{2,5}$};
\draw (374.28,131.47) node  [color={rgb, 255:red, 0; green, 0; blue, 0 }  ,opacity=1 ]  {$y_{3,5}$};
\draw (447.46,193.54) node  [color={rgb, 255:red, 0; green, 0; blue, 0 }  ,opacity=1 ]  {$y_{4,5}$};
\draw (92.28,117.47) node  [color={rgb, 255:red, 0; green, 0; blue, 0 }  ,opacity=1 ]  {$z_{2,1}$};
\draw (122.28,160.47) node  [color={rgb, 255:red, 0; green, 0; blue, 0 }  ,opacity=1 ]  {$z_{3,1}$};
\draw (152.28,160.47) node  [color={rgb, 255:red, 0; green, 0; blue, 0 }  ,opacity=1 ]  {$z_{3,2}$};
\draw (198.28,259.47) node  [color={rgb, 255:red, 0; green, 0; blue, 0 }  ,opacity=1 ]  {$z_{4,1}$};
\draw (237.62,259.33) node  [color={rgb, 255:red, 0; green, 0; blue, 0 }  ,opacity=1 ]  {$z_{4,2}$};
\draw (277.28,260.47) node  [color={rgb, 255:red, 0; green, 0; blue, 0 }  ,opacity=1 ]  {$z_{4,3}$};
\draw (362.28,473.47) node  [color={rgb, 255:red, 0; green, 0; blue, 0 }  ,opacity=1 ]  {$z_{5,1}$};
\draw (423.62,473.33) node  [color={rgb, 255:red, 0; green, 0; blue, 0 }  ,opacity=1 ]  {$z_{5,2}$};
\draw (487.28,473.47) node  [color={rgb, 255:red, 0; green, 0; blue, 0 }  ,opacity=1 ]  {$z_{5,3}$};
\draw (544.28,473.47) node  [color={rgb, 255:red, 0; green, 0; blue, 0 }  ,opacity=1 ]  {$z_{5,4}$};

\end{tikzpicture}

}

Next, we show that the operator $T_1$ given by
\[
T_1(e_{i_s,i_t})=p_{_{K'_s}}T(e_{i_s,i_t})p_{_{L'_s}} =z _{s,t},
\]
for every $1\leq t<s<\infty$ which is bounded on $[e_{i_s,i_t}]_{1\leq t<s<\infty}$.\footnote{
If $E$ is a separable Banach symmetric sequence space, then $P_{\{K'_s,L'_s\}_{s=2}^\infty}$ is a diagonal projection, which is bounded. 
Hence, 
 $T_1=P_{\{K'_s,L'_s\}_{s=2}^\infty}T|_{[e_{i_s,i_t}]_{1\leq t<s<\infty}}$   is automatically bounded.
However,  if $E$ is just a separable  quasi-Banach symmetric sequence space which is not normed, then one cannot guarantee   the boundedness of a diagonal projection   on $\mathcal C_E$. 
}

For any finite sequence $\{\alpha_{s,t}\}_{1\leq t<s\leq n+1}$ of scalars, we have
\begin{eqnarray*}
&&\left\Vert\sum_{1\leq t<s\leq n+1}\alpha_{s,t}x_{s,t}\right\Vert_{\mathcal C_E}\\
&\stackrel{\eqref{at}}{=}&
\left\Vert\sum_{1\leq t<s\leq n+1}\alpha_{s,t}x_{n-1+s,t}\right\Vert_{\mathcal C_E}\\
&\stackrel{\eqref{x-y-z-st}\eqref{T=x+y+z}}{=}&\left\Vert P_{[\bigcup_{l=1}^nK'_{n+l}],[\bigcup_{l=1}^nL'_l]}T\bigg(\sum_{1\leq t<s\leq n+1}\alpha_{s,t}e_{i_{n-1+s},i_t}\bigg)\right\Vert_{\mathcal C_E}\\
&\leq& \Vert T\Vert\left\Vert \sum_{1\leq t<s\leq n+1}\alpha_{s,t}e_{i_{n-1+s},i_t}\right\Vert_{\mathcal C_E}\\
&=&\Vert T\Vert \left\Vert \sum_{1\leq t<s\leq n+1}\alpha_{s,t}e_{i_s,i_t}\right\Vert_{\mathcal C_E},
\end{eqnarray*}
and similarly, by \eqref{bt}, we have 
\[
\left\Vert\sum_{1\leq t<s\leq n+1}\alpha_{s,t}y_{t,s}\right\Vert_{\mathcal C_E}\leq\Vert T\Vert\left\Vert\sum_{1\leq t<s\leq n+1}\alpha_{s,t}e_{i_s,i_t}\right\Vert_{\mathcal C_E}.
\]
Consequently, 
there exists a bounded operator $S_2:[e_{i_s,i_t}]_{1\leq t<s<\infty}\to\mathcal C_E$ with
\begin{align}\label{s1}
S_2(e_{i_s,i_t})=x_{s,t}+y_{t,s}
\end{align}
for every $1\leq t<s<\infty$, and hence,  we have $T_1= T|_{[e_{i_s,i_t}]_{1\leq t<s<\infty}}-S_2$\;\footnote{Note that $ (T|_{[e_{i_s,i_t}]_{1\leq t<s<\infty}}-S_2 ) (e_{i_s,i_t})  =x_{s,t}+y_{t,s}+z_{s,t}-x_{s,t}-y_{t,s} = z_{s,t} =T_1 (e_{i_s,i_t}) $.}, which is bounded.  
\end{proof}

\begin{lemma}\label{thma2}
Under the assumptions of Lemma \ref{thma1}, 
there is an increasing sequence $\{t_k\}_{k=2}^\infty$ of positive integers such that there is a bounded operator $S_3:[e_{i_{t_k},i_{t_l}}]_{2\leq l<k<\infty}\to\mathcal C_E$ (see  \eqref{s3} below) such that $S_3$ is a small perturbation of $S_2|_{[[e_{i_{t_k},i_{t_l}}]_{2\leq l<k<\infty}}$ associated with the Schauder decomposition $\big\{[e_{i_{t_k},i_{t_l}}]_{k=l+1}^\infty\big\}_{l=2}^\infty$.
\end{lemma}

\begin{proof}
For any finite sequence $\{\alpha_t\}_{t=1}^n$ of scalars we have    
\begin{eqnarray*}
\left\Vert\sum_{t=1}^n\alpha_ta_t\right\Vert_{\mathcal C_E(H')}&\stackrel{\eqref{at}}{=}&\left\Vert\sum_{t=1}^n\alpha_tx_{n+1,t}\right\Vert_{\mathcal C_E}\\
&\leq&\left\Vert  \sum_{t=1}^n\alpha_t (x_{n+1,t} +y_{t,n+1})  \Big)\right\Vert_{\mathcal C_E}\\
&= &\left\Vert S_2\Big(\sum_{t=1}^n\alpha_te_{i_{n+1},i_t}\Big)\right\Vert_{\mathcal C_E}\leq\Vert S_2\Vert\left(\sum_{t=1}^n\vert\alpha_t\vert^2\right)^{1/2},
\end{eqnarray*}
and similarly, 
\[
\left\Vert\sum_{t=1}^n\alpha_tb_t\right\Vert_{\mathcal C_E(H'')}\leq\Vert S_2\Vert\left(\sum_{t=1}^n\vert\alpha_t\vert^2\right)^{1/2}.
\]
This implies that the maps $(\alpha_t)_{t=1}^\infty\in\ell_2\mapsto\sum_{t=1}^\infty\alpha_ta_t$ and $(\alpha_t)_{t=1}^\infty\in\ell_2\mapsto\sum_{t=1}^\infty\alpha_tb_t$ are bounded, and thus $\{a_t\}_{t=1}^\infty$ is a weakly null sequence in $\mathcal C_E(H')$ and $\{b_t\}_{t=1}^\infty$ is a weakly null sequence in $\mathcal C_E(H'')$.

Note that $\big\{v(L'_t)\big\}_{t=1}^\infty$ and $\big\{u'(K'_t)\big\}_{t=1}^\infty$ are sequences of mutually orthogonal finite dimensional subspaces of $H'$ (resp. $H''$). For every $t\geq1$, we denote\footnote{
For subspaces $ X_1,X_2$ of a linear space $X$,  $X_1+X_2$ stands for 
\[
 \{x_1+x_2:x_1\in X_1\;{\rm and}\;x_2\in X_2\},
\]
which  is a subspace of $X$. Since ${\rm im }(a_i)$ is a finite dimensional subspace for every $i$, it follows that
\[
{\rm im}(a_1)+\cdots+{\rm im}(a_t)
\]
is finite dimensional. In particular,  ${\rm im}(a_1)+\cdots+{\rm im}(a_t)$ is closed for every $t$.}  
\[
F_t={\rm im}(a_1)+\cdots+{\rm im}(a_t)\stackrel{\eqref{at}}{\subseteq}{\rm im}(u_{t+1}),
\]
\[
\quad F'_t={\ker(b_1)}^\perp+\cdots+{\ker(b_t)}^\perp\stackrel{\eqref{bt}}{\subseteq}{\ker ({v'_{t+1}}^\ast)}^\perp={\rm im}(v'_{t+1}).
\]
Then, 
\[
a_t=p_{_{F_t}}a_tp_{_{v([\bigcup_{j=1}^tL'_j])}}\;\;\;\;{\rm and}\;\;\;\;b_t=p_{_{u'([\bigcup_{j=1}^tK'_j])}}b_tp_{_{F'_t}}\;\;\;\;\;\mbox{for every}\;t\geq1.
\]
By standard perturbation argument and using Theorem \ref{3e}, there is a sequence $\{t_l\}_{l=2}^\infty$ of positive integers, a sequence $\{a'_l\}_{l=2}^\infty$ in $\mathcal C_E(H')$ and a sequence $\{b'_l\}_{l=2}^\infty$ in $\mathcal C_E(H'')$ such that 
\begin{align}\label{atbt}
\Vert a_{t_l}-a'_l\Vert\;\;\;\;{\rm and}\;\;\;\;\Vert b_{t_l}-b'_l\Vert\leq\delta_l\;\;\;\;\;\;\mbox{for every }l\geq2,
\end{align}
where $\{a'_l\}_{l=2}^\infty$ and $\{b'_l\}_{l=2}^\infty$ are in the form of
\begin{align}\begin{split}\label{a'lb'l}
a'_l=\sum_{m=1}^{l-1}\big(a'_{l,m}+a'_{m,l}\big)+a'_{l,l}\;\;\;\;\;\;l=2,3,\dots,\\
b'_l=\sum_{m=1}^{l-1}\big(b'_{l,m}+b'_{m,l}\big)+b'_{l,l}\;\;\;\;\;\;l=2,3,\dots,
\end{split}
\end{align} respectively,  
\begin{enumerate}
\item [(1)]
$\hat{F}_l={F_{t_{l-1}}}^\perp\cap F_{t_l}$ (put $F_0=\{0\}$) and $\tilde{F}_l={F'_{t_{l-1}}}^\perp\cap F'_{t_l}$ (put $F'_0=\{0\}$), and $\hat{L}_l=[\bigcup_{j=t_{l-1}+1}^{t_l}v(L'_j)]=v\big([\bigcup_{j=t_{l-1}+1}^{t_l}L'_j]\big)$ and $\tilde{K}_l=[\bigcup_{j=t_{l-1}+1}^{t_l}u'(K'_j)]=u'\big([\bigcup_{j=t_{l-1}+1}^{t_l}(K'_j)]\big)$ for every $l\geq1$ (put $t_0=0$),
\item [(2)]
$\{a'_{l,m}\}_{(l,m)\neq (1,1)}\stackrel{\mbox{\tiny b.s}}{\boxplus}\{\hat{F}_l\}_{l=1}^\infty\otimes\{\hat{L}_m\}_{m=1}^\infty$ such that
\begin{itemize}
    \item 
    $\{a'_{l,m}\}_{l=m+1}^\infty\stackrel{\mbox{\tiny b.s}}{\boxplus}\big\{\hat{u}_l:\hat{F}_l\to\hat{H}\big\}_{l=2}^\infty\otimes\big\{\hat{v}_m:\hat{L}_m\to\hat{H}\big\}_{m=1}^\infty\curvearrowleft c_m$ are consistent for all $m\geq1$, and
\begin{align}\label{cmc1}
\Vert c_m\Vert_{\mathcal C_E}\leq\delta_m\Vert c_1\Vert_{\mathcal C_E}\;\;\;\;\;\;\;\;\mbox{for every}\;m\geq2,
\end{align}
 \item 
$\{a'_{m,l}\}_{l=m+1}^\infty\stackrel{\mbox{\tiny b.s}}{\boxplus}\big\{\hat{u}'_m:\hat{F}_m\to\hat{H}\big\}_{m=1}^\infty\otimes\big\{\hat{v}'_l:\hat{L}_l\to\hat{H}\big\}_{l=2}^\infty\curvearrowleft c'_m$ are consistent for all $m\geq1$, and
\begin{align}\label{cm'c1}
\Vert c'_m\Vert_{\mathcal C_E}\leq\delta_m\Vert c'_1\Vert_{\mathcal C_E}\;\;\;\;\;\;\;\;\mbox{for every}\;m\geq2,
\end{align}
 \item 
 $a'_{l,l}=p_{_{\hat{F}_l}}a_{t_l}p_{_{\hat{L}_l}}$ for every $l\geq2$,
\end{itemize}

\item [(3)]
$\{b'_{m,l}\}_{(m,l)\neq (1,1)}\stackrel{\mbox{\tiny b.s}}{\boxplus}\{\tilde{K}_m\}_{m=1}^\infty\otimes\{\tilde{F}_l\}_{l=1}^\infty$ such that 
\begin{itemize}
    \item
$\{b'_{m,l}\}_{l=m+1}^\infty\stackrel{\mbox{\tiny b.s}}{\boxplus}\big\{\tilde{u}_m:\tilde{K}_m\to\tilde{H}\big\}_{m=1}^\infty\otimes\big\{\tilde{v}_l:\tilde{F}_l\to\tilde{H}\big\}_{l=2}^\infty\curvearrowleft d_m$ are consistent for all $m\geq1$, and
\begin{align}\label{dmd1}
\Vert d_m\Vert_{\mathcal C_E}\leq\delta_m\Vert d_1\Vert_{\mathcal C_E}\;\;\;\;\;\;\mbox{for every}\;m\geq2,
\end{align}
    \item
$\{b'_{l,m}\}_{l=m+1}^\infty\stackrel{\mbox{\tiny b.s}}{\boxplus}\big\{\tilde{u}'_l:\tilde{K}_l\to\tilde{H}\big\}_{l=2}^\infty\otimes\big\{\tilde{v}'_m:\tilde{L}_m\to\tilde{H}\big\}_{m=1}^\infty\curvearrowleft d'_m$ are consistent for all $m\geq1$, and
\begin{align}\label{dm'd1}
\Vert d'_m\Vert_{\mathcal C_E}\leq\delta_m\Vert d'_1\Vert_{\mathcal C_E}\;\;\;\;\;\;\;\;\mbox{for every}\;m\geq2,
\end{align}
    \item
$b'_{l,l}=p_{_{\tilde{K}_l}}b_{t_l}p_{_{\tilde{F}_l}}$ for every $l\geq2$.
\end{itemize}

\end{enumerate}

Put 
\begin{equation}\label{xkl'ylk'}
x'_{k,l}:=u_{t_{k}}^\ast a'_lv\;\;\;\;{\rm and}\;\;\;\;y'_{l,k}:={u'}^\ast b'_lv'_{t_k}
\end{equation}
for every $2\leq l<k<\infty$.

For any $l\geq2$ and a  finite sequence $\{\alpha_k\}_{k=l+1}^n$ of scalars,
we have 
\begin{eqnarray*}
&&\Bigg\Vert\sum_{k=l+1}^n\alpha_kS_2(e_{i_{t_k},i_{t_l}})-\sum_{k=l+1}^n\alpha_k(x'_{k,l}+y'_{l,k})\Bigg\Vert_{\mathcal \mathcal C_E}\\
&\stackrel{\eqref{s1}}{\leq} &2^{\frac{1}{r}-1}\Bigg(\Big\Vert\sum_{k=l+1}^n\alpha_k(x_{t_k,t_l}-x'_{k,l})\Big\Vert_{\mathcal C_E}+\Bigg\Vert\sum_{k=l+1}^n\alpha_k(y_{t_l,t_k}-y'_{l,k})\Bigg\Vert_{\mathcal C_E}\bigg)\\
&\stackrel{\eqref{at},\eqref{bt},\eqref{xkl'ylk'}}{=}&2^{\frac{1}{r}-1}\Bigg(\bigg(\sum_{k=l+1}^n\vert\alpha_k\vert^2\bigg)^{1/2}\left\Vert a_{t_l}-a'_l \right\Vert_{\mathcal C_E}+\bigg(\sum_{k=l+1}^n\vert\alpha_k\vert^2\bigg)^{1/2} \left\Vert b_{t_l}-b'_l\right\Vert_{\mathcal C_E}\Bigg)\\
&\stackrel{\eqref{atbt}}{\leq}&2^{\frac{1}{r}}\delta_l\Bigg\Vert\sum_{k=l+1}^n\alpha_ke_{i_{t_k},i_{t_l}}\Bigg\Vert_{\mathcal C_E}.
\end{eqnarray*}


Note that 
\begin{eqnarray*}
x'_{k,l}&=&u_{t_{k}}^\ast a'_lv\\
&=&\sum_{m=1}^{l-1}\Big(u_{t_k}^\ast a'_{l,m}v+u_{t_k}^\ast a'_{m,l}v\Big)+u_{t_k}^\ast p_{_{\hat{F}_l}}a_{t_l}p_{_{\hat{L}_l}}v\\
&\stackrel{\eqref{cmc1},\eqref{cm'c1}}{=}& \sum_{m=1}^{l-1}\Big(({\hat{u}_l}u_{t_k})^\ast c_m(\hat{v}_m v)+(\hat{u}'_mu_{t_k})^\ast c'_m({\hat{v}_l}'v)\Big)+u_{t_k}^\ast p_{_{\hat{F}_l}}a_{t_l}p_{_{\hat{L}_l}}v,
\end{eqnarray*}
and,  similarly, 
\begin{align*}
y'_{l,k}={u'}^\ast b'_lv'_{t_k}=\sum_{m=1}^{l-1}\Big(({\tilde{u}'_l}u')^\ast d'_m(\tilde{v}'_m v'_{t_k})+(\tilde{u}_mu')^\ast d_m({\tilde{v}_l}v'_{t_k})\Big)+{u'}^\ast p_{_{\tilde{K}_l}}b_{t_l}p_{_{\tilde{F}_l}}v'_{t_k}.
\end{align*}
Put
\[
\hat{F}^{_{-1}}_{k,l}=(u_{t_k})^\ast(\hat{F}_l)\subset{K'_{t_k}}\;\;\;\;{\rm and}\;\;\;\;\tilde{F}^{_{-1}}_{l,k}=(v'_{t_k})^\ast(\tilde{F}_l)\subset{L'_{t_k}}
\]
for every $1\leq l<k<\infty$, and
\[
\hat{L}^{_{-1}}_l=\Big[{\bigcup}_{j=t_{l-1}+1}^{t_l}L'_j\Big]\;\;\;\;{\rm and}\;\;\;\;\tilde{K}^{_{-1}}_l=\Big[{\bigcup}_{j=t_{l-1}+1}^{t_l}K'_j\Big]
\]
for every $l\geq1$. Clearly, $\big\{\hat{F}^{_{-1}}_{k,l}\big\}_{1\leq l<k<\infty}$ and $\big\{\tilde{F}^{_{-1}}_{l,k}\big\}_{1\leq l<k<\infty}$ are sequences of mutually orthogonal finite dimensional subspace of Hilbert spaces respectively. For any $1\leq l<k<\infty$ and $1\leq m<k<\infty$, we put
\begin{align}\label{xuv}
x_{_{(k,l),m}}=
\begin{cases}
0~,&  m=l=1,\;k=2\\
u_{t_k}^\ast p_{_{\hat{F}_l}}a_{t_l}p_{_{\hat{L}_l}}v~,& 2\leq m=l<k<\infty\\
({\hat{u}_l}u_{t_k})^\ast c_m(\hat{v}_m v)~,& 1\leq m<l<k<\infty\\
(\hat{u}'_lu_{t_k})^\ast c'_l({\hat{v}_m}'v)~,& 1\leq l<m<k<\infty\\
\end{cases}
\end{align}
and
\begin{align}\label{yuv}
y_{_{m,(l,k)}}=
\begin{cases}
0~,&  m=l=1,\;k=2\\
{u'}^\ast p_{_{\tilde{K}_l}}b_{t_l}p_{_{\tilde{F}_l}}v'_{t_k}~,& 2\leq m=l<k<\infty\\
({\tilde{u}_m}u')^\ast d_m(\tilde{v}_l v'_{t_k})~,& 1\leq m<l<k<\infty\\
(\tilde{u}'_mu')^\ast d'_l({\tilde{v}_l}'v'_{t_k})~,& 1\leq l<m<k<\infty\\
\end{cases}
\end{align}
In particular, 
\begin{itemize}
    \item [(1)]
$\big\{x_{_{(k,l),m}}\big\}_{
\substack{ _{1\leq l<k<\infty}  \\ _{1\leq m<k<\infty} } }
\stackrel{\mbox{{\rm\tiny b.s}}}{\boxplus}\{\hat{F}^{_{-1}}_{k,l}\}_{1\leq l<k<\infty}\otimes\{\hat{L}^{_{-1}}_m\}_{m=1}^\infty$ such that
$x_{_{(2,1),1}}=0$,
\begin{align*}
\big\{x_{_{(k,l),m}}\big\}_{2\leq m=l<k<\infty}
\stackrel{\mbox{{\rm\tiny b.s}}}{\boxplus}&\big\{u_{t_k}|_{\hat{F}^{_{-1}}_{k,l}}:\hat{F}^{_{-1}}_{k,l}\to H'\big\}_{2\leq l<k<\infty}\otimes\big\{v|_{\hat{L}^{_{-1}}_m}:\hat{L}^{_{-1}}_m\to H'\big\}_{l=2}^\infty\\
&\curvearrowleft p_{_{\hat{F}_l}}a_{t_l}p_{_{\hat{L}_l}}\;\;\;\;\;\;{\rm are\;consistent\;for\;all}\;l\geq2,
\end{align*}
\begin{align}\label{'xklm}
\big\{x_{_{(k,l),m}}\big\}_{1\leq m<l<k<\infty}
\stackrel{\mbox{{\rm\tiny b.s}}}{\boxplus}&\big\{\hat{u}_lu_{t_k}|_{\hat{F}^{_{-1}}_{k,l}}:\hat{F}^{_{-1}}_{k,l}\to\hat{H}\big\}_{2\leq l<k<\infty}\otimes\big\{\hat{v}_m v|_{\hat{L}^{_{-1}}_m}:\hat{L}^{_{-1}}_m\to\hat{H}\big\}_{m=1}^\infty\\
&\curvearrowleft c_m\;\;\;\;\;\;{\rm are\;consistent\;for\;all}\;m\geq1,\notag
\end{align}
\begin{align}\label{'xkml}
\big\{x_{_{(k,m),l}}\big\}_{1\leq m<l<k<\infty}
\stackrel{\mbox{{\rm\tiny b.s}}}{\boxplus}&\big\{\hat{u}'_mu_{t_k}|_{\hat{F}^{_{-1}}_{k,m}}:\hat{F}^{_{-1}}_{k,m}\to\hat{H}'\big\}_{1\leq m<k-1<\infty}\otimes\big\{{\hat{v}_l}'v|_{\hat{L}^{_{-1}}_l}:\hat{L}^{_{-1}}_l\to\hat{H}\big\}_{l=2}^\infty\\&\curvearrowleft c'_m\;\;\;\;\;\;{\rm are\;consistent\;for\;all}\;m\geq1,\notag
\end{align}

\item [(2)]
$\big\{y_{_{m,(l,k)}}\big\}_{
\substack{ _{1\leq m<k<\infty}  \\ _{1\leq l<k<\infty} } }
\stackrel{\mbox{{\rm\tiny b.s}}}{\boxplus}\{\tilde{K}^{_{-1}}_m\}_{m=1}^\infty\otimes\{\tilde{F}^{_{-1}}_{l,k}\}_{1\leq l<k<\infty}$ such that $y_{_{1,(1,2)}}=0$,
\begin{align*}
\big\{y_{_{m,(l,k)}}\big\}_{2\leq m=l<k<\infty}
\stackrel{\mbox{{\rm\tiny b.s}}}{\boxplus}&\big\{u'|_{\tilde{K}^{_{-1}}_m}:\tilde{K}^{_{-1}}_m\to H''\big\}_{l=2}^\infty\otimes\big\{v'_{t_k}|_{\tilde{F}^{_{-1}}_{l,k}}:\tilde{F}^{_{-1}}_{l,k}\to H''\big\}_{2\leq l<k<\infty}\\
&\curvearrowleft p_{_{\tilde{K}_l}}b_{t_l}p_{_{\tilde{F}_l}}\;\;\;\;\;\;{\rm are\;consistent\;for\;all}\;l\geq2, 
\end{align*}
\begin{align}\label{'ymlk}
\big\{y_{_{m,(l,k)}}\big\}_{1\leq m<l<k<\infty}
\stackrel{\mbox{{\rm\tiny b.s}}}{\boxplus}&\big\{\tilde{u}_m u'|_{\tilde{K}^{_{-1}}_m}:\tilde{K}^{_{-1}}_m\to\tilde{H}\big\}_{m=1}^\infty\otimes\big\{\tilde{v}_lv'_{t_k}|_{\tilde{F}^{_{-1}}_{l,k}}:\tilde{F}^{_{-1}}_{l,k}\to\tilde{H}\big\}_{2\leq l<k<\infty}\\
&\curvearrowleft d_m\;\;\;\;\;\;{\rm are\;consistent\;for\;all}\;m\geq1,\notag
\end{align}
\begin{align}\label{'ylmk}
\big\{y_{_{l,(m,k)}}\big\}_{1\leq m<l<k<\infty}
\stackrel{\mbox{{\rm\tiny b.s}}}{\boxplus}&\big\{{\tilde{u}_l}'u'|_{\tilde{K}^{_{-1}}_l}:\tilde{K}^{_{-1}}_l\to\tilde{H}\big\}_{l=2}^\infty\otimes\big\{\tilde{v}'_mv'_{t_k}|_{\tilde{F}^{_{-1}}_{m,k}}:\tilde{F}^{_{-1}}_{m,k}\to\tilde{H}'\big\}_{1\leq m<k-1<\infty}\\
&\curvearrowleft d'_m\;\;\;\;\;\;{\rm are\;consistent\;for\;all}\;m\geq1,\notag
\end{align}
\item [(3)]
\[
x'_{k,l}=\sum_{m=1}^{l-1}\big(x_{_{(k,l),m}}+x_{_{(k,m),l}}\big)+x_{_{(k,l),l}}\quad 
\mbox{ and }\quad 
y'_{l,k}=\sum_{m=1}^{l-1}\big(y_{_{l,(m,k)}}+y_{_{m,(l,k)}}\big)+y_{_{l,(l,k)}}
\]
for every $2\leq l<k<\infty$. 
\end{itemize}
Consequently, by Theorem \ref{2d}, there is a bounded operator $S_3:[e_{i_{t_k},i_{t_l}}]_{2\leq l<k<\infty}\to\mathcal C_E$ such that
\begin{align}\label{s3}\begin{split}
 S_3(e_{i_{t_k},i_{t_l}}) &=x'_{k,l}+y'_{l,k}\\
&=\sum_{m=1}^{l-1}\big(x_{_{(k,l),m}}+x_{_{(k,m),l}}\big)+\sum_{m=1}^{l-1}\big(y_{_{l,(m,k)}}+y_{_{m,(l,k)}}\big)+x_{_{(k,l),l}}+y_{_{l,(l,k)}}
\end{split}
\end{align}
for every $2\leq l<k<\infty$, which is a small perturbation of $S_2|_{[[e_{i_{t_k},i_{t_l}}]_{2\leq l<k<\infty}}$ associated with the Schauder decomposition $\big\{[e_{i_{t_k},i_{t_l}}]_{k=l+1}^\infty\big\}_{l=2}^\infty$.  
\end{proof}

\begin{lemma}\label{thma3}
The operator  $S_3$ in Lemma \ref{thma2} can be written as $S_3=T_2+T_3+T_4$, where $T_2$,  $T_3$ and $T_4$ are bounded operators from $[e_{i_{t_k},i_{t_l}}]_{2\leq l<k<\infty}$ to $\mathcal C_E$ such that
\[
T_2(e_{i_{t_k},i_{t_l}})=x_{_{(k,l),l}}+y_{_{l,(l,k)}},\;\;\;\;\;\;\;\;2\leq l<k<\infty,
\]
\[
T_3(e_{i_{t_k},i_{t_l}})=\sum_{m=1}^{l-1}\big(x_{_{(k,l),m}}+y_{_{m,(l,k)}}\big),\;\;\;\;\;\;\;\;2\leq l<k<\infty,
\]
and
\[
T_4(e_{i_{t_k},i_{t_l}})=\sum_{m=1}^{l-1}\big(x_{_{(k,m),l}}+y_{_{l,(m,k)}}\big),\;\;\;\;\;\;\;\;2\leq l<k<\infty.
\]
\end{lemma}

\begin{proof}
First, we show that $T_3$ is bounded on $[e_{i_{t_k},i_{t_l}}]_{2\leq l<k<\infty}$. If $c_1=0$ and $d_1=0$, 
then $c_m,c_m',d_m,d_m' =0$ (see \eqref{cmc1}, \eqref{cm'c1}, \eqref{dm'd1}, \eqref{dmd1} above). Hence, by the definitions of $x_{(k,l),m},x_{_{(k,m),l}}, y_{m,(l,k)}$ and $y_{_{l,(m,k)}}$, we obtain  
$T_3=0$.

If $c_1$ or $d_1$ is nonzero, then for any finite sequence $\{\alpha_k\}_{k=3}^n$ of scalars, we have 
\begin{eqnarray*}
&&\Bigg(\sum_{k=3}^n\vert\alpha_k\vert^2\Bigg)^{1/2}\cdot\left\Vert\left(
\begin{smallmatrix}
0&d_1\\c_1&0
\end{smallmatrix}\right)\right\Vert_{\mathcal C_E}\\
&\stackrel{\eqref{'xklm},\eqref{'ymlk},{\rm Th.}\;\ref{hi}}{=}&\Bigg\Vert\sum_{k=3}^n\alpha_k\big(x_{_{(k,2),1}}+y_{_{1,(2,k)}}\big)\Bigg\Vert_{\mathcal C_E}\\ 
&=&\Bigg\Vert\big(P_{[\bigcup_{j=2}^\infty \tilde{K}^{_{_{-1}}}_j],\hat{L}^{_{_{-1}}}_1}+P_{\tilde{K}^{_{_{-1}}}_1,[\bigcup_{j=2}^\infty \hat{L}^{_{_{-1}}}_j]}\big)S_3\Big(\sum_{k=1}^n\alpha_ke_{i_{t_k},i_{t_{k-1}}}\Big)\Bigg\Vert_{\mathcal C_E} \\
&\stackrel{\eqref{diag-proj}}{\leq}&
4^{\frac{1}{r}-1}\Vert S_3\Vert\cdot\Bigg\Vert\sum_{k=3}^n\alpha_ke_k\Bigg\Vert_E,
\end{eqnarray*}
where the second equality follows from the facts that
\[
\big\{x_{_{(k,l),m}}\big\}_{
\substack{ _{1\leq l<k<\infty}  \\ _{1\leq m<k<\infty }}}
\stackrel{\mbox{{\rm\tiny b.s}}}{\boxplus}\{\hat{F}^{_{-1}}_{k,l}\}_{1\leq l<k<\infty}\otimes\{\hat{L}^{_{-1}}_m\}_{m=1}^\infty\]
and
\[\big\{y_{_{m,(l,k)}}\big\}_{
\substack{ _{1\leq m<k<\infty}  \\ _{1\leq l<k<\infty} } }
\stackrel{\mbox{{\rm\tiny b.s}}}{\boxplus}\{\tilde{K}^{_{-1}}_m\}_{m=1}^\infty\otimes\{\tilde{F}^{_{-1}}_{l,k}\}_{1\leq l<k<\infty}.
\]
In particular, we have \[
\left\Vert\left(
\begin{smallmatrix}
0&d_1\\c_1&0
\end{smallmatrix}\right)\right\Vert_{\mathcal C_E}\cdot \left \Vert\cdot\right\Vert_{\ell_2}\leq4^{\frac{1}{r}-1}\Vert S_3\Vert\cdot \left\Vert\cdot \right \Vert_E.
\]
Hence, for any sequence $\{\alpha_{k,l}\}_{2\leq l<k<\infty}$ of scalars with finitely many  nonzero elements, we have 
\begin{align}\label{ineq tri S3}\begin{split}
\Bigg(\sum_{2\leq l<k<\infty}\vert\alpha_{k,l}\vert^2\Bigg)^{1/2}\cdot\left\Vert\left(
\begin{smallmatrix}
0&d_1\\c_1&0
\end{smallmatrix}\right)\right\Vert_{\mathcal C_E}&=\left\Vert\left(
\begin{smallmatrix}
0&d_1\\c_1&0
\end{smallmatrix}\right)\right\Vert_{\mathcal C_E}\cdot\Bigg\Vert\sum_{2\leq l<k<\infty}\alpha_{k,l}e_{i_{t_k},i_{t_l}}\Bigg\Vert_{C_{\ell_2}} \\
&\leq4^{\frac{1}{r}-1}\Vert S_3\Vert\cdot\Bigg\Vert\sum_{2\leq l<k<\infty}\alpha_{k,l}e_{i_{t_k},i_{t_l}}\Bigg\Vert_{\mathcal C_E}.
\end{split}
\end{align}
Consequently, for any sequence $\{\alpha_{k,l}\}_{2\leq l<k<\infty}$ of scalars with finitely many nonzero elements, we have
\begin{eqnarray}\label{hil}
&& \Bigg\Vert T_3\bigg(\sum_{2\leq l<k<\infty}\alpha_{k,l}e_{i_{t_k},i_{t_l}}\bigg)\Bigg\Vert_{\mathcal C_E}\\
&=&\Bigg\Vert\sum_{2\leq l<k<\infty}\alpha_{k,l}\sum_{m=1}^{l-1}\big(x_{_{(k,l),m}}+y_{_{m,(l,k)}}\big)\Bigg\Vert_{\mathcal C_E}\nonumber\\
&\stackrel{\eqref{qt-eq}}{\leq}&4^{\frac{1}{r}}\Bigg(\sum_{m=1}^\infty\Bigg\Vert\sum_{m+1\leq l<k<\infty}\alpha_{k,l}\big(x_{_{(k,l),m}}+y_{_{m,(l,k)}}\big)\Bigg\Vert_{\mathcal C_E}^r\Bigg)^{1/r}\nonumber\\
&\stackrel{\eqref{'xklm},\eqref{'ymlk},{\rm Th}\;\ref{hi}}{=}&4^{\frac{1}{r}}\Bigg(\sum_{m=1}^\infty\left\Vert\left(
\begin{smallmatrix}
0&d_m\\c_m&0
\end{smallmatrix}\right)\right\Vert_{\mathcal C_E}^r\Bigg)^{1/r}\Bigg(\sum_{m+1\leq l<k<\infty}\vert\alpha_{k,l}\vert^2\Bigg)^{1/2}\nonumber\\
&\stackrel{\eqref{cmc1},\eqref{dmd1}}{\leq}&4^{\frac{1}{r}}\Bigg(1+2^{1-r}\sum_{m=2}^\infty{\delta_m}^r\Bigg)^{1/r}\left\Vert\left(
\begin{smallmatrix}
0&d_1\\c_1&0
\end{smallmatrix}\right)\right\Vert_{\mathcal C_E}\Bigg(\sum_{2\leq l<k<\infty}\vert\alpha_{k,l}\vert^2\Bigg)^{1/2}\nonumber\\
&\stackrel{\eqref{ineq tri S3}}{\leq}&4^{\frac{2}{r}-1}\Bigg(1+2^{1-r}\sum_{m=2}^\infty{\delta_m}^r\Bigg)^{1/r}\Vert S_3\Vert\cdot\Bigg\Vert\sum_{2\leq l<k<\infty}\alpha_{k,l}e_{i_{t_k},i_{t_l}}\Bigg\Vert_{\mathcal C_E}\nonumber. 
\end{eqnarray}
Hence, $T_3$ is bounded.

Next, we show $T_4$ is bounded on $[e_{i_{t_k},i_{t_l}}]_{2\leq l<k<\infty}$. For any finite sequence $\{\alpha_{k,l}\}_{2\leq l<k\leq n+2}$ of scalars,
\begin{eqnarray*}
&&\bigg\Vert\sum_{2\leq l<k\leq n+2}\alpha_{k,l}\sum_{m=1}^{l-1}x_{_{(k,m),l}}\bigg\Vert_{\mathcal C_E}^r\\
&\stackrel{\eqref{'xkml}}{=}&\bigg\Vert\sum_{2\leq l<k\leq n+2}\alpha_{k,l}\sum_{m=1}^{l-1}x_{_{(2n+k,m),n+l}}\bigg\Vert_{\mathcal C_E}^r\\
&=&\bigg\Vert\sum_{2\leq l<k\leq n+2}\alpha_{k,l}\sum_{m=1}^{n}x_{_{(2n+k,m),n+l}}-\sum_{2\leq l<k\leq n+2}\alpha_{k,l}\sum_{m=l}^nx_{_{(2n+k,m),n+l}}\bigg\Vert_{\mathcal C_E}^r\\
&=&\bigg\Vert\sum_{2\leq l<k\leq n+2}\alpha_{k,l}\sum_{m=1}^{n}x_{_{(2n+k,m),n+l}}-\sum_{l=2}^n\sum_{m=l}^n\sum_{k=l+1}^{n+2}\alpha_{k,l}x_{_{(2n+k,m),n+l}}\bigg\Vert_{\mathcal C_E}^r\\
&\stackrel{\eqref{qt-eq}}{\leq}&4\bigg\Vert\sum_{2\leq l<k\leq n+2}\alpha_{k,l}\sum_{m=1}^{n}x_{_{(2n+k,m),n+l}}\bigg\Vert_{\mathcal C_E}^r+4\sum_{l=2}^n\sum_{m=l}^n\bigg\Vert\sum_{k=l+1}^{n+2}\alpha_{k,l}x_{_{(2n+k,m),n+l}}\bigg\Vert_{\mathcal C_E}^r\\
&\stackrel{\eqref{'xkml}}{=}&4\bigg\Vert P_{[\bigcup_{i,j=1}^n\hat{F}^{_{-1}}_{2n+2+i,j}],[\bigcup_{j=1}^n \hat{L}^{_{-1}}_{n+1+j}]}S_3\Big(\sum_{2\leq l<k\leq n+2}\alpha_{k,l}e_{i_{t_{2n+k}},i_{t_{n+l}}}\Big)\bigg\Vert_{\mathcal C_E}^r\\
&&+4\sum_{l=2}^n\sum_{m=l}^n\Vert c'_m\Vert_{\mathcal C_E}^r\bigg(\sum_{k=l+1}^{n+2}\vert\alpha_{k,l}\vert^2\bigg)^{r/2}\\
&\stackrel{\eqref{qt-eq}}{\leq}&4\Vert S_3\Vert^r\bigg\Vert\sum_{2\leq l<k\leq n+2}\alpha_{k,l}e_{i_{t_{2n+k}},i_{t_{n+l}}}\bigg\Vert_{\mathcal C_E}^r+4\sum_{l=2}^n\sum_{m=l}^n\Vert c'_m\Vert_{\mathcal C_E}^r\bigg(\sum_{k=l+1}^{n+2}\vert\alpha_{k,l}\vert^2\bigg)^{r/2}\\
&=&4\Vert S_3\Vert^r\Bigg\Vert\sum_{2\leq l<k\leq n+2}\alpha_{k,l}e_{i_{t_k},i_{t_l}}\Bigg\Vert_{\mathcal C_E}^r+4\sum_{l=2}^n\sum_{m=l}^n\Vert c'_m\Vert_{\mathcal C_E}^r\bigg(\sum_{k=l+1}^{n+2}\vert\alpha_{k,l}\vert^2\bigg)^{r/2}\\
&\leq&4\bigg(\Vert S_3\Vert^r+\sum_{l=2}^n\sum_{m=l}^n{\delta_m}^r\Vert c'_1\Vert_{\mathcal C_E}^r\bigg)\Bigg\Vert\sum_{2\leq l<k\leq n+2}\alpha_{k,l}e_{i_{t_k},i_{t_l}}\Bigg\Vert_{\mathcal C_E}^r\\
&\leq&4\Bigg(\Vert S_3\Vert^r+\Vert c'_1\Vert_{\mathcal C_E}^r\sum_{m=2}^\infty m{\delta_m}^r\Bigg)\Bigg\Vert\sum_{2\leq l<k\leq n+2}\alpha_{k,l}e_{i_{t_k},i_{t_l}}\Bigg\Vert_{\mathcal C_E}^r,
\end{eqnarray*}
and similarly,
\begin{align*}
&\quad \Bigg\Vert\sum_{2\leq l<k\leq n+2}\alpha_{k,l}\sum_{m=1}^{l-1}y_{_{l,(m,k)}}\Bigg\Vert_{\mathcal C_E}^r\\
&\leq4\bigg(\Vert S_3\Vert^r+\Vert d'_1\Vert_{\mathcal C_E}^r\sum_{m=2}^\infty m{\delta_m}^r\bigg)\Bigg\Vert\sum_{2\leq l<k\leq n+2}\alpha_{k,l}e_{i_{t_k},i_{t_l}}\Bigg\Vert_{\mathcal C_E}^r.
\end{align*}
Since $\{\delta_m\}_{m=1}^\infty$ is small enough, without loss generality, we may assume that $\sum_{m=2}^\infty m{\delta_m}^r<\infty$. Consequently, $T_4$ is bounded. 

Since $T_2=S_3-T_3-T_4$, it follows that $T_2$ is bounded on 
$[e_{i_{t_k},i_{t_l}}]_{2\leq l<k<\infty}$\footnote{If $E$ is a Banach space, then the diagonal projection on $\mathcal C_E$ is automatically bounded (see \eqref{projection inequality}). Noting 
\[
T_2=\big(P_{\big\{[\bigcup_{k'=l'+1}^\infty\hat{F}^{_{-1}}_{k',l'}],\hat{L}^{_{_{-1}}}_{l'}\big\}_{l'=2}^\infty}+P_{\big\{\tilde{K}^{_{_{-1}}}_{l'},[\bigcup_{k'=l'+1}^\infty\tilde{F}^{_{-1}}_{l',k'}]\big\}_{l'=2}^\infty}\big)S_3,
\]
we obtain that  $T_2$ is bounded. By $T_4=S_3-T_2-T_3$, we obtain   the boundedness of $T_4$.}.
\end{proof}

The proof of Theorem \ref{thma} is complete.   
\end{proof}

Below, we study properties of 
 operators $T_1$, $T_2$, $T_3$, and $T_4$ given in  Theorem \ref{thma}. 
 In Theorem \ref{thmb} (resp. Theorem \ref{thma'}) below, we show that if $c_0, \ell_2\not\hookrightarrow E$ (resp. $c_0\not\hookrightarrow E$), then there exist two operators from $\mathcal B([e_{\xi'_k,\eta'_l}]_{k,l=1}^\infty,\mathcal C_E)$ (resp. three operators from $\mathcal B\big(\mathcal T_{E,\{\xi'_k,\eta'_k\}_{k=1}^\infty},\mathcal C_E$\big)), where $\{\xi_k'\}_{k=1}^\infty$ and $\{\eta_k'\}_{k=1}^\infty$ are two orthonormal sequences in $H$, with more concrete representations,  whose sum restricted on  $\mathcal T_{E,\{\xi'_k,\eta'_k\}_{k=1}^\infty}(=[e_{\xi'_k,\eta'_l}]_{1\leq l\leq k<\infty})$   is a small perturbation of 
   $\left(T_1+T_2+T_3+T_4\right)|_{\mathcal T_{E,\{\xi'_k,\eta'_k\}_{k=1}^\infty}}$. 

From now on, we always assume that the sequence $\{\varepsilon_l\}_{l=1}^\infty$ in Theorem \ref{thma} consists of sufficiently small positive numbers.

\begin{lemma}\label{t2t3}
Let $T_2$ and $T_3$ be the bounded operators given in Theorem \ref{thma}. If $\hat{b}_1$ or $\tilde{b}_1$ is nonzero, then
\begin{itemize}
\item [(1)]
$\big\{T_3(e_{\xi_{i_k},\eta_{i_l}})\big\}_{2\leq l<k<\infty}\sim\big\{(T_2+T_3)(e_{\xi_{i_k},\eta_{i_l}})\big\}_{2\leq l<k<\infty}\sim\big\{e^H_{(k,l)}\big\}_{2\leq l<k<\infty}$, and for any  sequence $\{\alpha_{k,l}\}_{2\leq l<k<\infty}$ with $\norm{\{\alpha_{k,l}\}_{2\leq l<k<\infty}}_2<\infty$, we have 
\begin{eqnarray}
\begin{split}\label{t3l2}&\quad 
2^{1-\frac{1}{r}}\left(1-2^{\frac{4}{r}-1}\bigg(\sum_{m=2}^\infty{\varepsilon_m}^r\bigg)^{1/r}\right)\left\Vert\left(
\begin{smallmatrix}
0&\tilde{b}_1\\\hat{b}_1&0
\end{smallmatrix}\right)
\right\Vert_{\mathcal C_E}\bigg(\sum_{2\leq l<k<\infty}\vert\alpha_{k,l}\vert^2\bigg)^{1/2}\\
&\leq\Bigg\Vert\sum_{2\leq l<k<\infty}\alpha_{k,l}T_3 (e_{\xi_{i_k},\eta_{i_l}})\Bigg\Vert_{\mathcal C_E}\\
&\leq2^{\frac{1}{r}-1}\left(1+2^{\frac{3}{r}}\bigg(\sum_{m=2}^\infty{\varepsilon_m}^r\bigg)^{1/r}\right)\left\Vert\left(
\begin{smallmatrix}
0&\tilde{b}_1\\\hat{b}_1&0
\end{smallmatrix}\right)
\right\Vert_{\mathcal C_E}\bigg(\sum_{2\leq l<k<\infty}\vert\alpha_{k,l}\vert^2\bigg)^{1/2}.
\end{split}
\end{eqnarray}
\item [(2)]
Putting 
\begin{equation}
\hat{b}:=\sum_{m=1}^\infty \hat{b}_m\otimes (\cdot|e_1)e_m\;\;\;\;\mbox{and}\;\;\;\;\tilde{b}: =\sum_{m=1}^\infty \tilde{b}_m\otimes(\cdot|e_m)e_1,
\end{equation}
we have
\[
\left\Vert\cdot\right\Vert_{\ell_2}\leq\frac{16^{\frac{1}{r}-1}\Vert T\Vert}{\left\Vert\left(
\begin{smallmatrix}
0&\tilde{b}\\\hat{b}&0
\end{smallmatrix}\right)
\right\Vert_{\mathcal C_E}}\left\Vert\cdot\right\Vert_E,
\]
i.e., $E\subset\ell_2$.
In particular,
$  
\left\Vert\cdot\right\Vert_{\mathcal C_{\ell_2}}\leq\frac{16^{\frac{1}{r}-1}\Vert T\Vert}{\left\Vert\left(
\begin{smallmatrix}
0&\tilde{b}\\\hat{b}&0
\end{smallmatrix}\right)
\right\Vert_{\mathcal C_E}}\left\Vert\cdot\right\Vert_{\mathcal C_E}.
$  
\end{itemize}
\end{lemma}

\begin{proof}
\begin{itemize}
\item [(1).]
By Theorem \ref{hi}, for any  sequence $\{\alpha_{k,l}\}_{2\leq l<k<\infty}$ with $\norm{\{\alpha_{k,l}\}_{2\leq l<k<\infty}}_2<\infty$, we  have 
\begin{equation}\label{eq:l2 57}
\left\Vert\sum_{2\leq l<k<\infty}\alpha_{k,l}\big(x_{_{(k,l),1}}+y_{_{1,(l,k)}}\big)\right\Vert_{\mathcal C_E}=
\left\Vert\left(
\begin{smallmatrix}
0&\tilde{b}_1\\\hat{b}_1&0
\end{smallmatrix}\right)
\right\Vert_{\mathcal C_E}\bigg(\sum_{2\leq l<k<\infty}\vert\alpha_{k,l}\vert^2\bigg)^{1/2}.
\end{equation}
Arguing similarly as  \eqref{hil}, we have 
\begin{align*}
 &\Bigg\Vert\sum_{2\leq l<k<\infty}\alpha_{k,l}\bigg(T_3(e_{\xi_{i_k},\eta_{i_l}})-\big(x_{_{(k,l),1}}+y_{_{1,(l,k)}}\big)\bigg)\Bigg\Vert_{\mathcal C_E}\\
 \leq&4^{\frac{1}{r}}\bigg(\sum_{m=2}^\infty\left\Vert\left(
\begin{smallmatrix}
0&\tilde{b}_m\\\hat{b}_m&0
\end{smallmatrix}\right)\right\Vert_{\mathcal C_E}^r\bigg)^{1/r}\Bigg(\sum_{2\leq l<k<\infty}\vert\alpha_{k,l}\vert^2\Bigg)^{1/2}\\
\leq&2^{\frac{3}{r}}\Bigg(\sum_{m=2}^\infty{\varepsilon_m}^r\Bigg)^{1/r}\left\Vert\left(
\begin{smallmatrix}
0&\tilde{b}_1\\\hat{b}_1&0
\end{smallmatrix}\right)
\right\Vert_{\mathcal C_E}\Bigg(\sum_{2\leq l<k<\infty}\vert\alpha_{k,l}\vert^2\Bigg)^{1/2}.
\end{align*}
This together with \eqref{eq:l2 57} proves  \eqref{t3l2}, and thus  $\big\{T_3(e_{\xi_{i_k},\eta_{i_l}})\big\}_{2\leq l<k<\infty}\sim\big\{e^H_{(k,l)}\big\}_{2\leq l<k<\infty}$. By  \eqref{hatc} and Remark \ref{rem}, $\{x_{_{(k,l),l}}\}_{2\leq l<k<\infty}\simeq\{\hat{c}_l\otimes(\cdot|e_l)e_{_{(k,l)}}\}_{2\leq l<k<\infty}$, where $\{e_{_{(k,l)}}\}_{k,l}$ is unit basis of $\ell_2(\mathbb N\times\mathbb N)$. For any finite sequence of scalars $\{\alpha_{k,l}\}_{2\leq l<k\leq n}$, we have 
\begin{eqnarray*}
&&\left\Vert\sum_{2\leq l<k\leq n}\alpha_{k,l}x_{_{(k,l),l}}\right\Vert_{\mathcal C_E}\\
&=&\left\Vert\sum_{2\leq l<k\leq n}\alpha_{k,l}\big(\hat{c}_l\otimes(\cdot|e_l)e_{_{(k,l)}}\big)\right\Vert_{\mathcal C_E}\\
&=&\left\Vert\sum_{l=2}^n\bigg(\hat{c}_l\otimes\bigg((\cdot|e_l)\sum_{k=l+1}^n\alpha_{k,l}e_{_{(k,l)}}\bigg)\bigg)\right\Vert_{\mathcal C_E}\\
&=&\left\Vert\sum_{l=2}^n\bigg(\sum_{k=l+1}^n\vert\alpha_{k,l}\vert^2\bigg)^{1/2}\hat{c}_l\otimes(\cdot|e_l)e_l\right\Vert_{\mathcal C_E}   \\
&\stackrel{\eqref{hatc}}{=}&\left\Vert\sum_{l=2}^n\bigg(\sum_{k=l+1}^n\vert\alpha_{k,l}\vert^2\bigg)^{1/2}x_{_{(n+1,l),l}}\right\Vert_{\mathcal C_E}\\
&\stackrel{\eqref{T2}}{=}& \left\Vert P_{K'_{n+1},[\bigcup_{j=1}^nL'_j]}T_2\left(\sum_{l=2}^n\bigg(\sum_{k=l+1}^n\vert\alpha_{k,l}\vert^2\bigg)^{1/2}e_{\xi_{i_{n+1}},\eta_{i_l}}\right)\right\Vert_{\mathcal C_E}\\
&\leq& \Vert T_2\Vert\bigg(\sum_{2\leq l<k\leq n}\vert\alpha_{k,l}\vert^2\bigg)^{1/2}.
\end{eqnarray*}
Arguing similarly, we have 
\[
\left\Vert\sum_{2\leq l<k\leq n}\alpha_{k,l}y_{_{l,(l,k)}}\right\Vert_{\mathcal C_E}\leq\Vert T_2\Vert\bigg(\sum_{2\leq l<k\leq n}\vert\alpha_{k,l}\vert^2\bigg)^{1/2}.
\]
Furthermore, since 
\[
\Big(P_{[\bigcup_{j=2}^\infty K'_j],L'_1}+P_{K'_1,[\bigcup_{j=2}^\infty L'_j]}\Big)(T_2+T_3)(e_{\xi_{i_k},\eta_{i_l}})\stackrel{\eqref{T2},\eqref{T3}}{=}x_{_{(k,l),1}}+y_{_{1,(l,k)}}
\]
for every $2\leq l<k<\infty$, it follows from    the fact that $\big\{T_3(e_{\xi_{i_k},\eta_{i_l}})\big\}_{2\leq l<k<\infty}\sim\big\{e^H_{(k,l)}\big\}_{2\leq l<k<\infty}$ and \eqref{eq:l2 57}  that there exists  some positive number $C>1$ such that for any  sequence $\{\alpha_{k,l}\}_{2\leq l<k<\infty}$ of finitely many non-zero scalars, we have 
\begin{align*}
C^{-1}\Bigg(\sum_{2\leq l<k<\infty}\vert\alpha_{k,l}\vert^2\Bigg)^{1/2}&\leq\Bigg\Vert\sum_{2\leq l<k<\infty}\alpha_{k,l}(T_2+T_3)(e_{\xi_{i_k},\eta_{i_l}})\Bigg\Vert_{\cC_E}\\
&=\Bigg\Vert\sum_{2\leq l<k<\infty}\alpha_{k,l}\Big(x_{_{(k,l),l}}+y_{_{l,(l,k)}}+T_3(e_{\xi_{i_k},\eta_{i_l}})\Big)\Bigg\Vert_{\cC_E} \\
&\leq C\Bigg(\sum_{2\leq l<k<\infty}\vert\alpha_{k,l}\vert^2\Bigg)^{1/2}.
\end{align*}
Therefore, $\big\{(T_2+T_3)(e_{\xi_{i_k},\eta_{i_l}})\big\}_{2\leq l<k<\infty}\sim\big\{e^H_{(k,l)}\big\}_{2\leq l<k<\infty}$.

\item [(2).]
For any positive integer $N$,  set 
\[
\hat{\beta}_N:=\sum_{m=1}^N\hat{b}_m\otimes (\cdot|e_1)e_m\;\;\;\;{\rm and}\;\;\;\;\tilde{\beta}_N:=\sum_{m=1}^N\tilde{b}_m\otimes(\cdot|e_m)e_1.
\]
For any sequence $\{\alpha_k\}_{k=1}^\infty\in E$, we have 
\begin{eqnarray} 
&&\notag \left\Vert\left(
\begin{smallmatrix}
0&\tilde{\beta}_N\\\hat{\beta}_N&0
\end{smallmatrix}\right)
\right\Vert_{\mathcal C_E} \left\Vert\sum_{k=1}^\infty\alpha_ke_k\right\Vert_{\ell_2}\\\notag
&\stackrel{\eqref{hatb},\eqref{tildec}}{=}&\left\Vert\sum_{k=1}^\infty\alpha_k\sum_{m=1}^N\big(x_{_{(N+k+2,N+k+1),m}}+y_{_{m,(N+k+1,N+k+2)}}\big)\right\Vert_{\mathcal C_E}\\\notag
&=&\Bigg\Vert\Big(P_{[\bigcup_{j=N+2}^\infty F_{j+1,j}],[\bigcup_{j=1}^NL'_j]}+P_{[\bigcup_{j=1}^{N}K'_j],[\bigcup_{j=N+2}^\infty G_{j,j+1}]}\Big)\\\label{proj+proj}
&&\;\;\;(T_1+T_2+T_3+T_4)\bigg(\sum_{k=1}^\infty\alpha_ke_{\xi_{i_{N+k+2}},\eta_{i_{N+k+1}}}\bigg)\Bigg\Vert_{\mathcal C_E}\\\notag
&\stackrel{\eqref{diag-proj}}{\leq}&4^{\frac{1}{r}-1}\Bigg(2^{\frac{1}{r}-1}\Vert T\Vert+2^{\frac{1}{r}-1}4^{\frac{1}{r}}\bigg(\sum_{k=N}^\infty{\varepsilon_k}^r\bigg)^{\frac{1}{r}}\Bigg)\left\Vert\sum_{k=1}^\infty\alpha_ke_{\xi_{i_{N+k+2}},\eta_{i_{N+k+1}}}\right\Vert_{\mathcal C_E}\\
&=&8^{\frac{1}{r}-1}\Bigg(\Vert T\Vert+4^{\frac{1}{r}}\bigg(\sum_{k=N}^\infty{\varepsilon_k}^r\bigg)^{\frac{1}{r}}\Bigg)\left\Vert\sum_{k=1}^\infty\alpha_ke_k\right\Vert_E.\notag
\end{eqnarray}
Noting that $\hat{\beta}_N\to \hat{b}$ and $\tilde{\beta}_N\to \tilde{b}$ in $\mathcal C_E$ as $N\to \infty $, 
we have 
\[
\left\Vert\cdot \right\Vert_{\ell_2}\leq\frac{16^{\frac{1}{r}-1}\Vert T\Vert}{\left\Vert\left(
\begin{smallmatrix}
0&\tilde{b}\\\hat{b}&0
\end{smallmatrix}\right)
\right\Vert_{\mathcal C_E}}\left\Vert\cdot\right\Vert_E.
\]

\end{itemize}
The proof is complete.
\end{proof}

\begin{notation}\label{notationAk}
For any nonempty subset $A\subset\mathbb N$, we define
\[
K_A=\left[{\bigcup}_{i\in A}K_i\right],\;\;\;\;\;\;\;L_A=\left[{\bigcup}_{i\in A}L_i\right].
\]
Put
\[
K=\left[{\bigcup}_{i=1}^\infty K_i\right]\;\;\;\;{\rm and}\;\;\;\;L=\left[{\bigcup}_{i=1}^\infty L_i\right].
\]

Passing to a  subsequence of $\{i_k\}_{k=1}^\infty$ in the Theorem \ref{thma} if necessary, without loss of generality,  we may  assume that there exists a sequence $\{B_k\}_{k=1}^\infty$ of nonempty subsets of $\mathbb N$ with $B_k\subset\{n_{k-1}+1,\dots,n_k\}$ such that
\[
K'_k=K_{B_k},\;\;\;\;\;\;\;L'_k=L_{B_k},
\]
and $\mathbb N\setminus\bigcup_{k=1}^\infty B_k$ is an infinite set. 
Moreover, without loss of generality, we may  assume
\[
\sum_{i=1}^\infty\dim(K_i)=\sum_{i=1}^\infty\dim(L_i)=\infty.
\]
Hence, we may  assume that there exists  a sequence $\{A_k\}_{k=1}^\infty$ of mutually disjoint (infinite) subsets of $\mathbb N$ with $B_k\subset A_k$ such that
\[
\dim(K_{A_k})=\dim(L_{A_k})=\infty
\]
for every $k$. It is easy to construct two sequences $\{F_{(k,l)}\}_{2\leq l<k<\infty}$ and $\{G_{(l,k)}\}_{2\leq l<k<\infty}$ of mutually orthogonal closed subspaces of $H$ such that 
\[
F_{k,l}\subset F_{(k,l)}\subset K_{A_k},\;\;\;\;\;\;\;\;\dim(F_{(k,l)})=\infty,
\]
and
\[
G_{l,k}\subset G_{(l,k)}\subset L_{A_k},\;\;\;\;\;\;\;\;\dim(G_{(l,k)})=\infty.
\]
\end{notation}

\begin{lemma}\label{t3}
Let $T_3$ be the bounded operators given in Theorem \ref{thma}. Then there is a bounded operator $\tilde{R}_3:[e_{\xi_{i_k},\eta_{i_l}}]_{k,l=2}^\infty\to\mathcal C_E$ having the form
\[
\tilde{R}_3(e_{\xi_{i_k},\eta_{i_l}})=\left\{
\begin{array}{rcl}
x_{_{(k,l)}}+y_{_{(l,k)}}, &  & {2\leq l<k<\infty,}\\
0\;\;\;\;\;\;\;\;, &  & \text{otherwise,}\\
\end{array}
\right.
\]
such that $\tilde{R}_3|_{[e_{\xi_{i_k},\eta_{i_l}}]_{2\leq l<k<\infty}}$ is a small perturbation of $T_3$ associated with the Schauder decomposition $\big\{[e_{\xi_{i_k},\eta_{i_l}}]_{k=l+1}^\infty\big\}_{l=1}^\infty$, where
\[
\{x_{_{(k,l)}}\}_{2\leq l<k<\infty}\stackrel{\mbox{{\rm\tiny b.s}}}{\boxplus}\{F_{(k,l)}\}_{2\leq l<k<\infty}\otimes L\curvearrowleft\hat{b},
\]
\[
\{y_{_{(l,k)}}\}_{2\leq l<k<\infty}\stackrel{\mbox{{\rm\tiny b.s}}}{\boxplus}K\otimes\{G_{(l,k)}\}_{2\leq l<k<\infty}\curvearrowleft\tilde{b},
\]
and
$\left\Vert\left(
\begin{smallmatrix}
0& \tilde{b} \\ \hat{b}&0
\end{smallmatrix}\right)
\right\Vert_{\mathcal C_E} \left\Vert\cdot\right\Vert_{\ell_2}\leq16^{\frac{1}{r}-1}\Vert T\Vert\cdot \left\Vert\cdot \right\Vert_E$.
Moreover, we have 
\[
\lim_{N\to\infty}\left\Vert P_{[\bigcup_{k=N}^\infty K_k],[\bigcup_{k=N}^\infty L_k]}\tilde{R}_3\right\Vert=0.
\]
\end{lemma}

\begin{proof}
If both $\hat{b}_1$ and $\tilde{b}_1$ are zero, then we define  $\tilde{R}_3=0$. 

Assume that 
 $\hat{b}_1$ or $\tilde{b}_1$ is nonzero, and
\[
\big\{x_{_{(k,l),m}}\big\}_{1\leq m<l<k<\infty}\stackrel{\mbox{{\rm\tiny b.s}}}{\boxplus}\big\{\hat{u}'_{k,l}:F_{k,l}\to\hat{H}'\big\}_{2\leq l<k<\infty}\otimes\big\{\hat{v}'_m:L'_m\to\hat{H}'\big\}_{m=1}^\infty\curvearrowleft\hat{b}_m
\]
and
\[
\big\{y_{_{m,(l,k)}}\big\}_{1\leq m<l<k<\infty}
\stackrel{\mbox{{\rm\tiny b.s}}}{\boxplus}\big\{\tilde{u}'_m:K'_m\to\tilde{H}'\big\}_{m=1}^\infty\otimes\big\{\tilde{v}'_{l,k}:G_{l,k}\to\tilde{H}'\big\}_{2\leq l<k<\infty}\curvearrowleft\tilde{b}_m
\]
for every $m\geq1$.

We may construct a sequence of isometries $\big\{\hat{u}_{(k,l)}':F_{(k,l)}\to\hat{H}'\otimes_{_2}\ell_2\big\}_{2\leq l<k<\infty}$ and an isometry $\hat{v}':L\to \hat{H}'\otimes_{_2}\ell_2$ such that $\hat{u}_{(k,l)}'$ is an extension of $\hat{u}'_{k,l}(\cdot)\otimes_{_2} e_1$, i.e.
\[
\hat{u}'_{(k,l)}(f)=\hat{u}'_{k,l}(f)\otimes_{_2} e_1\;\;\;\;\;\;\;\;{\rm for\;any}\;f\in F_{k,l}
\]
and
\[
\hat{u}'_{(k,l)}\big(F_{(k,l)}\big)=\hat{H}'\otimes_{_2}e_1,
\]
and
\[
\hat{v}':\sum_{m=1}^\infty f_m\longmapsto\sum_{m=1}^\infty\hat{v}'_m(f_m)\otimes_{_2} e_m,\;\;\;\;\;\;\;\;{\rm where}\;f_m\in L'_m,\;m\geq1.
\]
Similarly, we may construct an isometry $\tilde{u}':K\to \tilde{H}'\otimes_{_2}\ell_2$ and a sequence of isometries $\big\{\tilde{v}_{(l,k)}':G_{(l,k)}\to\tilde{H}'\otimes_{_2}\ell_2\big\}$ such that
\[
\tilde{u}':\sum_{m=1}^\infty g_m\longmapsto\sum_{m=1}^\infty\tilde{u}'_m(g_m)\otimes_{_2} e_m,\;\;\;\;\;\;\;\;{\rm where}\;g_m\in  K'_m,\;m\geq1,
\]
and $\tilde{v}_{(l,k)}'$ is an extension of $\tilde{v}'_{l,k}(\cdot)\otimes_{_2} e_1$, i.e.
\[
\tilde{v}'_{(l,k)}(g)=\tilde{v}'_{l,k}(g)\otimes_{_2} e_1\;\;\;\;\;\;\;\;{\rm for\;any}\;g\in G_{k,l}
\]
and 
\[
\tilde{v}'_{(l,k)}\big(G_{(l,k)}\big)=\tilde{H}'\otimes_{_2}e_1.
\]
It follows that (see Remark \ref{rem} (3))
\begin{align}\label{xklm}
\big\{x_{_{(k,l),m}}\big\}_{1\leq m<l<k<\infty}\stackrel{\mbox{{\rm\tiny b.s}}}{\boxplus}\big\{\hat{u}_{(k,l)}':F_{(k,l)}\to\hat{H}'\otimes_{_2}\ell_2\big\}_{2\leq l<k<\infty}\otimes\big\{\hat{v}':L\to \hat{H}'\otimes_{_2}\ell_2\big\}\curvearrowleft\hat{b}_m\otimes e_{1,m}
\end{align}
and
\begin{align}\label{ymlk}
\big\{y_{_{m,(l,k)}}\big\}_{1\leq m<l<k<\infty}\stackrel{\mbox{{\rm\tiny b.s}}}{\boxplus}\big\{\tilde{u}':K\to \tilde{H}'\otimes_{_2}\ell_2\big\}\otimes\big\{\tilde{v}_{(l,k)}':G_{(l,k)}\to\tilde{H}'\otimes_{_2}\ell_2\big\}_{2\leq l<k<\infty}\curvearrowleft\tilde{b}_m\otimes e_{m,1}.
\end{align}
for every $m\geq1$, where $e_{1,m}=(\cdot|e_m)e_1$ and $e_{m,1}=(\cdot|e_1)e_m$.
Note that
\[
\sum_{m=1}^\infty\Vert\hat{b}_m\otimes e_{1,m}\Vert_{\mathcal C_E}^r\stackrel{{\rm Th}\;\ref{thma}}{\leq}\Bigg(1+\sum_{m=2}^\infty{\varepsilon_m}^r\Bigg)\Vert\hat{b}_1\Vert_{\mathcal C_E}^r<\infty
\]
and
\[
\sum_{m=1}^\infty\Vert\tilde{b}_m\otimes e_{m,1}\Vert_{\mathcal C_E}^r \stackrel{{\rm Th}\;\ref{thma}}{\leq}\bigg(1+\sum_{m=2}^\infty{\varepsilon_m}^r\bigg)\Vert\tilde{b}_1\Vert_{\mathcal C_E}^r<\infty.
\]
Therefore, the series $\sum_{m=1}^\infty\hat{b}_m\otimes e_{1,m}$ converges to $\hat{b}$, and the series $\sum_{m=1}^\infty\tilde{b}_m\otimes e_{m,1}$ converges to $\tilde{b}$. Define
\begin{equation}\label{t3ex}
x_{_{(k,l)}}:=(\hat{u}_{(k,l)}')^{\ast}\hat{b}\hat{v}'\;\;\;\;\;\;{\rm and}\;\;\;\;\;\;y_{_{(l,k)}}:=\tilde{u}'^\ast\tilde{b}\tilde{v}'_{(l,k)},
\end{equation}
for every $2\leq l<k<\infty$. We obtain
\begin{equation}\label{x(kl)}
\{x_{(k,l)}\}_{2\leq l<k<\infty}\stackrel{\mbox{{\rm\tiny b.s}}}{\boxplus}\big\{\hat{u}_{(k,l)}':F_{(k,l)}\to\hat{H}'\otimes_{_2}\ell_2\big\}_{2\leq l<k<\infty}\otimes\big\{\hat{v}':L\to \hat{H}'\otimes_{_2}\ell_2\big\}\curvearrowleft \hat{b}
\end{equation}
and
\begin{equation}\label{y(lk)}
\{y_{(l,k)}\}_{2\leq l<k<\infty}
\stackrel{\mbox{{\rm\tiny b.s}}}{\boxplus}\big\{\tilde{u}':K\to \tilde{H}'\otimes_{_2}\ell_2\big\}\otimes\big\{\tilde{v}_{(l,k)}':G_{(l,k)}\to\tilde{H}'\otimes_{_2}\ell_2\big\}_{2\leq l<k<\infty}\curvearrowleft\tilde{b}.
\end{equation}
Let  $\{e_{(k,l)}\}_{k,l}$ be the  unit basis of $\ell_2(\mathbb N\times\mathbb N)\cong\ell_2$. 
For any fixed 
  $l\geq2$ and any  sequence $\{\alpha_k\}_{k=l+1}^\infty$ of scalars with $\norm{ \{\alpha_k\}_{k=l+1}^\infty}_2<\infty$,
\begin{eqnarray*}
&&\left\Vert T_3\Big(\sum_{k=l+1}^\infty \alpha_ke_{\xi_{i_k},\eta_{i_l}}\Big)-\sum_{k=l+1}^\infty
\alpha_k\Big(x_{_{(k,l)}}+y_{_{(l,k)}}\Big)\right\Vert_{\mathcal C_E}^r\\
&\stackrel{\eqref{T3}}{=}& \Bigg\Vert\sum_{k=l+1}^\infty\alpha_k\bigg(\sum_{m=1}^{l-1}x_{_{(k,l),m}}-x_{_{(k,l)}}\bigg)+\sum_{k=l+1}^\infty\alpha_k\bigg(\sum_{m=1}^{l-1}y_{_{m,(l,k)}}-y_{_{(l,k)}}\bigg)\Bigg\Vert_{\mathcal C_E}^r\\
&\stackrel{\eqref{xklm},\eqref{ymlk},\eqref{x(kl)}\eqref{y(lk)}}{=}&\Bigg\Vert\sum_{k=l+1}^\infty\alpha_k\left({\sum}_{m=l}^\infty\hat{b}_m\otimes e_{1,m}\right)\otimes(\cdot|e_1)e_{_{(k,l)}}\\
&&\qquad \qquad  +\sum_{k=l+1}^\infty\alpha_k\left({\sum}_{m=l}^\infty\tilde{b}_m\otimes e_{m,1}\right)\otimes(\cdot|e_{_{(k,l)}})e_1\Bigg\Vert_{\mathcal C_E}^r\\
&\stackrel{\eqref{sum-infty}}{\leq}&4\sum_{m=l}^\infty\big(\Vert\hat{b}_m\Vert_{\mathcal C_E}^r+\Vert\tilde{b}_m\Vert_{\mathcal C_E}^r\big)\bigg(\sum_{k=l+1}^\infty\vert\alpha_k\vert^2\bigg)^{r/2}\\
&\leq&4\bigg(\big(\Vert\hat{b}_1\Vert_{\mathcal C_E}^r+\Vert\tilde{b}_1\Vert_{\mathcal C_E}^r\big)\sum_{m=l}^\infty{\varepsilon_m}^r\bigg)\left\Vert\sum_{k=l+1}^\infty \alpha_ke_{\xi_{i_k},\eta_{i_l}}\right\Vert_{\mathcal C_E}^r.
\end{eqnarray*}
Without loss of generality, we may assume that the sequence
  $\big\{\sum_{m=l}^\infty{\varepsilon_m}^r\big\}_{l=2}^\infty$ of positive numbers is small enough.
By Theorem \ref{2d}, there is an operator $R_3\in\mathcal B\big([e_{\xi_{i_k},\eta_{i_l}}]_{2\leq l<k<\infty},\mathcal C_E\big)$ with 
\[
R_3(e_{\xi_{i_k},\eta_{i_l}})=x_{_{(k,l)}}+y_{_{(l,k)}}
\]
for every $2\leq l<k<\infty$ such that $R_3$ is a perturbation of $T_3$ associated with the Schauder decomposition $\big\{[e_{\xi_{i_k},\eta_{i_l}}]_{k=l+1}^\infty\big\}_{l=2}^\infty$ for $\big\{4^{\frac{1}{r}}\big(\Vert\hat{b}_1\Vert_{\mathcal C_E}^r+\Vert\tilde{b}_1\Vert_{\mathcal C_E}^r\big)^{1/r}\big(\sum_{m=l}^\infty{\varepsilon_m}^r\big)^{1/r}\big\}_{l=2}^\infty$.

For any finitely non-zero sequence $\{\alpha_{k,l}\}_{k,l=2}^\infty$ of scalars,  we have 
\begin{eqnarray}\label{t3s}
&&\bigg\Vert\sum_{2\leq l<k<\infty}\alpha_{k,l}\big(x_{_{(k,l)}}+y_{_{(l,k)}}\big)\bigg\Vert_{\mathcal C_E}^r\\\notag
&\leq&4\bigg\Vert\sum_{2\leq l<k<\infty}\alpha_{k,l}x_{_{(k,l)}}\bigg\Vert_{\mathcal C_E}^r+4\bigg\Vert\sum_{2\leq l<k<\infty}\alpha_{k,l}y_{_{(l,k)}}\bigg\Vert_{\mathcal C_E}^r\\\notag
&\stackrel{\eqref{x(kl)},\eqref{y(lk)}}{=}&4\big(\Vert\hat{b}\Vert_{\mathcal C_E}^r+\Vert\tilde{b}\Vert_{\mathcal C_E}^r\big)\bigg(\sum_{2\leq l<k<\infty}\vert\alpha_{k,l}\vert^2\bigg)^{r/2}\\\notag
&\leq&4\big(\Vert\hat{b}\Vert_{\mathcal C_E}^r+\Vert\tilde{b}\Vert_{\mathcal C_E}^r\big)\bigg\Vert\sum_{l,k=2}^\infty\alpha_{k,l}e_{\xi_{i_k},\eta_{i_l}}\bigg\Vert_{\mathcal C_{\ell_2}}^r\\\notag
&\stackrel{{\rm Lem}\;\ref{t2t3}}{\leq}&2^{7-4r}\Vert T\Vert^r\bigg\Vert\sum_{l,k=2}^\infty\alpha_{k,l}e_{\xi_{i_k},\eta_{i_l}}\bigg\Vert_{\mathcal C_E}^r\notag.
\end{eqnarray}

Therefore, $R_3$ can be extended to a bounded operator $\tilde{R}_3:[e_{\xi_{i_k},\eta_{i_l}}]_{k,l=2}^\infty\to\mathcal C_E$ with
\[
\tilde{R}_3(e_{\xi_{i_k},\eta_{i_l}})=\left\{
\begin{array}{rcl}
x_{_{(k,l)}}+y_{_{(k,l)}}, &  & {2\leq l<k<\infty,}\\
0\;\;\;\;\;\;\;\;, &  & \text{otherwise.}\\
\end{array}
\right.
\]
By Lemma \ref{t2t3}, we have  $\left\Vert\left(
\begin{smallmatrix}
0&\hat{b}\\\tilde{b}&0
\end{smallmatrix}\right)
\right\Vert_{\mathcal C_E}\Vert\cdot\Vert_{\ell_2}\leq16^{\frac{1}{r}-1}\Vert T\Vert\cdot\Vert\cdot\Vert_E$.

Using that facts \eqref{hatb} and \eqref{tildeb}, and arguing similarly as \eqref{t3s}, we obtain that 
\begin{equation}
\lim_{N\to\infty}\big\Vert P_{[\bigcup_{k=N}^\infty K_k],[\bigcup_{k=N}^\infty L_k]}\tilde{R}_3\big\Vert=0.
\end{equation}
\end{proof}

\begin{lemma}\label{t4}
Let $T_4$ be the  bounded operator given in Theorem \ref{thma}. 
There is a bounded operator $\tilde{R}_4:[e_{\xi_{i_{2k+1}},\eta_{i_{2l}}}]_{k,l=1}^\infty\to\mathcal C_E$ having the form
\[
\tilde{R}_4(e_{\xi_{i_{2k+1}},\eta_{i_{2l}}})=x_{2k+1,2l}+y_{2l,2k+1}\;\;\;\;\;\;\;\;\mbox{for every}\;k,l\geq1,
\]
such that $\tilde{R}_4|_{[e_{\xi_{i_{2k+1}},\eta_{i_{2l}}}]_{1\leq l\leq k<\infty}}$ is a small perturbation of $T_4|_{[e_{\xi_{i_{2k+1}},\eta_{i_{2l}}}]_{1\leq l\leq k<\infty}}$ associated with the Schauder decomposition $\big\{[e_{\xi_{i_{2k+1}},\eta_{i_{2l}}}]_{k=l}^\infty\big\}_{l=1}^\infty$, where
\[
\{x_{2k+1,2l}\}_{k,l=1}^\infty\stackrel{\mbox{{\rm\tiny b.s}}}{\boxplus}\{K_{A_{2k+1}}\}_{k=1}^\infty\otimes\{L_{A_{2l}}\}_{l=1}^\infty\curvearrowleft\;,
\]
and
\[
\{y_{2l,2k+1}\}_{k,l=1}^\infty\stackrel{\mbox{{\rm\tiny b.s}}}{\boxplus}\{K_{A_{2l}}\}_{l=1}^\infty\otimes\{L_{A_{2k+1}}\}_{k=1}^\infty\curvearrowleft\;.
\]

Moreover, if $T_4\neq0$ (i.e. $\hat{a}_1$ or $\tilde{a}_1$ is nonzero), then $\tilde{R}_4$ is an isomorphic embedding.
\end{lemma}

\begin{proof}
Without loss of generality, we may assume that $\hat{a}_m\in\mathcal B(H)$ and $\tilde{a}_m\in\mathcal B(H)$, and
\[
\big\{x_{_{(k,m),l}}\big\}_{1\leq m<l<k<\infty}\stackrel{\mbox{{\rm\tiny b.s}}}{\boxplus}\big\{\hat{u}_{k,m}:F_{k,m}\to H\big\}_{1\leq m<k-1<\infty}\otimes\big\{\hat{v}_l:L'_l\to H\big\}_{l=2}^\infty\curvearrowleft\hat{a}_m
\]
and
\[
\big\{y_{_{l,(m,k)}}\big\}_{1\leq m<l<k<\infty}
\stackrel{\mbox{{\rm\tiny b.s}}}{\boxplus}\big\{\tilde{u}_l:K'_l\to H\big\}_{l=2}^\infty\otimes\big\{\tilde{v}_{m,k}:G_{m,k}\to H\big\}_{1\leq m<k-1<\infty}\curvearrowleft\tilde{a}_m
\]
for every $m\geq1$.

There is  a sequence of isometries
$\big\{u_i:K_{A_i}\to H\otimes_{_2}\ell_2\big\}_{i=2}^\infty$ satisfying
\[
u_{2k}(f)=\tilde{u}_{2k}(f)\otimes_{_2} e_1\;\;\;\;\;\;\;\;{\rm for\;any}\;f\in K'_{2k}=K_{B_{2k}},
\]
\[
u_{2k}(K_{A_{2k}})=H\otimes_{_2}e_1,
\]
and 
\[
u_{2k+1}\left({\sum}_{m=1}^{2k-1}f_m\right)=\hat{u}_{2k+1,m}(f_m)\otimes_{_2} e_m,\;\;\;\;\;\;\;\;{\rm where}\;f_m\in F_{2k+1,m},\;1\leq m<2k,
\]
\[
{\rm im}(\hat{a}_m)\otimes_{_2} e_m\subset u_{2k+1}(K_{A_{2k+1}})\;\;\;\;\;\;\;\;{\rm for\;every}\;m\geq2k,
\]
and a sequence of isometries $\big\{v_i:L_{A_i}\to H\otimes_{_2}\ell_2\big\}_{i=2}^\infty$ satisfying
\[
v_{2k}(g)=\hat{v}_{2k}(g)\otimes_{_2} e_1\;\;\;\;\;\;\;\;{\rm for\;any}\;g\in L'_{2k}=L_{B_{2k}},
\]
\[
v_{2k}(L_{A_{2k}})=H\otimes_{_2}e_1,
\]
and
\[
v_{2k+1}\left({\sum}_{m=1}^{2k-1}g_m\right)=\tilde{v}_{m,2k+1}(g_m)\otimes_{_2} e_m,\;\;\;\;\;\;\;\;{\rm where}\;g_m\in G_{m,2k+1},\;1\leq m<2k, 
\]
\[
{\rm ker}(\tilde{a}_m)^{\perp}\otimes_{_2} e_m\subset v_{2k+1}(L_{A_{2k+1}})\;\;\;\;\;\;\;\;{\rm for\;every}\;m\geq2k.
\]
It follows that (see Remark \ref{rem} (2))
\begin{align*}
&\big\{x_{_{(2k+1,m),2l}}\big\}_{1\leq m<2l<2k+1<\infty}\\
\stackrel{\mbox{{\rm\tiny b.s}}}{\boxplus}&\big\{u_{2k+1}:K_{A_{2k+1}}\to H\otimes_{_2}\ell_2\big\}_{k=1}^\infty\otimes\big\{v_{2l}:L_{A_{2l}}\to H\otimes_{_2}\ell_2\big\}_{l=1}^\infty\curvearrowleft\hat{a}_m\otimes e_{m,1}
\end{align*}
and
\begin{align*}
&\big\{y_{_{2l,(m,2k+1)}}\big\}_{1\leq m<2l<2k+1<\infty}\\
\stackrel{\mbox{{\rm\tiny b.s}}}{\boxplus}&
\big\{u_{2l}:K_{A_{2l}}\to H\otimes_
{_2}\ell_2\big\}_{l=1}^\infty\otimes\big\{v_{2k+1}:L_{2k+1}\to H\otimes_{_2}\ell_2\big\}_{k=1}^\infty\curvearrowleft\tilde{a}_m\otimes e_{1,m}
\end{align*}
for every $m\geq1$, where $e_{m,1}=(\cdot|e_1)e_m$ and $e_{1,m}=(\cdot|e_m)e_1$. Note that 
\[
\sum_{m=1}^\infty\Vert\hat{a}_m\otimes e_{m,1}\Vert_{\mathcal C_E}^r\stackrel{{\rm Th}\;\ref{thma}}{\leq}\bigg(1+\sum_{m=2}^\infty{\varepsilon_m}^r\bigg)\Vert\hat{a}_1\Vert_{\mathcal C_E}^r<\infty
\]
and
\[
\sum_{m=1}^\infty\Vert\tilde{a}_m\otimes e_{1,m}\Vert_{\mathcal C_E}^r\stackrel{{\rm Th}\;\ref{thma}}{\leq}\bigg(1+\sum_{m=2}^\infty{\varepsilon_m}^r\bigg)\Vert\tilde{a}_1\Vert_{\mathcal C_E}^r<\infty.
\]
Suppose that the series $\sum_{m=1}^\infty\hat{a}_m\otimes e_{m,1}$ converges to $\hat{a}$, and the series $\sum_{m=1}^\infty\tilde{a}_m\otimes e_{1,m}$ converges to $\tilde{a}$. Denote 
\[
x_{2k+1,2l}:=(u_{2k+1})^\ast\hat{a}v_{2l}\;\;\;\;\;\;{\rm and}\;\;\;\;\;\;y_{2l,2k+1}= (u_{2l})^\ast\tilde{a}v_{2k+1}
\]
for every $k,l\geq1$.  
We obtain
\begin{equation}\label{x2k+12l}
\{x_{2k+1,2l}\}_{k,l=1}^\infty\stackrel{\mbox{{\rm\tiny b.s}}}{\boxplus}\big\{u_{2k+1}:K_{A_{2k+1}}\to H\otimes_{_2}\ell_2\big\}_{k=1}^\infty\otimes\big\{v_{2l}:L_{A_{2l}}\to H\otimes_{_2}\ell_2\big\}_{l=1}^\infty\curvearrowleft\hat{a}
\end{equation}
and
\begin{equation}\label{y2l2k+1}
\{y_{2l,2k+1}\}_{k,l=1}^\infty
\stackrel{\mbox{{\rm\tiny b.s}}}{\boxplus}\big\{u_{2l}:K_{A_{2l}}\to H\otimes_
{_2}\ell_2\big\}_{l=1}^\infty\otimes\big\{v_{2k+1}:L_{2k+1}\to H\otimes_{_2}\ell_2\big\}_{k=1}^\infty\curvearrowleft\tilde{a}.
\end{equation}
Fix $l\geq1$, for any sequence $\{\alpha_k\}_{k=l}^\infty$ of scalars with $\norm{ \{\alpha_k\}_{k=l}^\infty}_2<\infty$, we have
\begin{eqnarray}\label{t4s}
 &&\left\Vert T_4\Big(\sum_{k=l}^\infty \alpha_ke_{\xi_{i_{2k+1}},\eta_{i_{2l}}}\Big)-\sum_{k=l}^\infty
\alpha_k\Big(x_{2k+1,2l}+y_{2l,2k+1}\Big)\right\Vert_{\mathcal C_E}^r\\\notag 
&\stackrel{\eqref{hata},\eqref{tildea}}{=}&\bigg\Vert\sum_{k=l}^\infty\alpha_k\bigg(\Big({\sum}_{m=2l}^\infty\hat{a}_m\otimes e_{m,1}\Big)\otimes(\cdot|e_{2l})e_{2k+1}\\\notag
&&\;\;\;\;\;\;\;\;\;\;\;\;\;\;\;\;+\Big({\sum}_{m=2l}^\infty\tilde{a}_m\otimes e_{1,m}\Big)\otimes(\cdot|e_{2k+1})e_{2l}\bigg)\bigg\Vert_{\mathcal C_E}^r\\\notag
&\leq&4\sum_{m=2l}^\infty\bigg(\bigg\Vert\sum_{k=l}^\infty\alpha_k\big(\hat{a}_m\otimes e_{m,1}\big)\otimes(\cdot|e_{2l})e_{2k+1}\bigg\Vert_{\mathcal C_E}^r\\\notag
&&\;\;\;\;\;\;\;\;\;\;+\bigg\Vert\sum_{k=l}^\infty\alpha_k\big(\tilde{a}_m\otimes e_{1,m}\big)\otimes(\cdot|e_{2k+1})e_{2l}\bigg\Vert_{\mathcal C_E}^r\bigg)\\\notag
&\leq&4\sum_{m=2l}^\infty\big(\Vert\hat{a}_m\Vert_{\mathcal C_E}^r+\Vert\tilde{a}_m\Vert_{C_E}^r\big)\bigg(\sum_{k=l}^\infty\vert\alpha_k\vert^2\bigg)^{r/2}\\\notag
&\stackrel{{\rm Th}\;\ref{thma}}{\leq}&4\bigg(\big(\Vert\hat{a}_1\Vert_{\mathcal C_E}^r+\Vert\tilde{a}_1\Vert_{C_E}^r\big)\sum_{m=2l}^\infty{\varepsilon_m}^r\bigg)\left\Vert\sum_{k=l}^\infty \alpha_ke_{\xi_{i_{2k+1}},\eta_{i_{2l}}}\right\Vert_{\mathcal C_E}^r\notag.
\end{eqnarray}
Without loss of generality, we may assume that the sequence
  $\big\{\sum_{m=l}^\infty{\varepsilon_m}^r\big\}_{l=2}^\infty$ of positive numbers is small enough.
By Theorem \ref{2d}, there is an operator $R_4\in\mathcal B\big([e_{\xi_{i_{2k+1}},\eta_{i_{2l}}}]_{1\leq l\leq k<\infty},\mathcal C_E\big)$ with 
\[
R_4(e_{\xi_{i_{2k+1}},\eta_{i_{2l}}})=x_{2k+1,2l}+y_{2l,2k+1}
\]
for every $1\leq l\leq k<\infty$ such that $R_4$ is a perturbation of $T_4$ associated with the Schauder decomposition $\big\{[e_{\xi_{i_{2k+1}},\eta_{i_{2l}}}]_{k=l}^\infty\big\}_{l=1}^\infty$ for $\big\{4^{\frac{1}{r}}\big(\Vert\hat{a}_1\Vert_{\mathcal C_E}^r+\Vert\tilde{a}_1\Vert_{\mathcal C_E}^r\big)^{1/r}\big(\sum_{m=2l}^\infty{\varepsilon_m}^r\big)^{1/r}\big\}_{l=1}^\infty$.

We define   isometry $\wideparen{u}$ (respectively, $\wideparen{v}$)   by formula \eqref{u},  which is induced by the sequence $\{u_i\}_{i=1}^\infty$ (respectively, $\{v_i\}_{i=1}^\infty$) of isometries above.
For any finitely non-zero sequence $\{\alpha_k\}_{k=1}^\infty$ of scalars, by Remark~\ref{rem} (2), we have 
\begin{eqnarray}\label{ineqCEtoE}
&&
\left\Vert\sum_{k=1}^\infty\alpha_k\big(\hat{a}\otimes(\cdot|e_{2k})e_{2k+1}+\tilde{a}\otimes(\cdot|e_{2k+1})e_{2k}\big)\right\Vert_{\mathcal C_E}\\
&=&
\left\Vert\wideparen{u}\bigg(\sum_{k=1}^\infty\alpha_k (x_{2k+1,2l}+y_{2l,2k+1} )\bigg)\wideparen{v} \right\Vert_{\mathcal C_E}\nonumber \\
&\stackrel{\eqref{x2k+12l},\eqref{y2l2k+1}}{\leq} &
  \Vert R_4\Vert\left\Vert\sum_{k=1}^\infty\alpha_ke_{\xi_{i_{2k+1}},\eta_{i_{2k}}}\right\Vert_{\mathcal C_E}\nonumber \\
  &=& \Vert R_4\Vert\left\Vert\sum_{k=1}^\infty\alpha_ke_k\right\Vert_E\nonumber.
\end{eqnarray}
Using Proposition \ref{ie}, for any finitely non-zero sequence $\{\alpha_{k,l}\}_{k,l=1}^\infty$ of scalars, and setting   $x=\sum_{k,l=1}^\infty\alpha_{k,l}e_{\xi_{i_{2k+1}},\eta_{i_{2l}}}$, we have 
\begin{eqnarray*}
&&\left\Vert\sum_{k,l=1}^\infty\alpha_{k,l}\big(x_{2k+1,2l}+y_{2l,2k+1}\big)\right\Vert_{\mathcal C_E}\\
&\stackrel{\eqref{x2k+12l},\eqref{y2l2k+1}}{=}& \left\Vert\sum_{k,l=1}^\infty\alpha_{k,l}\big(\hat{a}\otimes(\cdot|e_{2l})e_{2k+1}+\tilde{a}\otimes(\cdot|e_{2k+1})e_{2l}\big)\right\Vert_{\mathcal C_E}\\
&\stackrel{{\rm Prop}\;\ref{ie}}{=}&\left\Vert\sum_{k=1}^\infty\alpha_{k,l}\left(
\begin{smallmatrix}
0&\tilde{a}\\\hat{a}&0
\end{smallmatrix}\right)\otimes(\cdot|e_l)e_k\right\Vert_{\mathcal C_E}\\
&=&\left\Vert\left(
\begin{smallmatrix}
0&\tilde{a}\\\hat{a}&0
\end{smallmatrix}\right)\otimes\left({\sum}_{k,l=1}^\infty\alpha_{k,l}(\cdot|e_l)e_k\right)\right\Vert_{\mathcal C_E}\\
&=&\left\Vert\left(
\begin{smallmatrix}
0&\tilde{a}\\\hat{a}&0
\end{smallmatrix}\right)\otimes x\right\Vert_{\mathcal C_E}\\
&=&\left\Vert\left(
\begin{smallmatrix}
0&\tilde{a}\\\hat{a}&0
\end{smallmatrix}\right)\otimes{\sum}_{k=1}^\infty s_k(x)(\cdot|e_k)e_k\right\Vert_{\mathcal C_E}\\
&\stackrel{{\rm Prop}\;\ref{ie}}{=}&\left\Vert\sum_{k=1}^\infty s_k(x)\big(\hat{a}\otimes(\cdot|e_{2k})e_{2k+1}+\tilde{a}\otimes(\cdot|e_{2k+1})e_{2k}\big)\right\Vert_{\mathcal C_E}\\
&\stackrel{\eqref{ineqCEtoE}}{\leq}&\Vert R_4\Vert\left\Vert\sum_{k=1}^\infty s_k(x)e_k\right\Vert_E=\Vert R_4\Vert\cdot\Vert x\Vert_{\mathcal C_E}.
\end{eqnarray*}
Therefore, $R_4$ can be extended to a bounded operator $\tilde{R}_4:[e_{\xi_{i_{2k+1}},\eta_{i_{2l}}}]_{k,l=1}^\infty\to \mathcal C_E$ with
\[
\tilde{R}_4(e_{\xi_{i_{2k+1}},\eta_{i_{2l}}})=x_{2k+1,2l}+y_{2l,2k+1}
\]
for all $k,l\geq1$.

By Theorem \ref{thma} and the definitions of $\hat{a}$ and $\tilde {a}$, the operator  $\hat{a}_1$ (respectively,  $\tilde{a}_1$) is nonzero if and only if $\hat{a}$ (respectively, $\tilde{a}$) is nonzero. Without loss of generality,  we may assume $\hat{a}\neq0$. In particular, $\norm{\hat{a}}_\infty\neq0$.   For any finitely non-zero sequence $\{\alpha_{k,l}\}_{k,l=1}^\infty$ of scalars, we have 
\begin{eqnarray*}
&&\bigg\Vert\sum_{k,l=1}\tilde{R}_4\big(\alpha_{k,l}e_{\xi_{i_{2k+1}},\eta_{i_{2l}}}\big)\bigg\Vert_{\mathcal C_E}\\
&\stackrel{\eqref{x2k+12l},\eqref{y2l2k+1}}{=}&\left\Vert\sum_{k,l=1}^\infty\alpha_{k,l}\big(\hat{a}\otimes(\cdot|e_{2l})e_{2k+1}+\tilde{a}\otimes(\cdot|e_{2k+1})e_{2l}\big)\right\Vert_{\mathcal C_E}\\
&\geq&\left\Vert\sum_{k,l=1}^\infty\alpha_{k,l}\hat{a}\otimes(\cdot|e_{2l})e_{2k+1}\right\Vert_{\mathcal C_E}  \\ 
&\geq&  \norm{\hat{a}}_\infty   \bigg\Vert\sum_{k,l=1}\alpha_{k,l}(\cdot|e_{2l})e_{2k+1}\bigg\Vert_{\mathcal C_E} \\ 
&=&\norm{\hat{a}}_\infty  \bigg\Vert\sum_{k,l=1}\alpha_{k,l}e_{\xi_{i_{2k+1}},\eta_{i_{2l}}}\bigg\Vert_{\mathcal C_E}.
\end{eqnarray*}
In other words,  if $T_4\neq0$, then $\tilde{R}_4$ is an isomorphic embeddings.
\end{proof}

\begin{remark}\label{rem-t4}
It can be seen that
\[
\{x_{2k+1,2l}+y_{2l,2k+1}\}_{k,l=1}^\infty\simeq\left\{
\left(
\begin{smallmatrix}
0&\tilde{a}\\\hat{a}&0
\end{smallmatrix}\right)\otimes e_{k,l}
\right\}_{k,l=1}^\infty,
\]
where $\hat{a}=\sum_{m=1}^\infty\hat{a}_m\otimes e_{m,1}$ and $\tilde{a}=\sum_{m=1}^\infty\tilde{a}_m\otimes e_{1,m}$, in which $e_{k,l}$ is given by $(\cdot|e_l^H)e_k^H$.
\end{remark}

\begin{lemma}\label{t2}
Let $T_2$ be the bounded operator given in Theorem \ref{thma}. Then
either
\begin{itemize}
\item [(1)]
there is a positive integer $N$ such that $\big\{T_2(e_{\xi_{i_{2k+1}},\eta_{i_{2l}}})\big\}_{N\leq l\leq k<\infty }\sim\big\{e^H_{(k,l)}\big\}_{N\leq l\leq k<\infty}$, or
\item [(2)]
there is an increasing subsequence $\{i'_l\}_{l=1}^\infty$ of $\{i_l\}_{l=2}^\infty$ and a normalized sequence $\{\eta'_l\}_{l=1}^\infty$ with $\eta'_l\in[\eta_{i_{2s}}]_{i'_{2l}\leq i_{2s}<i'_{2l+1}}$ such that $T_2|_{[e_{\xi_{i'_{2k+1}},\eta'_l}]_{1\leq l\leq k<\infty}}$ is a small perturbation of the zero operator associated with the Schauder decomposition $\big\{[e_{\xi_{i'_{2k+1}},\eta'_l}]_{k=l}^\infty\big\}_{l=1}^\infty$.
\end{itemize}
\end{lemma}

\begin{proof}
Let
\[
c_l=\left(
\begin{smallmatrix}
0&\tilde{c}_l\\\hat{c}_l&0
\end{smallmatrix}\right)
\]
for every $l\geq2$. By Remark \ref{rem} (2), for any finitely non-zero sequence of scalars $\{\alpha_l\}_{l=1}^n$, we have 
\begin{eqnarray*}
\left\Vert\sum_{l=1}^n\alpha_lT_2(e_{\xi_{i_{2n+1}},\eta_{i_{2l}}})\right\Vert_{\mathcal C_E}&\stackrel{\eqref{T2},\eqref{hatc},\eqref{tildec}}{=}&\left\Vert\sum_{l=1}^n\alpha_l\big(\hat{c}_{2l}\otimes(\cdot|e_{2l})e_{2l+1}+\tilde{c}_{2l}\otimes(\cdot|e_{2l+1})e_{2l}\big)\right\Vert_{\mathcal C_E}\\
&\stackrel{{\rm Prop}\;\ref{ie}}{=}&\left\Vert\sum_{l=1}^n\alpha_l\big(c_{2l}\otimes(\cdot|e_l)e_l\big)\right\Vert_{\mathcal C_E}.
\end{eqnarray*}
It follows that
\[
\left\Vert\sum_{l=1}^n\alpha_l\big(c_{2l}\otimes(\cdot|e_l)e_l\big)\right\Vert_{\mathcal C_E}\leq\Vert T_2\Vert\bigg(\sum_{l=1}^n\vert\alpha_l\vert^2\bigg)^{1/2}.
\]
Therefore, the operator $Q$ defined by 
\[
Q:\{\alpha_l\}_{l=1}^\infty\in\ell_2\mapsto\sum_{l=1}^\infty\alpha_l\big(c_{2l}\otimes(\cdot|e_l)e_l\big)
\]
is bounded.

For any positive integer $N$ and any finite sequence of scalars $\{\alpha_{k,l}\}_{N\leq l\leq k\leq n}$, we have 

\begin{eqnarray*}
&&\left\Vert\sum_{N\leq l\leq k\leq n}\alpha_{k,l}T_2(e_{\xi_{i_{2k+1}},\eta_{i_{2l}}})\right\Vert_{\mathcal C_E}\\
&\stackrel{\eqref{T2}}{=}&
\left\Vert\sum_{N\leq l\leq k\leq n}\alpha_{k,l}\big(x_{_{(2k+1,2l),2l}}+y_{_{2l,(2l,2k+1)}}\big)\right\Vert_{\mathcal C_E}\\
&\stackrel{\eqref{hatc},\eqref{tildec}}{=}&\left\Vert\sum_{N\leq l\leq k\leq n}\alpha_{k,l}\big(\hat{c}_{2l}\otimes(\cdot|e_{2^{2l}})e_{2^{2l}(2k+1)}+\tilde{c}_{2l}(\cdot|e_{2^{2l}(2k+1)})e_{2^{2l}}\big)\right\Vert_{\mathcal C_E}\\
&=&\left\Vert\sum_{l=N}^n\sum_{k=l} ^n\alpha_{k,l}\big(\hat{c}_{2l}\otimes(\cdot|e_{2^{2l}})e_{2^{2l}(2k+1)}+\tilde{c}_{2l}(\cdot|e_{2^{2l}(2k+1)})e_{2^{2l}}\big)\right\Vert_{\mathcal C_E}\\
&\stackrel{{\rm Prop}\;\ref{ie}}{=}&\left\Vert\sum_{l=N}^n\bigg(c_{2l}\otimes\bigg((\cdot|e_{2^{2l}})\sum_{k=l}^n\alpha_{k,l}e_{2^{2l}(2k+1)}\bigg)\bigg)\right\Vert_{\mathcal C_E}\\
&=&\left\Vert\sum_{l=2}^n\bigg(\sum_{k=l+1}^n\vert\alpha_{k,l}\vert^2\bigg)^{1/2}c_{2l}\otimes(\cdot|e_l)e_l\right\Vert_{\mathcal C_E}   \\
&=&\left\Vert Q\left(\sum_{l=N}^n\bigg(\sum_{k=l}^n\vert\alpha_{k,l}\vert^2\bigg)^{1/2}e_l\right)\right\Vert_{\mathcal C_E}.
\end{eqnarray*}
\begin{itemize}
\item [Case I:]
There exists a positive integer $N$ such that $Q|_{[e_i]_{i=N}^\infty}$ is an isomorphic embedding. In particular, we have  $\big\{T_2(e_{\xi_{i_{2k+1}},\eta_{i_{2l}}})\big\}_{N\leq l\leq k <\infty }\sim\big\{e^H_{(k,l)}\big\}_{N\leq l<k<\infty}$.
\item [Case II:]
Otherwise, there exists  an increasing sequence of positive integers $\{l_t\}_{t=1}^\infty$ and $\alpha^{^{_{(t)}}}=\sum_{l_t\leq l<l_{t+1}}\alpha^{^{_{(t)}}}_le_l\in[e_l]_{l_t\leq l<l_{t+1}}$ with $\Vert\alpha^{^{_{(t)}}}\Vert_2=1$ for every $t$ such that (here, $\{\varepsilon_t \}_{t\ge 1}$ is a sequence of positive numbers which are small enough)
\[
\big\Vert Q\big(\alpha^{^{_{(t)}}}\big)\big\Vert_{\mathcal C_E}\leq\varepsilon_t.
\]
Set 
\[
\eta'_{t}=\sum_{l_t\leq l<l_{t+1}}\alpha^{^{_{(t)}}}_l\eta_{i_{2l}}.
\]
Note that $\norm{\eta_t'}_H =1$ for all $t\ge 1$. Fix each $t\ge1$, 
for any finite sequence of scalars $\{\beta_s\}_{s=t+1}^n$, we have

\begin{align*}
\left\Vert T_2\Big(\sum_{s=t+1}^n\beta_se_{\xi_{i_{2l_s+1}},\eta'_t}\Big)\right\Vert_{\mathcal C_E}&=\left\Vert\sum_{l_t\leq l<l_{t+1}}\sum_{s=t+1}^n\beta_s\alpha^{^{_{(t)}}}_lT_2(e_{\xi_{i_{2l_s+1}},\eta_{i_{2l}}})\right\Vert_{\mathcal C_E}\\
&=\left\Vert Q\left(\sum_{l_t\leq l<l_{t+1}}\bigg(\sum_{s=t+1}^n\vert\beta_s\vert^2\bigg)^{1/2}\alpha^{^{_{(t)}}}_le_l\right)\right\Vert_{\mathcal C_E}\\
&=\bigg(\sum_{s=t+1}^n\vert\beta_s\vert^2\bigg)^{1/2}\big\Vert Q\big(\alpha^{^{_{(t)}}}\big)\big\Vert_{\mathcal C_E}\\
&\leq\varepsilon_t\left\Vert\sum_{s=t+1}^n\beta_se_{\xi_{i_{2l_s+1}},\eta'_t}\right\Vert_{\mathcal C_E}.
\end{align*}
By Theorem \ref{2d}, the proof is complete. 
\end{itemize}
\end{proof}

\begin{remark}\label{t2+t3-rem}
Combining Lemma \ref{t2t3} and Lemma \ref{t2}, we obtain that either
\[
\big\{(T_2+T_3)(e_{\xi_{i_{2k+1}},\eta_{i_{2l}}})\big\}_{N\leq l\leq k<\infty }\sim\big\{e^H_{(k,l)}\big\}_{N\leq l\leq k<\infty},
\] 
or
$(T_2+T_3)|_{[e_{\xi_{i'_{2k+1}},\eta'_l}]_{1\leq l\leq k<\infty}}$ is a small perturbation of the zero operator associated with the Schauder decomposition $\big\{[e_{\xi_{i'_{2k+1}},\eta'_l}]_{k=l}^\infty\big\}_{l=1}^\infty$. 
\end{remark}

Combining Lemma \ref{t4} and Remark \ref{t2+t3-rem}, we obtain the following theorem.
\begin{theorem}\label{thma'}
   Under the assumption of Theorem \ref{thma}, there is an increasing sequence $\{i'_k\}_{k=1}^\infty$ of positive integers, a sequence $\{A_k\}_{k=1}^\infty$ of mutually disjoint subsets of $\mathbb N$, a normalized sequence $\{\eta'_k\}_{k=1}^\infty$ with $\eta'_k\in[\eta_i]_{i'_{2k}\leq i<i'_{2k+1}}$, and three operators $S_1,S_2,S_3\in\mathcal B\big(\mathcal T_{E,\{\xi_{i'_{2k+1}},\eta'_k\}_{k=1}^\infty},\mathcal C_E\big)$ such that $S_1+S_2+S_3$ is a small perturbation of $T|_{\mathcal T_{E,\{\xi_{i'_{2k+1}},\eta'_k\}_{k=1}^\infty}}$ associated with the Schauder decomposition $\big\{[e_{\xi_{i'_{2k+1}},\eta'_l}]_{k=l}^\infty\big\}_{l=1}^\infty$, where
\begin{itemize}
    \item [(1)]
$S_1(e_{\xi_{i'_{2k+1}},\eta'_l})=p_{_{K_{A_{2k+1}}}}T(e_{\xi_{i'_{2k+1}},\eta'_l})p_{_{L_{A_{2k+1}}}}$ for every $1\leq l\leq k<\infty$,
\item [(2)]
$\big\{S_2(e_{\xi_{i'_{2k+1}},\eta'_l})\big\}_{1\leq l\leq k<\infty }\sim\big\{e^H_{(k,l)}\big\}_{1\leq l\leq k<\infty}$ or $S_2=0$, and
\item [(3)] $S_3(e_{\xi_{i'_{2k+1}},\eta'_l})=x_{2k+1,2l}+y_{2l,2k+1}$ for every $1\leq k\leq l<\infty$, where 
\[
\{x_{2k+1,2l}\}_{1\leq l\leq k<\infty}\stackrel{\mbox{{\rm\tiny b.s}}}{\boxplus}\{K_{A_{2k+1}}\}_{k=1}^\infty\otimes\{L_{A_{2l}}\}_{l=1}^\infty\curvearrowleft\;
\]
and
\[
\{y_{2l,2k+1}\}_{1\leq l\leq k<\infty}\stackrel{\mbox{{\rm\tiny b.s}}}{\boxplus}\{K_{A_{2l}}\}_{l=1}^\infty\otimes\{L_{A_{2k+1}}\}_{k=1}^\infty\curvearrowleft\;.
\]
\end{itemize}
\end{theorem}

\begin{remark}
Let $F$ be another separable quasi-Banach symmetric sequence space. If 
we consider a bounded operator $T:\mathcal T_{F,\{\xi_i,\eta_i\}_{i=1}^\infty}\to \mathcal C_E$ with $p_{_{[\bigcup_{i=1}^\infty K_i]}}T(\cdot)p_{_{[\bigcup_{i=1}^\infty L_i]}}=T$, 
  then, by Remark~\ref{tftoce}, a mutatis mutandis repetition of  the proofs of the previous lemmas yields that the conclusion of Theorem \ref{thma'} remains valid. Moreover,
\begin{itemize}
\item [(i)]
if $F\nsubseteq\ell_2$, then $S_2=0$ in the  above theorem. If $F\nsubseteq E$, then $S_3=0$  in the  above theorem.
\item [(ii)] 
if $\{K_i\}_{i=1}^\infty$ and $\{L_i\}_{i=1}^\infty$ are chosen so that $p_{K_i}({\rm im}(T))p_{L_j}=\{0\}$ for every $1\leq i<j<\infty$, then ``$y_{2l,2k+1}$" parts are $0$. In particular, $K_i=[\xi_i]$, $L_i=[\eta_i]$ and $T\in\mathcal B\big(\mathcal T_{F,\{\xi_i,\eta_i\}_{i=1}^\infty}\big)$, so we obtain $S_1,S_2,S_3\in \mathcal B\big(\mathcal T_{F,\{\xi_i,\eta_i\}_{i=1}^\infty}\big)$.
\end{itemize}
\end{remark}

\begin{lemma}\label{c0l2}
Let $T_1$ and $T_2$ be the bounded operators given in Theorem \ref{thma}. If, in addition, $\ell_2\not\hookrightarrow E$, then there is
\begin{itemize}
    \item an increasing subsequence $\{i'_k\}_{k=1}^\infty$ of $\{i_k\}_{k=2}^\infty$,
    \item a normalized sequence $\{\xi'_k\}_{k=1}^\infty$ with $\xi'_k\in[\xi_{i_{2s+1}}]_{i'_{2k+1}\leq i_{2s+1}<i'_{2k+2}}$, and
    \item  a normalized sequence $\{\eta'_k\}_{k=1}^\infty$ with $\eta'_k\in[\eta_{i_{2s}}]_{i'_{2k}\leq i_{2s}<i'_{2k+1}}$ 
\end{itemize}
  such that  both $T_1|_{[e_{\xi'_k,\eta'_l}]_{1\leq l\leq k<\infty}}$ and 
$T_2|_{[e_{\xi'_k,\eta'_l}]_{1\leq l\leq k<\infty}}$ are the small perturbations of the zero operator associated with the  Schauder decomposition $\big\{[e_{\xi'_k,\eta'_l}]_{k=l}^\infty\big\}_{l=1}^\infty$.

Consequently, if $T$ is an isomorphic embedding, then $\hat{a}_1$ or $\tilde{a}_1$ is nonzero (i.e. $T_4\neq0$).
\end{lemma}
\begin{proof}
The same argument used in  Lemma \ref{t2} yields that  there is an increasing subsequence $\{i''_l\}_{l=1}^\infty$ of $\{i_l\}_{l=2}^\infty$ and a normalized sequence $\{\eta''_l\}_{l=1}^\infty$ with $\eta''_l\in[\eta_{i_{2s}}]_{i''_{2l}\leq i_{2s}<i''_{2l+1}}$ such that $T_2|_{[e_{\xi_{i''_{2k+1}},\eta''_l}]_{1\leq l\leq k<\infty}}$ is a perturbation of the zero operator associated with the Schauder decomposition $\big\{[e_{\xi_{i''_{2k+1}},\eta''_l}]_{k=l}^\infty\big\}_{l=1}^\infty$ for $\{\varepsilon_l\}_{l=1}^\infty$.

Note that for each fixed $l$,   $[T_1e_{\xi_{i''_{2k+1}},\eta''_l}]_{k=l}^\infty$ is isometrically isomorphic to a subspace of $E$ (see e.g. \cite[Proposition 3.3]{HSS}). For each fixed $k\geq1$, consider the operator
\begin{align*}
\{\alpha_s\}_{s=0}^\infty\in\ell_2\longmapsto\left(\sum_{s=0}^\infty\alpha_sT_1e_{\xi_{i''_{2(k+s)+1}},\eta''_1},\dots,\sum_{s=0}^\infty\alpha_sT_1e_{\xi_{i''_{2(k+s)+1}},\eta''_k}\right)\in \underbrace{\left(   \mathcal C_E \oplus \cdots \oplus \mathcal C_E \right) }_k . 
\end{align*}
Obviously, it is bounded. Using the assumption $\ell_2\not\hookrightarrow E$ and Fact \ref{lp}, we  find an increasing sequence of positive integers $\{k_s\}_{s=1}^\infty$ and $\alpha^{^{_{(s)}}}=\big(\alpha^{^{_{(s)}}}_k\big)_{k_{2s+1}\leq k<k_{2s+2}}$ with $\Vert\alpha^{^{_{(s)}}}\Vert_2=1$ for every $s$ such that
\[
\left\Vert\sum_{k_{2s+1}\leq k< k_{2s+2}}\alpha^{^{_{(s)}}}_kT_1e_{\xi_{i''_{2k+1}},\eta''_{k_t}}\right\Vert_{\mathcal C_E}\leq\varepsilon_s\;\;\;\;\;\;t=1,\dots,2s
\]
Put
\[
\xi'_s=\sum_{k_{2s+1}\leq k< k_{2s+2}}\alpha^{^{_{(s)}}}_k\xi_{i''_{2k+1}}\;\;\;{\rm and}\;\;\;\eta'_s=\eta''_{k_{2s}}\;\;\;\;\;\;s=1,2\cdots.
\]
For each fixed $t$, for any sequence of scalars $\{\beta_s\}_{s=t}^\infty$ with with $\norm{\{\beta_s\}_{s=t}^\infty}_2<\infty$, we have
\begin{align*}
\left\Vert T_1\bigg(\sum_{s=t}^\infty \beta_se_{\xi'_s,\eta'_t}\bigg)\right\Vert_{\mathcal C_E}\leq4^{\frac{1}{r}}\bigg(\sum_{s=t}^\infty\big(\varepsilon_s\vert\beta_s\vert\big)^r\bigg)^{\frac{1}{r}}& \leq4^{\frac{1}{r}}\bigg(\sum_{s=t}^\infty{\varepsilon_s}^q\bigg)^{\frac{1}{q}}\bigg(\sum_{s=t}^\infty\vert\beta_s\vert^2\bigg)^{\frac{1}{2}}\\
&\leq4^{\frac{1}{r}}\bigg(\sum_{s=t}^\infty{\varepsilon_s}^q\bigg)^{\frac{1}{q}}\bigg\Vert\sum_{s=t}^\infty \beta_se_{\xi'_s,\eta'_t}\bigg\Vert_{\mathcal C_E}.
\end{align*}
where $\frac{1}{r}=\frac{1}{q}+\frac{1}{2}$. Therefore, $T_1|_{[e_{\xi'_s,\eta'_t}]_{1\leq t\leq s<\infty}}$ is a perturbation of the zero operator associated with the Schauder decomposition $\big\{[e_{\xi'_s,\eta'_t}]_{s=t}^\infty\big\}_{t=1}^\infty$ for $\big\{4^{\frac{1}{r}}\big(\sum_{s=t}^\infty{\varepsilon_s}^q\big)^{\frac{1}{q}}\big\}_{t=1}^\infty$ (Without loss of generality, we may assume that the sequence $\big\{\big(\sum_{s=t}^\infty{\varepsilon_s}^q\big)^{\frac{1}{q}}\big\}_{t=1}^\infty$ of positive numbers is small enough). Since $[e_{\xi'_s,\eta'_t}]_{s=t}^\infty\subset[e_{\xi_{i''_{2k+1}},\eta''_{k_{2t}}}]_{k=k_{2t}}^\infty$ for every $t$, it follows  that $T_2|_{[e_{\xi'_s,\eta'_t}]_{1\leq t\leq s<\infty}}$ is a small perturbation of the zero operator associated with the Schauder decomposition $\big\{[e_{\xi'_s,\eta'_t}]_{s=t}^\infty\big\}_{t=1}^\infty$ for $\{\varepsilon_{k_{2t}}\}_{t=1}^\infty$. The first assertion of the lemma follows by setting $i'_{2s}=i''_{2k_{2s}}$ and $i'_{2s+1}=i''_{2k_{2s+1}+1}$ for all $s\geq1$.

Assume that $T$ is an isomorphic embedding. 
If $\hat{a}_1$ and $\tilde{a}_1$ are both zero (i.e. $T_4=0$), then $T_3|_{[e_{\xi'_k,\eta'_l}]_{N\leq l\leq k<\infty}}$ is an isomorphic embedding for some sufficiently large positive integer $N$. By Lemma \ref{t2t3}, $[e_{\xi'_k,\eta'_l}]_{N\leq l\leq k<\infty}$ is isomorphic to $\ell_2$, and thus,
\[
E\cong[e_{\xi'_k,\eta'_k}]_{k=N}^\infty\approx \ell_2.
\]
This contradicts the hypothesis. In other words, $\hat{a}_1$ or $\tilde{a}_1$ is nonzero.
\end{proof}

Now, combining Lemmas \ref{t3}, \ref{t4} and \ref{c0l2}, and using Remark \ref{rembl}, we obtain the following theorem.

\begin{theorem}\label{thmb}
Suppose that $E$ is a separable quasi-Banach symmetric sequence space such that   $c_0,\ell_2\not\hookrightarrow E$ and $T\in\mathcal B(\mathcal T_{E,\{\xi_i,\eta_i\}_{i=1}^\infty},\mathcal C_E)$ with $p_{_{[\bigcup_{i=1}^\infty K_i]}}T(\cdot)p_{_{[\bigcup_{i=1}^\infty L_i]}}=T$.
Let $\{\varepsilon_l\}_{l=1}^\infty$ be  a sequence  of positive numbers. 

There exist an increasing sequence $\{i_k\}_{k=1}^\infty$ of positive integers, two orthonormal sequences $\{\xi'_k\}_{k=1}^\infty$ with $\xi'_k\in[\xi_i]_{i_{2k+1}\leq i<i_{2k+2}}$ and $\{\eta'_k\}_{k=1}^\infty$ with $\eta'_k\in[\eta_i]_{i_{2k}\leq i<i_{2k+1}}$, and two operators $\tilde{R}_1,\tilde{R}_2\in\mathcal B\big([e_{\xi'_k,\eta'_l}]_{k,l=1}^\infty,\mathcal C_E\big)$ having the forms
\begin{equation}\label{tilder1}
\tilde{R}_1(e_{\xi'_k,\eta'_l})=\left\{
\begin{array}{rcl}
x_{_{(2k+1,2l)}}+y_{_{(2l,2k+1)}}, &  & {1\leq l\leq k<\infty,}\\
0\;\;\;\;\;\;\;\;\;\;\;\;\;, &  & \text{otherwise,}\\
\end{array}
\right.
\end{equation}
\begin{equation}\label{R2tilde}
\tilde{R}_2(e_{\xi'_k,\eta'_l})=x_{2k+1,2l}+y_{2l,2k+1}\;\;\;\;\;\;\;\;\mbox{for every}\;k,l\geq1,
\end{equation}
such that $\big(\tilde{R}_1+\tilde{R}_2\big)|_{\mathcal T_{E,\{\xi'_k,\eta'_k\}_{k=1}^\infty}}$ is a perturbation of $T|_{\mathcal T_{E,\{\xi'_k,\eta'_k\}_{k=1}^\infty}}$ associated with the Schauder decomposition $\big\{[e_{\xi'_k,\eta'_l}]_{k=l}^\infty\big\}_{l=1}^\infty$ for $\{\varepsilon_l\}_{l=1}^\infty$, where
\begin{itemize}
\item [(1)]
$\{A_k\}_{k=1}^\infty$ is a sequence of mutually disjoint subsets of $\mathbb N$, and $\{F_{(2k+1,2l)}\}_{1\leq l\leq k<\infty}$ and $\{G_{(2l,2k+1)}\}_{1\leq l\leq k<\infty}$ are two sequences of mutually orthogonal closed subspaces of $H$ such that
\[
F_{(2k+1,2l)}\subset K_{A_{2k+1}}\;\;\;\;\mbox{and}\;\;\;\;G_{(2l,2k+1)}\subset L_{A_{2k+1}}\;\;\;\;\mbox{for every}\;1\leq l\leq k<\infty,
\]
\item [(2)]
\[
\{x_{_{(2k+1,2l)}}\}_{1\leq l\leq k<\infty}\stackrel{\mbox{{\rm\tiny b.s}}}{\boxplus}\{F_{(2k+1,2l)}\}_{1\leq l\leq k<\infty}\otimes L\curvearrowleft\hat{b},
\]
\[
\{y_{_{(2l,2k+1)}}\}_{1\leq l\leq k<\infty}\stackrel{\mbox{{\rm\tiny b.s}}}{\boxplus} K\otimes\{G_{(2l,2k+1)}\}_{1\leq l\leq k<\infty}\curvearrowleft\tilde{b},
\]
and $\left\Vert\left(
\begin{smallmatrix}
0&\tilde{b}\\ \hat{b}&0
\end{smallmatrix}\right)
\right\Vert_{\mathcal C_E} \left\Vert\cdot \right\Vert_{\ell_2}\leq 16^{\frac{1}{r}-1}\left\Vert T \right\Vert\cdot \left\Vert\cdot\right\Vert_E$,
\item [(3)]
\[
\{x_{2k+1,2l}\}_{k,l=1}^\infty\stackrel{\mbox{{\rm\tiny b.s}}}{\boxplus}\{K_{A_{2k+1}}\}_{k=1}^\infty\otimes\{L_{A_{2l}}\}_{l=1}^\infty\curvearrowleft\;
\]
and
\[
\{y_{2l,2k+1}\}_{k,l=1}^\infty\stackrel{\mbox{{\rm\tiny b.s}}}{\boxplus}\{K_{A_{2l}}\}_{l=1}^\infty\otimes\{L_{A_{2k+1}}\}_{k=1}^\infty\curvearrowleft\;.
\]
\end{itemize}

Moreover, we have $\lim_{N\to\infty}\big\Vert P_{[\bigcup_{k=N}^\infty K_k],[\bigcup_{k=N}^\infty L_k]}\tilde{R}_1\big\Vert=0$;  if $T$ is an isomorphic embedding, then $\tilde{R}_2$ is an isomorphic embedding.
\end{theorem}

\begin{remark}\label{tftoce'}
Let $F$ be another separable quasi-Banach symmetric sequence space. If 
we consider a bounded operator $T:\mathcal T_{F,\{\xi_i,\eta_i\}_{i=1}^\infty}\to \mathcal C_E$ with $p_{_{[\bigcup_{i=1}^\infty K_i]}}T(\cdot)p_{_{[\bigcup_{i=1}^\infty L_i]}}=T$, 
  then, by remark \ref{tftoce}, a mutatis mutandis repetition of  the proofs of the previous lemmas yields that all  conclusions of Theorem \ref{thmb} remain valid, except  that 
one needs an additional condition that  $\ell_2\not\hookrightarrow F$ to guarantee that $\tilde{R}_2$ is an isomorphic embedding provided that $T$ is an isomorphic embedding (see \cite{GHSY}), which recovers \cite[Corollary 5.3]{Arazy4}. Moreover,
\begin{itemize}
\item [(i)]
if $F\nsubseteq\ell_2$, then $\tilde{R}_1=0$ in the above theorem. If $F\nsubseteq E$, then $\tilde{R}_2=0$  in the above theorem.
\item [(ii)] 
if $\{K_i\}_{i=1}^\infty$ and $\{L_i\}_{i=1}^\infty$ are chosen such that $p_{K_i}({\rm im}(T))p_{L_j}=\{0\}$ for every $1\leq i<j<\infty$ (for example, $K_i=[\xi_i]$, $L_i=[\eta_i]$ and im$(T)\subset[e_{\xi_i,\eta_j}]_{1\leq j\leq i<\infty}$), then those ``$y_{(2l,2k+1)}$" and ``$y_{2l,2k+1}$" parts are $0$.
\end{itemize}
\end{remark}

\subsection{Applications}

Recall that for a separable Banach symmetric sequence space $E$, we have  $\mathcal T_E\not\approx\mathcal C_E$ if and only if the triangular projection $T^{E,\{\xi_i,\eta_i\}_{i=1}^\infty}$ (see \eqref{tria-proj} for the definition) is unbounded (Theorem \ref{tria}). 
The necessity of this fact holds in the setting of  quasi-Banach spaces, see Lemma \ref{qb-tenotce} below.

The following proposition for the case of Banach spaces $E$ is  due to Arazy \cite[Proposition 4.8]{Arazy1}, whose proof  relies on the boundedness of the diagonal projection (see \eqref{bs-diag-proj}). However, for general quasi-Banach symmetric sequence spaces $E$, one cannot guarantee the boundedness of the diagonal projection on $\cC_E$ (see \eqref{diag-proj-unb}). Below, we     present a different proof that circumvents the usage of diagonal projections.

\begin{proposition}\label{te=te2}
For every separable quasi-Banach symmetric sequence space $E$, we have 
\[
\mathcal C_E\approx\mathcal C_E\oplus\mathcal C_E\;\;\;\;\mbox{and}\;\;\;\;\mathcal T_E\approx\mathcal T_E\oplus\mathcal T_E.
\]
\end{proposition}
\begin{proof}
{\bf The case for  $\mathcal C_E$}. We may choose an infinite dimensional closed subspace $K$ of $H$ such that $K^{\perp}$ is infinite dimensional. It follows that
\[
\mathcal C_E\approx p_{_K}\mathcal C_E\oplus p_{_{K^\perp}}\mathcal C_E\approx\mathcal C_E\oplus\mathcal C_E.
\]

{\bf The case for  $\mathcal T_E$}. We have $\mathcal T_E\cong[e_{\xi_i,\eta_j}]_{1\leq j\leq i<\infty}$, and thus
\begin{align*}
\mathcal T_E&\cong\sum_{0\leq s,t\leq1}p_{_{[\xi_{2k-s}]_{k=1}^\infty}}[e_{\xi_i,\eta_j}]_{1\leq j\leq i<\infty}p_{_{[\eta_{2k-t}]_{k=1}^\infty}}\\
&\approx[e_{\xi_{2k-1},\eta_{2l-1}}]_{l\leq k}\oplus [e_{\xi_{2k},\eta_{2l-1}}]_{l\leq k}\oplus[e_{\xi_{2k-1},\eta_{2l}}]_{l<k}\oplus[e_{\xi_{2k},\eta_{2l}}]_{l\leq k}\approx \sum_{i=1}^4\oplus\mathcal T_E,
\end{align*}
and
\begin{align*}
\mathcal T_E&\cong\sum_{0\leq s,t\leq2}p_{_{[\xi_{3k-s}]_{k=1}^\infty}}[e_{\xi_i,\eta_j}]_{1\leq j\leq i<\infty}p_{_{[\eta_{3k-t}]_{k=1}^\infty}}\\
&\approx\bigg(\sum_{0\leq t\leq s\leq2}\oplus[e_{\xi_{3k-s},\eta_{3l-t}}]_{l\leq k}\bigg)\oplus\bigg(\sum_{0\leq s<t\leq2}\oplus[e_{\xi_{3k-s},\eta_{3l-t}}]_{l<k}\bigg)\approx \sum_{i=1}^9\oplus\mathcal T_E.
\end{align*}
Consequently,
\[
\mathcal T_E\oplus\mathcal T_E\approx\mathcal T_E\oplus\bigg(\sum_{i=1}^9\oplus\mathcal T_E\bigg)\approx\mathcal T_E\oplus\mathcal T_E\oplus\bigg(\sum_{i=1}^4\oplus\mathcal T_E\bigg)^2\approx\sum_{i=1}^4\oplus\mathcal T_E\approx\mathcal T_E.
\]
\end{proof}

\begin{lemma}\label{qb-tenotce}
Let $E$ be a separable quasi-Banach symmetric sequence space. If $\mathcal T_E\not\approx\mathcal C_E$, then $T^{E,\{\xi_i,\eta_i\}_{i=1}^\infty}$ is unbounded.
\end{lemma}
\begin{proof}
If $T^{E,\{\xi_i,\eta_i\}_{i=1}^\infty}$ is bounded, then
\[
\mathcal C_E\cong[e_{\xi_i,\eta_j}]_{i,j=1}^\infty\approx[e_{\xi_i,\eta_j}]_{1\leq j\leq i<\infty}\oplus[e_{\xi_i,\eta_j}]_{1\leq i<j<\infty}\approx\mathcal T_E\oplus\mathcal T_E\approx\mathcal T_E,
\]
contradicting the hypothesis.
\end{proof}
\begin{lemma}
Let $E$ be a separable quasi-Banach symmetric sequence space. If  $T^{E,\{\xi_i,\eta_i\}_{i=1}^\infty}$ is bounded, then $P_{\{[\xi_k],[\eta_k]\}_{k=1}^\infty}$ is bounded on $\mathcal C_E$.
\end{lemma}
\begin{proof}
    For any $x=\sum_{i,j}\alpha_{i,j}e_{\xi_i,\eta_j}\in{\rm span}\big(\{e_{\xi_i,\eta_j}\}_{i,j=1}^\infty\big)$,
\begin{align*}
    \big\Vert P_{\{[\xi_k],[\eta_k]\}_{k=1}^\infty}(x)\big\Vert_{\mathcal C_E}&=\bigg\Vert\sum_i\alpha_{i,i}e_{\xi_i,\eta_i}\bigg\Vert_{\mathcal C_E}\leq\big\Vert T^{E,\{\xi_k,\eta_k\}_{k=1}^\infty}\big\Vert\bigg\Vert\sum_{i\leq j}\alpha_{i,j}e_{\xi_i,\eta_j}\bigg\Vert_{\mathcal C_E}\\
    &=\big\Vert T^{E,\{\xi_k,\eta_k\}_{k=1}^\infty}\big\Vert\bigg\Vert\sum_{i\leq j}\alpha_{i,j}e_{\xi_j,\eta_i}\bigg\Vert_{\mathcal C_E}\\
    &\leq\big\Vert T^{E,\{\xi_k,\eta_k\}_{k=1}^\infty}\big\Vert^2\bigg\Vert\sum_{i,j}\alpha_{i,j}e_{\xi_j,\eta_i}\bigg\Vert_{\mathcal C_E}\\ &=\big\Vert T^{E,\{\xi_k,\eta_k\}_{k=1}^\infty}\big\Vert^2\bigg\Vert\sum_{i,j}\alpha_{i,j}e_{\xi_i,\eta_j}\bigg\Vert_{\mathcal C_E}.
\end{align*}
 This yields that $P_{\{[\xi_k],[\eta_k]\}_{k=1}^\infty}$ is bounded on $\mathcal C_E$.
\end{proof}

\begin{example}\label{exmple}
    By \eqref{diag-proj-unb}, it follows immediately that $T^{\ell_p,\{\xi_k,\eta_k\}_{k=1}^\infty}$ is unbounded for each $0<p\leq1$ (for the case when $p=1$,  see \cite{KwapienPelczynski}).
\end{example}

When $\mathcal C_E\approx\mathcal T_E$, 
it is convenient to study operators on $\mathcal T_E$ instead of $\mathcal C_E$. 
However, if $\mathcal T_E\not\approx \mathcal C_E$, 
then the situation changes dramatically (e.g., the triangular projection  $T^{E,\{\xi_i,\eta_i\}_{i=1}^\infty}$ is unbounded).
In particular, $\mathcal C_E$ does not have an unconditional finite-dimensional Schauder decomposition when $E$ is separable Banach space (e.g. see \cite{Arazy1} and \cite{KwapienPelczynski}), which may cause  technical difficulties. 

In Theorem \ref{thmc} below, we establish an analogue of Theorem \ref{thmb} in the setting of operators on the whole `square', i.e.,  operators from $B\big([e_{\xi_i,\eta_j}]_{i,j=1}^\infty,\mathcal C_E\big)$. 
Perturbations associated with  a  Schauder decomposition play   key roles, which allow  the application of Lemma \ref{2d''}.
Before proceeding to the proof of Theorem~\ref{thmc}, we need the following auxiliary lemma.

\begin{lemma}\label{per=}
Suppose that $E$ is a separable quasi-Banach symmetric sequence space such that    $c_0,\ell_2\not\hookrightarrow E$, and $\mathcal T_E\not\approx \mathcal C_E$. 
Assume that $T,S\in\mathcal B\big([e_{\xi_i,\eta_j}]_{i,j=1}^\infty,\mathcal C_E\big)$ satisfy $p_{_{[\bigcup_{i=1}^\infty K_i]}}T(\cdot)p_{_{[\bigcup_{i=1}^\infty L_i]}}=T$ and $p_{_{[\bigcup_{i=1}^\infty K_i]}}S(\cdot)p_{_{[\bigcup_{i=1}^\infty L_i]}}=S$,  and assume that   $S|_{[e_{\xi_i,\eta_j}]_{1\leq i<j<\infty}}$ is a perturbation of $T|_{[e_{\xi_i,\eta_j}]_{1\leq i<j<\infty}}$  associated with the  Schauder decomposition $\big\{[e_{\xi_i,\eta_j}]_{j=i+1}^\infty\big\}_{i=1}^\infty$,

Then there exist an increasing sequence of positive integers $\{i_k\}_{k=1}^\infty$, two orthonormal sequences $\{\xi'_k\}_{k=1}^\infty$ with $\xi'_k\in[\xi_i]_{i_{2k+1}\leq i<i_{2k+2}}$ and $\{\eta'_k\}_{l=1}^\infty$ with $\eta'_k\in[\eta_i]_{i_{2k}\leq i<i_{2k+1}}$, and a bounded operator $R:[e_{\xi'_k,\eta'_l}]_{k,l=1}^\infty\to\mathcal C_E$ having the form
\[
R(e_{\xi'_k,\eta'_l})=\left\{
\begin{array}{rcl}
x_{_{(k,l)}}+y_{_{(l,k)}}, &  & {1\leq l\leq k<\infty,}\\
0\;\;\;\;\;\;\;\;, &  & \text{otherwise,}\\
\end{array}
\right.
\]
such that $S|_{[e_{\xi'_k,\eta'_l}]_{k,l=1}^\infty}+R$ is a small perturbation of $T|_{[e_{\xi'_k,\eta'_l}]_{k,l=1}^\infty}$ associated with the Schauder decomposition $\big\{[e_{\xi'_k,\eta'_l}]_{\min\{k,l\}=s}^\infty\big\}_{s=1}^\infty$, where
\begin{itemize}
\item [(1)]
$\{F_{(k,l)}\}_{1\leq l\leq k<\infty}$ is a sequences of mutually orthgononal closed subspaces of $K$, and $\{G_{(l,k)}\}_{1\leq l\leq k<\infty}$ is a sequence of mutually orthogonal closed subspaces of $L$,
\item [(2)]
\[
\{x_{_{(k,l)}}\}_{1\leq l\leq k<\infty}\stackrel{\mbox{{\rm\tiny b.s}}}{\boxplus}\{F_{(k,l)}\}_{1\leq l\leq k <\infty}\otimes L\curvearrowleft\hat{b},
\]
\[
\{y_{_{(l,k)}}\}_{1\leq l\leq k<\infty}\stackrel{\mbox{{\rm\tiny b.s}}}{\boxplus}K\otimes\{G_{(l,k)}\}_{1\leq l\leq k<\infty}\curvearrowleft\tilde{b}.
\]
\end{itemize}
Moreover,
\[
\lim_{N\to\infty}\big\Vert P_{[\bigcup_{i=N}^\infty K_i],[\bigcup_{i=N}^\infty L_i]}R\big\Vert=0.
\]
\end{lemma}

\begin{proof}
Without loss of generality, we may assume  that $(T-S)|_{[e_{\xi_i,\eta_j}]_{1\leq i<j<\infty}}$ is a small perturbation of the zero operator associated with the Schauder decomposition $\big\{[e_{\xi_i,\eta_j}]_{j=i+1}^\infty\big\}_{i=1}^\infty$. By Lemma~\ref{2d''}, there exists an operator $\tilde{R}\in\mathcal B\big([e_{\xi_i,\eta_j}]_{i,j=1}^\infty,\mathcal C_E\big)$ satisfying that 
\[
\tilde{R}|_{[e_{\xi_i,\eta_j}]_{1\leq j\leq i<\infty}}=(T-S)|_{[e_{\xi_i,\eta_j}]_{1\leq j\leq i<\infty}}\;\;\;\;{\rm and}\;\;\;\;\tilde{R}|_{[e_{\xi_i,\eta_j}]_{1\leq i<j<\infty}}=0,
\]
and $\tilde{R}$ is a small perturbation of $T-S$ associated with the Schauder decomposition $\big\{[e_{\xi_i,\eta_j}]_{\min\{i,j\}=s}^\infty\big\}_{s=1}^\infty$.

Applying Theorem \ref{thmb} to $\tilde{R}|_{[e_{\xi_i,\eta_j}]_{1\leq j\leq i<\infty}}$, we obtain two operators  $\tilde{R}_1,\tilde{R}_2\in\mathcal B\big([e_{\xi'_k,\eta'_l}]_{k,l=1}^\infty,\mathcal C_E\big)$ satisfying that the assertions in Theorem \ref{thmb} such that  $\big(\tilde{R}_1+\tilde{R}_2\big)|_{[e_{\xi'_k,\eta'_l}]_{1\leq l\leq k<\infty}}$ is a small perturbation of $\tilde{R}|_{[e_{\xi'_k,\eta'_l}]_{1\leq l\leq k<\infty}}$ associated with the Schauder decomposition $\big\{[e_{\xi'_k,\eta'_l}]_{k=l}^\infty\big\}_{l=1}^\infty$. Using Lemma \ref{2d''}, there exists an operator $\tilde{Q}\in\mathcal B\big([e_{\xi'_k,\eta'_l}]_{k,l=1}^\infty,\mathcal C_E\big)$ with $p_{_{[\bigcup_{i=1}^\infty K_i]}}\tilde{Q}(\cdot)p_{_{[\bigcup_{i=1}^\infty L_i]}}=\tilde{Q}$ satisfying
\begin{equation}\label{tildeT}
\tilde{Q}|_{[e_{\xi'_k,\eta'_l}]_{1\leq l\leq k<\infty}}=(\tilde{R}_1+\tilde{R}_2)|_{[e_{\xi'_k,\eta'_l}]_{1\leq l\leq k<\infty}}\;\;\;\;{\rm and}\;\;\;\;\tilde{Q}|_{[e_{\xi'_k,\eta'_l}]_{1\leq k<l<\infty}}=0,
\end{equation}
and such that $\tilde{Q}$ is a small perturbation of $\tilde{R}|_{[e_{\xi'_k,\eta'_l}]_{k,l=1}^\infty}$ associated with the Schauder decomposition $\big\{[e_{\xi'_k,\eta'_l}]_{\min\{k,l\}=s}^\infty\big\}_{s=1}^\infty$.

Suppose that
\[
\{x_{2k+1,2l}\}_{k,l=1}^\infty\stackrel{\mbox{{\rm\tiny b.s}}}{\boxplus}\{K_{A_{2k+1}}\}_{k=1}^\infty\otimes\{L_{A_{2l}}\}_{l=1}^\infty\curvearrowleft \hat{a}
\]
and
\[
\{y_{2l,2k+1}\}_{k,l=1}^\infty\stackrel{\mbox{{\rm\tiny b.s}}}{\boxplus}\{K_{A_{2l}}\}_{l=1}^\infty\otimes\{L_{A_{2k+1}}\}_{k=1}^\infty\curvearrowleft \tilde{a}.
\]
We claim that 
 $\tilde{R}_2=0$. Indeed, otherwise,  $\hat{a}$ or $\tilde{a}$ is nonzero (see \eqref{R2tilde}). Without loss of generality, we assume that $\hat{a}\neq0$.
For any finitely non-zero sequence $\{\alpha_{k,l}\}_{k,l=1}^\infty$ of scalars, we have
\begin{eqnarray*}
&& \Vert\hat{a}\Vert_\infty\bigg\Vert\sum_{1\leq l\leq k<\infty}\alpha_{k,l}e_{\xi'_k,\eta'_l}\bigg\Vert_{\mathcal C_E}\\
&\leq&\bigg\Vert\sum_{1\leq l\leq k<\infty}\alpha_{k,l}\hat{a}\otimes(\cdot|e_{2l})e_{2k+1}\bigg\Vert_{\mathcal C_E}\\
&\leq&\bigg\Vert\sum_{1\leq l\leq k<\infty}\alpha_{k,l}\big(\hat{a}\otimes(\cdot|e_{2l})e_{2k+1}+\tilde{a}\otimes(\cdot|e_{2k+1})e_{2l})\bigg\Vert_{\mathcal C_E}\\
&\stackrel{{\rm Rem.}\;\ref{rem}}{=}&\bigg\Vert\sum_{1\leq l\leq k<\infty}\alpha_{k,l}\big(\tilde{Q}-\tilde{R}_1\big)(e_{\xi'_k,\eta'_l})\bigg\Vert_{\mathcal C_E}\\
&\stackrel{\eqref{tilder1},\eqref{tildeT}}{=}&\bigg\Vert\sum_{k,l=1}^\infty\alpha_{k,l}\big(\tilde{Q}-\tilde{R}_1\big)(e_{\xi'_k,\eta'_l})\bigg\Vert_{\mathcal C_E}\\
&\leq&\big\Vert\tilde{Q}-\tilde{R}_1\big\Vert\bigg\Vert\sum_{k,l=1}^\infty\alpha_{k,l}e_{\xi'_k,\eta'_l}\bigg\Vert_{\mathcal C_E}.
\end{eqnarray*}
This yields that the triangular projection $T^{E,\{\xi'_k,\eta'_k\}_{k=1}^\infty}$ is bounded, which contradicts  the conclusion of Lemma \ref{qb-tenotce}, and hence $\tilde{R}_2=0$. We set $R=\tilde{R}_1$. It follows that $R$ is a small perturbation of $(T-S)|_{[e_{\xi'_k,\eta'_l}]_{k,l=1}^\infty}$ associated with the Schauder decomposition $\big\{[e_{\xi'_k,\eta'_l}]_{\min\{k,l\}=s}^\infty\big\}_{s=1}^\infty$, and thus $S|_{[e_{\xi'_k,\eta'_l}]_{k,l=1}^\infty}+R$ is a small perturbation of $T|_{[e_{\xi'_k,\eta'_l}]_{k,l=1}^\infty}$ associated with the Schauder decomposition $\big\{[e_{\xi'_k,\eta'_l}]_{\min\{k,l\}=s}^\infty\big\}_{s=1}^\infty$.
\end{proof}

\begin{theorem}\label{thmc}
Suppose that $E$ is a 
separable quasi-Banach symmetric sequence space such that    $c_0,\ell_2\not\hookrightarrow E$,  and $\mathcal T_E\not\approx\mathcal  C_E$. 
Let $T\in\mathcal B\big([e_{\xi_i,\eta_j}]_{i,j=1}^\infty,\mathcal C_E\big)$ with $p_{_{[\bigcup_{i=1}^\infty K_i]}}T(\cdot)p_{_{[\bigcup_{i=1}^\infty L_i]}}=T$, and let  $\{\varepsilon_l\}_{l=1}^\infty$ be a sequence  of positive numbers. 

Then there exist an increasing sequence $\{i_k\}_{k=1}^\infty$of positive integers, two orthonormal sequences $\{\xi'_k\}_{k=1}^\infty$ with $\xi'_k\in[\xi_i]_{i_{2k+1}\leq i<i_{2k+2}}$ and $\{\eta'_k\}_{k=1}^\infty$ with $\eta'_k\in[\eta_i]_{i_{2k}\leq i<i_{2k+1}}$, and two operators $R_1,R_2\in\mathcal B\big([e_{\xi'_k,\eta'_l}]_{k,l=1}^\infty\mathcal C_E\big)$ having the forms
\begin{equation}
R_1(e_{\xi'_k,\eta'_l})=x_{_{(k,l)}}+y_{_{(l,k)}}\;\;\;\;\mbox{for every}\;k,l\geq1,
\end{equation}
\begin{equation}\label{thmcr2}
R_2(e_{\xi'_k,\eta'_l})=x_{2k+1,2l}+y_{2l,2k+1}\;\;\;\;\;\;\;\;\mbox{for every}\;k,l\geq1,
\end{equation}
such that $R_1+R_2$ is a perturbation of $T|_{[e_{\xi'_k,\eta'_l}]_{k,l=1}^\infty}$ associated with the Schauder decomposition $\big\{[e_{\xi'_k,\eta'_l}]_{\min{k,l}=s}^\infty\big\}_{s=1}^\infty$ for $\{\varepsilon_s\}_{s=1}^\infty$, where
\begin{itemize}
\item [(1)]
there are  two sequences $\{F_{(k,l)}\}_{1\leq l\leq k<\infty}$ and $\{F_{(k,l)}\}_{1\leq k<l<\infty}$ of mutually orthogonal subspaces of $K$, and  two sequences $\{G_{(l,k)}\}_{1\leq l\leq k<\infty}$ and $\{G_{(l,k)}\}_{1\leq k<l<\infty}$ of mutually orthogonal subspaces of $L$ such that
\[
\{x_{_{(k,l)}}\}_{1\leq l\leq k<\infty}\stackrel{\mbox{{\rm\tiny b.s}}}{\boxplus}\{F_{(k,l)}\}_{1\leq l\leq k<\infty}\otimes L\curvearrowleft\hat{b},
\]
\[
\{y_{_{(l,k)}}\}_{1\leq l\leq k<\infty}\stackrel{\mbox{{\rm\tiny b.s}}}{\boxplus}K\otimes\{G_{(l,k)}\}_{1\leq l\leq k<\infty}\curvearrowleft\tilde{b},
\]
and
\[
\{x_{_{(k,l)}}\}_{1\leq k<l<\infty}\stackrel{\mbox{{\rm\tiny b.s}}}{\boxplus}\{F_{(k,l)}\}_{1\leq k<l<\infty}\otimes L\curvearrowleft\tilde{b'},
\]
\[
\{y_{_{(l,k)}}\}_{1\leq k<l<\infty}\stackrel{\mbox{{\rm\tiny b.s}}}{\boxplus}K\otimes\{G_{(l,k)}\}_{1\leq k<l<\infty}\curvearrowleft\hat{b'},
\]
and $\left\Vert\left(
\begin{smallmatrix}
0&\tilde{b}\\ \hat{b}&0
\end{smallmatrix}\right)
\right\Vert_{\mathcal C_E}
\left\Vert\cdot\right\Vert_{\ell_2},\left\Vert\left(
\begin{smallmatrix}
0&\hat{b'}\\ \tilde{b'}&0
\end{smallmatrix}\right)
\right\Vert_{\mathcal C_E} \left \Vert\cdot\right\Vert_{\ell_2}\leq16^{\frac{1}{r}-1}\Vert T\Vert\cdot \left\Vert\cdot\right\Vert_E$;
\item [(2)]
there is a sequence $\{A_k\}_{k=1}^\infty$ of mutually disjoint subsets of\;$\mathbb N$ such that
\[
\{x_{2k+1,2l}\}_{k,l=1}^\infty\stackrel{\mbox{{\rm\tiny b.s}}}{\boxplus}\{K_{A_{2k+1}}\}_{k=1}^\infty\otimes\{L_{A_{2l}}\}_{l=1}^\infty\curvearrowleft\;
\]
and
\[
\{y_{2l,2k+1}\}_{k,l=1}^\infty\stackrel{\mbox{{\rm\tiny b.s}}}{\boxplus}\{K_{A_{2l}}\}_{l=1}^\infty\otimes\{L_{A_{2k+1}}\}_{k=1}^\infty\curvearrowleft\;.
\]
\end{itemize}

Moreover, we have $\lim_{N\to\infty}\big\Vert P_{[\bigcup_{i=N}^\infty K_i],[\bigcup_{i=N}^\infty L_i]}R_1\big\Vert=0$; if $T|_{[e_{\xi_i,\eta_j}]_{1\leq j\leq i<\infty}}$ is an isomorphic embedding, then $R_2$ is an isomorphic embedding.
\end{theorem}

\begin{proof}

Applying Theorem \ref{thmb} to $T|_{[e_{\xi_i,\eta_j}]_{1\leq j\leq i<\infty}}$, we obtain two operators  $\tilde{R}_1,\tilde{R}_2\in\mathcal B\big([e_{\xi'_k,\eta'_l}]_{k,l=1}^\infty,\mathcal C_E\big)$ satisfying that the assertions in Theorem \ref{thmb}, such that  $\big(\tilde{R}_1+\tilde{R}_2\big)|_{[e_{\xi'_k,\eta'_l}]_{1\leq l\leq k<\infty}}$ is a small perturbation of $T|_{[e_{\xi'_k,\eta'_l}]_{1\leq l\leq k<\infty}}$ associated with the Schauder decomposition $\big\{[e_{\xi'_k,\eta'_l}]_{k=l}^\infty\big\}_{l=1}^\infty$, where $\{i_k\}_{k=1}^\infty$ is  an increasing sequence of positive integers, and two orthonormal sequences $\{\xi'_k\}_{k=1}^\infty$ with $\xi'_k\in[\xi_i]_{i_{2k+1}\leq i<i_{2k+2}}$ and $\{\eta'_k\}_{k=1}^\infty$ with $\eta'_k\in[\eta_i]_{i_{2k}\leq i<i_{2k+1}}$. 

 Applying  Lemma \ref{per=}  to (the transpositions of)  operators $T|_{[e_{\xi'_k,\eta'_l}]_{k,l=1}^{\infty}}$ and $\big(\tilde{R}_1+\tilde{R}_2\big)|_{[e_{\xi'_k,\eta'_l}]_{k,l=1}^{\infty}}$, we  obtain an operator $R\in\mathcal B([e_{\xi''_s,\eta''_t}]_{s,t=1}^\infty,\mathcal C_E)$ such that $\big(\tilde{R}_1+\tilde{R}_2\big)|_{[e_{\xi''_s,\eta''_t}]_{s,t=1}^\infty}+R$ is a small perturbation of $T|_{[e_{\xi''_s,\eta''_t}]_{s,t=1}^\infty}$ associated with the Schauder decomposition $\big\{[e_{\xi''_s,\eta''_t}]_{\min\{s,t\}=m}^\infty\big\}_{m=1}^\infty$, where there is an increasing sequence $\{l_s\}_{s=1}^\infty$ of positive integers, and two orthonormal sequences $\{\eta''_s\}_{s=1}^\infty$ with $\eta''_s\in[\eta'_l]_{l_{2s+1}\leq l<l_{2s+2}}$ and $\xi''_s\in[\xi'_l]_{l_{2s}\leq l<l_{2s+1}}$. 
Moreover, we may assume that there are  $x_{_{(s,t)}}$ and $y_{_{(t,s)}}$ from $\mathcal C_E$ such that
\[
R(e_{\xi''_{s+1},\eta''_t})=\left\{
\begin{array}{rcl}
0\;\;\;\;\;\;\;\;\;, &  & {t\leq s,}\\
x_{_{(s,t)}}+y_{_{(t,s)}}, &  & {t> s,}\\
\end{array}
\right.
\]
with the same property of (1) in Lemma \ref{per=}.

Note that
\[
\xi''_s\in[\xi'_l]_{l_{2s}\leq l<l_{2s+1}}=[\xi'_l]_{l_{2s}\leq l\leq l_{2s+1}-1}\subset[\xi_i]_{i_{2l_{2s}+1}\leq i<i_{2(l_{2s+1}-1)+2}}\subset[\xi_i]_{i_{2l_{2s}}\leq i<i_{2l_{2s+1}}}
\]
and
\[
\eta''_s\in[\eta'_l]_{l_{2s+1}\leq l<l_{2s+2}}=[\eta'_l]_{l_{2s+1}\leq l\leq l_{2s+2}-1}\subset[\eta_i]_{i_{2l_{2s+1}}\leq i<i_{2(l_{2s+2}-1)+1}}\subset[\eta_i]_{i_{2l_{2s+1}}\leq i<i_{2l_{2s+2}}}.
\]
Denote by $i'_s=i_{2l_{s+1}}$. 
Then, $\{i'_s\}_{s=1}^\infty$ is an increasing sequence of positive integers, and put
\[
\xi'''_s:=\xi''_{s+1}\in[\xi'_l]_{l_{2s+2}\leq l<l_{2s+3}}\subset [\xi_i]_{i'_{2s+1}\leq i<i'_{2s+2}}
\]
and
\[
\eta'''_s:=\eta''_s\in[\eta'_l]_{l_{2s+1}\leq l<l_{2s+2}}\subset[\eta_i]_{i'_{2s}\leq i<i'_{2s+1}},
\]
whose $\ell_2$-norms are $1$.
By Remark \ref{rembl},  operators $\tilde{R}_1|_{[e_{\xi'''_s,\eta'''_t}]_{s,t=1}^\infty}$ and $\tilde{R}_2|_{[e_{\xi'''_s,\eta'''_t}]_{s,t=1}^\infty}$ also  possess the properties of the conclusion stated in Theorem~\ref{thmb}. Without loss of generality, we may assume that there are  $x_{_{(s,t)}}$ and $y_{_{(t,s)}}$ from $\mathcal C_E$ such that 
\[
\tilde{R}_1(e_{\xi'''_s,\eta'''_t})=\left\{
\begin{array}{rcl}
x_{_{(s,t)}}+y_{_{(t,s)}}, &  & {t\leq s,}\\
0\;\;\;\;\;\;\;\;, &  & {t>s,}\\
\end{array}
\right.
\]

Let
\[
R_1=\tilde{R}_1|_{[e_{\xi'''_s,\eta'''_t}]_{s,t=1}^\infty}+R\;\;\;\;\;\;{\rm and}\;\;\;\;\;\;R_2=\tilde{R}_2|_{[e_{\xi'''_s,\eta'''_t}]_{s,t=1}^\infty}.
\]

We get the desired operators. Without loss of generality, we may assume that
\[
\{x_{_{(k,l)}}\}_{1\leq l\leq k<\infty}\stackrel{\mbox{{\rm\tiny b.s}}}{\boxplus}\{F_{(k,l)}\}_{1\leq l\leq k<\infty}\otimes L\curvearrowleft\hat{b},
\]
\[
\{y_{_{(l,k)}}\}_{1\leq l\leq k<\infty}\stackrel{\mbox{{\rm\tiny b.s}}}{\boxplus}K\otimes\{G_{(l,k)}\}_{1\leq l\leq k<\infty}\curvearrowleft\tilde{b},
\]
and
\[
\{x_{_{(k,l)}}\}_{1\leq k<l<\infty}\stackrel{\mbox{{\rm\tiny b.s}}}{\boxplus}\{F_{(k,l)}\}_{1\leq k<l<\infty}\otimes L\curvearrowleft\tilde{b'},
\]
\[
\{y_{_{(l,k)}}\}_{1\leq k<l<\infty}\stackrel{\mbox{{\rm\tiny b.s}}}{\boxplus}K\otimes\{G_{(l,k)}\}_{1\leq k<l<\infty}\curvearrowleft\hat{b'},
\]
where both $\{F_{(k,l)}\}_{1\leq l\leq k<\infty}$ and $\{F_{(k,l)}\}_{1\leq k<l<\infty}$ are  sequences of  mutually orthogonal subspaces of $K$, and both $\{G_{(l,)k}\}_{1\leq l\leq k<\infty}$ and $\{G_{(l,k)}\}_{1\leq k<l<\infty}$ are sequences  of mutually orthogonal subspaces of $L$. Note that (by Theorem \ref{thmb} and Remark~\ref{rembl}) the operators $\tilde{b}$ and $\hat{b}$ satisfy 
\[
\left\Vert\left(
\begin{smallmatrix}
0&\tilde{b}\\ \hat{b}&0
\end{smallmatrix}\right)
\right\Vert_{\mathcal C_E}  \left\Vert\cdot \right\Vert_{\ell_2}\leq16^{\frac{1}{r}-1}\Big\Vert T|_{{[e_{\xi_i,\eta_j}]_{1\leq j\leq i<\infty}}}\Big\Vert\cdot\left\Vert\cdot\right\Vert_E.
\]
Arguing similarly as the proof of Lemma \ref{t2t3} (2), we obtain that for any $\varepsilon>0$, there are two positive integers $N_1$ and $N_2$ such that for any sequence $\{\alpha_s\}_{s=1}^\infty\in E$ and any  positive integer $n>N_2$ such that
\begin{align*}
&\left(2^{1-\frac{1}{r}}\left\Vert\left(
\begin{smallmatrix}
0&\tilde{b'}\\ \hat{b'}&0
\end{smallmatrix}\right)
\right\Vert_{\mathcal C_E}-\varepsilon\right)\left\Vert\sum_{m=1}^\infty\alpha_me_m
\right\Vert_{\ell_2}\\
\leq&\bigg\Vert\bigg( P_{[\bigcup_{j=N_1+1}^\infty K_i],[\bigcup_{i=1}^{N_1}L_i]}+P_{[\bigcup_{j=1}^{N_1}K_i],[\bigcup_{i=N_1+1}^\infty L_i]}\bigg)R\left(\sum_{m=1}^\infty\alpha_me_{\xi'''_{n+m},\eta'''_{n+m+1}}\right)\bigg\Vert_{\mathcal C_E}\\
\leq&\bigg\Vert \bigg( P_{[\bigcup_{j=N_1+1}^\infty K_i],[\bigcup_{i=1}^{N_1}L_i]}+P_{[\bigcup_{j=1}^{N_1}K_i],[\bigcup_{i=N_1+1}^\infty L_i]}\bigg)R|_{[e_{\xi'''_{n+m},\eta'''_{n+m+1}}]_{m=1}^\infty}\bigg\Vert\cdot\left\Vert\sum_{m=1}^\infty\alpha_me_{\xi'''_{n+m},\eta'''_{n+m+1}}\right\Vert_{\mathcal C_E}\\
=& \bigg\Vert \bigg( P_{[\bigcup_{j=N_1+1}^\infty K_i],[\bigcup_{i=1}^{N_1}L_i]}+P_{[\bigcup_{j=1}^{N_1}K_i],[\bigcup_{i=N_1+1}^\infty L_i]}\bigg)R|_{[e_{\xi'''_{n+m},\eta'''_{n+m+1}}]_{m=1}^\infty}\bigg\Vert\left\Vert\sum_{m=1}^\infty\alpha_me_m\right\Vert_E
\end{align*}
Recall that $\big(\tilde{R}_1+\tilde{R}_2\big)|_{[e_{\xi''_s,\eta''_t}]_{s,t=1}^\infty}+R$ is a small perturbation of $T|_{[e_{\xi''_s,\eta''_t}]_{s,t=1}^\infty}$ associated with the Schauder decomposition $\big\{[e_{\xi''_s,\eta''_t}]_{\min\{s,t\}=m}^\infty\big\}_{m=1}^\infty$.  
By Lemma \ref{sdper}, we have 
\[
\big\Vert(R+\tilde{R}_1+\tilde{R}_2-T)|_{[e_{\xi'''_{n+m},\eta'''_{n+m+1}}]_{m=1}^\infty}\big\Vert\to0\;\;\;\;{\rm as}\;n\to\infty.
\]
Moreover, we have 
\[
\tilde{R}_1|_{[e_{\xi'''_{n+m},\eta'''_{n+m+1}}]_{m=1}^\infty}=0,
\]
and
\[
\bigg( P_{[\bigcup_{j=N_1+1}^\infty K_i],[\bigcup_{i=1}^{N_1}L_i]}+P_{[\bigcup_{j=1}^{N_1}K_i],[\bigcup_{i=N_1+1}^\infty L_i]}\bigg)\tilde{R}_2|_{[e_{\xi'''_{n+m},\eta'''_{n+m+1}}]_{m=1}^\infty}=0
\]
when $n$ is sufficiently large. Arguing similarly as \eqref{proj+proj}, and letting $n\to\infty$, we obtain 
\[
\left(2^{1-\frac{1}{r}}\left\Vert\left(
\begin{smallmatrix}
0&\tilde{b'}\\ \hat{b'}&0
\end{smallmatrix}\right)
\right\Vert_{\mathcal C_E}-\varepsilon\right)\left\Vert\sum_{m=1}^\infty\alpha_me_m
\right\Vert_{\ell_2}\leq8^{\frac{1}{r}-1}\Vert T\Vert\left\Vert\sum_{m=1}^\infty\alpha_me_m\right\Vert_E
\]
Since $\varepsilon>0$ is arbitrarily taken, it follows that 
\[
\left\Vert\left(
\begin{smallmatrix}
0&\tilde{b'}\\ \hat{b'}&0
\end{smallmatrix}\right)
\right\Vert_{\mathcal C_E}\left\Vert\cdot\right\Vert_{\ell_2}\leq16^{\frac{1}{r}-1}\Vert T\Vert\cdot \left\Vert\cdot\right \Vert_E.
\]
The proof is complete.
\end{proof}

\begin{remark}\label{thmc-rem}
\noindent
\begin{itemize}
    \item [(1)]
      Theorem \ref{thmc}  remains valid if  the condition $\mathcal T_E\not\approx\mathcal C_E$ is  replaced by  the condition that $T^{E,\{\xi_i,\eta_i\}_{i=1}^\infty}$ is unbounded.
\item [(2)]
   Since $\lim_{N\to\infty}\big\Vert(T-(R_1+R_2))|_{[e_{\xi'_k,\eta'_l}]_{k,l=N}^\infty}\big\Vert=0$ and $\lim_{N\to\infty}\big\Vert P_{[\bigcup_{i=N}^\infty K_i],[\bigcup_{i=N}^\infty L_i]}R_1\big\Vert=0$, it follows that   
\begin{eqnarray*}
  &&\lim_{N\to\infty}\big\Vert P_{[\bigcup_{i=N}^\infty K_{A_i}],[\bigcup_{i=N}^\infty L_{A_i}]}T|_{[e_{\xi'_k,\eta'_l}]_{N\leq l\leq k<\infty}}\big\Vert\\
  &\leq&\lim_{N\to\infty}2^{\frac{1}{r}-1}\big\Vert R_2|_{[e_{\xi'_k,\eta'_l}]_{N\leq l\leq k <\infty}}\big\Vert\\
  &\stackrel{{\rm Th}\;\ref{thmc}\;(2)}{=}&2^{\frac{1}{r}-1}\big\Vert R_2|_{[e_{\xi'_k,\eta'_l}]_{1\leq l\leq k<\infty}}\big\Vert\\
  &\stackrel{{\rm Th}\;\ref{thmc}\;(2)}{=}&\lim_{N\to\infty}2^{\frac{1}{r}-1}\big\Vert R_2|_{[e_{\xi'_k,\eta'_l}]_{N\leq k< l<\infty}}\big\Vert\\
  &\leq&4^{\frac{1}{r}-1}\lim_{N\to\infty}\big\Vert P_{[\bigcup_{i=N}^\infty K_{A_i}],[\bigcup_{i=N}^\infty L_{A_i}]}T|_{[e_{\xi'_k,\eta'_l}]_{N\leq k< l<\infty}}\big\Vert
   \end{eqnarray*}
   
   Let $E$ be a Lorentz sequence space (i.e., $E:=\{ x\in c_0: \norm{x}_E:= \sum_{n=1}^\infty \mu(n;x) \phi_n <\infty  \}$, where $\{\phi_n\}_{n\ge 1}$ is a positive sequence decreasing  to $0$) such that $E\supset \ell_2$ (i.e., $\norm{\cdot}_E \le \norm{\cdot }_{\ell_2}$) and $E$ has trivial Boyd indices. For example, let $\phi_n  = \frac1n $. In particular, $E$ is a $KB$-space and contains exactly two mutually non-equivalent symmetric basic sequences\cite[Theorem 4.e.5]{LT}.  
Now, consider  an operator $T:[e_{\xi_i,\eta_j}]_{i,j=1}^\infty\subset\mathcal C_1\to \mathcal C_E$ given by
\[
T:\sum_{i,j}\alpha_{i,j}e_{\xi_i,\eta_j}\longmapsto\sum_{j\leq i}\alpha_{i,j}e_{\xi_i,\eta_j}.
\]
Since
\[
\bigg\Vert \sum_{j\leq i}\alpha_{i,j}e_{\xi_i,\eta_j}\bigg\Vert_{\mathcal C_E}\le \bigg\Vert \sum_{j\leq i}\alpha_{i,j}e_{\xi_i,\eta_j}\bigg\Vert_2\leq \bigg\Vert \sum_{i,j}\alpha_{i,j}e_{\xi_i,\eta_j}\bigg\Vert_2\leq\bigg\Vert \sum_{i,j}\alpha_{i,j}e_{\xi_i,\eta_j}\bigg\Vert_1,
\]
it follows that $T$ is a bounded operator. Moreover, $c_0,\ell_2\not\hookrightarrow\ell_1,E$, and $\mathcal T_1\not\approx\mathcal C_1$ and $\mathcal T_E\not\approx\mathcal C_E$. Note that
\[
\lim_{N\to\infty}\big\Vert P_{[\bigcup_{i=N}^\infty K_{A_i}],[\bigcup_{i=N}^\infty L_{A_i}]}T|_{[e_{\xi'_k,\eta'_l}]_{N\leq l\leq k<\infty}}\big\Vert=1 
\]
and
\[
\lim_{N\to\infty}\big\Vert P_{[\bigcup_{i=N}^\infty K_{A_i}],[\bigcup_{i=N}^\infty L_{A_i}]}T|_{[e_{\xi'_k,\eta'_l}]_{N\leq k<l<\infty}}\big\Vert=0.
\]
This shows that the conclusion of Theorem \ref{thmc} does not hold  for  bounded operators $T$   between two distinct symmetric operator ideals $\mathcal C_F$ and $\mathcal C_E$, even when $c_0,\ell_2\not\hookrightarrow E,F$, and $\mathcal T_E\not\approx\mathcal C_E$ and $\mathcal T_F\not\approx\mathcal C_F$.
\end{itemize}

\end{remark}

\begin{theorem}\label{emb1}
Suppose that $E$ is a separable quasi-Banach symmetric sequence space such that    $c_0,\ell_2\not\hookrightarrow E$, and $\mathcal T_E\not\approx\mathcal C_E$. If $T\in\mathcal B\big([e_{\xi_i,\eta_j}]_{i,j=1}^\infty,\mathcal C_E\big)$ satisfies that  $T|_{[e_{\xi_i,\eta_j}]_{1\leq j\leq i<\infty}}$ is an isomorphic embedding, then there is an increasing sequence of positive integers $\{i_k\}_{k=1}^\infty$, and two orthonormal sequences $\{\xi'_k\}_{k=1}^\infty$ with $\xi'_k\in[\xi_i]_{i_{2k+1}\leq i<i_{2k+2}}$ and $\{\eta'_k\}_{k=1}^\infty$ with $\eta'_k\in[\eta_i]_{i_{2k}\leq i<i_{2k+1}}$ such that $T|_{[e_{\xi'_k,\eta'_l}]_{k,l=1}^\infty}$ is an isomorphic embedding  and $T([e_{\xi'_k,\eta'_l}]_{k,l=1}^\infty)$ is complemented in $\mathcal C_E$.
\end{theorem}

\begin{proof}
Without loss of generality, we assume that $\{K_i\}_{i=1}^\infty$ and $\{L_i\}_{i=1}^\infty$ are two sequences of mutually orthogonal finite dimensional subspaces of $H$ with $[\bigcup_{i=1}^\infty K_i]=[\bigcup_{i=1}^\infty L_i]=H$.

Applying Theorem \ref{thmc} to $T$, we obtain two operators  $R_1,R_2\in\mathcal B\big([e_{\xi'_k,\eta'_l}]_{k,l=1}^\infty,\mathcal C_E\big)$ satisfying that the assertions in Theorem \ref{thmc} with $R_2$ is an isomorphic embedding since  $T|_{[e_{\xi_i,\eta_j}]_{1\leq j\leq i<\infty}}$ is an isomorphic embedding. 
In particular, there is a positive number $C$ such that
\begin{equation}\label{r2emb}
\Vert R_2(x)\Vert_{\mathcal C_E}\geq C\Vert x\Vert_{\mathcal C_E}\;\;\;\;\;{\rm for\;all}\;x\in[e_{\xi'_k,\eta'_l}]_{k,l=1}^\infty.
\end{equation}
Note that (see Theorem \ref{thmc} above)
\[
\big\Vert(T-(R_1+R_2))|_{[e_{\xi'_k,\eta'_l}]_{k,l=n}^\infty}\big\Vert \to0\;\;\;\;{\rm and}\;\;\;\;\big\Vert P_{[\bigcup_{k=n}^\infty K_k],[\bigcup_{k=n}^\infty L_k]}R_1\big\Vert\to0\;\;\;\;{\rm as}\;n\to\infty.
\]
Therefore, for a given  positive number $(1+2^{1-\frac{1}{r}})\delta<4^{1-\frac{1}{r}}C$ which is small enough, there is a positive integer $N_0$ such that 
\begin{align}\label{T-R1-R2}
\left\Vert(T-(R_1+R_2))|_{[e_{\xi'_k,\eta'_l}]_{k,l=N_0}^\infty}\right\Vert,\left\Vert P_{[\bigcup_{k=N_0}^\infty K_k],[\bigcup_{k=N_0}^\infty L_k]}R_1\right\Vert\leq\delta.
\end{align}
We now choose a positive integer $N>N_0$ such that $\{1,2,\dots,N_0\}\cap\bigcup_{k=2N}^\infty A_k=\emptyset$ (see Notation \ref{notationAk}). 
In particular, we have $R_2|_{[e_{\xi'_k,\eta'_l}]_{k,l=N}^\infty}\stackrel{\eqref{thmcr2}}{=} P_{[\bigcup_{k=N}^\infty K_k],[\bigcup_{k=N}^\infty L_k]}R_2|_{[e_{\xi'_k,\eta'_l}]_{k,l=N}^\infty}$. 
Then, for any $x\in[e_{\xi'_k,\eta'_l}]_{l,k=N}^\infty$, we have
\begin{eqnarray*}
\Vert T(x)\Vert_{\mathcal C_E}&\geq&\big\Vert P_{[\bigcup_{k=N}^\infty K_k],[\bigcup_{k=N}^\infty L_k]}T(x)\big\Vert_{\mathcal C_E}\\
&\stackrel{\eqref{T-R1-R2}}{\geq}&2^{1-\frac{1}{r}}\big\Vert P_{[\bigcup_{k=N}^\infty K_k],[\bigcup_{k=N}^\infty L_k]}(R_1+R_2)(x)\big\Vert_{\mathcal C_E}-\delta\Vert x\Vert_{\mathcal C_E}\\
&\geq&4^{1-\frac{1}{r}}\Vert R_2(x)\Vert_{\mathcal C_E}-2^{1-\frac{1}{r}}\big\Vert P_{[\bigcup_{k=N}^\infty K_k],[\bigcup_{k=N}^\infty L_k]}R_1(x)\big\Vert_{\mathcal C_E}-\delta\Vert x\Vert_{\mathcal C_E}\\
&\geq&(4^{1-\frac{1}{r}}C-2^{1-\frac{1}{r}}\delta-\delta)\Vert x\Vert_{\mathcal C_E}.
\end{eqnarray*}
This shows that $T|_{[e_{\xi'_k,\eta'_l}]_{l,k=N}^\infty}$ is an isomorphic embedding. Next, we will show that $T([e_{\xi'_k,\eta'_l}]_{k,l=N}^\infty)$ is complemented in $\mathcal C_E$ for sufficiently large $N$.

Without loss of generality, we may assume that 
\[
\{x_{2k+1,2l}\}_{k,l=1}^\infty\stackrel{\mbox{{\rm\tiny b.s}}}{\boxplus}\{K_{A_{2k+1}}\}_{k=1}^\infty\otimes\{L_{A_{2l}}\}_{l=1}^\infty\curvearrowleft\hat{a}\in\mathcal B(\hat{H})
\]
and $\hat{a}\neq0$. Hence, there exist $\xi,\eta\in\hat{H}$ such that $\alpha:=(\hat{a}\eta|\xi)\neq0$. We may choose two orthonormal sequences $\{\xi''_k\}_{k=1}^\infty$ and $\{\eta''_k\}_{k=1}^\infty$ in $H$ such that $\xi''_k\in K_{A_{2k+1}}$, $\eta''_k\in L_{A_{2k}}$ for every $k\geq1$,
\[
\{p_{_{[\xi''_k]}}x_{2k+1,2l}p_{_{[\eta''_l]}}\}_{k,l=1}^\infty\stackrel{\mbox{{\rm\tiny b.s}}}{\boxplus}\{K_{A_{2k+1}}\}_{k=1}^\infty\otimes\{L_{A_{2l}}\}_{l=1}^\infty\curvearrowleft\alpha(\cdot|\eta)\xi.
\]
Hence, we obtain
\[
P_{[\xi''_k]_{k=N}^\infty,[\eta''_k]_{k=N}^\infty}(x_{2k'+1,2l'}+y_{2l',2k'+1})=p_{_{[\xi''_{k'}]}}x_{2k'+1,2l'}p_{_{[\eta''_{l'}]}}=\alpha(\cdot|\eta''_{l'})\xi''_{k'}=\alpha e_{\xi''_{k'},\eta''_{l'}} 
\]
 for all $k',l'\geq N$.
Setting  $W:=P_{[\xi''_k]_{k=N}^\infty,[\eta''_k]_{k=N}^\infty}R_2|_{[e_{\xi'_k,\eta'_l}]_{k,l=N}^\infty}$, we have
\begin{equation}\label{W}
W(e_{\xi'_{k'},\eta'_{l'}})=P_{[\xi''_k]_{k=N}^\infty,[\eta''_k]_{k=N}^\infty}R_2(e_{\xi'_{k'},\eta'_{l'}})=\alpha e_{\xi''_{k'},\eta''_{l'}}\;\;\;\;\;\;{\rm for\;all}\;k',l'\geq N.
\end{equation}
It follows that
\begin{eqnarray*}
& &\Big\Vert P_{[\xi''_k]_{k=N}^\infty,[\eta''_k]_{k=N}^\infty}TW^{-1}-I_{[e_{\xi''_k,\eta''_l}]_{k,l=N}^\infty}\Big\Vert\\
&\leq&\Big\Vert P_{[\xi''_k]_{k=N}^\infty,[\eta''_k]_{k=N}^\infty}T|_{[e_{\xi'_k,\eta'_l}]_{k,l=N}^\infty}-W\Big\Vert\Vert W^{-1}\Vert\\
&=&\alpha^{-1}\big\Vert P_{[\xi''_k]_{k=N}^\infty,[\eta''_k]_{k=N}^\infty}(T-R_2)|_{[e_{\xi'_k,\eta'_l}]_{k,l=N}^\infty}\big\Vert\\
&\leq&\frac{2^{\frac{1}{r}-1}}{\alpha}\left\Vert(T-(R_1+R_2))|_{[e_{\xi'_k,\eta'_l}]_{k,l=N_0}^\infty}\right\Vert+\frac{2^{\frac{1}{r}-1}}{\alpha}\left\Vert P_{[\bigcup_{k=N_0}^\infty K_k],[\bigcup_{k=N_0}^\infty L_k]}R_1\right\Vert\\
&\stackrel{\eqref{T-R1-R2}}{\leq}&\frac{2^{\frac{1}{r}}\delta}{\alpha}.
\end{eqnarray*}
If we take $\delta<\alpha/2^{\frac{1}{r}}$, then it follows that $P_{[\xi''_k]_{k=N}^\infty,[\eta''_k]_{k=N}^\infty}TW^{-1}$ is an isomorphism from $[e_{\xi''_k,\eta''_l}]_{k,l=N}^\infty$ onto $[e_{\xi''_k,\eta''_l}]_{k,l=N}^\infty$. Hence, we have that
\begin{align*}
{\rm im}(P_{[\xi''_k]_{k=N}^\infty,[\eta''_k]_{k=N}^\infty})=[e_{\xi''_k,\eta''_l}]_{k,l=N}^\infty&~=~{\rm im}(P_{[\xi''_k]_{k=N}^\infty,[\eta''_k]_{k=N}^\infty}TW^{-1})\\
&\stackrel{\eqref{W}}{=}P_{[\xi''_k]_{k=N}^\infty,[\eta''_k]_{k=N}^\infty}\big(T([e_{\xi''_k,\eta''_l}]_{k,l=N}^\infty)\big),
\end{align*}
and $P_{[\xi''_k]_{k=N}^\infty,[\eta''_k]_{k=N}^\infty}|_{T([e_{\xi''_k,\eta''_l}]_{k,l=N}^\infty)}$ is an isomorphism onto its image $[e_{\xi''_k,\eta''_l}]_{k,l=N}^\infty$. 
The mapping given by 

\[
\big(P_{[\xi''_k]_{k=N}^\infty,[\eta''_k]_{k=N}^\infty}|_{T([e_{\xi''_k,\eta''_l}]_{k,l=N}^\infty)}\big)^{-1}P_{[\xi''_k]_{k=N}^\infty,[\eta''_k]_{k=N}^\infty}
\]
is  a bounded  projection from $\mathcal C_E$ onto $T([e_{\xi''_k,\eta''_l}]_{k,l=N}^\infty)$. 
\end{proof}

\begin{remark}\label{te-emb-te}
Assume that $T\in\mathcal B(\mathcal T_{E,\{\xi_i,\eta_i\}_{i=1}^\infty})$ is a isomorphic embedding when $c_0,\ell_2\not\hookrightarrow E$. 
Let $K_i=[\xi_i]$ and $L_i=[\eta_i]$ for every $i$. 
By Remark \ref{tftoce} (2) and  Theorem \ref{thmb}, 
we obtain two operators  $\tilde{R}_1,\tilde{R}_2\in\mathcal B\big([e_{\xi'_k,\eta'_l}]_{k,l=1}^\infty,\mathcal C_E\big)$ satisfying  the assertions in Theorem \ref{thmb} and 
\[
\tilde{R}_2(e_{\xi'_k,\eta'_l})=x_{2k+1,2l}\;\;\;\;\;\;{\rm for\;every}\;k,l\geq1,
\]
(i.e. $y_{2l,2k+1}=0$). Using the same method as in Theorem \ref{emb1}, one  can   show that there is a sufficiently large  positive integer $N$   such that $T([e_{\xi'_k,\eta'_l}]_{N\leq k<l<\infty})$ is complemented in $\mathcal T_{E,\{\xi_i,\eta_i\}_{i=1}^\infty}$. 
Note, however, that   we cannot 
guarantee $\tilde{R}_2([e_{\xi'_k,\eta'_l}]_{1\leq k<l<\infty})\subset\mathcal T_{E,\{\xi_i,\eta_i\}_{i=1}^\infty}$. Nevertheless, it can be seen from the proof of Theorem \ref{emb1} that 
it suffices to find a projection $P_{[\xi''_k]_{k=N}^\infty,[\eta''_k]_{k=N}^\infty}$ such that $P_{[\xi''_k]_{k=N}^\infty,[\eta''_k]_{k=N}^\infty}(x_{2k'+1,2l'})\in\mathcal T_{E,\{\xi_i,\eta_i\}_{i=1}^\infty}$ for every $1\leq l'\leq k'<\infty$. This can be achieved by  the construction of $\tilde{R}_2$ (see Lemma \ref{t4}) and \eqref{hata}.
\end{remark}

Combining Theorem \ref{emb1} and Remark \ref{te-emb-te}, we extend  \cite[Theorem 5.4]{Arazy4} to the setting of quasi-Banach spaces.

\begin{theorem}\label{ce-hc}
Suppose that $E$ is a separable quasi-Banach symmetric sequence space such that    $c_0,\ell_2\not\hookrightarrow E$. If $X$ is a subspace of $\mathcal C_E$ which is isomorphic to $\mathcal C_E$, then there exists a subspace $Y$ of $X$ that  is both isomorphic to $\mathcal C_E$ and complemented in $\mathcal C_E$.
\end{theorem}

\begin{remark}\label{chE}
    For general (quasi-)Banach spaces, the property described above is referred to as complementable homogeneity, see Definition \ref{ch}.
      A natural question is whether general symmetric sequence spaces are complementably homogeneous. 
      It is well-known that  $c_0$ and $\ell_p$, $1\leq p<\infty$ satisfy this property (see \cite[Theorem 2.a.3]{LT}).
      Moreover,  two observations from Johnson are worth noting:
\begin{itemize}
    \item Weighted Lorentz sequence spaces are complementably homogeneous. This can be obtained by using \cite[Theorem 4.e.3]{LT} and a standard perturbation argument.
    \item If  $X$ is a subspace of $L_p(0,1)$, $1<p<2$, that has a symmetric basis and contains an isomorphic copy of $L_p(0,1)$, then $X$ is not complementably homogeneous. Indeed, the existence of such a space can be found in  \cite[Example 10.4]{Johnson Maurey}. Since $L_p(0,1)$ is complementably homogeneous (see \cite[Theorem 9.1]{Johnson Maurey}), it follows that there exists a complemented subspace of $X$ which is an isomorphic copy of $L_p(0,1)$. If $X$ is complementably homogeneous, this yields that $X$ is isomorphic to a complemented subspace of $L_p
    (0,1)$. By the classical Pe{\l}czy\'nski decomposition technique, we have $X\approx L_p(0,1)$. This contradicts the fact that $L_p(0,1)$ has no symmetric basis. 
\end{itemize}
Indeed,  \cite[Theorem 9.1]{Johnson Maurey} shows that  a broad class of rearrangement-invariant function spaces are complementably homogeneous. However,  a similar result for  symmetric sequence spaces  is still unknown.
\end{remark}

\begin{definition}\label{stri-sing}
Suppose that $X$, $Y$ and $Z$ are (quasi-)Banach spaces. 
Recall that a bounded operator $T:X\to Y$ is called \emph{$Z$-strictly singular} (e.g. see \cite[page 25]{PR}) provided that for every subspace of  $X$  that is isomorphic to $Z$, the 
 restriction of $T$ to that subspace  fails to be   an isomorphic embedding.
\end{definition}

We shall prove a somewhat unexpected result, Theorem \ref{cs=ts}, which states that in the setting of $\mathcal B(\mathcal C_E)$, the $\mathcal C_E$-strictly singularity is equivalent to the $\mathcal T_E$-strictly singularity. 
Before proceeding to the proof of Theorem \ref{cs=ts}, we  need the following lemma.

\begin{lemma}\label{te}
Suppose that $E$ is a separable quasi-Banach symmetric sequence space such that    $c_0,\ell_2\not\hookrightarrow E$. Let $X$ be a subspace of $\mathcal C_E$, which is  isomorphic to $\mathcal T_E$. Then there is a sequence $\{x_{i,j}\}_{i,j=1}^\infty$ in $\mathcal C_E$ such that 
\[
x_{i,j}\in\left\{
\begin{array}{rcl}
{X\;}, &  & {1\leq j\leq i<\infty,}\\
{\mathcal C_E}, &  & {1\leq i<j<\infty,
}\\
\end{array}
\right.
\]
 and $\{x_{i,j}\}_{i,j=1}^\infty\sim\{e_{\xi_i,\eta_j}\}_{i,j=1}^\infty$.
\end{lemma}

\begin{proof}
We may assume that $\{K_i\}_{i=1}^\infty$ and $\{L_i\}_{i=1}^\infty$ are the two sequences of mutually orthogonal finite dimensional subspaces of $H$ with $[\bigcup_{i=1}^\infty K_i]=[\bigcup_{i=1}^\infty L_i]=H$. Since $X\approx \mathcal T_E$, it follows that there exists  an isomorphism $T$ from $\{e_{\xi_i,\eta_j}\}_{1\leq j\leq i<\infty}$ onto $X$.

Applying Theorem \ref{thmb} to $T$, we obtain two operators  $\tilde{R}_1,\tilde{R}_2\in\mathcal B\big([e_{\xi'_k,\eta'_l}]_{k,l=1}^\infty,\mathcal C_E\big)$ satisfying that the assertions in Theorem \ref{thmb}. In particular, since $T$ is an isomorphic embedding,  $\tilde{R}_2$ is also an isomorphic embedding; hence,  there is a positive number $C$ such that
\[
\Vert\tilde{R}_2(x)\Vert_{\mathcal C_E}\geq C\Vert x\Vert_{\mathcal C_E}\;\;\;\;\;{\rm for\;all}\;x\in[e_{\xi'_k,\eta_l}]_{l,k=1}^\infty.
\]
Using Lemma \ref{2d''}, 
there exists a bounded operator $S:[e_{\xi'_k,\eta'_l}]_{k,l=1}^\infty\to\mathcal C_E$ such that
\[
S(e_{\xi'_k,\eta'_l})=\left\{
\begin{array}{rcl}
{T(e_{\xi'_k,\eta'_l})}\;\;\;\;\;\;, &  & {1\leq l\leq k<\infty,}\\
{(\tilde{R}_1+\tilde{R}_2)(e_{\xi'_k,\eta'_l})}, &  & {1\leq k<l<\infty.}\\
\end{array}
\right.
\]
and
\[
\big\Vert(S-(\tilde{R}_1+\tilde{R}_2))|_{[e_{\xi'_k,\eta'_l}]_{l,k=n}^\infty}\big\Vert\to0\;\;\;\;{\rm as
}\;n\to\infty.
\]
Repeating the proof of the first assertion of  Theorem \ref{emb1}  mutatis mutandis, we find
   a positive number $N$ such that $\big\{S(e_{\xi'_k,\eta'_l})\big\}_{k,l=N}^\infty\sim\{e_{\xi'_k,\eta'_l}\}_{k,l=N}^\infty$ and $S(e_{\xi'_k,\eta'_l})\in X$ for all $N\leq l\leq k<\infty$.
\end{proof}

\begin{theorem}\label{cs=ts}
Suppose that $E$ is a (quasi-)Banach symmetric sequence space such that    $c_0,\ell_2\not\hookrightarrow E$. Then,  an operator $T\in\mathcal B(\mathcal C_E)$ is $\mathcal C_E$-strictly singular if and only if it is $\mathcal T_E$-strictly singular.
\end{theorem}

\begin{proof}
The proof is trivial when  $\mathcal T_E\approx\mathcal C_E$.

We now assume $\mathcal T_E\not\approx\mathcal C_E$. Let $T\in\mathcal B(\mathcal C_E)$. If $T$ is $\mathcal T_E$-strictly singular, then it is obvious that $T$ is $\mathcal C_E$-strictly singular.

Conversely, If $T$ is not $\mathcal T_E$-strictly singular, then there is a subspace $X$ of $\mathcal C_E$, isomorphic to $\mathcal T_E$, such that $T|_X$ is an isomorphic embedding. By Lemma~\ref{te}, there is a sequence $\{x_{i,j}\}_{i,j=1}^\infty$ in $\mathcal C_E$ such that $x_{i,j}\in X$ for all $j\leq i$ and $\{x_{i,j}\}_{i,j=1}^\infty\sim\{e_{\xi_i,\eta_j}\}_{i,j=1}^\infty$, where $\{\xi_i\}_{i=1}^\infty$ and $\{\eta_i\}_{i=1}^\infty$ are two orthonormal sequences. Let us consider the operator $T|_{[x_{i,j}]_{i,j=1}^\infty}$;   the proof is completed by applying Theorem~\ref{emb1}.
\end{proof}

When $E$ is separable Banach symmetric sequence space, several  equivalent conditions for the $\mathcal T_{E,\{\xi_i,\eta_i\}_{i=1}^\infty}$ (see \eqref{ts}) to be complemented in $C_E$ have been 
systematically studied in \cite{KwapienPelczynski} and \cite{Arazy1}. 
We end this section by applying the preceding conclusions to show that 
if $E$ is a separable (quasi-)Banach symmetric sequence
space such that
$c_0,\ell_2\not\hookrightarrow E$,  then 
  $\mathcal T_E$ is  isomorphic to a complemented subspace of $\mathcal C_E$ if and only if   $\mathcal T_{E,\{\xi_i,\eta_i\}_{i=1}^\infty}$ is  complemented in $\mathcal C_E$.  
This strengthens the classical result in  \cite{KwapienPelczynski} and \cite{Arazy1} (see Theorem \ref{tria} above).
In particular, $\mathcal T_1$ is not isomorphic to any complemented subspace of $\mathcal C_1$, which will be   used to prove that $\mathcal T_1\oplus \mathcal C_1$ is a Wintner space, i.e., Corollary \ref{t1+c1}.

\begin{theorem}\label{teicce}
Suppose that $E$ is a separable quasi-Banach symmetric sequence space such that $c_0,\ell_2\not\hookrightarrow E$. The followings are equivalent:
\begin{itemize}
\item [(a)] The triangular projection 
$T^{E,\{\xi_i,\eta_i\}_{i=1}^\infty}$ is bounded on the span$\big(\{e_{\xi_i,\eta_j}\}_{i,j=1}^\infty\big)$;
\item [(b)]
$\mathcal T_{E,\{\xi_i,\eta_i\}_{i=1}^\infty}$ is complemented in $\mathcal C_E$;
\item [(c)]
$\mathcal T_E$ is isomorphic to a complemented subspace of $\mathcal C_E$;
\item [(d)]
$\mathcal T_E\approx\mathcal C_E$.
\end{itemize}
\end{theorem}
\begin{proof} 
(a)$\Longrightarrow$(b)$\Longrightarrow(c)$: This is obvious. 

(c)$\Longrightarrow$(d): Let $Q$ be the projection on $\mathcal C_E$ such that im$(Q)\approx\mathcal T_E$. 
In particular,  $Q$ is not $\mathcal T_E$-strictly singular. By Theorem \ref{cs=ts}, $Q$ is not $\mathcal C_E$-strictly singular. This yields that $\mathcal C_E$ is isomorphic to a subspace $Z$ of $\mathcal T_E$. 
By  Theorem~\ref{ce-hc}, $Z$ contains a subspace isomorphic to $\mathcal \cC_E$ which is complemented in $\cC_E$. 
Since $\mathcal{T}_E$ can be viewed as a subspace of $\cC_E$, 
  it follows that $\cC_E$ is isomorphic to a complemented subspace of  $\mathcal{T}_E$. 
 Next, we employ the classical Pe{\l}czy\'nski decomposition technique to show $\mathcal T_E\approx\mathcal C_E$. There are two subspaces $X$ and $Y$ of $\mathcal C_E$ such that 
\[
\mathcal C_E\approx\mathcal T_E\oplus X\;\;\;\;{\rm and}\;\;\;\;\mathcal T_E\approx\mathcal C_E\oplus Y.
\]
By Proposition \ref{te=te2},
\[
\mathcal C_E\approx\mathcal T_E\oplus X\approx\mathcal T_E\oplus\mathcal T_E\oplus X\approx\mathcal T_E\oplus\mathcal C_E,
\]
and similarly, $\mathcal T_E\approx\mathcal T_E\oplus\mathcal C_E$. Consequently, $\mathcal T_E\approx\mathcal C_E$.

(d)$\Longrightarrow$(a): Now, we may let $K_i=[\xi_i]$ and $L_i=[\eta_i]$ for every $i$, and $T$ be an isomorphism from $[e_{\xi_i,\eta_j}]_{i,j=1}^\infty$ onto $\mathcal T_{E,\{\xi_i,\eta_i\}_{i=1}^\infty}$. Suppose, by contrary, that $T^{E,\{\xi_i,\eta_i\}_{i=1}^\infty}$ is unbounded.

By Remark \ref{thmc-rem} (1), applying Theorem \ref{thmc} to $T$, we obtain two operators $R_1, R_2\in\mathcal B\big([e_{\xi'_k,\eta'_l}]_{k,l=1}^\infty,\mathcal C_E\big)$, where $\{i_k\}_{k=1}^\infty$ is an increasing sequence of positive integers, and $\{\xi'_k\}_{k=1}^\infty$ and $\{\eta'_k\}_{k=1}^\infty$ are two orthonormal sequences with $\xi'_k\in[\xi_i]_{i_{2k+1}\leq i<i_{2k+2}}$ and  $\eta'_k\in[\eta_i]_{i_{2k}\leq i<i_{2k+1}}$, satisfying that the assertions in Theorem \ref{thmc}, such that  $R_2$ is an isomorphic embedding and
\begin{align}\label{TR1R2}
\big\Vert(T-(R_1+R_2))|_{[e_{\xi'_k,\eta'_l}]_{k,l=N}^\infty}\big\Vert \to0
\end{align}
and
\begin{align}\label{PR1 to 0}
\big\Vert P_{[\bigcup_{k=N}^\infty K_k],[\bigcup_{k=N}^\infty L_k]}R_1\big\Vert\to0\;\;\;\;{\rm as}\;N\to\infty.
\end{align}
Moreover, by Remark \ref{tftoce} (2) and the construction of $R_2$ (see proof of Theorem~\ref{thmc} and Lemma \ref{t4}), we know that
\[
R_2(e_{\xi'_k,\eta'_l})=x_{2k+1,2l}\;\;\;\;\;\;\;\;{\rm for\;every}\;k,l\geq1,
\]
and we may choose an increasing sequence $\{m_i\}_{i=2}^\infty$ of positive integers such that
\begin{equation}\label{pxp-bs}
\big\{p_{_{K'_{2k+1}}}x_{2k+1,2l}p_{_{L'_{2l}}}\big\}_{k,l=1}^\infty\stackrel{\mbox{{\rm\tiny b.s}}}{\boxplus}\{K'_{2k+1}\}_{k=1}^\infty\otimes\{L'_{2l}\}_{l=1}^\infty\curvearrowleft\hat{a}_0\in\mathcal B(\hat{H})
\end{equation}
and $\hat{a}_0\neq0$, where $K'_{2k+1}=[\bigcup_{i=m_{2k+1}+1}^{m_{2k+2}}K_i]$ and $L'_{2k}=[\bigcup_{i=m_{2k}+1}^{m_{2k+1}}L_i]$ for every $k\geq1$.

By \eqref{TR1R2} and \eqref{PR1 to 0}, we have
\[
\big\Vert P_{K'_{2N+1},L'_{2N+2}}(T-R_2)|_{[e_{\xi'_k,\eta'_l}]_{k,l=N}^\infty}\big\Vert\to0\;\;\;\;{\rm as}\;N\to\infty.
\]
Note that $P_{K'_{2N+1},L'_{2N+2}}\big(\mathcal T_{E,\{\xi_i,\eta_i\}_{i=1}^\infty}\big)=\{0\}$, and thus $P_{K'_{2N+1},L'_{2N+2}}T=0$. This yields
\[
\big\Vert P_{K'_{2N+1},L'_{2N+2}}R_2|_{[e_{\xi'_k,\eta'_l}]_{k,l=N}^\infty}\big\Vert\to0\;\;\;\;{\rm as}\;N\to\infty.
\]
However, we have 
\[
\big\Vert P_{K'_{2N+1},L'_{2N+2}}R_2(e_{\xi'_{N},\eta'_{N+1}})\big\Vert_{\mathcal C_E}=\Vert\hat{a}_0\Vert_{\mathcal C_E}\;\;\;\;\;\;\;{\rm for\;all}\;N,
\]
which is a contradiction.
\end{proof}

\section{Arazy-type decomposition theorem for operators on $\mathcal{C}_p$ and its applications}
\label{Sec:Arazy's decomposition for cp}

\subsection{Decompositions  of  operators on $\mathcal{C}_p$}
When $E=\ell_p$, $0<p<\infty$, we write $\mathcal C_{\ell_p}$ as $\mathcal C_p$, $\mathcal T_{\ell_p}$ as $\mathcal T_p$, $\mathcal T_{\ell_p,\{\xi_i,\eta_i\}_{i=1}^\infty}$ as $\mathcal T_{p,\{\xi_i,\eta_i\}_{i=1}^\infty}$ and $\left\Vert\cdot \right\Vert_{\ell_p}$ as $\left\Vert\cdot \right\Vert_p$ for short. Note that $\left\Vert\cdot \right\Vert_p$ is $\min\{p,1\}$-subadditive (e.g. see \cite{McCarthy}), so that  the modulus of concavity of $\left\Vert\cdot \right \Vert_p$ is $2^{\max\{\frac{1}{p},1\}-1}$.

In this section, we always assume that  $\{\xi_i\}_{i=1}^\infty$ and $\{\eta_i\}_{i=1}^\infty$ are  two sequences of orthonormal vectors of $H$, and $\{K_i\}_{i=1}^\infty$ and $\{L_i\}_{i=1}^\infty$ are  two sequences of mutually orthogonal finite dimensional subspaces of $H$, unless otherwise stated.
Recall that for a subset $A\subset \mathbb{N}$,  we  denote 
\begin{equation}
K_A=\left[{\bigcup}_{i\in A}K_i\right],\;\;\;\;\;\;\;L_A=\left[{\bigcup}_{i\in A}L_i\right],
\end{equation}
and 
 \begin{equation}
K=\left[{\bigcup}_{i=1}^\infty K_i\right]\;\;\;\;{\rm and}\;\;\;\;L=\left[{\bigcup}_{i=1}^\infty L_i\right].
\end{equation}

Assume that 
$\{x_k\}_{k=1}^\infty$ is a normalized sequence such that $x_k=p_{_{K_k}}x_k p_{_{L_k}}$ in $\mathcal C_p$ for every $k$. 
Then $\{x_k\}_{k=1}^\infty$ is isometrically equivalent to the unit basis $\big\{e_k^{\ell_p}\big\}_{k=1}^\infty$ of $\ell_p$. 
Moreover, for $0<p< 2$, we have $\big\{e_k^{\ell_p}\big\}_{k=1}^\infty\sim\big\{ e_k^{\ell_2}\oplus e_k^{\ell_p}\big\}_{k=1}^\infty$.
As a consequence of  Theorem \ref{3e}, Corollary \ref{3e'} and Remark \ref{3a'}, we obtain the following  variant of \cite[Lemma 2.4]{Arazy2}.

\begin{theorem}\label{cpl2}
Let $0<p<2$. Suppose that $\{x_k\}_{k=2}^\infty$ is a bounded sequence in $\mathcal C_p$, having  no subsequence equivalent to the unit basis of $\ell_p$, with the form 
\[
x_k=p_{_{[\bigcup_{l=1}^kK_l]}}x_kp_{_{[\bigcup_{l=1}^kL_l]}}-p_{_{[\bigcup_{l=1}^{k-1}K_l]}}x_kp_{_{[\bigcup_{l=1}^{k-1}L_l]}}\;\;\;\;\;\;k=2,3,\dots.
\]
For a given  sequence of positive numbers $\{\varepsilon_i\}_{i=2}^\infty$,  
there exist an increasing subsequences $\{k_i\}_{i=2}^\infty$ of positive integers, and a sequence $\{y_i\}_{i=2}^\infty$ in $\mathcal C_p$ having the form 
\[
y_i=\sum_{j=1}^{i-1}\big(a_{i,j}+a_{j,i}\big)\;\;\;\;\;\;i=2,3,\dots,
\]
such that $\Vert x_{k_i}-y_i\Vert_p\leq\varepsilon_i$ for every $i\geq2$, where
\begin{itemize}
\item [(1)]
$K'_i=[\bigcup_{l=k_{i-1}+1}^{k_i}K_l]$ and $L'_i=[\bigcup_{l=k_{i-1}+1}^{k_i}L_l]$ for every $i\geq1$ (put $n_0=0$),
\item [(2)]
$\{a_{i,j}\}_{1\leq i\neq j<\infty}\stackrel{\mbox{{\rm\tiny b.s}}}{\boxplus}\{K'_i\}_{i=1}^\infty\otimes\{L'_j\}_{j=1}^\infty$,
\item [(3)]
$\{a_{i,j}\}_{i=j+1}^\infty\stackrel{\mbox{{\rm\tiny b.s}}}{\boxplus}\{K'_i\}_{i=2}^\infty\otimes\{L'_j\}_{j=1}^\infty\curvearrowleft d_j$ are consistent for all $j\geq1$, and \[
\Vert d_j\Vert_p\leq\varepsilon_j\Vert d_1\Vert_p\;\;\;\;\;\;\mbox{for every}\; j\geq 2,
\]
\item [(4)]
$\{a_{j,i}\}_{i=j+1}^\infty\stackrel{\mbox{{\rm\tiny b.s}}}{\boxplus}\{K'_j\}_{j=1}^\infty\otimes\{L'_i\}_{i=2}^\infty\curvearrowleft d'_j$ are consistent for all $j\geq 1$, and
\[
\Vert d'_j\Vert_p\leq\varepsilon_j\Vert d'_1\Vert_p\;\;\;\;\;\;\mbox{for every}\; j\geq 2.
\]
\end{itemize}
\end{theorem}

By standard perturbation arguments (see, e.g.,  Corollary \ref{3e'}), we obtain  the following corollary (cf.  \cite[Proposition 2.3]{Arazy2}, \cite[Proposition 4]{Arazy Lindenstrauss}, and \cite{Friedman} for  similar results).

\begin{corollary}\label{cpl2'}
Let $0<p<2$. Suppose that $\{x_k\}_{k=1}^\infty$ is a semi-normalized $\sigma$-weakly null sequence in $\mathcal C_p$ such that  $p_{_{[\bigcup_{i=1}^\infty K_i]}}x_kp_{_{[\bigcup_{j=1}^\infty L_i]}}=x_k$ for every $k$, which  contains  no subsequence equivalent to the unit basis of $\ell_p$. For a given $\varepsilon>0$,   there exist two increasing subsequences $\{k_i\}_{i=2}^\infty$ and $\{n_i\}_{i=1}^\infty$ of positive integers, and a sequence $\{y_i\}_{i=2}^\infty$ in $\mathcal C_p$ having the form
\[
y_i=a_i+b_i\;\;\;\;\;\;i=2,3,\dots,
\]
such that for every finitely nonzero sequence of scalars $\{t_i\}_{i=2}^\infty$, we have 
\[
\bigg\Vert\sum_{i=2}^\infty t_i(x_{k_i}-y_i)\bigg\Vert_p\leq\varepsilon\bigg\Vert\sum_{i=2}^\infty t_ix_{k_i}\bigg\Vert_p,
\]
where 
\begin{itemize}
\item [(1)]
$K'_i=[\bigcup_{l=n_{i-1}+1}^{n_i}K_l]$ and $L'_i=[\bigcup_{l=n_{i-1}+1}^{n_i}L_l]$ for every $i\geq1$ (put $n_0=0$),
\item [(2)]
$\{a_i\}_{i=2}^\infty\stackrel{\mbox{{\rm\tiny b.s}}}{\boxplus}\{K'_i\}_{i=2}^\infty\otimes L'_1\curvearrowleft\;$,
\item [(3)]
$\{b_i\}_{i=2}^\infty\stackrel{\mbox{{\rm\tiny b.s}}}{\boxplus}K'_1\otimes\{L'_i\}_{i=2}^\infty\curvearrowleft\;$.
\end{itemize}

\end{corollary}

The next proposition is a consequence of \cite[Theorem 2.2]{Arazy Friedman}. For the sake of completeness, we present the proof below.

\begin{proposition}\label{cpc}
Let $1\leq p<\infty$, and $\{F_i\}_{i=2}^\infty$ and $\{G_i\}_{i=2}^\infty$ be two sequences of mutually orthogonal closed subspaces of $H$. Suppose that
\[
\{x_{k,l}\}_{k,l=1}^\infty\stackrel{\mbox{{\rm\tiny b.s}}}{\boxplus}\{F_{2k+1}\}_{k=1}^\infty\otimes \{G_{2l}\}_{l=1}^\infty\curvearrowleft x,
\]
\[
\{y_{l,k}\}_{k,l=1}^\infty\stackrel{\mbox{{\rm\tiny b.s}}}{\boxplus}\{F_{2l}\}_{l=1}^\infty\otimes\{G_{2k+1}\}_{k=1}^\infty\curvearrowleft y,
\]
and $\left\Vert\left(
\begin{smallmatrix}
0&y \\ x&0
\end{smallmatrix}\right)
\right\Vert_p=1$. Then $\{x_{k,l}+y_{l,k}\}_{k,l=1}^\infty\simeq\{e_{\xi_k,\eta_l}\}_{k,l=1}^\infty$, and $[x_{k,l}+y_{l,k}]_{k,l=1}^\infty$ is $1$-complemented in $\mathcal C_p$ (the projection operator is given in  \eqref{cpproj} below).
\end{proposition}
\begin{proof}
Note that $\{x_{k,l}+y_{l,k}\}_{k,l=1}^\infty\simeq\{x\otimes(\cdot|e_{2l})e_{2k+1}+y\otimes(\cdot|e_{2k+1})e_{2l}\}_{k,l=1}^\infty$. By Proposition \ref{ie}, we have  $\{x_{k,l}+y_{l,k}\}_{k,l}^\infty\simeq\{e_{\xi_k,\eta_l}\}_{k,l}^\infty$.

Let  $\lambda_{2i-1} := s_i(x)=$ and $\lambda_{2i}:=s_i(y)$ for each $i$. 
We have  $(\sum_{i=1}^\infty\vert\lambda_i\vert^p)^{1/p}=1$. 
This implies the existence of a sequence  $\{\mu_i\}_{i=1}^\infty\in\ell_{p'}$ of norm $1$ such that $\sum_{i=1}^\infty\lambda_i\mu_i=1$. 
Moreover,  there are two sequences $\{\xi_{k,i}\}_i$ and $\{\eta_{k,i}\}_i$ of orthonormal vectors of $H$ with $\xi_{k,i}\in F_k$ and $\eta_{k,i}\in G_k$ for every $k,i$ such that
\[
x_{k,l}=\sum_i\lambda_{2i-1}e_{\xi_{2k+1,i},\eta_{2l,i}}\;\;\;\;\;\;{\rm and}\;\;\;\;\;\;y_{l,k}=\sum_{i}\lambda_{2i}e_{\xi_{2l,i},\eta_{2k+1,i}}.
\]
Define
\[
D_i=[\xi_{2k+1,i}]_{k=1}^\infty\;\;\;\;\;\;{\rm and}\;\;\;\;\;\;\;E_i=[\eta_{2l,i}]_{l=1}^\infty,
\]
\[
D'_i=[\xi_{2l,i}]_{l=1}^\infty\;\;\;\;\;\;{\rm and}\;\;\;\;\;\;\;E_i'=[\eta_{2k+1,i}]_{k=1}^\infty
\]
for every $i$. For any $a\in\mathcal C_p$, we have 
\begin{align*}
\Vert a\Vert_p\geq &\bigg(\sum_i\big(\Vert p_{_{D_i}}ap_{_{E_i}}\Vert_p^p+\Vert p_{_{D'_i}}ap_{_{E'_i}}\Vert_p^p\big)\bigg)^{1/p}\\
=&\left(\sum_i\left(\left\Vert{\sum}_{k,l}(a\eta_{2l,i}|\xi_{2k+1,i})e_{\xi_{2k+1,i},\eta_{2l,i}}\right\Vert_p^p+\left\Vert{\sum}_{k,l}(a\eta_{2k+1,i}|\xi_{2l,i})e_{\xi_{2l,i},\eta_{2k+1,i}}\right\Vert_p^p\right)\right)^{1/p}\\
=&\left(\sum_i\left(\left\Vert{\sum}_{k,l}(a\eta_{2l,i}|\xi_{2k+1,i})(\cdot |e_l)e_k\right\Vert_p^p+\left\Vert{\sum}_{k,l}(a\eta_{2k+1,i}|\xi_{2l,i})(\cdot |e_l)e_k\right\Vert_p^p\right)\right)^{1/p}\\
\geq&\sum_i\left(\mu_{2i-1}\left\Vert{\sum}_{k,l}(a\eta_{2l,i}|\xi_{2k+1,i})(\cdot |e_l)e_k\right\Vert_p+\mu_{2i}\left\Vert{\sum}_{k,l}(a\eta_{2k+1,i}|\xi_{2l,i})(\cdot |e_l)e_k\right\Vert_p\right)\\
\geq&\left\Vert\sum_{k,l}\sum_i\Big(\mu_{2i-1}(a\eta_{2l,i}|\xi_{2k+1,i})+\mu_{2i}(a\eta_{2k+1,i}|\xi_{2l,i})\Big)(\cdot |e_l)e_k\right\Vert_p\\
=&\left\Vert\sum_{k,l}\sum_i\Big(\mu_{2i-1}(a\eta_{2l,i}|\xi_{2k+1,i})+\mu_{2i}(a\eta_{2k+1,i}|\xi_{2l,i})\Big)(x_{k,l}+y_{l,k})\right\Vert_p.\\
\end{align*}
Hence, there exists a contractive operator on $\mathcal C_p$ defined  by
\begin{equation}\label{cpproj}
P:a\in\mathcal C_p\longmapsto \sum_{k,l}\sum_i\Big(\mu_{2i-1}(a\eta_{2l,i}|\xi_{2k+1,i})+\mu_{2i}(a\eta_{2k+1,i}|\xi_{2l,i})\Big)(x_{k,l}+y_{l,k}).
\end{equation}
Note that
\begin{align*}
&P(x_{k',l'}+y_{l',k'})\\
=&\sum_{k,l}\sum_i\Big(\mu_{2i-1}\big((x_{k',l'}+y_{l',k'})\eta_{2l',i}|\xi_{2k'+1,i}\big)+\mu_{2i}\big((x_{k',l'}+y_{l',k'})\eta_{2k'+1,i}|\xi_{2l',i}\big)\Big)(x_{k,l}+y_{l,k})\\
=&\sum_i\big(\mu_{2i-1}\lambda_{2i-1}+\mu_{2i}\lambda_{2i}\big)(x_{k',l'}+y_{l',k'})\\
=&x_{k',l'}+y_{l',k'},
\end{align*}
for every $k'$ and $l'$. Hence, we have $P|_{[x_{k,l}+y_{l,k}]_{k,l=1}^\infty}={\rm id}_{[x_{k,l}+y_{l,k}]_{k,l=1}^\infty}$. Furthermore, ${\rm im}(P)\subset[x_{k,l}+y_{l,k}]_{k,l=1}^\infty$; consequently,    $P^2=P$.  Thus, $P$ is a contractive projection from $\mathcal C_p$ onto $[x_{k,l}+y_{l,k}]_{k,l=1}^\infty$.
\end{proof}

When  $E=\ell_p$, $1\leq p<2$, if we replace Theorem \ref{3e} (resp. Corollary \ref{3e'}) in the proof of Theorem~\ref{thma} by Theorem \ref{cpl2} (resp. Corollary \ref{cpl2'}), then  we obtain $T_1,T_2=0$ in Theorem~\ref{thma}. For the case when  $  E=\ell_p$, $2<p<\infty$,    Lemma \ref{t2t3} implies that  $T_3=0$ in Theorem \ref{thma}.  

\begin{theorem}\label{thmalp}
Let  $ E=\ell_p$, $0<p\neq2<\infty $. 
\begin{itemize}
\item [(1)]
If $0<p<2$, then $T_1=T_2=0$ in Theorem \ref{thma}.
\item [(2)]
If $2<p<\infty$, then $T_3=0$ in Theorem \ref{thma}.
\end{itemize}
\end{theorem}

Therefore,   Lemma \ref{c0l2} is not needed for the proof of 
 Theorem \ref{thmb} when  $\mathcal C_E=\mathcal C_p$, $0<p<2$, leading  the following result.

\begin{theorem}\label{1p2}
Suppose that $0<p<2$ and $T\in\mathcal B(\mathcal T_{p,\{\xi_i,\eta_i\}_{i=1}^\infty},\mathcal C_p)$ satisfies $p_{_{[\bigcup_{i=1}^\infty K_i]}}T(\cdot)p_{_{[\bigcup_{i=1}^\infty L_i]}}=T$. Given a sequence $\{\varepsilon_l\}_{l=1}^\infty$ of positive numbers, there exist an increasing sequence $\{i_k\}_{k=1}^\infty$ of positive integers, and two operators $\tilde{R}_1,\tilde{R}_2\in\mathcal B\big([e_{\xi_{i_{2k+1}},\eta_{i_{2l}}}]_{k,l=1}^\infty,\mathcal C_p\big)$
having the forms
\[
\tilde{R}_1(e_{\xi_{i_{2k+1}},\eta_{i_{2l}}})=\left\{
\begin{array}{rcl}
x_{_{(2k+1,2l)}}+y_{_{(2l,2k+1)}}, &  & {1\leq l\leq k<\infty,}\\
0\;\;\;\;\;\;\;\;\;\;\;\;\;, &  & \text{otherwise,}\\
\end{array}
\right.
\]
\[
\tilde{R}_2(e_{\xi_{i_{2k+1}},\eta_{i_{2l}}})=x_{2k+1,2l}+y_{2l,2k+1}\;\;\;\;\;\;\;\;\mbox{for every}\;k,l\geq1,
\]
such that $\big(\tilde{R}_1+\tilde{R}_2\big)|_{\mathcal T_{p,\{\xi_{i_{2k+1}},\eta_{i_{2k}}\}_{k=1}^\infty}}$ is a perturbation of $T|_{\mathcal T_{p,\{\xi_{i_{2k+1}},\eta_{i_{2k}}\}_{k=1}^\infty}}$ associated with the  Schauder decomposition $\big\{[e_{\xi_{i_{2k+1}},\eta_{i_{2l}}}]_{k=l}^\infty\big\}_{l=1}^\infty$ for $\{\varepsilon_l\}_{l=1}^\infty$, where

\begin{itemize}
\item [(1)]
$\{A_k\}_{k=1}^\infty$ is a sequence of mutually disjoint subsets of\;$\mathbb N$, and $\{F_{(2k+1,2l)}\}_{1\leq l\leq k<\infty}$ and $\{G_{(2l,2k+1)}\}_{1\leq l\leq k<\infty}$ are two sequences of mutually orthogonal closed subspaces of $H$ such that
\[
F_{(2k+1,2l)}\subset K_{A_{2k+1}}\;\;\;\;{\rm and}\;\;\;\;G_{(2l,2k+1)}\subset L_{A_{2k+1}}\;\;\;\;\mbox{for every}\;1\leq l\leq k<\infty.
\]
\item [(2)]
\[
\{x_{_{(2k+1,2l)}}\}_{1\leq l\leq k<\infty}\stackrel{\mbox{{\rm\tiny b.s}}}{\boxplus}\{F_{(2k+1,2l)}\}_{1\leq l\leq k<\infty}\otimes L\curvearrowleft\hat{b},
\]
\[
\{y_{_{(2l,2k+1)}}\}_{1\leq l\leq k<\infty}\stackrel{\mbox{{\rm\tiny b.s}}}{\boxplus} K\otimes\{G_{(2l,2k+1)}\}_{1\leq l\leq k<\infty}\curvearrowleft\tilde{b},
\]
and $\left\Vert\left(
\begin{smallmatrix}
0&\tilde{b}\\ \hat{b}&0
\end{smallmatrix}\right)
\right\Vert_p\leq4^{\max\{\frac{1}{p},1\}-1}\Vert T\Vert$,
\item [(3)] 
\[
\{x_{2k+1,2l}\}_{k,l=1}^\infty\stackrel{\mbox{{\rm\tiny b.s}}}{\boxplus}\{K_{A_{2k+1}}\}_{k=1}^\infty\otimes\{L_{A_{2l}}\}_{l=1}^\infty\curvearrowleft\;
\]
and
\[
\{y_{2l,2k+1}\}_{k,l=1}^\infty\stackrel{\mbox{{\rm\tiny b.s}}}{\boxplus}\{K_{A_{2l}}\}_{l=1}^\infty\otimes\{L_{A_{2k+1}}\}_{k=1}^\infty\curvearrowleft\;.
\]
\end{itemize}
\end{theorem}

\begin{remark}
    Compared to Theorem \ref{thmb},  Theorem \ref{1p2} above allows us to   select a subsequence $\{\xi_{i_{2k+1}}\}_{k=1}^\infty$ of $\{\xi_i\}_{i=1}^\infty$ and a subsequence $\{\eta_{i_{2k}}\}_{k=1}^\infty$ of $\{\eta_i\}_{i=1}^\infty$ such that $\big(\tilde{R}_1+\tilde{R}_2\big)|_{\mathcal T_{p,\{\xi_{i_{2k+1}},\eta_{i_{2k}}\}_{k=1}^\infty}}$ is a perturbation of $T|_{\mathcal T_{p,\{\xi_{i_{2k+1}},\eta_{i_{2k}}\}_{k=1}^\infty}}$, whereas   Theorem \ref{thmb} requires  new sequences  $\{\xi'_k\}_{k=1}^\infty$ and $\{\eta'_k\}_{k=1}^\infty$. 
\end{remark}

If $\mathcal C_E=\mathcal C_p$, $2<p<\infty$, then we have  $\tilde{R}_1=0$
in  the proof of Theorem \ref{thmb}.

\begin{theorem}\label{2<p}
Suppose that $2<p<\infty$ and $T\in\mathcal B(\mathcal T_{p,\{\xi_i,\eta_i\}_{i=1}^\infty},\mathcal C_p)$ satisfies  $p_{_{[\bigcup_{i=1}^\infty K_i]}}T(\cdot)p_{_{[\bigcup_{i=1}^\infty L_i]}}=T$. Given a sequence $\{\varepsilon_l\}_{l=1}^\infty$ of positive numbers, there exist an increasing sequence $\{i_k\}_{k=1}^\infty$of positive integers, two orthonormal sequences $\{\xi'_k\}_{k=1}^\infty$ with $\xi'_k\in[\xi_i]_{i_{2k+1}\leq i<i_{2k+2}}$ and $\{\eta'_k\}_{k=1}^\infty$ with $\eta'_k\in[\eta_i]_{i_{2k}\leq i<i_{2k+1}}$, and an operator $\tilde{R}\in\mathcal B\big([e_{\xi'_k,\eta'_l}]_{k,l=1}^\infty,\mathcal C_p\big)$ having the form
\[
\tilde{R}(e_{\xi'_k,\eta'_l})=x_{2k+1,2l}+y_{2l,2k+1}\;\;\;\;\;\;\;\;\mbox{for every}\;k,l\geq1,
\]
where
\[
\{x_{2k+1,2l}\}_{k,l=1}^\infty\stackrel{\mbox{{\rm\tiny b.s}}}{\boxplus}\{K_{A_{2k+1}}\}_{k=1}^\infty\otimes\{L_{A_{2l}}\}_{l=1}^\infty\curvearrowleft\;
\]
\[
\{y_{2l,2k+1}\}_{k,l=1}^\infty\stackrel{\mbox{{\rm\tiny b.s}}}{\boxplus}\{K_{A_{2l}}\}_{l=1}^\infty\otimes\{L_{A_{2k+1}}\}_{k=1}^\infty\curvearrowleft\;,
\]
and $\{A_k\}_{k=1}^\infty$ is a sequence of mutually disjoint subsets of $\mathbb N$, such that $\tilde{R}|_{\mathcal T_{p,\{\xi'_k,\eta'_k\}_{k=1}^\infty}}$ is a perturbation of $T|_{\mathcal T_{p,\{\xi'_k,\eta'_k\}_{k=1}^\infty}}$ associated with  the Schauder decomposition $\big\{[e_{\xi'_k,\eta'_l}]_{k=l}^\infty\big\}_{l=1}^\infty$ for $\{\varepsilon_l\}_{l=1}^\infty$.
\end{theorem}

It is well-known that     $\mathcal C_p$  has the follow property: 
  for any sequence $\{z_{i,j}\}_{i,j=1}^\infty$ in $\mathcal C_p$ satisfying 
\[
\{z_{i,j}\}_{i,j=1}^\infty\stackrel{\mbox{\tiny b.s}}{\boxplus}\{F_i\}_{i=1}^\infty\otimes\{G_j\}_{j=1}^\infty\curvearrowleft c,
\]
where $\{F_i\}_{i=1}^\infty$ and $\{G_i\}_{i=1}^\infty$ are sequences of mutually orthogonal closed subspaces of $H$, and $\Vert c\Vert_p=1$,  we have 
$$\{z_{i,j}\}_{i,j=1}^\infty\simeq\{e_{\xi_i,\eta_j}\}_{i,j=1}^\infty.$$ This property enables us to avoid using Theorems \ref{t3} and \ref{t4} by employing  a simpler 
argument:    set 
\[
x_{(k,l)}=x_{_{(k,l),1}},\;\;\;\;y_{(l,k)}=y_{_{1,(l,k)},},\;\;\;\;x_{k,l}=x_{_{(k,1),l}},\;\;\;\;y_{l,k}=y_{_{l,(1,k)}}
\]
(with $x_{_{(k,l),1}}$, $y_{_{1,(l,k)},}$, $x_{_{(k,1),l}}$ and $y_{_{l,(1,k)}}$ given in Theorem \ref{thma}, and note that their supports are all finite-dimensional projections). Then we obtain  operator-perturbation versions of Theorems \ref{1p2} and \ref{2<p}.

\begin{theorem}\label{1p2'}
Suppose that $0<p<2$ and $T\in\mathcal B(\mathcal T_{p,\{\xi_i,\eta_i\}_{i=1}^\infty},\mathcal C_p)$ with $p_{_{[\bigcup_{i=1}^\infty K_i]}}T(\cdot)p_{_{[\bigcup_{i=1}^\infty L_i]}}=T$. Given a positive number $\varepsilon$, there exist two increasing sequences $\{i_k\}_{k=1}^\infty$ and $\{n_k\}_{k=1}^\infty$ of positive integers, and two operators $\tilde{R}_1,\tilde{R}_2\in\mathcal B\big([e_{\xi_{i_{2k+1}},\eta_{i_{2l}}}]_{k,l=1}^\infty,\mathcal C_p\big)$
of the form
\[
\tilde{R}_1(e_{\xi_{i_{2k+1}},\eta_{i_{2l}}})=\left\{
\begin{array}{rcl}
x_{_{(2k+1,2l)}}+y_{_{(2l,2k+1)}}, &  & {1\leq l\leq k<\infty,}\\
0\;\;\;\;\;\;\;\;\;\;\;\;\;, &  & \text{otherwise,}\\
\end{array}
\right.
\]
\[
\tilde{R}_2(e_{\xi_{i_{2k+1}},\eta_{i_{2l}}})=x_{2k+1,2l}+y_{2l,2k+1}\;\;\;\;\;\;\;\;\mbox{for every}\;k,l\geq1,
\]
such that $\big\Vert T|_{\mathcal T_{p,\{\xi_{i_{2k+1}},\eta_{i_{2k}}\}_{k=1}^\infty}}-(R_1+R_2)|_{\mathcal T_{p,\{\xi_{i_{2k+1}},\eta_{i_{2k}}\}_{k=1}^\infty}}\big\Vert\leq\varepsilon$, where

\begin{itemize}
\item [(1)]
$K'_k=[\bigcup_{i=n_{k-1}+1}^{n_k}K_i]$ and $L'_k=[\bigcup_{i=n_{k-1}+1}^{n_k}L_i]$ for every $k\geq1$ (put $n_0=0$),
\item [(2)]
$\{F_{(2k+1,2l)}\}_{1\leq l\leq k<\infty}$ and $\{G_{(2l,2k+1)}\}_{1\leq l\leq k<\infty}$ are two sequences of mutually orthogonal closed subspaces of $H$ such that
\[
F_{(2k+1,2l)}\subset K'_{2k+1}\;\;\;\;{\rm and}\;\;\;\;G_{(2l,2k+1)}\subset L'_{2k+1}\;\;\;\;{\rm for\;every}\;1\leq l\leq k<\infty.
\]
\item [(3)]
\[
\{x_{_{(2k+1,2l)}}\}_{1\leq l\leq k<\infty}\stackrel{\mbox{{\rm\tiny b.s}}}{\boxplus}\{F_{(2k+1,2l)}\}_{1\leq l\leq k<\infty}\otimes L'_1\curvearrowleft\hat{b}_1,
\]
\[
\{y_{_{(2l,2k+1)}}\}_{1\leq l\leq k<\infty}\stackrel{\mbox{{\rm\tiny b.s}}}{\boxplus} K'_1\otimes\{G_{(2l,2k+1)}\}_{1\leq l\leq k<\infty}\curvearrowleft\tilde{b}_1,
\]
and $\left\Vert\left(
\begin{smallmatrix}
0&\tilde{b}_1\\ \hat{b}_1&0
\end{smallmatrix}\right)
\right\Vert_p\leq4^{\max\{\frac{1}{p},1\}-1}\Vert T\Vert$,
\item [(4)] 
\[
\{x_{2k+1,2l}\}_{k,l=1}^\infty\stackrel{\mbox{{\rm\tiny b.s}}}{\boxplus}\{K'_{2k+1}\}_{k=1}^\infty\otimes\{L'_{2l}\}_{l=1}^\infty\curvearrowleft\;
\]
and
\[
\{y_{2l,2k+1}\}_{k,l=1}^\infty\stackrel{\mbox{{\rm\tiny b.s}}}{\boxplus}\{K'_{2l}\}_{l=1}^\infty\otimes\{L'_{2k+1}\}_{k=1}^\infty\curvearrowleft\;.
\]
\end{itemize}
\end{theorem}

\begin{theorem}\label{2<p'}
Suppose that $2<p<\infty$ and $T\in\mathcal B(\mathcal T_{p,\{\xi_i,\eta_i\}_{i=1}^\infty},\mathcal C_p)$ with $p_{_{[\bigcup_{i=1}^\infty K_i]}}T(\cdot)p_{_{[\bigcup_{i=1}^\infty L_i]}}=T$. Given a sequence $\{\varepsilon_l\}_{l=1}^\infty$ of positive numbers, there are two increasing sequences $\{i_k\}_{k=1}^\infty$ and $\{n_k\}_{k=1}^\infty$ of positive integers, and an operator $\tilde{R}\in\mathcal B\big([e_{\xi_{i_{2k+1}},\eta_{i_{2l}}}]_{k,l=1}^\infty,\mathcal C_p\big)$
having the forms
\[
\tilde{R}(e_{\xi'_k,\eta'_l})=x_{2k+1,2l}+y_{2l,2k+1}\;\;\;\;\;\;\;\;\mbox{for every}\;k,l\geq1,
\]
where
\[
\{x_{2k+1,2l}\}_{k,l=1}^\infty\stackrel{\mbox{{\rm\tiny b.s}}}{\boxplus}\{K'_{2k+1}\}_{k=1}^\infty\otimes\{L'_{2l}\}_{l=1}^\infty\curvearrowleft\;
\]
\[
\{y_{2l,2k+1}\}_{k,l=1}^\infty\stackrel{\mbox{{\rm\tiny b.s}}}{\boxplus}\{K'_{2l}\}_{l=1}^\infty\otimes\{L'_{2k+1}\}_{k=1}^\infty\curvearrowleft\;,
\]
and $K'_k=[\bigcup_{i=n_{k-1}+1}^{n_k}K_i]$ and $L'_k=[\bigcup_{i=n_{k-1}+1}^{n_k}L_i]$ for every $k\geq1$ (put $n_0=0$), such that $\big\Vert T|_{\mathcal T_{p,\{\xi_{i_{2k+1}},\eta_{i_{2k}}\}_{k=1}^\infty}}-\tilde{R}|_{\mathcal T_{p,\{\xi_{i_{2k+1}},\eta_{i_{2k}}\}_{k=1}^\infty}}\big\Vert\leq\varepsilon$.
\end{theorem}

\begin{remark}\label{remtptotp}
If $\{K_i\}_{i=1}^\infty$ and $\{L_i\}_{i=1}^\infty$ are chosen to be $K_i=[\xi_i]$, $L_i=[\eta_i]$ and $T\in\mathcal B(\mathcal T_{p,\{\xi_i,\eta_i\}_{i=1}^\infty},\mathcal C_p\big)$ with im$(T)\subset\mathcal T_{p,\{\xi_i,\eta_i\}_{i=1}^\infty}$, then   those ``$y_{(2l,2k+1)}$" and ``$y_{2l,2k+1}$" parts in Theorems \ref{1p2'} and \ref{2<p'} are $0$. In this case,
one advantage to use 
 Theorems \ref{1p2'} and \ref{2<p'} (rather than Theorems \ref{1p2} and \ref{2<p}) is that they guarantee that the  operators $R_1$, $R_2$ and $R$   have their ranges contained in $\mathcal T_{p,\{\xi_i,\eta_i\}_{i=1}^\infty}$, in contrast to Theorems \ref{1p2} and \ref{2<p},
 which are useful tools in the study of algebra $\mathcal B(\mathcal T_p)$.
\end{remark}

\begin{theorem}\label{pneq2}
Suppose that $0<p\neq2<\infty$ and $T\in\mathcal B(\mathcal T_{p,\{\xi_i,\eta_i\}_{i=1}^\infty},\mathcal C_p)$ with $p_{_{[\bigcup_{i=1}^\infty K_i]}}T(\cdot)p_{_{[\bigcup_{i=1}^\infty L_i]}}=T$. Given a positive number $\varepsilon$, there are two increasing sequences $\{i_k\}_{k=1}^\infty$ and $\{n_k\}_{k=1}^\infty$ of positive integers, a sequence
\[
\{a_{k,l}\}_{k,l=1}^\infty\stackrel{\mbox{\tiny b.s}}{\boxplus}\big\{[\xi_i]_{i_{2k+1}\leq i<i_{2k+2}}\big\}_{k=1}^\infty\otimes \big\{[\eta_i]_{i_{2l}\leq i<i_{2l+1}}\big\}_{l=1}^\infty\curvearrowleft a
\]
with $\Vert a\Vert_p=1$, and an operator $\tilde{R}\in\mathcal B\big([a_{k,l}]_{k,l=1}^\infty,\mathcal C_p\big)$ having the form
\[
\tilde{R}(a_{k,l})=x_{2k+1,2l}+y_{2l,2k+1}\;\;\;\;\;\;\;\;\mbox{for every}\;k,l\geq1,
\]
where
\[
\{x_{2k+1,2l}\}_{k,l=1}^\infty\stackrel{\mbox{{\rm\tiny b.s}}}{\boxplus}\{K'_{2k+1}\}_{k=1}^\infty\otimes\{L'_{2l}\}_{l=1}^\infty\curvearrowleft\;
\]
\[
\{y_{2l,2k+1}\}_{k,l=1}^\infty\stackrel{\mbox{{\rm\tiny b.s}}}{\boxplus}\{K'_{2l}\}_{l=1}^\infty\otimes\{L'_{2k+1}\}_{k=1}^\infty\curvearrowleft\;,
\]
and $K'_k=[\bigcup_{i=n_{k-1}+1}^{n_k}K_i]$ and $L'_k=[\bigcup_{i=n_{k-1}+1}^{n_k}L_i]$ for every $k\geq1$ (put $n_0=0$), such that $\big\Vert T|_{[a_{k,l}]_{1\leq l\leq k<\infty}}-\tilde{R}|_{[a_{k,l}]_{1\leq l\leq k<\infty}}\big\Vert\leq\varepsilon$.
\end{theorem}

\begin{proof}
If $2<p<\infty$, then the assertions follow from Theorem \ref{2<p'}.

Assume that $0<p<2$. Using Theorem \ref{1p2'}, we obtain operators  $\tilde{R}_1$ and $\tilde{R}_2$ satisfying those statements in Theorem \ref{1p2'} and $\big\Vert T|_{[e_{\xi_{i_{2k+1}},\eta_{i_{2l}}}]_{1\leq l\leq k<\infty}}-(R_1+R_2)|_{[e_{\xi_{i_{2k+1}},\eta_{i_{2l}}}]_{1\leq l\leq k<\infty}}\big\Vert\leq\frac{\varepsilon}{2^{\max\{\frac{1}{p},1\}}}$. We choose a positive integer $N$ such that $N^{1/2-1/p}<\frac{\varepsilon}{2^{\max\{\frac{1}{p},1\}+1}\big(1+8^{\max\{\frac{1}{p},1\}-1}\Vert T\Vert\big)}$. Put
\begin{equation}
a_{k,l}=\frac{1}{N^{1/p}}\sum_{m=1}^Ne_{\xi'_{2(Nk+m)+1},\eta'_{2(Nl+m)}}\;\;\;\;\;\;\;\;{\rm for\;every}\;k,l\geq1.
\end{equation}\label{akl}
By Theorem \ref{3a'}, we know that 
\[
\{a_{k,l}\}_{k,l=1}^\infty\stackrel{\mbox{{\rm\tiny b.s}}}{\boxplus}\big\{[\xi'_{2(Nk+m)+1}]_{m=1}^N\big\}_{k=1}^\infty\otimes\big\{[\eta'_{2(Nl+m)}]_{m=1}^N\big\}_{l=1}^\infty\curvearrowleft\frac{1}{N^{1/p}}\sum_{m=1}^N(\cdot|e_m)e_m.
\]
For any finitely non-zero sequence $\{\alpha_{k,l}\}_{k,l=1}^\infty$ of scalars, we have

\begin{eqnarray*}
&&\bigg\Vert \tilde{R}_1\bigg(\sum_{k,l=1}^\infty\alpha_{k,l}a_{k,l}\bigg)\bigg\Vert_p\\
&=&\frac{1}{N^{1/p}}\bigg\Vert \sum_{k,l=1}^\infty\alpha_{k,l}\sum_{m=1}^N\tilde{R}_1\bigg(e_{\xi'_{2(Nk+m)+1},\eta'_{2(Nl+m)}}\bigg)\bigg\Vert_p\\
&=&\frac{1}{N^{1/p}}\bigg\Vert \sum_{k,l=1}^\infty\sum_{m=1}^N\alpha_{k,l}\big(x_{_{(2(Nk+m)+1,2(Nl+m))}}+y_{_{(2(Nl+m),2(Nk+m)+1)}}\big)\bigg\Vert_p\\
&\leq&\frac{2^{\max\{\frac{1}{p},1\}-1}}{N^{1/p}}\Bigg(\bigg\Vert \sum_{1\leq l\leq k<\infty}\sum_{m=1}^N\alpha_{k,l}x_{_{(2(Nk+m)+1,2(Nl+m))}}\bigg\Vert_p\\
&&\;\;\;\;\;\;\;\;\;\;\;\;+\bigg\Vert \sum_{1\leq l\leq k<\infty}\sum_{m=1}^N\alpha_{k,l}y_{_{(2(Nl+m),2(Nk+m)+1)}}\bigg\Vert_p\Bigg)\\
&\stackrel{{\rm Th}\;\ref{1p2'}\;(3)}{\leq}&\frac{8^{\max\{\frac{1}{p},1\}-1}}{N^{1/p}}\Bigg(\Vert T\Vert\bigg( N\sum_{1\leq l\leq k<\infty}\vert\alpha_{k,l}\vert^2\bigg)^{1/2}+\Vert T\Vert\bigg( N\sum_{1\leq l\leq k<\infty}\vert\alpha_{k,l}\vert^2\bigg)^{1/2}\Bigg)\\
&=&2\cdot8^{\max\{\frac{1}{p},1\}-1}N^{1/2-1/p}\Vert T\Vert\bigg(\sum_{1\leq l\leq k<\infty}\vert\alpha_{k,l}\vert^2\bigg)^{1/2}\\
&\stackrel{(\Vert\cdot\Vert_2\leq\Vert\cdot\Vert_p)}{\leq}&\frac{\varepsilon}{2^{\max\{\frac{1}{p},1\}}}\bigg\Vert\sum_{k,l=1}^\infty\alpha_{k,l}e_{\xi_k,\eta_l}\bigg\Vert_p=\frac{\varepsilon}{2^{\max\{\frac{1}{p},1\}}}\bigg\Vert\sum_{k,l=1}^\infty\alpha_{k,l}a_{k,l}\bigg\Vert_p.
\end{eqnarray*}
Let $\tilde{R}=\tilde{R}_2|_{[a_{k,l}]_{k,l=1}^\infty}$. We have
\begin{align*}
&\big\Vert T|_{[a_{k,l}]_{1\leq l\leq k<\infty}}-\tilde{R}|_{[a_{k,l}]_{1\leq l\leq k<\infty}}\big\Vert\\
\leq&2^{\max\{\frac{1}{p},1\}-1}\Big(\big\Vert\tilde{R}_1|_{[a_{k,l}]_{1\leq l\leq k<\infty}}\big\Vert+\big\Vert T|_{[a_{k,l}]_{1\leq l\leq k<\infty}}-(\tilde{R}_1+\tilde{R}_2)|_{[a_{k,l}]_{1\leq l\leq k<\infty}}\big\Vert\Big)\leq\varepsilon,
\end{align*}
and by Theorem \ref{3a'}, $\tilde{R}$ is the desired operator.
\end{proof}

It is known that $T^{p,\{\xi_k,\eta_k\}_{k=1}^\infty}$ is unbounded for $0<p\leq1$ (see Example \ref{exmple}).
Now,
arguing in a way similar to the proofs of Theorem \ref{1p2} and Theorem \ref{thmc}, 
we obtain a variant of Theorem \ref{thmc} for $\cC_p$ when $0<p\leq1$.

\begin{theorem}\label{p=1}
Let $0<p\leq1$. Suppose that $T\in\mathcal B\big([e_{\xi_i,\eta_j}]_{i,j=1}^\infty,\mathcal C_p\big)$ with $p_{_{[\bigcup_{i=1}^\infty K_i]}}T(\cdot)p_{_{[\bigcup_{i=1}^\infty L_i]}}=T$.
For a given sequence $\{\varepsilon_l\}_{l=1}^\infty$ of positive numbers,  there exist an increasing sequence $\{i_k\}_{k=1}^\infty$of positive integers, and two operators $R_1,R_2\in\mathcal B\big([e_{\xi_{i_{2k}},\eta_{i_{2l}}}]_{k,l=1}^\infty,\mathcal C_p\big)$ of the form
\[
R_1(e_{\xi_{i_{2k+1}},\eta_{i_{2l}}})=x_{_{(k,l)}}+y_{_{(l,k)}}\;\;\;\;\;\;\;\;\mbox{for\;every}\;k,l\geq1,
\]
\[
R_2(e_{\xi_{i_{2k+1}},\eta_{i_{2l}}})=x_{2k+1,2l}+y_{2l,2k+1}\;\;\;\;\;\;\;\;\mbox{for every}\;k,l\geq1.
\]
such that $R_1+R_2$ is a perturbation of $T|_{[e_{\xi_{i_{2k+1}},\eta_{i_{2l}}}]_{k,l=1}^\infty}$ associated with the Schauder decomposition $\big\{[e_{\xi_{i_{2k+1}},\eta_{i_{2l}}}]_{\min{k,l}=s}^\infty\big\}_{s=1}^\infty$ for $\{\varepsilon_s\}_{s=1}^\infty$, where
\begin{itemize}
\item [(1)]
there are  two sequences $\{F_{(k,l)}\}_{1\leq l\leq k<\infty}$ and $\{F_{(k,l)}\}_{1\leq k<l<\infty}$ of mutually orthogonal subspaces of $K$, and  two sequences $\{G_{(l,k)}\}_{1\leq l\leq k<\infty}$ and $\{G_{(l,k)}\}_{1\leq k<l<\infty}$ of mutually orthogonal subspaces of $L$ such that
\[
\{x_{_{(k,l)}}\}_{1\leq l\leq k<\infty}\stackrel{\mbox{{\rm\tiny b.s}}}{\boxplus}\{F_{(k,l)}\}_{1\leq l\leq k<\infty}\otimes L\curvearrowleft\hat{b},
\]
\[
\{y_{_{(l,k)}}\}_{1\leq l\leq k<\infty}\stackrel{\mbox{{\rm\tiny b.s}}}{\boxplus}K\otimes\{G_{(l,k)}\}_{1\leq l\leq k<\infty}\curvearrowleft\tilde{b},
\]
and
\[
\{x_{_{(k,l)}}\}_{1\leq k<l<\infty}\stackrel{\mbox{{\rm\tiny b.s}}}{\boxplus}\{F_{(k,l)}\}_{1\leq k<l<\infty}\otimes L\curvearrowleft\tilde{b'},
\]
\[
\{y_{_{(l,k)}}\}_{1\leq k<l<\infty}\stackrel{\mbox{{\rm\tiny b.s}}}{\boxplus}K\otimes\{G_{(l,k)}\}_{1\leq k<l<\infty}\curvearrowleft\hat{b'},
\]
and $\left\Vert\left(
\begin{smallmatrix}
0&\tilde{b}\\ \hat{b}&0
\end{smallmatrix}\right)
\right\Vert_p,\left\Vert\left(
\begin{smallmatrix}
0&\hat{b'}\\ \tilde{b'}&0
\end{smallmatrix}\right)
\right\Vert_p\leq4^{\frac{1}{p}-1}\Vert T\Vert$,
\item [(4)]
there is a sequence $\{A_k\}_{k=1}^\infty$ of mutually disjoint subsets of\;$\mathbb N$ such that
\[
\{x_{2k+1,2l}\}_{k,l=1}^\infty\stackrel{\mbox{{\rm\tiny b.s}}}{\boxplus}\{K_{A_{2k+1}}\}_{k=1}^\infty\otimes\{L_{A_{2l}}\}_{l=1}^\infty\curvearrowleft
\]
and
\[
\{y_{2l,2k+1}\}_{k,l=1}^\infty\stackrel{\mbox{{\rm\tiny b.s}}}{\boxplus}\{K_{A_{2l}}\}_{l=1}^\infty\otimes\{L_{A_{2k+1}}\}_{k=1}^\infty\curvearrowleft\;.
\]
\end{itemize}
\end{theorem}

\begin{remark}
    Compared to Theorem \ref{thmc},  Theorem \ref{p=1} allows us to  select a subsequence $\{\xi_{i_{2k+1}}\}_{k=1}^\infty$ of $\{\xi_i\}_{i=1}^\infty$ and a subsequence $\{\eta_{i_{2k}}\}_{k=1}^\infty$ of $\{\eta_i\}_{i=1}^\infty$ such that  $R_1+R_2$ is a perturbation of $T|_{[e_{\xi_{i_{2k+1}},\eta_{i_{2l}}}]_{k,l=1}^\infty}$, whereas Theorem \ref{thmc} only provides   new sequences.
\end{remark}

Arguing mutatis mutandis as in the proof of Theorem \ref{pneq2}, we obtain the following theorem.

\begin{theorem}\label{p=1'}
Let $0<p\leq1$. 
Suppose that $T\in\mathcal B\big([e_{\xi_i,\eta_j}]_{i,j=1}^\infty,\mathcal C_p\big)$ with $p_{_{[\bigcup_{i=1}^\infty K_i]}}T(\cdot)p_{_{[\bigcup_{i=1}^\infty L_i]}}=T$. Given a positive number $\varepsilon$. Then there is an increasing sequence $\{i'_k\}_{k=1}^\infty$ of positive integers,  a sequence
\[
\{a_{k,l}\}_{k,l=1}^\infty\stackrel{\mbox{{\rm\tiny b.s}}}{\boxplus}\big\{[\xi_i]_{i'_{2k+1}\leq i<i'_{2k+2}}\big\}_{k=1}^\infty\otimes \big\{[\eta_i]_{i'_{2l}\leq i<i'_{2l+1}}\big\}_{l=1}^\infty\curvearrowleft a
\]
with $\Vert a\Vert_p=1$, and an operator $R\in\mathcal B\big([a_{k,l}]_{k,l=1}^\infty,\mathcal C_p\big)$ having the form
\[
R(a_{k,l})=x_{2k+1,2l}+y_{2l,2k+1}\;\;\;\;\;\;\;\;\mbox{for every}\;k,l\geq1,
\]
where
\[
\{x_{2k+1,2l}\}_{k,l=1}^\infty\stackrel{\mbox{{\rm\tiny b.s}}}{\boxplus}\{K_{A_{2k+1}}\}_{k=1}^\infty\otimes\{L_{A_{2l}}\}_{l=1}^\infty\curvearrowleft\;,
\]
\[
\{y_{2l,2k+1}\}_{k,l=1}^\infty\stackrel{\mbox{{\rm\tiny b.s}}}{\boxplus}\{K_{A_{2l}}\}_{l=1}^\infty\otimes\{L_{A_{2k+1}}\}_{k=1}^\infty\curvearrowleft\;,
\]
and $\{A_k\}_{k=1}^\infty$ is a sequence of mutually disjoint subsets of\;$\mathbb N$, such that $\big\Vert T|_{[a_{k,l}]_{k,l=1}^\infty}-R\big\Vert\leq\varepsilon$.
\end{theorem}

\subsubsection{$\mathcal{C}_p$-strictly singular operators}

\begin{theorem}\label{c1s}
Let $0<p\leq1$, and $T\in\mathcal B(\mathcal C_p)$. The following are equivalent:
\begin{itemize}
\item [(1)]
$T$ is $\mathcal C_p$-strictly singular.
\item [(2)]
For every subspace $X$ of $\mathcal C_p$ with $X\approx\mathcal C_p$ and for any $\varepsilon>0$, there exists a subspace $Y\subset X$ such that $Y\approx \mathcal C_p$ and $\Vert T|_Y\Vert<\varepsilon$.
\item [(3)]
$T$ is $\mathcal T_p$-strictly singular.
\end{itemize}
\end{theorem}

\begin{proof}
For convenience, we only prove the case $p=1$. The case for $p<1$ follows from a similar argument.

(2)$\Longrightarrow$(1): This follows directly from the definition of a $\mathcal C_1$-strictly singular operator.

(1)$\Longrightarrow$(2): Suppose that $T$ is a strictly singular operator on $\mathcal C_1$. Consider any subspace $X\subset\mathcal C_1$ with $X\approx\mathcal C_1$ and  any $\varepsilon>0$. Let $S: \mathcal C_1\rightarrow X$ be a linear isomorphism.
By Theorem \ref{p=1'}, there is a subspace $Z$ of $\mathcal C_1$ with $Z\cong\mathcal C_1$, and an operator $R\in\mathcal B(Z,\mathcal C_1)$ such that there exists a constant  $\beta\ge 0$ such that
\[
\Vert R(z)\Vert_1=\beta\Vert z\Vert_1\;\;\;\;\;\;{\rm for\;all\;}z\in Z,
\]
and $\Vert TS|_Z-R\Vert<\varepsilon/(2\Vert S^{-1}\Vert)$. Hence, for any $z\in Z$, we have 
\[
\Vert TS(z)\Vert _1\geq\Vert R(z)\Vert_1 -\Vert(TS-R)(z)\Vert_1 \geq\big(\beta-\varepsilon/(2\Vert S^{-1}\Vert)\big)\Vert z\Vert_1.
\]
Since $T$ is $\mathcal C_1$-strictly singular, it follows that  $\beta\leq\varepsilon/(2\Vert S^{-1}\Vert)$. Setting $Y=S(Z)$, we have 
\[
\Vert T|_Y\Vert\leq\Vert TS|_Z\Vert\cdot\Vert S^{-1}\Vert\leq\big(\Vert TS|_Z-R\Vert+\Vert R\Vert\big)\Vert S^{-1}\Vert<\big(\varepsilon/(2\Vert S^{-1}\Vert)+\beta\big)\Vert S^{-1}\Vert<\varepsilon.
\]

The equivalence   (1)$\iff$(3) follows from  Theorem \ref{cs=ts}.
\end{proof}

\begin{remark}\label{c1srema}
In Theorem \ref{c1s} (2), we do not claim that the restriction  $T|_Y$ of $T$  is compact, even if
  $T$ is  $\mathcal C_1$-strictly singular  on $\mathcal C_1$ such that $T|_{Z}$ is an  isomorphic embedding for every subspace $Z$ of $\mathcal C_1$ with $Z\approx\ell_2$ (note that such a $Y$ necessarily contains a  subspace  isomorphic to $\ell_2$). For example, let 
$\{\xi_i\}_{i=1}^\infty$ be an orthonormal basis of $H$, and let $T\in\mathcal B(\mathcal C_1)$ be such that 
 \[
Te_{\xi_{i},\xi_{j}}=e_{\xi_{2^i(2j+1)},\xi_{1}}\;\;\;\;\;\;{\rm for\;every\;}1\leq i,j<\infty.
 \]
Note that $[e_{\xi_{2^i(2j+1)},\xi_{1}}]_{i,j\ge 1}$ is isomorphic to $\ell_2$ and $T$ is a $\cC_1$-strictly singular. For any subspace $Z$ of $\mathcal C_1$ with $Z\approx\ell_2$, by  \cite[Proposition 2.2 (i)]{Arazy2} (or Lemma \ref{l2cp}), there exists a positive integer $N$ such that $(I-P_{[\xi_i]_{i=N}^\infty,[\xi_i]_{i=N}^\infty})|_Z$ is an isomorphic embedding. Note that $T$ is an  isomorphic embedding  on im$(I-P_{[\xi_i]_{i=N}^\infty,[\xi_i]_{i=N}^\infty})=[e_{\xi_i,\xi_j}]_{\min\{i,j\}\leq N}$, and
\begin{align*}
&\big\Vert T(I-P_{[\xi_i]_{i=N}^\infty,[\xi_i]_{i=N}^\infty})(z)\big\Vert_1\\
\leq&\left(\big\Vert TP_{[\xi_i]_{i=N}^\infty,[\xi_i]_{i=N}^\infty}(z)\big\Vert_1^2+\big\Vert T(I-P_{[\xi_i]_{i=N}^\infty,[\xi_i]_{i=N}^\infty})(z)\big\Vert_1^2\right)^{1/2}=\left\Vert Tz\right\Vert_1
\end{align*}
for every $z\in Z$. Therefore, $T|_Z$ is an isomorphic embedding (therefore, $T|_Y$ is  not compact for any $Y\subset X$ with $Y\approx \cC_1$). 
\end{remark}

\begin{theorem}\label{tps}
Let $0<p<\infty$. Then $T\in\mathcal B(\mathcal T_p)$ is $\mathcal T_p$-strictly singular if and only if for every subspace $X$ of $\mathcal T_p$ with $X\approx\mathcal T_p$ and for any $\varepsilon>0$, there exists a subspace $Y\subset X$ such that $Y\approx \mathcal T_p$ and $\Vert T|_Y\Vert<\varepsilon$.
\end{theorem}

\begin{proof}
If $p=2$, then $\mathcal T_2\cong\ell_2$. The statement follows from  the fact that the normalized block basic sequence in $\ell_2$ is isometrically isomorphic to the unit basis of $\ell_2$.

If $p\neq2$, then the equivalence is proved by   arguing mutatis mutandis as in the proof of  the equivalence  $(1)\iff(2)$ in Theorem~\ref{c1s}, using  replacing Theorem~\ref{p=1'} in place of  Theorem~\ref{pneq2}.
\end{proof}

The fact that   $\mathcal C_p\approx\mathcal T_p$ for any $1<p<\infty$ together with Theorem \ref{tps} and Theorem \ref{c1s} yields   the
following corollary.
\begin{corollary}\label{cps0<p<}
Let $0<p<\infty$. Then $T\in\mathcal B(\mathcal C_p)$ is $\mathcal C_p$-strictly singular if and only if for every subspace $X$ of $\mathcal C_p$ with $X\approx\mathcal C_p$ and for any $\varepsilon>0$, there exists a subspace $Y\subset X$ such that $Y\approx \mathcal C_p$ and $\Vert T|_Y\Vert<\varepsilon$.    
\end{corollary}

\begin{remark}\label{add-opincp}
Let $B$ be a quasi-Banach space. An immediate consequence of Theorems \ref{c1s} and \ref{tps} is that the set of  all $\mathcal C_p$-strictly singular operators in $\cB(B,\mathcal C_p)$, $0<p\leq1$, is closed under addition.
In particular,   all $\mathcal C_p$-strictly singular operators on $\mathcal C_p$ are  closed under addition and therefore,  form an ideal of $\mathcal B(\mathcal C_p)$;  and the same holds for $\mathcal T_p$, $0< p<\infty$. More general results are given in  Section \ref{Sec:largest proper ideal}.
\end{remark}

\subsection{$\mathcal{C}_1$ is $1^+$-complementably homogeneous}\label{C1homogeneous}

The following property is a widely studied in the Banach spaces theory, here we adopt the terminology from  \cite{Chen Johnson Bentuo Zheng}.

\begin{definition}\label{ch}\cite{Chen Johnson Bentuo Zheng}
A (quasi-)Banach space $X$ is said to be \emph{complementably homogeneous} if every subspace of $X$ isomorphic to $X$ must contain a smaller subspace  isomorphic to $X$ and is complemented in $X$.
\end{definition}

\begin{definition}\label{uch}
A (quasi-)Banach space $X$ is said to be \emph{$\lambda$-complementably homogeneous}, where $1\leq\lambda<\infty$, if every subspace of $X$ isomorphic to $X$ must contain a smaller subspace that is $\lambda$-isomorphic to $X$ and is $\lambda$-complemented in $X$.
\end{definition}

\begin{definition}\label{+ch}
A (quasi-)Banach space $X$ is said to be \emph{$\lambda^+$-complementably homogeneous}, where $1\leq\lambda<\infty$, if $X$ is $k$-complementably homogeneous for every constant $k>\lambda$.
\end{definition}

Arazy \cite{Arazy2} proved that $\mathcal T_p$, $1\leq p<\infty$, is $2^+$-complementably homogeneous, and it is shown in  \cite[Theorem 5.4]{Arazy4} (see also \cite{Arazy83b}) that $\mathcal C_1$ is complementably homogeneous. Below, we show that  they are $1^+$-complementably homogeneous.

\begin{theorem}\label{c1ch}
$\mathcal C_1$ is $1^+$-complementably homogeneous.
\end{theorem}
\begin{proof}
Let $\mathcal C_1=[e_{\xi_i,\eta_j}]_{i,j=1}^\infty$, $K_i=[\xi_i]$ and $L_i=[\eta_i]$. Let $X$ be subspace of $\mathcal C_1$ with $X\approx\mathcal C_1$, and hence there is an isomorphic embedding $T\in\mathcal B([e_{\xi_i,\eta_j}]_{i,j=1}^\infty)$ with im$(T)=X$. There is a positive number $c$ such that $\Vert T(y)\Vert_1\geq c\Vert y\Vert_1$ for any $y\in [e_{\xi_i,\eta_j}]_{i,j=1}^\infty$. 

Let $\delta$ $(<c)$ be a  positive number which is small enough. By Theorem~\ref{p=1'}, there exist a sequence $\{a_{k,l}\}_{k,l=1}^\infty$,  which is generated by an operator of norm $1$, and an operator $R\in\mathcal B\big([a_{k,l}]_{k,l=1}^\infty,\mathcal C_1\big)$ satisfying:
\begin{itemize}
\item [(1)]
There is a sequence $\{A_k\}_{k=1}^\infty$ of mutually disjoint subsets of\;$\mathbb N$, and a sequence
\[
\{x_{2k+1,2l}\}_{k,l=1}^\infty\stackrel{\mbox{\tiny b.s}}{\boxplus}\{K_{A_{2k+1}}\}_{k=1}^\infty\otimes\{L_{A_{2l}}\}_{l=1}^\infty\curvearrowleft,
\]
\[
\{y_{2l,2k+1}\}_{k,l=1}^\infty\stackrel{\mbox{{\rm\tiny b.s}}}{\boxplus}\{K_{A_{2l}}\}_{l=1}^\infty\otimes\{L_{A_{2k+1}}\}_{k=1}^\infty\curvearrowleft\;,
\]
such that
\[
R(a_{k,l})=x_{2k+1,2l}+y_{2l,2k+1}\;\;\;\;{\rm for\;every}\;k,l\geq1.
\]
\item [(2)]
$\big\Vert T|_{[a_{k,l}]_{k,l=1}^\infty}-R\big\Vert\leq\delta$.
\end{itemize}
Obviously, $[x_{2k+1,2l}+y_{2l,2k+1}]_{k,l=1}^\infty\cong\mathcal C_1$. By Theorem \ref{cpc}, $[x_{2k+1,2l}+y_{2l,2k+1}]_{k,l=1}^\infty$ is $1$-complemented in $[e_{\xi_i,\eta_j}]_{i,j=1}^\infty$.
On the other hand, there exists  a non-negative number $\beta$ such that
\[
\Vert R(y)\Vert_1=\beta\Vert y\Vert_1\;\;\;\;\;\;\;\;{\rm for\;every\;}y\in[a_{k,l}]_{k,l=1}^\infty.
\]
Then,  $\beta\Vert y\Vert_1\geq\Vert T(y)\Vert_1-\Vert (R-T)(y)\Vert_1\geq(c-\delta)\Vert y\Vert_1$. Hence,  $\beta\geq c-\delta$. Note that $TR^{-1}$ is an isomorphic embedding and
\begin{align*}
\left\Vert TR^{-1}-\iota_{[x_{2k+1,2l}+y_{2l,2k+1}]_{k,l=1}^\infty\hookrightarrow [e_{\xi_i,\eta_j}]_{i,j=1}^\infty}\right\Vert&\leq\Vert T|_{[a_{k,l}]_{k,l=1}^\infty}-R\Vert\cdot\Vert R^{-1}\Vert\\
&\leq\delta(c-\delta)^{-1}<1.
\end{align*}
Therefore, $T([a_{k,l}]_{k,l=1}^\infty)$ is $c(c-2\delta)^{-1}$-isomorphic to $[x_{2k+1,2l}+y_{2l,2k+1}]_{k,l=1}^\infty\cong\mathcal C_1$.
Since $[x_{2k+1,2l}+y_{2l,2k+1}]_{k,l=1}^\infty$ is $1$-complement in $[e_{\xi_i,\eta_j}]_{i,j=1}^\infty$, 
it follows from  Lemma~\ref{comp} that $T([a_{k,l}]_{k,l=1}^\infty)$ is $c(c-2\delta)^{-1}$-complemented. 
Since  $\delta$ can be chosen arbitrarily, it follows  that $\mathcal C_1$ is $1^+$-complementably homogeneous.
\end{proof}

The following result improves a result by Arazy 
 \cite{Arazy2}, who showed that 
 lower triangular part $\mathcal{T}_p$  of  the Schatten $p$-class is $2^+$-complementably homogeneous.

\begin{theorem}\label{tpch}
Let $1\leq p<\infty$. Then, $\mathcal T_p$ is $1^+$-complementably homogeneous.
\end{theorem}
\begin{proof}
If $p=2$, then $\mathcal T_2\cong\ell_2$, and the proof for this case is trivial. 

Assume that  $p\neq2$. 
Let $\{\xi_i\}_{i=1}^\infty$ and $\{\eta_i\}_{i=1}^\infty$ be two orthonormal bases of $H$, $K_i=[\xi_i]$ and $L_i=[\eta_i]$, and $\mathcal T_p=[e_{\xi_i,\eta_j}]_{1\leq j\leq i<\infty}$. Let $X$ be a subspace of $\mathcal T_p$ with $X\approx\mathcal T_p$. 
In particular, there is an isomorphic embedding $T\in\mathcal B([e_{\xi_i,\eta_j}]_{1\leq j\leq i<\infty},\mathcal C_p)$ with im$(T)=X$. There is a positive number $c$ such that $\Vert T(y)\Vert_p\geq c\Vert y\Vert_p$ for any $y\in [e_{\xi_i,\eta_j}]_{i,j=1}^\infty$.

Let $\delta$ be a  positive number chosen sufficiently small. By Remark \ref{remtptotp} and Theorem \ref{pneq2}, there exist a sequence $\{a_{k,l}\}_{k,l=1}^\infty$ which is generated by an operator of norm $1$, and an operator $\tilde{R}\in\mathcal B\big([a_{k,l}]_{k,l=1}^\infty,\mathcal C_p\big)$ satisfying:
\begin{itemize}
\item [(1)]
There exists an increasing sequence $\{n_k\}_{k=1}^\infty$ of positive integers, and a sequence
\begin{equation}\label{tpcbl}
\{x_{2k+1,2l}\}_{k,l=1}^\infty\stackrel{\mbox{{\rm\tiny b.s}}}{\boxplus}\{K'_{2k+1}\}_{k=1}^\infty\otimes\{L'_{2l}\}_{l=1}^\infty\curvearrowleft\;
\end{equation}
such that
\[
\tilde{R}(a_{k,l})=x_{2k+1,2l}+y_{2l,2k+1}\;\;\;\;\;\;\;\;{\rm for\;every}\;k,l\geq1,
\]
where $K'_k=[\bigcup_{i=n_{k-1}+1}^{n_k}K_i]$ and $L'_k=[\bigcup_{i=n_{k-1}+1}^{n_k}L_i]$ for every $k\geq1$ (with $n_0=0$),
\item [(2)]
$\big\Vert T|_{[a_{k,l}]_{1\leq l\leq k<\infty}}-\tilde{R}|_{[a_{k,l}]_{1\leq l\leq k<\infty}}\big\Vert\leq\delta$.
\end{itemize}
Note that  $[x_{2k+1,2l}]_{1\leq l\leq k<\infty}\cong\mathcal T_p$, and $[x_{2k+1,2l}]_{1\leq l\leq k<\infty}\subset[e_{\xi_i,\eta_j}]_{1\leq j\leq i<\infty}$. By Theorem \ref{cpc}, there exists  a contractive projection $P$ from $[e_{\xi_i,\eta_j}]_{i,j=1}^\infty$ onto $[x_{2k+1,2l}]_{k,l=1}^\infty$ given by \eqref{cpproj}. By \eqref{tpcbl} and the construction of $P$ in Theorem \ref{cpc},  $P|_{[e_{\xi_i,\eta_j}]_{1\leq j\leq i<\infty}}$ is a contractive projection from $[e_{\xi_i,\eta_j}]_{1\leq j\leq i<\infty}$ onto $[x_{2k+1,2l}]_{1\leq l\leq k<\infty}$. The rest of the proof is the same as that  of Theorem \ref{c1ch}. 
\end{proof}

\section{The largest proper ideal in $\mathcal B(\mathcal C_E)$}\label{Sec:largest proper ideal}

In this section, we always assume that $E$ is a separable Banach symmetric sequence space, and we show that  $\mathcal B(\mathcal C_E)$ has a (unique) largest proper ideal. Note that most of the results in this section  can be extended to quasi-Banach spaces. 
However, for the purpose of proving the main result in  Section \ref{section:C1}, we only consider   Banach spaces in this section.

Let $X$ be a Banach space. Recall that an operator $S\in\mathcal B(X)$ is said to \emph{factor through $T\in\mathcal B(X)$} provided there exist $A,B\in\mathcal B(X)$ such that $S=ATB$. The following useful notation  is given  by Dosev and Johnson \cite{Dosev Johnson},
\begin{equation}\label{mx}
\mathcal{M}_X=\{T\in\mathcal B(X):I_X~{\rm does~not ~factor~through}~T\}.
\end{equation}
Obviously, $\mathcal{M}_X$ is close under left and right multiplication by operators in $\mathcal B(X)$. Therefore,
to show that 
$\mathcal{M}_X$ is an ideal, it suffices to show that   $\mathcal{M}_X$ is closed under addition. Moreover, if $\mathcal{M}_X$ is an ideal, 
then it is automatically the largest proper ideal (hence closed\footnote{Since d$(I_X,\mathcal M_X)\geq1$, it follows that $\overline{\mathcal M_X}\neq \mathcal B(X)$, and thus $\mathcal M_{X}=\overline{\mathcal M_X}$.}) in $\mathcal B(X)$.

The following result is certainly known to experts, which   underlies the ideas used in several papers (see e.g. \cite{Chen Johnson Bentuo Zheng,Dosev Johnson Schechtman,Bentuo Zheng}). For completeness, we include a short proof below.   

\begin{lemma}\label{mxeq}
Let $X$ be a Banach space and $T$ be a bounded operator on $X$. The followings are equivalent:
\begin{itemize}
\item [(1)]
$T\notin\mathcal M_X$.
\item [(2)]
There exists a subspace $Y\subset X$ with $Y\approx X$ such that $T|_Y$ is an isomorphic embedding, and $TY$ is complemented in $X$.
\item [(3)]
There exists a complemented subspace $Y\subset X$ with $Y\approx X$ such that $T|_Y$ is an isomorphic embedding, and $TY$ is complemented in $X$.
\end{itemize}
\end{lemma}
\begin{proof}
(1)$\Longrightarrow$(2): If $T\notin\mathcal M_X$, then there exist two operators $A,B\in\mathcal B(X)$ such that $ATB=I_X$. It follows that  $A$ is surjective, $B$ and $T|_{{\rm im}(B)}$ are both isomorphic embeddings, and $TBA$ is projection operator. The space    $Y:={\rm im}(B)$ is a subspace of $X$ with  desired properties.

(2)$\Longrightarrow$(1): By assumption, there is an isomorphic embedding operator $B\in\mathcal B(X)$ with ${\rm im}(B)=Y$, and a projection $P$ from $X$ onto $TY$. Hence, we may denote an operator $A$ on $X$ by $A(x)=(TB)^{-1}P(x)$. It follows that $ATB(x)= (TB)^{-1}P TB(x)= (TB)^{-1}  TB(x) =x $, i.e., $ATB=I_X$. Hence, $T\not\in \cM_X$. 

(3)$\Longrightarrow$(2): This is trivial. 

(2)$\Longrightarrow$(3): Let  $P$ be the projection from $X$ onto $TY$. Then  $(T|_Y)^{-1}PT$ is 
 a    projection from $X$ onto $Y$.
\end{proof}

Theorem \ref{c1s} characterizes the operators  in $\cB (\cC_1)$ which are $\mathcal  C_1$-strictly singular. 
Below, for more general Banach spaces, 
we give a  characterization in terms of the set  given in \eqref{mx} above (however, it does not  directly yield the norm estimate in Theorem \ref{c1s} directly). 

\begin{corollary}\label{chmx}
Let $X$ be a complementably homogeneous Banach space. Then, the set $\mathcal{M}_X$ coincides with  the set of all $X$-strictly singular operators on $X$.
\end{corollary}

\begin{remark}
It is easy to see that Lemma \ref{mxeq} and Corollary \ref{chmx} also hold for quasi-Banach spaces. Hence, for $0<p<1$, the set of all $\mathcal C_p$-strictly singular operators on $\mathcal C_p$ coincides with $\mathcal{M}_{\mathcal C_p}$ by Theorem \ref{ce-hc} and Remark \ref{add-opincp}, and thus it is the largest proper ideal in $\mathcal B(\mathcal C_p)$.
\end{remark}

The following proposition comes from \cite[Proposition 5.1]{Dosev Johnson}.

\begin{proposition}\label{dj5.1}
Let $X$ be a Banach space. Then $\mathcal M_X$ is the largest proper ideal in $\mathcal B(X)$ if and only if for every $T\in\mathcal B(X)$ we have $T\notin\mathcal M_X$ or $I-T\notin\mathcal M_X$.
\end{proposition}

The following lemma provides a sufficient condition for $\cM_X$ to be the largest ideal in $\cB(X)$. 
The idea underlying this lemma is useful in  the study of the primarity of Banach spaces, i.e., a Banach space $X$ is said to be  primary if  
for any decomposition of $X$ into the direct sum of two Banach spaces, at least one summand  is  isomorphic to $X$. 
\begin{lemma}\label{mxsuf}
Let $X$ be a Banach space. Assume  that for any $T\in\mathcal B(X)$, there exist a scalar $\alpha$, a Banach space $Y$ isomorphic to $X$, and $J\in\mathcal B(Y,X)$, $K\in\mathcal B(X,Y)$ with $KJ=I_Y$ such that
\begin{equation}\label{alpha JK}
\Vert KTJ-\alpha I_Y\Vert<1/2.
\end{equation}
Then $\mathcal M_X$ is the largest  proper ideal in $\mathcal B(X)$.
\end{lemma}
\begin{proof}
For any $T\in\mathcal B(X)$, 
we choose
a scalar $\alpha$, a Banach space $Y$ which is isomorphic to $X$, and $J\in\mathcal B(Y,X)$, $K\in\mathcal B(X,Y)$ with $KJ=I_Y$ satisfying \eqref{alpha JK}.

If $\vert\alpha\vert\geq1/2$, then since \eqref{alpha JK},
it follows that $KTJ$ is an invertible operator on $Y$.   Let $Q$ be an isomorphism from $X$ onto $Y$. We have 
\[
(Q^{-1}(KTJ)^{-1}K)T(JQ)=Q^{-1}(KTJ)^{-1}(KTJ)Q=I_X,
\]
and $Q^{-1}(KTJ)^{-1}K,JQ\in\mathcal B(X)$. 
Hence, $I_X$ factors through $T$, and thus $T\notin\mathcal M_X$ (see \eqref{mx} above).


If $\vert\alpha\vert<1/2$, then $\vert1-\alpha\vert>1/2$, and
\[
\Vert K(I_X-T)J-(1-\alpha)I_Y\Vert=\Vert KTJ-\alpha I_Y\Vert<1/2.
\] The same argument yields that  $I_X-T\notin\mathcal M_X$.

By Proposition \ref{dj5.1}, the proof is complete.
\end{proof}

Arazy \cite[Lemma 2.1]{Arazy3} shows that $\mathcal T_E$ satisfies the assumption in Lemma \ref{mxsuf}. Hence, 
we   obtain the following theorem. 

\begin{theorem}\label{tela} 
Let $E$ be a separable Banach symmetric sequence space. 
Then, $\mathcal M_{\mathcal T_E}$ is the largest proper ideal in $\mathcal B(\mathcal T_E)$.
\end{theorem}

In order to prove the main result of Section \ref{section:C1}, we need some auxiliary results, which are inspired by an idea used in Arazy's paper\cite[Lemma 2.1 and Theorem 3.1]{Arazy3}. Lemma \ref{tei} below can be viewed as a variant of \cite[Lemma 2.1]{Arazy3}.  

Now, let $\{\xi_i\}_{i=1}^\infty$ and $\{\eta_i\}_{i=1}^\infty$ be two orthonormal bases of $H$, and put
\[
\mathcal T_E=[e_{\xi_i,\eta_j}]_{1\leq j\leq i<\infty}.
\]

\begin{lemma}\label{tei}
For any operator $T\in \mathcal B(\mathcal T_E,\mathcal C_E)$ and sequence $\{\varepsilon_l\}_{l=1}^\infty$ of positive numbers,   there exist an increasing sequence $\{i_k\}_{k=1}^\infty$ of positive integers and a scalar $\alpha$ such that
\[
P_{[\xi_{i_{2k+1}}]_{k=1}^\infty,[\eta_{i_{2k}}]_{k=1}^\infty}T|_{[e_{\xi_{i_{2k+1}},\eta_{i_{2l}}}]_{1\leq l\leq k<\infty}}
\]
is a perturbation of $\alpha I_{[e_{\xi_{i_{2k+1}},\eta_{i_{2l}}}]_{1\leq l\leq k<\infty}}$ associated with  the Schauder decomposition $\big\{[e_{\xi_{i_{2k+1}},\eta_{i_{2l}}}]_{k=l}^\infty\big\}_{l=1}^\infty$ for $\{\varepsilon_l\}_{l=1}^\infty$.
\end{lemma}
\begin{proof}
Note that $\{T(e_{\xi_i,\eta_j})\}_{i=j}^\infty$ is weakly null for every $j$. By a standard perturbation argument, we may assume that there is an increasing sequence $\{i_k\}_{k=1}^\infty$ of positive integers such that
\[
T(e_{\xi_{i_{2k+1}},\eta_j})=p_{[\xi_s]_{s=1}^{i_{2k+2}-1}}T(e_{\xi_{i_{2k+1}},\eta_j})p_{[\eta_s]_{s=1}^{i_{2k+2}-1}}-p_{[\xi_s]_{s=1}^{i_{2k}}}T(e_{\xi_{i_{2k+1}},\eta_j})p_{[\eta_s]_{s=1}^{i_{2k}}}
\]
for every $k\geq1$ and $1\leq j\leq i_{2k}$. For  every $1\leq l\leq k<\infty$, we may write 
\begin{align}\label{xii2s1}
p_{_{[\xi_{i_{2s+1}}]_{s=1}^\infty}}T(e_{\xi_{i_{2k+1}},\eta_{i_{2l}}})=\sum_{j=1}^{i_{2k+2}-1}\alpha_j^{_{(k,l)}}e_{\xi_{i_{2k+1}},\eta_j},
\end{align}
where $\alpha_j^{(k,l)}$ is a  scalar. By a routine diagonal process of passing to a subsequence of $\{i_k\}_{k=1}^\infty$, we may  assume that there is a sequence $\big\{a^{_{(l)}}:  = (\alpha_j^{_{(l)}})_{j=1}^\infty\big\}_{l=1}^\infty$ in $\mathbb C^{\mathbb N}$ ($\mathbb R^{\mathbb N}$) such that
\[
\big(\alpha_1^{_{(k,l)}},\alpha_2^{_{(k,l)}},\dots,\alpha_{i_{2k+2}}^{_{(k,l)}},0,0,\dots\big)\longrightarrow a^{_{(l)}}\;\;\;\;{\rm as}\;k\to\infty\;\;\;\;\;({\rm by\;product\;topology}).
\]
It is easy to see that $\norm{ a^{_{(l)}}}_{\ell_2 }\le \Vert T  \Vert $ for every $l$. Choosing  a sequence $\{\delta_l\}_{l=1}^\infty$ of sufficiently small  positive numbers. After passing to a further subsequence of $\{i_k\}_{k=1}^\infty$, we may assume that (here, $i_{2k+2}$ is sufficiently greater than $i_{2k}$)
\begin{align}\label{alphai2sl}
\sum_{s=l+1}^\infty\big\vert\alpha_{i_{2s}}^{_{(l)}}\big\vert\leq\delta_l,
\end{align}
and (here, $i_{2k+1}$ is sufficiently greater than $i_{2k}$)
\begin{align}\label{alphai2skl}
\sum_{s=1}^k\big\vert\alpha_{i_{2s}}^{_{(k,l)}}-\alpha_{i_{2s}}^{_{(l)}}\big\vert\leq\delta_k\;\;\;\;\;\;{\rm for\;every\;}1\leq l\leq k<\infty.
\end{align}
 For each fixed $l\geq1$, and  for any finitely non-zero sequence $\{t_k\}_{k=l}^\infty$ of scalars, we have 
\begin{eqnarray*}
& &\left\Vert p_{_{[\xi_{i_{2s+1}}]_{s=1}^\infty}}T\left({\sum}_{k=l}^\infty t_ke_{\xi_{i_{2k+1}},\eta_{i_{2l}}}\right)p_{_{[\eta_{i_{2s}}]_{s=1}^\infty}}-{\sum}_{k=l}^\infty t_k{\sum}_{s=1}^l\alpha_{i_{2s}}^{_{(l)}}e_{\xi_{i_{2k+1}},\eta_{i_{2s}}}\right\Vert_{\mathcal C_E}\\
&\stackrel{\eqref{xii2s1}}{=}&\left\Vert{\sum}_{k=l}^\infty t_k\left({\sum}_{s=1}^k\alpha_{i_{2s}}^{_{(k,l)}}e_{\xi_{i_{2k+1}},\eta_{i_{2s}}}-{\sum}_{s=1}^l\alpha_{i_{2s}}^{_{(l)}}e_{\xi_{i_{2k+1}},\eta_{i_{2s}}}\right)\right\Vert_{\mathcal C_E}\\
&=&\left\Vert{\sum}_{k=l}^\infty t_k{\sum}_{s=1}^k(\alpha_{i_{2s}}^{_{(k,l)}}-\alpha_{i_{2s}}^{_{(l)}})e_{\xi_{i_{2k+1}},\eta_{i_{2s}}}-{\sum}_{k=l+1}^\infty t_k{\sum}_{s=l+1}^k\alpha_{i_{2s}}^{_{(l)}}e_{\xi_{i_{2k+1}},\eta_{i_{2s}}}\right\Vert_{\mathcal C_E}\\
&\stackrel{\eqref{alphai2skl}}{\leq}&{\sum}_{k=l}^\infty\vert t_k\vert\delta_k+\left\Vert{\sum}_{k=l+1}^\infty{\sum}_{s=l+1}^k t_k\alpha_{i_{2s}}^{_{(l)}}e_{\xi_{i_{2k+1}},\eta_{i_{2s}}}\right\Vert_{\mathcal C_E}\\
&\leq&\left({\sum}_{k=l}^\infty\vert\delta_k\vert^2\right)^{1/2}\left({\sum}_{k=l}^\infty\vert t_k\vert^2\right)^{1/2}+{\sum}_{s=l+1}^\infty\left\Vert{\sum}_{k=s}^\infty t_k\alpha_{i_{2s}}^{_{(l)}}e_{\xi_{i_{2k+1}},\eta_{i_{2s}}}\right\Vert_{\mathcal C_E}\\
&\leq&\left({\sum}_{k=l}^\infty\vert\delta_k\vert^2\right)^{1/2}\left({\sum}_{k=l}^\infty\vert t_k\vert^2\right)^{1/2}+{\sum}_{s=l+1}^\infty\big\vert \alpha_{i_{2s}}^{_{(l)}}\big\vert\left({\sum}_{k=l}^\infty\vert t_k\vert^2\right)^{1/2}\\
&\stackrel{\eqref{alphai2sl}}{\leq}&\left(\delta_l+\left({\sum}_{k=l}^\infty\vert\delta_k\vert^2\right)^{1/2}\right)\left\Vert{\sum}_{k=l}^\infty t_ke_{\xi_{i_{2k+1}},\eta_{i_{2l}}}\right\Vert_{\mathcal C_E}.
\end{eqnarray*}
By Theorem \ref{2d}, without loss of generality,  we may assume that
\begin{align}\label{PTPper}
p_{_{[\xi_{i_{2s+1}}]_{s=1}^\infty}}T(e_{\xi_{i_{2k+1}},\eta_{i_{2l}}})p_{_{[\eta_{i_{2s}}]_{s=1}^\infty}}=\sum_{s=1}^l\alpha_{i_{2s}}^{_{(l)}}e_{\xi_{i_{2k+1}},\eta_{i_{2s}}}\;\;\;\;\;\;\;\;{\rm for\;every\;}1\leq l\leq k<\infty.
\end{align}
For any finitely sequence $\{t_l\}_{l=1}^k$ of scalars, we have 
\begin{align*}
\left\Vert {\sum}_{l=1}^k t_l{\sum}_{s=1}^l\alpha_{i_{2s}}^{_{(l)}}e_{i_{2s}}\right\Vert_{\ell_2}&=\left\Vert {\sum}_{l=1}^k t_l\Big(\cdot\Big|{\sum}_{s=1}^l\alpha_{i_{2s}}^{_{(l)}}e_{i_{2s}}\Big)e_{i_{2k+1}}\right\Vert_{\mathcal C_E}\\
&=\left\Vert{\sum}_{l=1}^k t_l{\sum}_{s=1}^l\alpha_{i_{2s}}^{_{(l)}}e_{\xi_{i_{2k+1}},\eta_{i_{2s}}}\right\Vert_{\mathcal C_E}\leq\Vert T\Vert\left({\sum}_{l=1}^k\vert t_l\vert^2\right)^{1/2}.
\end{align*}
This implies that $\big\{\sum_{s=1}^l\alpha_{i_{2s}}^{_{(l)}}e_{i_{2s}}\big\}_{l=1}^\infty$ is a weakly null sequence in $\ell_2$.

Since $\big\{\alpha_{i_{2l}}^{_{(l)}}\big\}_{l=1}^\infty$ is a bounded sequence of scalars, without loss of generality, we may assume that $\big\{\alpha_{i_{2l}}^{_{(l)}}\big\}_{l=1}^\infty$ converges to a scalar $\alpha$. Next, we  inductively construct  an increasing sequence $\{n_l\}_{l=1}^\infty$ of positive integers. To start the induction, we  pick $n_1$ such that $\vert\alpha_{i_{2n_1}}^{_{(n_1)}}-\alpha\vert\leq\frac{\varepsilon_1}{\sqrt{2}}$. If $n_1,\dots,n_{l-1}$ ($l\geq2$) have been choose, then,  noting  that
\[
\big(\alpha_{i_{2n_1}}^{_{(m)}},\dots,\alpha_{i_{2n_{l-1}}}^{_{(m)}}\big)\stackrel{\Vert\cdot\Vert_2}{\longrightarrow} 0 \;\;\;\;{\rm as}\;m\to\infty,
\]
we may choose a sufficiently large integer  $n_l>n_{l-1}$ such that
\[
\left({\sum}_{s=1}^{l-1}\big\vert\alpha_{i_{2n_s}}^{_{(n_l)}}\big\vert^2\right)^{1/2}\leq\frac{\varepsilon_l}{\sqrt{2}}
\]
and
\[
\vert\alpha_{i_{2n_l}}^{_{(n_l)}}-\alpha\vert\leq\frac{\varepsilon_l}{\sqrt{2}}.
\]
This completes the induction. 
For simplicity, we  still denote   $i_{2l}:=i_{2n_l}$ and $i_{2l+1}:=i_{2n_l+1}$ for every $l$, and thus
\[
\left({\sum}_{s=1}^{l-1}\big\vert\alpha_{i_{2s}}^{_{(l)}}\big\vert^2\right)^{1/2}\;\;{\rm and }\;\;\;\big\vert\alpha_{i_{2l}}^{_{(l)}}-\alpha\big\vert\leq\frac{\varepsilon_l}{\sqrt{2}}\;\;\;\;\;\;\;\;{\rm for\;every\;}l\geq2.
\]
It follows that 
\begin{align}\label{ei2lper}
\left\Vert\alpha e_{i_{2l}}-{\sum}_{s=1}^l\alpha_{i_{2s}}^{_{(l)}}e_{i_{2s}}\right\Vert_{\ell_2}\leq\varepsilon_l\;\;\;\;\;\;\;\;{\rm for\;every\;}l\geq1.
\end{align}
For each fixed $l\geq1$, for any finitely non-zero sequence $\{t_k\}_{k=l}^\infty$ of scalars, we have 
\begin{eqnarray*}
&&\left\Vert p_{_{[\xi_{i_{2s+1}}]_{s=1}^\infty}}T\left({\sum}_{k=l}^\infty t_ke_{\xi_{i_{2k+1}},\eta_{i_{2l}}}\right)p_{_{[\eta_{i_{2s}}]_{s=1}^\infty}}-{\sum}_{k=l}^\infty t_k\alpha e_{\xi_{i_{2k+1}},\eta_{i_{2l}}}\right\Vert_{\mathcal C_E}\\
&\stackrel{\eqref{PTPper}}{=}&\left\Vert {\sum}_{k=l}^\infty t_k\Big(\Big({\sum}_{s=1}^l\alpha_{i_{2s}}^{_{(l)}}e_{\xi_{i_{2k+1}},\eta_{i_{2s}}}\Big)-\alpha e_{\xi_{i_{2k+1}},\eta_{i_{2l}}}\Big)\right\Vert_{\mathcal C_E}\\
&=& \left\Vert {\sum}_{k=l}^\infty t_k\Big(\cdot\Big|\alpha e_{i_{2l}}-{\sum}_{s=1}^l\alpha_{i_{2s}}^{_{(l)}}e_{i_{2s}}\Big)e_{i_{2k+1}}\right\Vert_{\mathcal C_E}\\
&\stackrel{\eqref{ei2lper}}{\leq}&\varepsilon_l\left({\sum}_{k=l}^\infty\vert t_k\vert^2\right)^{1/2}.
\end{eqnarray*}
It follows that
\[
P_{[\xi_{i_{2k+1}}]_{k=1}^\infty,[\eta_{i_{2k}}]_{k=1}^\infty}T|_{[e_{\xi_{i_{2k+1}},\eta_{i_{2l}}}]_{1\leq l\leq k<\infty}}
\]
is a perturbation of $\alpha I_{[e_{\xi_{i_{2k+1}},\eta_{i_{2l}}}]_{1\leq l\leq k<\infty}}$ associated with the Schauder decomposition $\big\{[e_{\xi_{i_{2k+1}},\eta_{i_{2l}}}]_{k=l}^\infty\big\}_{l=1}^\infty$ for $\{\varepsilon_l\}_{l=1}^\infty$.
\end{proof}

\begin{lemma}\label{cei}
Suppose that $\mathcal T_E\not\approx\mathcal C_E$. For any operator $T\in\mathcal B(\mathcal C_E)$ and sequence $\{\varepsilon_l\}_{l=1}^\infty$ of positive numbers,  there exist an increasing sequence $\{i_k\}_{k=1}^\infty$ of positive integers and a scalar $\alpha$ such that
\[
P_{[\xi_{i_{2k+1}}]_{k=1}^\infty,[\eta_{i_{2k}}]_{k=1}^\infty}T|_{[e_{\xi_{i_{2k+1}},\eta_{i_{2l}}}]_{k,l=1}^\infty}
\]
is a perturbation of $\alpha I_{[e_{\xi_{i_{2k+1}},\eta_{i_{2l}}}]_{k,l=1}^\infty}$ associated with the  Schauder decomposition $\big\{[e_{\xi_{i_{2k+1}},\eta_{i_{2l}}}]_{\min\{k,l\}=s}\big\}_{s=1}^\infty$ for $\{\varepsilon_l\}_{l=1}^\infty$.
\end{lemma}

\begin{proof}
Applying  Lemma \ref{tei} for $T|_{[e_{\xi_i,\eta_j}]_{1\leq j\leq i<\infty}}$ and Lemma \ref{2d''},
there exists an increasing sequence $\{i_k\}_{k=1}^\infty$ of positive integers, a scalar $\alpha$, and an operator $S\in\mathcal B([e_{\xi_{2k+1},\eta_{2l}}]_{k,l=1}^\infty,\mathcal C_E)$ such that
\[
S|_{[e_{\xi_{i_{2k+1}},\eta_{i_{2l}}}]_{1\leq l\leq k<\infty}}=\alpha I_{[e_{\xi_{i_{2k+1}},\eta_{i_{2l}}}]_{1\leq l\leq k<\infty}}, \]
\[S|_{[e_{\xi_{i_{2k+1}},\eta_{i_{2l}}}]_{1\leq k<l<\infty}}=P_{[\xi_{i_{2k+1}}]_{k=1}^\infty,[\eta_{i_{2k}}]_{k=1}^\infty}T|_{[e_{\xi_{i_{2k+1}},\eta_{i_{2l}}}]_{1\leq k<l<\infty}},
\]
and
\[
P_{[\xi_{i_{2k+1}}]_{k=1}^\infty,[\eta_{i_{2k}}]_{k=1}^\infty}T|_{[e_{\xi_{i_{2k+1}},\eta_{i_{2l}}}]_{k,l=1}^\infty}
\]
is a perturbation of $S$ associated with the Schauder decomposition $\big\{[e_{\xi_{i_{2k+1}},\eta_{i_{2l}}}]_{\min\{k,l\}=s}\big\}_{s=1}^\infty$ for $\{\varepsilon_l\}_{l=1}^\infty$. 
Arguing similarly for    the upper triangular part  of $[e_{\xi_{i_{2k+1}},\eta_{i_{2l}}}]_{k,l=1}^\infty$,  
Without loss of generality, we may  assume that there exists a scalar $\alpha'$ such that
\[
S|_{[e_{\xi_{i_{2k+1}},\eta_{i_{2l}}}]_{1\leq k<l<\infty}}=\alpha' I_{[e_{\xi_{i_{2k+1}},\eta_{i_{2l}}}]_{1\leq k<l<\infty}}.
\]
Note that
\[
(\alpha-\alpha')T^{E,\{\xi_{2k+1},\eta_{2k}\}_{k=1}^\infty}=S|_{{\rm span}(\{e_{\xi_{i_{2k+1}},\eta_{i_{2l}}}\}_{k,l=1}^\infty)}-\alpha'I_{{\rm span}(\{e_{\xi_{i_{2k+1}},\eta_{i_{2l}}}\}_{k,l=1}^\infty)}.
\]
If $\alpha\neq\alpha'$, then it yields that the  triangular projection $T^{E,\{\xi_{2k+1},\eta_{2k}\}_{k=1}^\infty}$ (see \eqref{tria-proj} for the definition) is bounded,  which contradicts  Theorem \ref{tria}. 
Hence, $S=\alpha I_{[e_{\xi_{i_{2k+1}},\eta_{i_{2l}}}]_{k,l=1}^\infty}$, this completes the proof.
\end{proof}

\begin{theorem}\label{cela}
Let $E$ be a separable Banach symmetric space. The set 
$\mathcal M_{\mathcal C_E}$ is the largest proper ideal in $\mathcal B(\mathcal C_E)$.
\end{theorem}

\begin{proof}
If $\mathcal C_E\approx\mathcal T_E$, then the assertion  follows from Theorem \ref{tela}. If $\mathcal C_E\not\approx\mathcal T_E$, then the assertion  follows from Lemmas \ref{cei} and \ref{mxsuf}.
\end{proof}

Now, we  consider  the case when $\mathcal C_E=\mathcal C_p$, $1\leq p<2$.
We end this section by proving  Proposition \ref{cpi} below,    which will be useful in the next section. 
Before proceeding to the proof of Proposition \ref{cpi}, we need the following lemma. 
\begin{lemma}\label{l2cp}
Let $1\leq p<2$. For any bounded operator $T:\ell_2\to\mathcal C_p$, we have 
\begin{equation}
\Vert P_{[\xi_i]_{i=N}^\infty,[\eta_i]_{i=N}^\infty}T\Vert\to0\;\;\;\;{\rm as}\;N\to\infty.
\end{equation}
\end{lemma}
\begin{proof}
Note that $\Vert P_{[\xi_i]_{i=N}^\infty,[\eta_i]_{i=N}^\infty}T\Vert$ is decreasing  as  $N$  increases.
Let   $c:= \lim_{N\to\infty}\Vert P_{[\xi_i]_{i=N}^\infty,[\eta_i]_{i=N}^\infty}T\Vert$.
Now suppose that $c>0$. Since $\lim_{N\to\infty}\Vert P_{[\xi_i]_{i=N}^\infty,[\eta_i]_{i=N}^\infty}T|_{[e_i]_{i=1}^n}\Vert=0$ for all $n\in\mathbb N^+$, it follows that for every $N$ and $n$,
\begin{align*}
\Vert P_{[\xi_i]_{i=N}^\infty,[\eta_i]_{i=N}^\infty}T|_{[e_i]_{i=n+1}^\infty}\Vert\geq&\lim_{N\to\infty}\Vert P_{[\xi_i]_{i=N}^\infty,[\eta_i]_{i=N}^\infty}T|_{[e_i]_{i=n+1}^\infty}\Vert\\
\geq&\lim_{N\to\infty}\Vert P_{[\xi_i]_{i=N}^\infty,[\eta_i]_{i=N}^\infty}T|_{[e_i]_{i=1}^\infty}\Vert-\lim_{N\to\infty}\Vert P_{[\xi_i]_{i=N}^\infty,[\eta_i]_{i=N}^\infty}T|_{[e_i]_{i=1}^n}\Vert\\
\geq& c.
\end{align*}
By standard perturbation arguments, there exist two increasing sequences $\{n_k\}_{k=1}^\infty$ and $\{m_k\}_{k=1}^\infty$ of positive integers, and a normalized disjoint  sequence $\{f_k\}_{k=1}^\infty$ in  $\ell_2$ with $f_k\in[e_i]_{i=n_k+1}^{n_{k+1}}$ for every $k$, such that
\[
\Big\Vert P_{[\xi_i]_{i=m_k+1}^{m_{k+1}},[\eta_i]_{i=m_k+1}^{m_{k+1}}}T(f_k)\Big\Vert_p \geq c/2\;\;\;\;\;\;{\rm for\;every}\;k.
\]
Since the operator $P_{\{[\xi_i]_{i=m_k+1}^{m_{k+1}},[\eta_i]_{i=m_k+1}^{m_{k+1}}\}_{k=1}^\infty}T|_{[f_k]_{k=1}^\infty}$ is bounded from $\ell_2$ into $\cC_p$, it follows that $(c/2)\left\Vert\cdot\right\Vert_{\ell_p}\leq \left\Vert\cdot\right\Vert_{\ell_2}$, which is a contradiction for $p<2$.
\end{proof}

\begin{proposition}\label{cpi}

For any operator $T\in\mathcal B(\mathcal C_1)$ and sequence $\{\varepsilon_l\}_{l=1}^\infty$ of positive numbers,   there exist an increasing sequence $\{i_k\}_{k=1}^\infty$ of positive integers and a scalar $\alpha$ such that
\[
P_{[\xi_{i_{2k+1}}]_{k=1}^\infty,[\eta_{i_{2k}}]_{k=1}^\infty}T|_{[e_{\xi_{i_{2k+1}},\eta_{i_{2l}}}]_{k,l=1}^\infty}
\]
is a perturbation of $\alpha I_{[e_{\xi_{i_{2k+1}},\eta_{i_{2l}}}]_{k,l=1}^\infty}$ associated with the Schauder decomposition $\big\{[e_{\xi_{i_{2k+1}},\eta_{i_{2l}}}]_{\min\{k,l\}=s}\big\}_{s=1}^\infty$ for $\{\varepsilon_l\}_{l=1}^\infty$, and
\[
\left\Vert P_{[\xi_{i_{2k+1}}]_{k=N}^\infty,[\eta_{i_{2k}}]_{k=N}^\infty}(T-\alpha I)|_{[e_{\xi_{i_{2k+1}},\eta_{i_{2l}}}]_{k,l=1}^\infty}\right\Vert\to0\;\;\;\;{\rm as}\;N\to\infty.
\]
\end{proposition}
\begin{proof}
By Lemma \ref{cei}, it suffices to prove that
\[
\lim_{N\to\infty}\left\Vert P_{[\xi_{i_{2k+1}}]_{k=N}^\infty,[\eta_{i_{2k}}]_{k=N}^\infty}(T-\alpha I)|_{[e_{\xi_{i_{2k+1}},\eta_{i_{2l}}}]_{k,l=1}^\infty}\right\Vert=0.
\]
For any positive integer $n$, we have 
\begin{eqnarray*}
&&\left\Vert P_{[\xi_{i_{2k+1}}]_{k=N}^\infty,[\eta_{i_{2k}}]_{k=N}^\infty}(T-\alpha I)|_{[e_{\xi_{i_{2k+1}},\eta_{i_{2l}}}]_{k,l=1}^\infty}\right\Vert\\
&\leq&2\left\Vert P_{[\xi_{i_{2k+1}}]_{k=N}^\infty,[\eta_{i_{2k}}]_{k=N}^\infty}(T-\alpha I)|_{[e_{\xi_{i_{2k+1}},\eta_{i_{2l}}}]_{\min\{k,l\}\leq n}}\right\Vert\\
& &\qquad \qquad +\left\Vert P_{[\xi_{i_{2k+1}}]_{k=N}^\infty,[\eta_{i_{2k}}]_{k=N}^\infty}(T-\alpha I)|_{[e_{\xi_{i_{2k+1}},\eta_{i_{2l}}}]_{k,l=n+1}^\infty}\right\Vert\\
&\stackrel{{\rm Lem}\;\ref{cei}}{\leq}&2\left\Vert P_{[\xi_{i_{2k+1}}]_{k=N}^\infty,[\eta_{i_{2k}}]_{k=N}^\infty}(T-\alpha I)|_{[e_{\xi_{i_{2k+1}},\eta_{i_{2l}}}]_{\min\{k,l\}\leq n}}\right\Vert+2\sum_{l=n+1}^\infty\varepsilon_l.
\end{eqnarray*}
Noting  that $[e_{\xi_{i_{2k+1}},\eta_{i_{2l}}}]_{\min\{k,l\}\leq n}\approx\ell_2$, by Lemma \ref{l2cp}, 
we have  $\left\Vert P_{[\xi_{i_{2k+1}}]_{k=N}^\infty,[\eta_{i_{2k}}]_{k=N}^\infty}(T-\alpha I)|_{[e_{\xi_{i_{2k+1}},\eta_{i_{2l}}}]_{\min\{k,l\}\leq n}}\right\Vert\to 0$ as $N \to \infty $, and therefore, 
\[
\varlimsup_{N\to\infty}\left\Vert P_{[\xi_{i_{2k+1}}]_{k=N}^\infty,[\eta_{i_{2k}}]_{k=N}^\infty}(T-\alpha I)|_{[e_{\xi_{i_{2k+1}},\eta_{i_{2l}}}]_{k,l=1}^\infty}\right\Vert\leq2\sum_{l=n+1}^\infty\varepsilon_l.
\]
Without loss of generality, we may assume that  $\varepsilon_l$'s are  small enough such that $\sum_{l=1}^\infty\varepsilon_l<\infty $.
Since $n$ is arbitrarily taken and $\sum_{l=n+1}^\infty\varepsilon_l\to0$ as $n\to0$, it follows that 
\[
\varlimsup_{N\to\infty}\left\Vert P_{[\xi_{i_{2k+1}}]_{k=N}^\infty,[\eta_{i_{2k}}]_{k=N}^\infty}(T-\alpha I)|_{[e_{\xi_{i_{2k+1}},\eta_{i_{2l}}}]_{k,l=1}^\infty}\right\Vert =0 .
\]
The proof is complete. 
\end{proof}

\begin{remark}
A result analogous to  Proposition \ref{cpi} holds  for $\mathcal T_p$, $1\leq p<2$.
\end{remark}

\section{The structure of commutator on $\mathcal C_1$}\label{section:C1}

In this section, we  apply results obtained in   previous sections to obtain a characterization of the commutators on $\mathcal C_1$.

Now, let $\{\xi_i\}_{i=1}^\infty$ and $\{\eta_i\}_{i=1}^\infty$ be two orthonormal bases of $H$.

\subsection{Every $T\in\mathcal M_{\mathcal C_1}$ is a commutator}

\begin{definition}\label{4.1}
	 A sequence  $\{X_i\}_{i=0}^\infty$   of closed subspaces of a Banach space ${X}$ is said to be an \emph{$\ell_p$-decomposition} of ${X}$ for $1\leq p<\infty$, or, $p=0$, if  the following two conditions are satisfied.
\begin{itemize}
	\item [(1)]  $\{X_i\}_{i=0}^\infty$ is a Schauder decomposition of ${X}$ for which there exists a positive constant $K$, such that for every convergent series $\sum_{i=0}^\infty x_i\in{X}$ with $x_i\in X_i$,
\[
\frac{1}{K}
\norm{  (\Vert x_i\Vert )_{i=1}^\infty   }_p 
\leq\bigg\Vert\sum_{i=0}^\infty x_i\bigg\Vert_X \leq K 
\norm{  (\Vert x_i\Vert )_{i=1}^\infty   }_p;
\]
      \item [(2)] The spaces $X_i$  ($i=0,1,2,....$) are uniformly  linear isomorphic to ${X}$.
\end{itemize}
\end{definition}

The following lemma is a combination of several  results due to  Dosev \cite[Lemma 5, Corollary 7 \& Theorem 8]{Dosev_l1}, whose proof can be found in  \cite[Lemma 4.7]{CY}. 
\begin{lemma}\label{y}
Let $\mathcal{D}=\{X_i\}_{i=0}^\infty$ be an $\ell_p$-decomposition of a Banach space ${X}$ and  $\widetilde{P}_n=\sum_{i=0}^n P_{\mathcal{D},i}$, where $P_{\mathcal{D},i}$ is  the natural projection from ${X}=\sum_{i=0}^\infty X_i$ onto $X_i$. Suppose that $T\in\mathcal{B}({X})$ satisfies
$$
\lim_{n\to\infty}\big\Vert(I-\widetilde{P}_n)T(I-P_{\mathcal{D},0})\big\Vert=\lim_{n\to\infty}\big\Vert T(I-\widetilde{P}_n)\big\Vert=0.
$$
Then $T$ is a commutator in $\mathcal B(X)$.
\end{lemma}

For convenience,   we state the  classical Pe{\l}czy\'nski decomposition technique below, see e.g. \cite[Theorem 2.2.3]{Albiac Kalton}.
\begin{lemma}\label{pelc}
Suppose that  $1\leq p<\infty$. Let $X$ be a Banach space satisfying  ${X}\approx(\sum{X})_p$, and $Y$ be a complemented subspace of ${X}$. If $Y$ contains a complemented subspace which is isomorphic to $X$, then $Y\approx{X}$.
\end{lemma}

By Theorem \ref{c1ch}, Lemma \ref{chmx} and Lemma \ref{cela},  
\begin{equation}\label{MC1}
\mathcal M_{\mathcal C_1}=\left\{
T\in\mathcal B(\mathcal C_1):T\;{\rm is}\;\mathcal C_1\mbox{-strictly~singular}
\right\}
\end{equation}
is the largest proper ideal (hence norm-closed) in $\mathcal B(\mathcal C_1)$.

\begin{lemma}\label{P(T-aI)P}
    For any operator $T\in\mathcal B(\mathcal C_1)$, there exist an increasing sequence $\{i_k\}_{k=1}^\infty$ of positive integers and a scalar $\alpha$ such that
    \[
\left\Vert P_{[\xi_{i_{2k+1}}]_{k=N}^\infty,[\eta_{i_{2k}}]_{k=N}^\infty}(T-\alpha I)P\right\Vert\to0\;\;\;\;{\rm as}\;N\to\infty,
\]
and
    \[
    P(T-\alpha I)P\in\mathcal M_{\mathcal C_1},
    \]
    where $P:=P_{[\xi_{i_{2k+1}}]_{k=1}^\infty,[\eta_{i_{2k}}]_{k=1}^\infty}$.
\end{lemma}
\begin{proof}
By Lemma \ref{cpi}, there exist  an increasing sequence $\{i_k\}_{k=1}^\infty$ of positive integers and a scalar $\alpha$ such that
\[
\left\Vert P_{[\xi_{i_{2k+1}}]_{k=N}^\infty,[\eta_{i_{2k}}]_{k=N}^\infty}(T-\alpha I)|_{[e_{\xi_{i_{2k+1}},\eta_{i_{2l}}}]_{k,l=1}^\infty}\right\Vert\to0\;\;\;\;{\rm as}\;N\to\infty.
\]
Set  $P:=P_{[\xi_{i_{2k+1}}]_{k=1}^\infty,[\eta_{i_{2k}}]_{k=1}^\infty}$. Then, we have
\begin{align*}
&\left\Vert P_{[\xi_{i_{2k+1}}]_{k=N}^\infty,[\eta_{i_{2k}}]_{k=N}^\infty}(T-\alpha I)P\right\Vert\\
=&\left\Vert P_{[\xi_{i_{2k+1}}]_{k=N}^\infty,[\eta_{i_{2k}}]_{k=N}^\infty}(T-\alpha I)|_{[e_{\xi_{i_{2k+1}},\eta_{i_{2l}}}]_{k,l=1}^\infty}P\right\Vert\to0\;\;\;\;{\rm as}\;N\to\infty.
\end{align*}
Observe the following decomposition 
\[
P(T-\alpha I)P=\big(P-P_{[\xi_{i_{2k+1}}]_{k=N}^\infty,[\eta_{i_{2k}}]_{k=N}^\infty}\big)(T-\alpha I)P+P_{[\xi_{i_{2k+1}}]_{k=N}^\infty,[\eta_{i_{2k}}]_{k=N}^\infty}(T-\alpha I)P.
\]
We have 
\[
\big(P-P_{[\xi_{i_{2k+1}}]_{k=N}^\infty,[\eta_{i_{2k}}]_{k=N}^\infty}\big)(T-\alpha I)P\to P(T-\alpha I)P\;\;\;\;{\rm as}\;N\to\infty.
\]
Since im$\big(P-P_{[\xi_{i_{2k+1}}]_{k=N}^\infty,[\eta_{i_{2k}}]_{k=N}^\infty}\big)\approx\ell_2$ when $N>1$, it follows from \eqref{MC1} that $P-P_{[\xi_{i_{2k+1}}]_{k=N}^\infty,[\eta_{i_{2k}}]_{k=N}^\infty}\in\mathcal M_{\mathcal C_1}$. 
Moreover, since $\mathcal M_{\mathcal C_1}$ is a norm-closed ideal, it follows that  
\[
P(T-\alpha I)P\in\mathcal M_{\mathcal C_1}.
\]
This completes the proof. 
\end{proof}

Below, we prove the main result of this section. 
\begin{theorem}\label{mc1c}
Every $T\in\mathcal{M}_{\mathcal C_1}$ is a commutator.
\end{theorem}
\begin{proof}
By Lemma \ref{P(T-aI)P}, there exist  an increasing sequence $\{i_k\}_{k=1}^\infty$ of positive integers and a scalar $\alpha$ such that
\[
\left\Vert P_{[\xi_{i_{2k+1}}]_{k=N}^\infty,[\eta_{i_{2k}}]_{k=N}^\infty}(T-\alpha I)P\right\Vert\to0\;\;\;\;{\rm as}\;N\to\infty,
\]
and
\[
P(T-\alpha I)P\in\mathcal M_{\mathcal C_1},
\]
where $P:=P_{[\xi_{i_{2k+1}}]_{k=1}^\infty,[\eta_{i_{2k}}]_{k=1}^\infty}$.

Noting that $\alpha P=PTP-P(T-\alpha I)P$ and $PTP\in\mathcal M_{\mathcal C_1}$,  we have $\alpha P\in\mathcal{M}_{\mathcal C_1}$. Since $P\notin\mathcal{M}_{\mathcal C_1}$ (see \eqref{MC1}), it follows that $\alpha=0$. Consequently, we have 
\[
\left\Vert P_{[\xi_{i_{2k+1}}]_{k=N}^\infty,[\eta_{i_{2k}}]_{k=N}^\infty}TP\right\Vert\to0\;\;\;\;{\rm as}\;N\to\infty.
\]

By Lemma \ref{pelc},  
we have $\ker(P)=(I-P)\mathcal C_1\approx \mathcal C_1$. For each $k\in\mathbb N^+$, define
\[
A_k:=\{i_{2^k(2l+1)+1}\}_{l=1}^\infty\;\;\;\;{\rm and}\;\;\;\;B_k:=\{i_{2^k(2l+1)}\}_{\mu=1}^\infty,
\]
and
\[
P_{A_k,B_k}:=P_{[\xi_i]_{i\in A_k},[\eta_i]_{i\in B_k}}. 
\]
Then, $P_{A_k,B_k}P$ is a projection with
\[
P_{A_k,B_k}P(\mathcal C_1)=[e_{\xi_i,\eta_j}]_{i\in A_k,j\in B_k}\cong\mathcal C_1
\]
for every $k\in\mathbb N^+$. Choose a sequence  $\{\delta_k\}_{k=1}^\infty$ of positive numbers with $\sum_{k=1}^\infty\delta_k<\infty$.
For every $k\in\mathbb N^+$, by Theorems \ref{c1s} and \ref{c1ch}, we can choose a closed subspace $X_k$ of $[e_{\xi_i,\eta_j}]_{i\in A_k,j\in B_k}$ such that $X_k$ is $2$-isomorphic to $\mathcal C_1$ and $2$-complemented in $[e_{\xi_i,\eta_j}]_{i\in A_k,j\in B_k}$, and $\Vert T|_{X_k}\Vert\leq\delta_k$. For each $k$,  let $Q_k$ be a projection from $[e_{\xi_i,\eta_j}]_{i\in A_k,j\in B_k}$ onto $X_k$ with $\Vert Q_k\Vert\leq 2$. Observe that 
\[
\left[{\bigcup}_{k=1}^\infty X_k\right]\cong\left(\sum_{k=1}^\infty\oplus\;X_k\right)_1\approx\left(\sum_{v=1}^\infty\oplus\;\mathcal C_1\right)_1,
\]
and the series  $\sum_{k=1}^\infty Q_kP_{A_k,B_k}P$ is strongly convergent, which  induces a projection from $\mathcal C_1$ onto $[\bigcup_{k=1}^\infty X_k]$. Let
$$
X_0=\ker\left(\sum_{k=1}^\infty Q_kP_{A_k,B_k}P\right)\supset\ker(P)\approx \cC_1 .
$$
Observe that $\ker(P)$ is a complemented subspace of $\cC_1$ and therefore, is a complemented subspace of $\ker\left(\sum_{k=1}^\infty Q_kP_{A_k,B_k}P\right)$. 
By Lemma \ref{pelc}, $X_0$ is isomorphic to $\mathcal C_1$. Thus,
$\mathcal{D}=\{X_k\}_{k=0}^\infty$ is an $\ell_1$-decomposition of $\mathcal C_1$ and satisfies that
\[
P_{\mathcal{D},0}=I-\sum_{k=1}^\infty Q_kP_{A_k,B_k}P,\;\;\;\;{\rm and}\;\;\;\;P_{\mathcal{D},k}= Q_kP_{A_k,B_k}P\;\;\;\;{\rm for\;all\;}k\in\mathbb N,
\]
where $P_{\mathcal{D},k}$ stands for the natural projection from ${X}=\sum_{k=0}^\infty X_k$ onto $X_k$. 

Denote $\widetilde{P}_n:=\sum_{k=0}^nP_{\mathcal{D},k}$. Then, by the triangular inequality, we have 
\[
\big\Vert T(I-\widetilde{P}_n)\big\Vert\leq\sum_{k=n+1}^\infty\Vert TP_{\mathcal{D},k}\Vert\leq2\sum_{k=n+1}^\infty\delta_k.
\]
Therefore,
\begin{equation}\label{d1}
\lim_{n\to\infty}\big\Vert T(I-\widetilde{P}_n)\big\Vert=0.
\end{equation}
On the other hand, it follows from $\norm{Q_k}\le 2$ and 
\begin{align*}
(I-\widetilde{P}_n)T(I-P_{\mathcal{D},0})&=(I-\widetilde{P}_n)TP(I-P_{\mathcal{D},0})\\
&=\left(\sum_{k=n+1}^\infty Q_kP_{A_k,B_k}P\right)TP\left(\sum_{k=1}^\infty Q_kP_{A_k,B_k}P\right)
\end{align*}
that
\[
\big\Vert(I-\widetilde{P}_n)T(I-P_{\mathcal{D},0})\big\Vert\leq 4\left\Vert P_{[\xi_{i_{2k+1}}]_{k=2^n}^\infty,[\eta_{i_{2k}}]_{k=2^n}^\infty}TP\right\Vert. 
\]
Therefore,
\begin{equation}\label{d2}
\lim_{n\to\infty}\big\Vert(I-\widetilde{P}_n)T(I-P_{\mathcal{D},0})\big\Vert=0.
\end{equation}
By \eqref{d1}, \eqref{d2}, and Lemma \ref{y}, $T$ is a commutator.
\end{proof}

\subsection{A characterization of commutators on $\mathcal C_1$}

For $X$ and $Y$ be two (closed) subspaces of a Banach space ${Z}$, let
\[
{\rm d}(S_X,Y)=\inf\{\Vert x-y\Vert:x\in S_X,~y\in Y\},
\]
where $S_X=\{x\in X:\norm{x}=1\}$ is the unit sphere of $X$.

The following two theorems are important tools in the study of commutator on Banach spaces,  see \cite[Lemmas 2.13, 2.14 \& Theorem. 2.15]{Chen Johnson Bentuo Zheng}. The ideas underlying  their proofs can be traced back to \cite{Apostol_lp}, \cite{Dosev_l1} and \cite{Dosev Johnson}.

\begin{theorem}\label{4}
Let $p\in[1,\infty]\cup\{0\}$, and $X$ be a complementably homogeneous Banach space isomorphic to $(\sum{X})_p$.
If  $T$ is a bounded linear operator on $X$ for which there is a subspace $Y$ of $X$ isomorphic to $X$ such that $T|_Y$ is an isomorphism and $d(S_Y,TY)>0$, then $T$ is a commutator.
\end{theorem}

\begin{theorem}\label{5}
Let $p\in[1,\infty]\cup\{0\}$, and $X$ be a complementably homogeneous Banach space isomorphic to $(\sum X)_p$. Suppose that the set of all $X$-strictly singular operators on $X$ 
is  an ideal in $\mathcal B(X)$.
If $T\in \cB(\cC_1)$ satisfies that 

\begin{enumerate}
    \item $T-\lambda'I$ is not $X$-strictly singular for each scalar $\lambda'$;
    \item there is a scalar $\lambda$ and a subspace $Y$ of $X$ isomorphic to $X$   such that $(T-\lambda I)|_Y$ is $X$-strictly singular,
\end{enumerate}  then $T$ is a commutator.
\end{theorem}

\begin{theorem}\label{6}
Let $T\in\mathcal B(\mathcal C_1)$ be such that $T-\lambda I$ is not $\mathcal C_1$-strictly singular for all $\lambda\in\mathbb C$. Then $T$ is a commutator.
\end{theorem}
\begin{proof}
Recall that\cite[p.300]{Arazy2}  $$\mathcal C_1\approx \left(\sum\mathcal C_1\right)_1 . $$

By Lemma \ref{P(T-aI)P}, there exists an increasing sequence $\{i_k\}_{k=1}^\infty$ of positive integers and a scalar $\alpha$ such that
\[
\left\Vert P_{[\xi_{i_{2k+1}}]_{k=N}^\infty,[\eta_{i_{2k}}]_{k=N}^\infty}(T-\alpha I)P\right\Vert\to0\;\;\;\;{\rm as}\;N\to\infty,
\]
and
\[
P(T-\alpha I)P\in\mathcal M_{\mathcal C_1}
\]
where $P:=P_{[\xi_{i_{2k+1}}]_{k=1}^\infty,[\eta_{i_{2k}}]_{k=1}^\infty}$.

(1) If $(I-P)TP$ is not $\mathcal C_1$-strictly singular, then there is a subspace $X$ of $\mathcal C_1$ isomorphic to $\mathcal C_1$ such that $(I-P)TP$ is an isomorphism on $X$. Denote by $Y=PX$. Then $Y\approx \mathcal C_1$. Since $(I-P)|_{TY}$ is an isomorphic embedding, it follows that there is a positive number $c$ such that
\[
\Vert (I-P)(y')\Vert_1\geq c\Vert y'\Vert_1\;\;\;\;\;\;{\rm for\;every}\;y'\in TY.
\]
For any $y\in S_Y=S_{PX}$ and $y'\in TY$, we have $\Vert (I-P)(y-y')\Vert_1=\Vert (1-P)(y')\Vert_1\geq c\Vert y'\Vert_1$. Thus, 
\begin{align*}
\Vert y-y'\Vert_1&=\frac{c\Vert y-y'\Vert_1+\Vert I-P\Vert\cdot\Vert y-y'\Vert_1}{c+\Vert I-P\Vert}\\
&\geq\frac{c\big(\Vert y-y'\Vert_1+\Vert y'\Vert_1\big)}{c+\Vert I-P\Vert}\geq\frac{c}{c+\Vert I-P\Vert}.
\end{align*}
Therefore, $d(S_Y,TY)>0$. By Theorem \ref{4}, $T$ is a commutator.

(2) if  $(I-P)TP$ is $\mathcal C_1$-strictly singular, then, from the identity
\[
(T-\alpha I)P=(I-P)TP+P(T-\alpha I)P
\]
and $P(T-\alpha I)P\in\mathcal M_{\mathcal C_1}$, it follows that   
$(T-\alpha I)|_{[e_{i_{2k+1},i_{2l}}]_{k,l=1}^\infty}$ is $\mathcal C_1$-strictly singular. By 
Theorem~\ref{5},  $T$ is a commutator.
\end{proof}

\begin{theorem}\label{thmd}
An operator $T\in\mathcal B(\mathcal C_1)$ is a commutator if and only if $T-\lambda I$ is not $\mathcal C_1$-strictly singular for all $\lambda\neq 0$.  Consequently, $\mathcal{C}_1$ is a Wintner space. 
\end{theorem}
\begin{proof}
The necessity  follows immediately from Wintner's theorem (i.e., Theorem \ref{win}) for quotient Banach algebra $\mathcal B(\mathcal C_1)/\mathcal M_{\mathcal C_1}$.

Sufficiency.  Suppose that $T-\lambda I\notin\mathcal M_{\mathcal C_1}$ for all $\lambda\neq 0$. If $T\in\mathcal{M}_{\mathcal C_1}$, it follows from  Theorem~\ref{mc1c} that  $T$ is a commutator  on $\cC_1$. If $T\notin\mathcal{M}_{\mathcal C_1}$, then $T-\lambda I\notin\mathcal{M}_{\mathcal C_1}$ for all scalars $\lambda$. 
By Theorem~\ref{6}, the operator $T$ is a commutator on $\cC_1$.
\end{proof}
The following corollary is an immediate consequence of Theorems \ref{cs=ts} and \ref{thmd}.
\begin{corollary}
An operator $T\in\mathcal B(\mathcal C_1)$ is a commutator if and only if $T-\lambda I$ is not $\mathcal T_1$-strictly singular for all $\lambda\neq 0$.
\end{corollary}

\begin{remark}Using results obtained  
in Sections \ref{Sec:Arazy's decomposition for cp} and \ref{Sec:largest proper ideal}
and  arguing mutatis mutandis as the above proof, one can show that,  for any   $1\leq p<\infty$, an operator $T\in\mathcal B(\mathcal T_p)$ is a commutator if and only if $T-\lambda I$ is not $\mathcal T_p$-strictly singular for all $\lambda\neq0$.
\end{remark}

\begin{corollary}\label{t1+c1}
    $\mathcal T_1\oplus\mathcal C_1$ is a Wintner space.
\end{corollary}
\begin{proof}
   Since both $\mathcal T_1$ and $\mathcal C_1$ are   Wintner spaces (see \cite[Theorem 5.4]{CY} and  Theorem \ref{thmd} above), it follows from   \cite[Corollary 5.4]{Dosev Johnson} that
   in order to prove the statement, 
   it suffices to  verify  the following two facts:
\begin{itemize}
    \item  All operators $T:\mathcal T_1\to\mathcal T_1$ that factor through $\mathcal C_1$ (i.e, there exist two operators $A:\mathcal T_1\to\mathcal C_1$ and $B:\mathcal C_1\to\mathcal T_1$ such that $T=BA$) are in the largest proper ideal $\mathcal M_{\mathcal T_1}$ in $\mathcal B(\mathcal T_1$) (Theorem~\ref{tela}).
    \item 
    All operators $T:\mathcal C_1\to\mathcal C_1$ that factor through $\mathcal T_1$ are in the largest proper ideal $\mathcal M_{\mathcal C_1}$ in $\mathcal B(\mathcal C_1)$.
\end{itemize}

The proof of fact 1:
If $T\notin\mathcal M_{\mathcal T_1}$, then there are two operators $U,V\in\mathcal B(\mathcal T_1)$ such that $UTV=I_{\mathcal T_1}$ (see \eqref{mx} above). Since $T$ factors through $\mathcal C_1$, i.e., $T=AB$, where $A:\mathcal T_1\to\mathcal C_1$ and $B:\mathcal C_1\to\mathcal T_1$, 
it follows that $\mathcal T_1$ is isomorphic to a complemented subspace of $\mathcal C_1$\footnote{Recall the following well-known fact: 
Let $X_0$ and $ X_1$ be two Banach spaces and let 
 $T_0: X_0\to X_1$ and $T_1:X_1\to X_0$ be 
 two bounded linear  operators. 
 Assume that $T_1\circ T_0 ={\rm id}_{X_0}$. Then,  $X_0  $ is isomorphic to a  complemented subspace of $X_1$.  
 A short proof of this fact can be found in Lemma 3.1 of \cite{HSZ25a}. }. By Theorem \ref{teicce}, we have $\mathcal T_1\approx\mathcal C_1$, a contradiction.

The proof of fact 2: 
Assume that there exist $A:\cC_1\to \mathcal{T}_1 $ and $B:\mathcal{T}_1 \to \cC_1$ such that $T=AB$. 
Since $\mathcal T_1$ has an  unconditional finite dimensional decomposition, it follows from \cite[Lemma 4.5 and Theorem 4.7]{Arazy1} that $\mathcal C_1$ is not isomorphic to any subspace of $\mathcal T_1$.
Hence, $A$ is $\cC_1$-strictly singular, and therefore, $T=AB$ is also $\cC_1$-strictly singular.
Hence, $T\in\mathcal M_{\mathcal C_1}$.
\end{proof}

The proof of Corollary \ref{t1+c1} together with the following facts:
\begin{itemize}
\item 
\cite[Corollary 3.2]{Arazy4}: 
$\ell_p \not\hookrightarrow \cC_q$ when $p\ne 2$ and $1\le p\ne q<\infty $. In particular, $\mathcal{T}_p , \cC_p\not\hookrightarrow \cC_q$. 
\item
$\mathcal C_p\oplus\mathcal C_2\approx\mathcal C_p\oplus\ell_2\approx\mathcal C_p$ for any $1\leq p<\infty$ and $\mathcal T_1\oplus\mathcal C_2\approx\mathcal T_1\oplus\ell_2\approx\mathcal T_1$.
\item 
$\mathcal C_p\approx\mathcal T_p$ for any $1<p<\infty$.
\end{itemize}
yields the following consequence.  

\begin{corollary}
For any $X_1,\dots,X_n\in\{\mathcal C_p:1\leq p<\infty\}\cup\{\mathcal T_p:1\leq p<\infty\}$,  $X_1\oplus\cdots\oplus X_n$ is a Wintner space.
\end{corollary}

\section*{Acknowledgment}
Jinghao Huang was supported by the NNSF of China (No. 12031004, 12301160 and 12471134). Fedor Sukochev was supported by the Australian Research Council (No. DP230100434). Zhizheng Yu would like to thank Professor Quanhua Xu for his generous support and encouragement. Our sincere thanks go to Professor William B. Johnson for generously sharing his ideas (see  Remark~\ref{chE}).

\appendix
\section{$\sigma$-weakly null sequences in  $\mathcal C_E$}\label{appendix}

In this section, we consider 
sequences in 
 a separable quasi-Banach operator ideal $\mathcal C_E$ which  converge to $0$ with respect to the $
\sigma$-weak operator topology of  $\cB(H)$,  and extend results established in \cite[Section 2]{Arazy4}. 
For convenience of the reader, we present
a complete version of our argument in terms of block sequences.

Note that  any weakly null sequence in a  quasi-Banach symmetric ideal $\mathcal C_E$ (i.e., sequences converging to $0$ with  respect to the $\sigma(\mathcal C_E,\mathcal C_E^*)$-topology)  are necessarily $\sigma$-weakly null.

The proofs for results in this section 
in the setting  when  $E$ is a general separable  quasi-Banach symmetric sequence space  are very similar to those  for  Banach spaces. 
 Throughout this section, we always
 \begin{itemize}
 \item 
 assume that the modulus of concavity of $\left\Vert\cdot\right\Vert_{\mathcal C_E}$ is $\kappa$, and
 \item 
 denote by  $\{K_i\}_{i=1}^\infty$ and $\{L_i\}_{i=1}^\infty$  two sequences of mutually orthogonal finite dimensional subspaces of $H$.
\end{itemize}

The following theorem was  proved in  \cite[Lemma 2.5]{Arazy3} in terms of matrix representations of operators on a Hilbert space. 
Below, we provide a more transparent and more self-contained proof.

\begin{theorem}\label{3c}
Let $E$ be a separable quasi-Banach symmetric space.
Suppose that $\{x_k\}_{k=2}^\infty$ is a bounded sequence in $\mathcal C_E$ with
\[
x_k=p_{_{[\bigcup_{l=1}^kK_l]}}x_kp_{_{[\bigcup_{l=1}^kL_l]}}-p_{_{[\bigcup_{l=1}^{k-1}K_l]}}x_kp_{_{[\bigcup_{l=1}^{k-1}L_l]}}, \;\;\;\;\;\;k=2,3,\dots.
\]
For a given sequence $\{\varepsilon_i\}_{i=2}^\infty$ of positive numbers,  there exist an increasing sequence $\{k_i\}_{i=1}^\infty$ of positive integers and a sequence $\{y_i\}_{i=2}^\infty$ in $\mathcal C_E$ having  the form
\[
y_i=\sum_{j=1}^{i-1}\big(a_{i,j}+a_{j,i}\big)+a_{i,i}\;\;\;\;\;\;i=2,3,\dots,
\]
such that $\Vert x_{k_i}-y_i\Vert_{\mathcal C_E}\leq\varepsilon_i$ for every $i\geq2$, 
where
\begin{itemize}
\item [(1)]
$K'_i=[\bigcup_{l=k_{i-1}+1}^{k_i}K_l]$ and $L'_i=[\bigcup_{l=k_{i-1}+1}^{k_i}L_l]$ for every $i\geq1$ (put $k_0=0$),
\item [(2)]
$\{a_{i,j}\}_{(i,j)\neq (1,1)}\stackrel{\mbox{{\rm\tiny b.s}}}{\boxplus}\{K'_i\}_{i=1}^\infty\otimes\{L'_j\}_{j=1}^\infty$,
\item [(3)]
$a_{i,i}=p_{_{K'_i}}x_{k_i}p_{_{L'_i}}$ for every $i\geq2$,
\item [(4)]
 $\{a_{i,j}\}_{i=j+1}^\infty\stackrel{\mbox{{\rm\tiny b.s}}}{\boxplus}\{K'_i\}_{i=2}^\infty\otimes\{L'_j\}_{j=1}^\infty\curvearrowleft\;$ are consistent\footnote{See Definition \ref{consistent}, i.e. there is a sequence of isometries $\big\{u_i:K'_i\to H'\big\}_{i=2}^\infty$, a sequence of isometries $\big\{v_j:L'_j\to H'\big\}_{j=1}^\infty$, and a sequence $\{d_j\}_{j=1}^\infty$ of in $\mathcal C_E(H')$ such that
\[
u_ia_{i,j}{v_j}^\ast=d_j\;\;\;\;\;\;\;\;{\rm for\;every}\;1\leq i<j<\infty.
\]
} for all $j\geq1$,
\item [(5)] 
$\{a_{j,i}\}_{i=j+1}^\infty\stackrel{\mbox{{\rm\tiny b.s}}}{\boxplus}\{K'_j\}_{j=1}^\infty\otimes\{L'_i\}_{i=2}^\infty\curvearrowleft\;$ are consistent for all $j\geq1$.
\end{itemize}
\end{theorem}

\begin{proof}
Choose an decreasing sequence $\{\delta_i\}_{i=1}^\infty$ of positive numbers such that $\delta_i\leq\varepsilon_i/2$ and $\sum_{i=1}^\infty\kappa^{i+1}\delta_i\leq1$.
We   write
\[
x_k=\sum_{l=1}^{k-1}\big(c_{k,l}+c_{l,k}\big)+c_{k,k}\;\;\;\;\;\;k=2,3,\dots,
\]
where $c_{k,l}=p_{_{K_k}}x_{{\max\{k,l\}}}p_{_{L_l}}$ for every $(k,l)\neq(1,1)$. Applying  Lemma \ref{3b} 
and  Remark \ref{3b'}
to  the lower triangle sequence $\{c_{k,l}\}_{1\le j\le i<\infty}$  and upper triangle sequence of $\{c_{k,l}\}_{1\le i\leq j <\infty}$,  we obtain that 
there exist an increasing sequence $\{k_i\}_{i=1}^\infty$ of positive integers and a sequence $\{b_{k_i,l}\}_{1\leq l\leq k_{i-1}}\cup\{b_{l,k_i}\}_{1\leq l\leq k_{i-1}}$ such that 
\begin{itemize}
\item [(1)]
$\Vert c_{k_i,l}-b_{k_i,l}\Vert_{\mathcal C_E} ,\Vert c_{l.k_i}-b_{l,k_i}\Vert_{\mathcal C_E}\leq\delta_i\delta_l$ for every $i\geq2$ and $1\leq l\leq k_{i-1}$,
\item [(2)]
$\{b_{k_i,l}\}_{l\leq k_{i-1}<\infty}\stackrel{\mbox{{\rm\tiny b.s}}}{\boxplus}\{K_{k_i}\}_{i=1}^\infty\otimes \{L_l\}_{l=1}^\infty\curvearrowleft\;$ are consistent for all $l$,
\item [(3)]
$\{b_{l,k_i}\}_{l\leq k_{i-1}<\infty}\stackrel{\mbox{{\rm\tiny b.s}}}{\boxplus}\{K_l\}_{l=1}^\infty\otimes \{L_{k_i}\}_{i=1}^\infty\curvearrowleft\;$ are consistent for all $l$.
\end{itemize}
Next, for all $1\leq j<i<\infty$, let
\begin{equation}\label{bl1}
a_{i,j}:=\sum_{l=k_{j-1}+1}^{k_j} b_{k_i,l},
\end{equation}
\begin{equation}\label{bl2}
a_{j,i}:=\sum_{l=k_{j-1}+1}^{k_j} b_{l,k_i},
\end{equation}
and
\begin{equation}\label{bl3}
a_{i,i}=p_{_{K'_i}}x_{k_i}p_{_{L'_i}},
\end{equation}
where $K'_i=[\bigcup_{l=k_{i-1}+1}^{k_i}K_l]$ and $L'_i=[\bigcup_{l=k_{i-1}+1}^{k_i}L_l]$ for every $i\geq1$.
The positions of $a_{i,j}$'s are demonstrated in the following picture. 
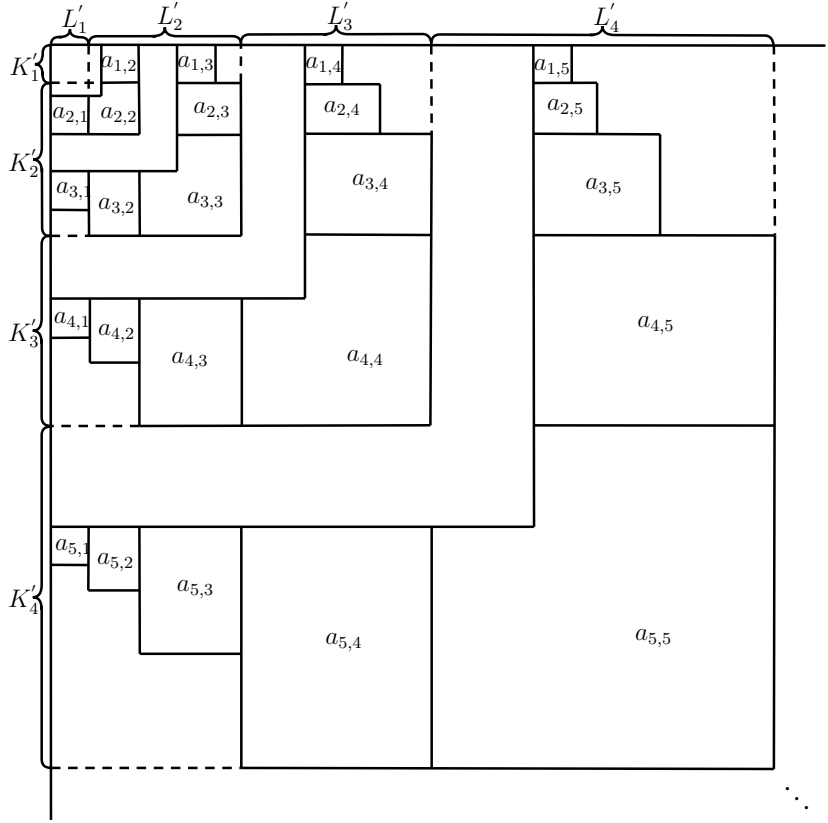
\tikzset{every picture/.style={line width=0.75pt}} 
\begin{center}
\scalebox{0.8}{ 
 \begin{tikzpicture}[x=0.75pt,y=0.75pt,yscale=-1,xscale=1] 

\draw [line width=1.2] (32.8,71.86) -- (521.34,72.4);
\draw [line width=1.2] (32.8,71.86) -- (33.06,560.94);
\draw [line width=1.2] (32.8,128.22) -- (88.85,128.41);
\draw [line width=1.2] (88.26,71.68) -- (88.85,128.41);
\draw [line width=1.2] (32.8,104.14) -- (64.8,104.14);
\draw [line width=1.2] (56.8,104) -- (56.59,127.82);
\draw [line width=1.2] (64.8,72) -- (64.8,104.14);
\draw [line width=1.2] (112.54,71.6) -- (112.54,151.6);
\draw [line width=1.2] (64.54,95.86) -- (88.54,95.6);
\draw [line width=1.2] (56.8,152) -- (57.06,192.4);
\draw [line width=1.2] [dash pattern={on 4.5pt off 3.6pt}] (152.76,96.04) -- (152.83,71.66);
\draw [line width=1.2] (33.34,151.6) -- (112.54,151.6);
\draw [line width=1.2] [dash pattern={on 4.5pt off 3.6pt}] (32.83,192.46) -- (57.06,192.4);
\draw [line width=1.2] (88.8,152) -- (89.06,192.4);
\draw [line width=1.2] (32.8,176) -- (56.66,176.2);
\draw [line width=1.2] (57.06,192.4) -- (153.06,192.4);
\draw [color={rgb, 255:red, 0; green, 0; blue, 0 }, draw opacity=1][line width=1.2] (152.76,96.04) -- (153.06,192.4);
\draw [line width=1.2] (136.8,72) -- (136.83,96.46);
\draw [line width=1.2] (112.8,128.8) -- (152.65,128.71);
\draw [line width=1.2] (112.8,96) -- (152.76,96.04);
\draw [line width=1.2] (33.6,232) -- (193.19,231.9);
\draw [line width=1.2] (88.8,232) -- (88.8,267.9) -- (88.8,312.4);
\draw [line width=1.2] (57.6,232.4) -- (57.6,272.4);
\draw [line width=1.2] (193.17,72) -- (193.19,231.9);
\draw [line width=1.2] (153.38,231.69) -- (153.38,312.09);
\draw [line width=1.2] (88.8,312.4) -- (272.26,311.95);
\draw [line width=1.2] (273.06,128.14) -- (272.26,311.95);
\draw [line width=1.2] (192.8,96.8) -- (240.5,96.76);
\draw [line width=1.2] [dash pattern={on 4.5pt off 3.6pt}] (88.8,312.4) -- (33.34,312.4);
\draw [line width=1.2] (240.5,96.76) -- (240.5,127.96);
\draw [line width=1.2] (57.6,272.4) -- (88.54,272.4);
\draw [line width=1.2] (33.6,256.8) -- (57.6,256.66);
\draw [line width=1.2] (192.8,128) -- (273.06,128.14);
\draw [line width=1.2] (193.04,191.56) -- (272.32,191.91);
\draw [line width=1.2] (216.8,72) -- (216.66,96.78);
\draw [line width=1.2] [dash pattern={on 4.5pt off 3.6pt}] (272.8,72) -- (273.06,128.14);
\draw [line width=1.2] (337.16,71.61) -- (337.6,376);
\draw [line width=1.2] (32.8,376) -- (337.6,376);
\draw [line width=1.2] (88.8,376) -- (89.06,456.14);
\draw [line width=1.2] (273.06,376.14) -- (273.06,528.14);
\draw [line width=1.2] (153.06,376.14) -- (153.06,528.14);
\draw [line width=1.2] (89.06,456.14) -- (153.06,456.14);
\draw [line width=1.2] (56.8,376) -- (56.6,416.13);
\draw [line width=1.2] (56.6,416.13) -- (89.06,416.14);
\draw [line width=1.2] (32.8,400) -- (57.06,400.14);
\draw [line width=1.2] (153.06,528.14) -- (488.8,528.8);
\draw [line width=1.2] (489.46,192.19) -- (488.8,528.8);
\draw [line width=1.2] (337.6,312) -- (489.28,312.13);
\draw [line width=1.2] [dash pattern={on 4.5pt off 3.6pt}] (33.06,528.14) -- (153.06,528.14);
\draw [line width=1.2] [dash pattern={on 4.5pt off 3.6pt}] (32.54,95.86) -- (64.54,95.86);
\draw [line width=1.2] [dash pattern={on 4.5pt off 3.6pt}] (56.8,72) -- (56.8,104);
\draw [line width=1.2] (337.6,192) -- (489.46,192.19);
\draw [line width=1.2] (337.6,128) -- (417.06,128.4);
\draw [line width=1.2] (417.06,128.4) -- (417.06,192.4);
\draw [line width=1.2] (336.8,96) -- (377.34,95.86);
\draw [line width=1.2] (377.34,95.86) -- (377.34,128.2);
\draw [line width=1.2] (361.34,71.86) -- (361.34,95.86);
\draw [line width=1.2] [dash pattern={on 4.5pt off 3.6pt}] (488.8,72) -- (489.46,192.19);

\draw [line width=1.2] (32.8,72.14) .. controls (29.5,72.14) and (27.86,73.78) .. (27.86,77.07) -- (27.86,77.07) .. controls (27.86,81.78) and (26.21,84.14) .. (22.92,84.14) .. controls (26.21,84.14) and (27.86,86.49) .. (27.86,91.19)(27.86,89.07) -- (27.86,91.19) .. controls (27.86,94.49) and (29.5,96.14) .. (32.8,96.14);
\draw [line width=1.2] (32.8,96) .. controls (29.06,96.01) and (27.21,97.88) .. (27.22,101.62) -- (27.31,136.35) .. controls (27.33,141.69) and (25.47,144.36) .. (21.74,144.37) .. controls (25.47,144.36) and (27.34,147.02) .. (27.36,152.35)(27.35,149.95) -- (27.46,187.08) .. controls (27.46,190.82) and (29.34,192.67) .. (33.07,192.66);
\draw [line width=1.2] (32.78,192.3) .. controls (29.05,192.3) and (27.18,194.17) .. (27.18,197.9) -- (27.18,244.3) .. controls (27.18,249.64) and (25.32,252.3) .. (21.58,252.3) .. controls (25.32,252.3) and (27.18,254.97) .. (27.18,260.3)(27.18,257.9) -- (27.18,306.7) .. controls (27.18,310.44) and (29.05,312.3) .. (32.78,312.3);
\draw [line width=1.2] (32.8,312.8) .. controls (29.06,312.8) and (27.2,314.66) .. (27.2,318.4) -- (27.2,412.42) .. controls (27.2,417.76) and (25.34,420.42) .. (21.6,420.42) .. controls (25.34,420.42) and (27.2,423.09) .. (27.2,428.42)(27.2,426.02) -- (27.2,522.44) .. controls (27.2,526.18) and (29.06,528.04) .. (32.8,528.04);
\draw [line width=1.2] (56.8,72.08) .. controls (56.83,68.82) and (55.22,67.18) .. (51.96,67.15) -- (51.96,67.15) .. controls (47.3,67.11) and (44.99,65.46) .. (45.02,62.2) .. controls (44.99,65.46) and (42.65,67.06) .. (38,67.02)(40.09,67.05) -- (38,67.02) .. controls (34.74,66.99) and (33.1,68.61) .. (33.06,71.86);
\draw [line width=1.2] (152.91,72.06) .. controls (152.9,68.32) and (151.04,66.46) .. (147.3,66.46) -- (112.86,66.47) .. controls (107.52,66.47) and (104.86,64.61) .. (104.85,60.87) .. controls (104.86,64.61) and (102.19,66.47) .. (96.86,66.48)(99.26,66.47) -- (62.4,66.49) .. controls (58.66,66.49) and (56.8,68.35) .. (56.8,72.09);
\draw [line width=1.2] (272.83,71.69) .. controls (272.82,67.95) and (270.95,65.1) .. (267.22,65.1) -- (220.8,65.22) .. controls (215.46,65.23) and (212.79,63.38) .. (212.78,59.64) .. controls (212.79,63.38) and (210.14,65.25) .. (204.8,65.26)(207.2,65.26) -- (158.38,65.38) .. controls (154.65,65.39) and (152.79,67.26) .. (152.8,71);
\draw [line width=1.2] (488.54,72.66) .. controls (488.54,68.93) and (486.69,66.06) .. (482.95,66.05) -- (388.75,65.74) .. controls (383.42,65.72) and (380.76,63.85) .. (380.78,60.11) .. controls (380.76,63.85) and (378.09,65.7) .. (372.75,65.68)(375.15,65.69) -- (278.56,65.37) .. controls (274.82,65.36) and (272.95,67.22) .. (272.94,70.95);

\draw (5.6,75.52) node [anchor=north west][inner sep=0.6pt] [font=\Large] {$K_{1}^{'}$};
\draw (4.8,136.32) node [anchor=north west][inner sep=0.6pt] [font=\Large] {$K_{2}^{'}$};
\draw (4.8,243.52) node [anchor=north west][inner sep=0.6pt] [font=\Large] {$K_{3}^{'}$};
\draw (4.8,411.52) node [anchor=north west][inner sep=0.6pt] [font=\Large] {$K_{4}^{'}$};
\draw (38.15,45.16) node [anchor=north west][inner sep=0.6pt] [font=\Large] {$L_{1}^{'}$};
\draw (373.6,44.32) node [anchor=north west][inner sep=0.6pt] [font=\Large] {$L_{4}^{'}$};
\draw (98.4,43.52) node [anchor=north west][inner sep=0.6pt] [font=\Large] {$L_{2}^{'}$};
\draw (205.6,43.52) node [anchor=north west][inner sep=0.6pt] [font=\Large] {$L_{3}^{'}$};
\draw (45.62,116.94) node [font=\Large] {$a_{2,1}$};
\draw (120.82,271.2) node [font=\Large] {$a_{4,3}$};
\draw (74.42,253.6) node [font=\Large] {$a_{4,2}$};
\draw (117.6,163.52) node [anchor=north west][inner sep=0.6pt] [font=\Large] {$a_{3,3}$};
\draw (34.94,157.52) node [anchor=north west][inner sep=0.6pt] [font=\Large] {$a_{3,1}$};
\draw (61.6,167.52) node [anchor=north west][inner sep=0.6pt] [font=\Large] {$a_{3,2}$};
\draw (205.18,85.59) node [font=\Large] {$a_{1,4}$};
\draw (124.55,84.4) node [font=\Large] {$a_{1,3}$};
\draw (34.4,382.72) node [anchor=north west][inner sep=0.6pt] [font=\Large] {$a_{5,1}$};
\draw (349.18,86.81) node [font=\Large] {$a_{1,5}$};
\draw (33.6,239.52) node [anchor=north west][inner sep=0.6pt] [font=\Large] {$a_{4,1}$};
\draw (121.6,106.72) node [anchor=north west][inner sep=0.6pt] [font=\Large] {$a_{2,3}$};
\draw (76.02,116.94) node [font=\Large] {$a_{2,2}$};
\draw (76.02,84.8) node [font=\Large] {$a_{1,2}$};
\draw (368.8,154.72) node [anchor=north west][inner sep=0.6pt] [font=\Large] {$a_{3,5}$};
\draw (401.6,241.92) node [anchor=north west][inner sep=0.6pt] [font=\Large] {$a_{4,5}$};
\draw (231.22,270.4) node [font=\Large] {$a_{4,4}$};
\draw (357.62,112) node [font=\Large] {$a_{2,5}$};
\draw (221.6,152.32) node [anchor=north west][inner sep=0.6pt] [font=\Large] {$a_{3,4}$};
\draw (216.82,112.8) node [font=\Large] {$a_{2,4}$};
\draw (111.2,408.32) node [anchor=north west][inner sep=0.6pt] [font=\Large] {$a_{5,3}$};
\draw (400,439.52) node [anchor=north west][inner sep=0.6pt] [font=\Large] {$a_{5,5}$};
\draw (204.8,441.92) node [anchor=north west][inner sep=0.6pt] [font=\Large] {$a_{5,4}$};
\draw (60.8,391.52) node [anchor=north west][inner sep=0.6pt] [font=\Large] {$a_{5,2}$};
\draw (498.1,535.05) node [anchor=north west][inner sep=0.6pt] [font=\LARGE, rotate=-45] {$\cdots$};

\end{tikzpicture}

}
\end{center}

Setting
\[
y_i=\sum_{j=1}^{i-1}\big(a_{i,j}+a_{j,i}\big)+a_{i,i},\;\;\;\;\;\;i=2,3,\dots,
\]
we have
\begin{align*}
\Vert x_{k_i}-y_i\Vert_{\mathcal C_E}&= \norm{\sum_{l=1}^{k_i-1}\big(c_{k,l}+c_{l,k}\big)+c_{k_i,k_i}  -\left(\sum_{j=1}^{i-1}\big(a_{i,j}+a_{j,i}\big)+a_{i,i} \right)}_{\mathcal C_E}\\
&= \norm{\sum_{l=1}^{k_{i-1} }\big(c_{k,l}+c_{l,k}\big)-\big(a_{i,j}+a_{j,i}\big)  +\left( \sum_{l=k_{i-1}+1}^{k_{i} -1 }\big(c_{k,l}+c_{l,k}\big)  +c_{k_i,k_i}- a_{i,i} \right)}_{\mathcal C_E}\\
&\leq\sum_{l=1}^{k_{i-1}}
\kappa^{l+1}\big(\Vert c_{k_i,l}-b_{k_i,l}\Vert_{\mathcal C_E}+\Vert c_{l.k_i}-b_{l,k_i}\Vert_{\mathcal C_E}\big)\\
&\leq\sum_{l=1}^{k_{i-1}}2\kappa^{l+1}\delta_i\delta_l \leq\varepsilon_i
\end{align*}
for every $i\geq 2$. By Remark \ref{rembl}, this completes the proof. 
\end{proof}

By the standard perturbation argument, we can obtain the following corollary.

\begin{corollary}
Let $E$ be a separable quasi-Banach symmetric space. Suppose that $\{x_k\}_{k=1}^\infty$ is a bounded $\sigma$-weakly  null  sequence   in $\mathcal C_E$ with $p_{[\bigcup_{i=1}^\infty K_i]}x_kp_{[\bigcup_{i=1}^\infty L_i]}=x_k$ for every $k$. For a given sequence $\{\varepsilon_i\}_{i=2}^\infty$ of positive numbers, 
there exist two increasing sequences $\{k_i\}_{i=2}^\infty$ and $\{n_i\}_{i=1}^\infty$ of positive integers, and a sequence $\{y_i\}_{i=2}^\infty$ in $\mathcal C_E$ having the form
\[
y_i=\sum_{j=1}^{i-1}\big(a_{i,j}+a_{j,i}\big)+a_{i,i}\;\;\;\;\;\;i=2,3,\dots,
\]
such that $\Vert x_{k_i}-y_i\Vert\leq\varepsilon_i$ for every $i\geq2$,
where
\begin{itemize}
\item [(1)]
$K'_i=[\bigcup_{l=n_{i-1}+1}^{n_i}K_l]$ and $L'_i=[\bigcup_{l=n_{i-1}+1}^{n_i}L_l]$ for every $i\geq1$ (put $n_0=0$),
\item [(2)]
$\{a_{i,j}\}_{(i,j)\neq (1,1)}\stackrel{\mbox{{\rm\tiny b.s}}}{\boxplus}\{K'_i\}_{i=1}^\infty\otimes\{L'_j\}_{j=1}^\infty$,
\item [(3)]
$a_{i,i}=p_{_{K'_i}}x_{k_i}p_{_{L'_i}}$ for every $i\geq2$,
\item [(4)]
$\{a_{i,j}\}_{i=j+1}^\infty\stackrel{\mbox{{\rm\tiny b.s}}}{\boxplus}\{K'_i\}_{i=2}^\infty\otimes\{L'_j\}_{j=1}^\infty\curvearrowleft\;$ are consistent for all $j\geq1$,
\item [(5)]
$\{a_{j,i}\}_{i=j+1}^\infty\stackrel{\mbox{{\rm\tiny b.s}}}{\boxplus}\{K'_j\}_{j=1}^\infty\otimes\{L'_i\}_{i=2}^\infty\curvearrowleft\;$ are consistent for all $j\geq1$.
\end{itemize}
\end{corollary}

Let $E$ be a quasi-Banach symmetric sequence space. 
The quasi-norm $\norm{\cdot}_E$ on $E$ is called order continuous if $\norm{x_\alpha}_E\downarrow_\alpha 0$ whenever $\{x_\alpha \}_\alpha$ is a downward directed net in $E$ satisfying $x_\alpha \downarrow 0$, equivalently, $E$ is separable\cite[Theorem 3.17]{NP}. 
 If for every upward directed sequence $\{a_\beta \}$ in $E^+$, satisfying $\sup_\beta \norm{a_\beta }_E <\infty $, there exists $a\in E^+$ such that $a_\beta \uparrow a$ in $E$ and $\norm{a}_E=\sup_\beta\norm{a_\beta }_E$, then $E$ is said to have the Fatou property \cite[Definition 5.3.7]{DPS}. 
 
A  Banach  symmetric sequence space $E$ is called a  Kantorovich--Banach space (KB-space for short) if it is separable  and has the Fatou property. 
 It is well-known that  
a symmetric sequence space $E$ is a KB-space if and only if 
 $c_0$ is not isomorphic to a subspace of $E$ (see e.g. \cite[Theorem 2.2]{BVL} or \cite[Theorem 5.9.6]{DPS}). 
Below, we incorporate a proof for the case of  quasi-Banach symmetric sequence spaces.
\begin{proposition}\label{KB c0}
Let $E$ be a quasi-Banach symmetric sequence space. 
Then, 
$E$ has the Fatou property and order continuous norm if and only if $c_0$ is not isomorphic to a subspace of $E$.
\end{proposition}
\begin{proof}
``The  implication $(\Rightarrow)$''.  
Since $E$ is separable, it follows that $E$ does not coincide with $\ell_\infty$, i.e., $E\subset c_0$. 
Assume by contradiction that there exists an isomorphism $T$ from  $c_0$ into  $E$. 
By \cite[Lemma 2.1]{KR07}, the standard basis of $c_0$ is equivalent to a block basis $(y_i)_{i=1}^\infty $ of the standard basis of $E$ (for convenience, we denote by $y_i=T(e^{c_0}_i)$). 
By the Fatou property of $E$, it follows that the unit ball of $E$ is closed in $\ell_\infty$ with respect to the pointwise convergence (see e.g. \cite[p.5]{HSZ25} or \cite[Theorem 5.1.10]{DPS}). 
Hence,  for any $\alpha=(\alpha_i)_{i=1}^\infty \in \ell_\infty $,
we have 
\begin{eqnarray*}\norm{T^{-1}} \norm{\alpha}_{\ell_\infty }&=&\sup_{1\le n <\infty } \norm{T^{-1}} 
\norm{ \sum_{i=1}^n \alpha_i  e_i^{c_0}  }_{c_0}  \le \sup_{1\le n <\infty }
\norm{ \sum_{i=1}^n \alpha_i T(e_i^{c_0}) }_E \\
&=&  \sup_{1\le n <\infty }
\norm{ \sum_{i=1}^n \alpha_i y_i }_E\stackrel{\tiny\rm Fatou~property}{=} \norm{ \sum_{i=1}^\infty \alpha_i y_i }_E  \\
&\stackrel{\tiny\rm Fatou~property}{=}& \sup_{1\le n <\infty }
\norm{ \sum_{i=1}^n \alpha_i y_i }_E = \sup_{1\le n <\infty }
\norm{ \sum_{i=1}^n \alpha_i T(e_i^{c_0}) }_E\\
&\le &\norm{T} \sup_{1\le n <\infty }
\norm{ \sum_{i=1}^n \alpha_i  e_i^{c_0}  }_{c_0} = \norm{T}  \sup_{1\le n <\infty }
\norm{ \alpha  }_{\ell_\infty },
\end{eqnarray*}
which implies that $\ell_\infty $ embeds into $E$. This contradicts the separability of $E$.

``The implication $(\Leftarrow)$''.   
Assume that $E$ is not separable, i.e., there exists an element $0\le  x \in E$ which is not order continuous\cite[Theorem 3.17]{NP}. In other words, there exists a positive number $\delta$ and  a sequence of mutually orthogonal projections $\{p_n\}_{n=1}^\infty $ in $\cB(H)$ such that $$\norm{p_n x }_E \ge \delta$$
for all $n\ge 1$.
For any $\alpha=(\alpha_1,\alpha_2,\cdots)\in c_0$, we have 
$$\norm{\alpha}_{c_0}  \le \frac{1}{\delta }\norm{ \sum_{n=1}^\infty \alpha_n p_n x }_E\le \frac{1}{\delta } \norm{ \sup_{n\ge 1}|\alpha_n|  x   }_E = \frac{\norm{\alpha}_{c_0}}{\delta } \norm{    x   }_E ,  $$
which is a contradiction. Hence, $E$ is separable. 
Now, it suffices to prove that for any increasing sequence $\{x_n\}_{n\ge 1}$ of positive elements in $E$ with $\norm{x_n}_E=1$, there exists $x\in E$ with $x_n \uparrow x$. 
Assume by contradiction that $x\not\in E$. 
There exists $\delta>0$ such that for any $N>0$, $\sup_{m\ge 1}\norm{ x_{N+m} -x_N  }_E >\delta $ (otherwise, for any $n$, there exists a number $N_n>0$ such that $\sup_{m\ge 1}\norm{ x_{N_n+m} -x_{N_n}  }_E \le \frac{1}{2^n} $, which implies that $\{x_n\}_{n\ge 1}$ converges to $x$ in norm). 
Hence, there exists a sequence $\{N_n\}_{n\ge 1}$ such that 
$$
\norm{ x_{N_{n+1}} -x_N  }_E >\delta.$$
Hence, for any finite sequence $\alpha=(\alpha_1,\alpha_2,\cdots \alpha_k )$, 
we have 
\begin{align*}
\norm{\alpha}_{c_0}\delta
&\le 
\sup _{1\le n \le k} \alpha_n \norm{ x_{N_{n+1}} -x_{N_n}  }_E\\
&\le 
\norm{ \sum_{n=1}^k  \alpha_n (x_{N_{n+1}} -x_{N_n} ) }_E \le \norm{ \sum_{n=1}^k 
\norm{\alpha}_{c_0} (x_{N_{n+1}} -x_{N_n} ) 
}_E \\
&= \norm{\alpha}_{c_0} \norm{ \sum_{n=1}^k  (x_{N_{n+1}} -x_{N_n} ) 
}_E
\le \norm{\alpha}_{c_0}\sup_{1\le n\le  k } \norm{x_{N_n} 
}_E=\norm{\alpha}_{c_0},
\end{align*}
which contradicts the assumption that $c_0$ does not embed into $E$.
\end{proof}

The following lemma was stated in  \cite[p. 310, Remark]{Arazy4}. Below, we include a short proof for completeness. 
\begin{lemma}\label{3d}
Suppose that $E$ is a quasi-Banach symmetric sequence space with $c_0\not\hookrightarrow E$. Assume that $\{a_k\}_{k=1}^\infty$ is a sequence in $\mathcal C_E$ with disjoint right supports (or  disjoint left supports) and
\[
\sup_{n\geq1}\bigg\Vert\sum_{k=1}^na_k\bigg\Vert_{\mathcal C_E}<\infty.
\]
Then, we have 
\[
\lim_{m\to\infty}\sup_{n\geq0}\bigg\Vert\sum_{k=m}^{m+n}a_k\bigg\Vert_{\mathcal C_E}=0
\]
\end{lemma}

\begin{proof}
Assume the contrary. There exist a positive number $\delta$ and a sequence $\{k_n\}_{n=0}^\infty$ of positive integers such that
\[
\Bigg\Vert\sum_{j=k_{n-1}+1}^{k_n}a_j\Bigg\Vert_{\mathcal C_E}>\delta\;\;\;\;\;\;n=1,2,\dots,
\]
Define  $b_n:= \sum_{j=k_{n-1}+1}^{k_n}a_j$ for every $n$.
There exists  an increasing sequence $\{G_n\}_{n=1}^\infty$ of finite-dimensional subspaces such that $\Vert p_{_{G_n}}b_n\Vert>\kappa^{-1}\delta$ for every $n$ (see e.g. \cite[Theorem 5.5.12]{DPS}). Using Corollary \ref{3a}, we deduce that $\lim_{n\to\infty}\Vert p_{_G}b_n\Vert=0$ 
for each finite-dimensional subspace $G$ of $H$. Indeed, 
otherwise, there exists a subsequence $\{p_Gb_{n_k}\}_{k=1}^\infty$ of $\{p_Gb_n\}_{n=1}^\infty$, which is isomorphic to the unit basis of $\ell_2$. This contradicts the $\sup_{m\geq1}\Vert\sum_{k=1}^mb_k\Vert_{\mathcal C_E}<\infty$.
Therefore,  there exists  an increasing sequence $\{n_m\}_{m=0}^\infty$ of positive integers such that  
\[
\big\Vert p_{_{{G_{n_{m-1}}}^\perp\cap G_{n_m}}}b_{n_m}\big\Vert_{\mathcal C_E}>\kappa^{-2}\delta.
\]
Put $c_m:=p_{_{{G_{n_{m-1}}}^\perp\cap G_{n_m}}}b_{n_m}$. Note that $\{c_m\}_{m=1}^\infty$ is a seminormalized unconditional basic sequence with $\sup_{N\geq1}\Vert\sum_{m=1}^N c_m\Vert_{\mathcal C_E}<\infty$, and thus $\{c_m\}_{m=1}^\infty$ is equivalent to the unit vector basis of $c_0$. 
However, 
the space 
$[c_m]_{m=1}^\infty$ is isomorphic to a subspace of $E$ (see e.g. \cite[Proposition 3.3]{HSS}), which is a contradiction. 
\end{proof}

\begin{remark}
Suppose that $E$ is a separable Banach symmetric sequence space and $1\leq p\neq2<\infty$. Note that $\mathcal C_E$ does not contain a subspace isomorphic to $\ell_p$ (resp. $c_0$) if and only if $E$ does not contain a subspace isomorphic to $\ell_p$ (resp. $c_0$)\cite[Corollary 3.2.]{Arazy4}.
\end{remark}
From now on, we only 
 consider  separable Banach symmetric sequence spaces $E$  for the sake of simplicity, 
although most of the results in this section remain valid for  quasi-Banach symmetric sequence spaces.

The following theorem is a generalization of   \cite[Corollary 2.8.]{Arazy4}.
\begin{theorem}\label{3e}
Suppose that $E$ is a  quasi-Banach symmetric sequence space with $c_0\not\hookrightarrow E$, and $\{x_k\}_{k=2}^\infty$ be a bounded sequence in $\mathcal C_E$ with
\[
x_k=p_{_{[\bigcup_{l=1}^kK_l]}}x_kp_{_{[\bigcup_{l=1}^kL_l]}}-p_{_{[\bigcup_{l=1}^{k-1}K_l]}}x_kp_{_{[\bigcup_{l=1}^{k-1}L_l]}}\;\;\;\;\;\;k=2,3,\cdots.
\]
For an arbitrary  sequence of positive numbers $\{\varepsilon_i\}_{i=2}^\infty$,  there exist an increasing sequence $\{k_i\}_{i=1}^\infty$ of positive integers and a sequence $\{y_i\}_{i=2}^\infty$ in $\mathcal C_E$ having the form
\[
y_i=\sum_{j=1}^{i-1}\big(a_{i,j}+a_{j,i}\big)+a_{i,i}, \;\;\;\;\;\;i=2,3,\dots,
\]
such that $\Vert x_{k_i}-y_i\Vert_{\mathcal C_E}\leq\varepsilon_i$ for every $i\geq2$, where
\begin{itemize}
\item [(1)]
$K'_i=[\bigcup_{l=k_{i-1}+1}^{k_i}K_l]$ and $L'_i=[\bigcup_{l=k_{i-1}+1}^{k_i}L_l]$ for every $i\geq1$ (put $k_0=0$),
\item [(2)]
$\{a_{i,j}\}_{(i,j)\neq (1,1)}\stackrel{\mbox{{\rm\tiny b.s}}}{\boxplus}\{K'_i\}_{i=1}^\infty\otimes\{L'_j\}_{j=1}^\infty$,
\item [(3)]
$a_{i,i}=p_{_{K'_i}}x_{k_i}p_{_{L'_i}}$ for every $i\geq2$,
\item [(4)]
$\{a_{i,j}\}_{i=j+1}^\infty\stackrel{\mbox{{\rm\tiny b.s}}}{\boxplus}\{K'_i\}_{i=2}^\infty\otimes\{L'_j\}_{j=1}^\infty\curvearrowleft d_j$ are consistent for all $j\geq1$, and \[
\Vert d_j\Vert_{\mathcal C_E}\leq\varepsilon_j\Vert d_1\Vert_{\mathcal C_E}\;\;\;\;\;\;\mbox{for every}\; j\geq 2,
\]
\item [(5)]
$\{a_{j,i}\}_{i=j+1}^\infty\stackrel{\mbox{{\rm\tiny b.s}}}{\boxplus}\{K'_j\}_{j=1}^\infty\otimes\{L'_i\}_{i=2}^\infty\curvearrowleft d'_j$ are consistent for all $j\geq 1$, and
\[
\Vert d'_j\Vert_{\mathcal C_E}\leq\varepsilon_j\Vert d'_1\Vert_{\mathcal C_E}\;\;\;\;\;\;\mbox{for every}\; j\geq 2.
\]
\end{itemize}
\end{theorem}
\begin{proof}

Let $\{\delta_i\}_{i=1}^\infty$ be a sequence of positive numbers which are sufficiently small, and let  $\{k_i\}_{i=1}^\infty$, $\{y_i\}_{i=2}^\infty$ in $\mathcal C_E$ and $\{a_{i,j}\}_{(i,j)\neq(1,1)}$  be as in the proof of  Theorem~\ref{3c}. The  argument used in the proof  of Theorem~\ref{3c}  shows that 
\begin{enumerate}
    \item there exists a sequence of isometries $\big\{w_i:K'_i\to\ell_2\big\}_{i=2}^\infty$, a sequence of isometries $\big\{\omega_l:L_l\to\ell_2\big\}_{l=1}^\infty$, and a sequence $\{b_l\}_{l=1}^\infty$ of in $\mathcal C_E(\ell_2)$ such that
\[
w_ib_{k_i,l}{\omega_l}^\ast=b_l
\]
for every $1\leq l\leq k_{i-1}<\infty$;
    \item there exists a sequence of isometries $\big\{w'_l:K_l\to\ell_2\big\}_{l=1}^\infty$, a sequence of isometries $\big\{\{\omega'_i:L'_i\to\ell_2\big\}_{i=2}^\infty$, and a sequence $\{b'_l\}_{l=1}^\infty$ of in $\mathcal C_E(\ell_2)$ such that
\[
w'_lb_{l,k_i}{\omega'_i}^\ast=b'_l
\]
for every $1\leq l\leq k_{i-1}<\infty$. 
\end{enumerate}
  Note that
\[
\sup_{n\geq1}\bigg\Vert\sum_{l=1}^nb_l\otimes(\cdot|e_l)e_1\bigg\Vert_{\mathcal C_E}=\sup_{n\geq1}\bigg\Vert\sum_{l=1}^n b_{k_{n+1},l}\bigg\Vert_{\mathcal C_E}\leq\sup_{n\geq1}\Vert y_n\Vert_{\mathcal C_E}<\infty
\]
and
\[
\sup_{n\geq1}\bigg\Vert\sum_{l=1}^nb'_l\otimes(\cdot|e_1)e_l\bigg\Vert_{\mathcal C_E}=\sup_{n\geq1}\bigg\Vert\sum_{l=1}^n b_{l,k_{n+1}}\bigg\Vert_{\mathcal C_E}\leq\sup_{n\geq1}\Vert y_n\Vert_{\mathcal C_E}<\infty. 
\]
By Lemma \ref{3d},  passing to a subsequence of $\{k_i\}_{i=1}^\infty$ if necessary, we may assume that
\[
\bigg\Vert\sum_{l=k_{i-1}+1}^{k_i}b_l\otimes(\cdot|e_l)e_1\bigg\Vert_{\mathcal C_E}\leq\varepsilon_i\bigg\Vert\sum_{l=1}^{k_1}b_l\otimes(\cdot|e_l)e_1\bigg\Vert_{\mathcal C_E}
\]
and
\[
\bigg\Vert\sum_{l=k_{i-1}+1}^{k_i}b'_l\otimes(\cdot|e_1)e_l\bigg\Vert_{\mathcal C_E}\leq\varepsilon_i\bigg\Vert\sum_{l=1}^{k_1}b'_l\otimes(\cdot|e_1)e_l\bigg\Vert_{\mathcal C_E}.
\]

Recall that $\{a_{i,j}\}_{(i,j)\neq(1,1)}$ are defined by  \eqref{bl1}, \eqref{bl2} and \eqref{bl3}. It suffices to define
\[
d_j=\sum_{l=k_{j-1}+1}^{k_j}b_l\otimes(\cdot|e_l)e_1\;\;\;\;\;\;\;\;{\rm for\;every}\;j\geq 1,
\]
and two sequences of isometries
\[
u_i:\xi\in K'_i\longmapsto w_i(\xi)\otimes e_1\in\ell_2\otimes_{_2}\ell_2\;\;\;\;\;\;\;\;{\rm for\;every}\;i\geq 2.
\]
and
\[
v_j:\sum_{l=k_{j-1}+1}^{k_j}\xi_l\in L'_j=\sum_{l=k_{j-1}+1}^{k_j}L_l\longmapsto\sum_{l=k_{j-1}+1}^{k_j}\omega_l(\xi_l)\otimes e_l\in\ell_2\otimes_{_2}\ell_2.
\]
for every $i\geq 2$. Consequently, by  Remark \ref{rem} (2), 
\[
u_ia_{i,j}{v_j}^\ast=d_j\;\;\;\;\;\;\;\;{\rm for\;every}\;1\leq i<j<\infty.
\]
Similarly, we can define   $\{d'_i\}_{i=1}^\infty$ $\{u'_i\}_{i=1}^\infty$ and $\{v'_j\}_{j=2}^\infty$. This completes the proof.
\end{proof}

\begin{corollary}\label{3e'}
Let  $\varepsilon>0$.   
Under the assumptions and conclusions of Theorem \ref{3e}, and further assume $\{x_k\}_{k=2}^\infty$ is semi-normalized in $\mathcal C_E$, and choose $\{\varepsilon_i\}_{i=2}^\infty$ with $\varepsilon_i$'s small enough. Let $a_i=a_{i,1}$, $b_i=a_{1,i}$, $c_i=a_{i,i}$, and
\[
z_i=a_i+b_i+c_i\;\;\;\;\;\;\;\;{\rm for\;every}\;i\geq 2.
\]
  We have 
\begin{equation}\label{st1}
\bigg\Vert\sum_{i=2}^\infty t_i(x_{k_i}-z_i)\bigg\Vert_{\mathcal C_E}\leq  \varepsilon\bigg\Vert\sum_{i=2}^\infty t_ix_{k_i}\bigg\Vert_{\mathcal C_E}
\end{equation}
 for every finitely nonzero sequence of scalars $\{t_i\}_{i=2}^\infty$. 
\end{corollary}
\begin{proof}
Since the proofs for Banach spaces and quasi-Banach spaces differ only slightly, we may assume without loss of generality that $\kappa=1$.

Fix a positive number $\delta$ such that $\Vert x_k\Vert_{\mathcal C_E}\geq\delta>0 $ for every $k\geq 2$.
For every finitely nonzero sequence of scalars $\{t_i\}_{i=2}^\infty$, we have   
\[
\bigg\Vert\sum_{i=2}^\infty t_i(x_{k_i}-y_i)\bigg\Vert_{\mathcal C_E}\leq\sum_{i=2}^\infty\varepsilon_i\vert t_i\vert\leq\sum_{i=2}^\infty\varepsilon_i\cdot\sup_{i\geq2}\vert t_i\vert\leq\bigg(\sum_{i=2}^\infty\varepsilon_i\bigg)\frac{2}{\delta}\bigg\Vert\sum_{i=2}^\infty t_ix_{k_i}\bigg\Vert_{\mathcal C_E}.
\]
On the other hand, since 
\begin{eqnarray*}
\bigg\Vert\sum_{i=2}^\infty t_iz_i\bigg\Vert_{\mathcal C_E}&\stackrel{\eqref{projection inequality}}{\geq}&\bigg\Vert (P_{[\cup_{i=2}^\infty K'_i],L'_1}+P_{K'_1,[\cup_{i=2}^\infty L'_i]})\sum_{i=2}^\infty t_iz_i\bigg\Vert_{\mathcal C_E}\\&=&\bigg\Vert\sum_{i=2}^\infty t_i(a_i+c_i)\bigg\Vert_{\mathcal C_E}\\
&=&\bigg\Vert\sum_{i=2}^\infty t_i(a_{i,1}+a_{1,i})\bigg\Vert_{\mathcal C_E}\\
&\stackrel{\tiny \mbox{Th \ref{hi}}}{=}&\left\Vert\left(
\begin{smallmatrix}
0&d'_1\\d_1&0
\end{smallmatrix}\right)
\right\Vert_{\mathcal C_E}\Big(\sum_{i=2}^\infty\vert t_i\vert^2\Big)^{1/2},
\end{eqnarray*}
it follows that 
\begin{align*}
\bigg\Vert\sum_{i=2}^\infty t_i(y_i-z_i)\bigg\Vert_{\mathcal C_E}&=\bigg\Vert\sum_{i=3}^\infty t_i\sum_{j=2}^{i-1}\big(a_{i,j}+a_{j,i}\big)\bigg\Vert_{\mathcal C_E}\\
&\leq\sum_{j=2}^\infty\bigg\Vert\sum_{i=j+1}^\infty t_i(a_{i,j}+a_{j,i})\bigg\Vert_{\mathcal C_E}\\
&\leq\sum_{j=2}^\infty\big(\Vert d_j\Vert_{C_E}+\Vert d'_j\Vert_{\mathcal C_E}\big)\Big(\sum_{i=2}^\infty\vert t_i\vert^2\Big)^{1/2}\\
&\leq\sum_{j=2}^\infty \varepsilon_j \big(\Vert d_1\Vert_{\mathcal C_E}+\Vert d'_1\Vert_{\mathcal C_E}\big)\Big(\sum_{i=2}^\infty\vert t_i\vert^2\Big)^{1/2}\\
&\leq2\bigg(\sum_{j=2}^\infty\varepsilon_j\bigg)\bigg\Vert\sum_{i=2}^\infty t_iz_i\bigg\Vert_{\mathcal C_E}.
\end{align*}

Let $\varepsilon>0$. 
For a suitable chosen sequence    $\{\varepsilon_i\}_{i=2}^\infty$,  we have \eqref{st1}. Indeed, 
let 
\[\varepsilon' := \max\left\{ 
\bigg(\sum_{i=2}^\infty\varepsilon_i\bigg)\frac{2}{\delta} ,~2\bigg(\sum_{j=2}^\infty\varepsilon_j\bigg) \right\}.
\]
letting
\[
X=\sum_{i=2}^\infty t_ix_{k_i},
~
Y=\sum_{i=2}^\infty t_iy_i,
~
 Z=\sum_{i=2}^\infty t_iz_i,
\]
we have
\[
\Vert X-Y\Vert\leq\varepsilon'\Vert X\Vert{\rm \;\;\;\;and\;\;\;\;}\Vert Y-Z\Vert\leq\varepsilon'\Vert Z\Vert.
\]
Then
\[
(1-\varepsilon')\Vert Z\Vert\leq\Vert Y\Vert\leq(1+\varepsilon')\Vert X\Vert,
\]
and finally
\[
\Vert X-Z\Vert\leq\Vert X-Y\Vert+\Vert Y-Z\Vert\leq\varepsilon'\Vert X\Vert+\varepsilon'\Vert Z\Vert\leq\varepsilon'(1+\frac{1+\varepsilon'}{1-\varepsilon'})\Vert X\Vert.
\]
If $\{\varepsilon_j\}$ is such that 
  $\varepsilon'(1+\frac{1+\varepsilon'}{1-\varepsilon'})<\varepsilon$, then the inequality   \eqref{st1} holds. 
\end{proof}

By a  standard perturbation argument, we   obtain the following two  corollaries.

\begin{corollary}
Suppose that $E$ is a quasi-Banach symmetric sequence space with $c_0\not\hookrightarrow E$, and $\{x_k\}_{k=2}^\infty$ is a bounded $\sigma$-weakly null sequence  (i.e. converges to $0$ in the $\sigma$-weak topology of $\cB(H)$) in $\mathcal C_E$ with $p_{[\bigcup_{i=1}^\infty K_i]}x_kp_{[\bigcup_{i=1}^\infty L_i]}=x_k$ for every $k$. For a given a sequence of positive numbers $\{\varepsilon_i\}_{i=2}^\infty$, there exist two increasing sequences $\{k_i\}_{i=2}^\infty$ and $\{n_i\}_{i=1}^\infty$ of positive integers, and a sequence $\{y_i\}_{i=2}^\infty$ in $\mathcal C_E$, having the form
\[
y_i=\sum_{j=1}^{i-1}\big(a_{i,j}+a_{j,i}\big)+a_{i,i}\;\;\;\;\;\;i=2,3,\dots,
\]
such that $\Vert x_{k_i}-y_i\Vert_{\mathcal C_E}\leq\varepsilon_i$ for every $i\geq2$, where 
\begin{itemize}
\item [(1)]
$K'_i=[\bigcup_{l=n_{i-1}+1}^{n_i}K_l]$ and $L'_i=[\bigcup_{l=n_{i-1}+1}^{n_i}L_l]$ for every $i\geq1$ (put $n_0=0$),
\item [(2)]
$\{a_{i,j}\}_{(i,j)\neq (1,1)}\stackrel{\mbox{{\rm\tiny b.s}}}{\boxplus}\{K'_i\}_{i=1}^\infty\otimes\{L'_j\}_{j=1}^\infty$,
\item [(3)]
$a_{i,i}=p_{_{K'_i}}x_{k_i}p_{_{L'_i}}$ for every $i\geq2$,
\item [(4)]
$\{a_{i,j}\}_{i=j+1}^\infty\stackrel{\mbox{{\rm\tiny b.s}}}{\boxplus}\{K'_i\}_{i=2}^\infty\otimes\{L'_j\}_{j=1}^\infty\curvearrowleft d_j$ are consistent for all $j\geq1$, and \[
\Vert d_j\Vert_{\mathcal C_E}\leq\varepsilon_j\Vert d_1\Vert_{C_E}\;\;\;\;\;\;\mbox{for every}\; j\geq 2,
\]
\item [(5)]
$\{a_{j,i}\}_{i=j+1}^\infty\stackrel{\mbox{{\rm\tiny b.s}}}{\boxplus}\{K'_j\}_{j=1}^\infty\otimes\{L'_i\}_{i=2}^\infty\curvearrowleft d'_j$ are consistent for all $j\geq 1$, and
\[
\Vert d'_j\Vert_{\mathcal C_E}\leq\varepsilon_j\Vert d'_1\Vert_{\mathcal C_E}\;\;\;\;\;\;\mbox{for every}\; j\geq 2.
\]
\end{itemize}
\end{corollary}

\begin{corollary}
Suppose that $E$ is a quasi-Banach symmetric sequence space  with $c_0\not\hookrightarrow E$, and $\{x_k\}_{k=2}^\infty$ is a semi-normalized $\sigma$-weakly null sequence in $\mathcal C_E$ with $p_{[\bigcup_{i=1}^\infty K_i]}x_kp_{[\bigcup_{i=1}^\infty L_i]}=x_k$ for every $k$.
For any given $\varepsilon>0$, 
there exist two increasing sequences $\{k_i\}_{i=2}^\infty$ and $\{n_i\}_{i=1}^\infty$ of positive integers, and a sequence $\{y_i\}_{i=2}^\infty$ in $\mathcal C_E$ having the form
\[
y_i=a_i+b_i+c_i\;\;\;\;\;\;i=2,3,\dots,
\]
such that for every finitely nonzero sequence of scalars $\{t_i\}_{i=2}^\infty$, we have
\[
\bigg\Vert\sum_{i=2}^\infty t_i(x_{k_i}-y_i)\bigg\Vert_{\mathcal C_E}\leq\varepsilon\bigg\Vert\sum_{i=2}^\infty t_ix_{k_i}\bigg\Vert_{\mathcal C_E},
\]
where
\begin{itemize}
\item [(1)]
$K'_i=[\bigcup_{l=n_{i-1}+1}^{n_i}K_l]$ and $L'_i=[\bigcup_{l=n_{i-1}+1}^{n_i}L_l]$ for every $i\geq1$ (put $n_0=0$),
\item [(2)]
$\{a_i\}_{i=2}^\infty\stackrel{\mbox{{\rm\tiny b.s}}}{\boxplus}\{K'_i\}_{i=2}^\infty\otimes L'_1\curvearrowleft\;$,
\item [(3)]
$\{b_i\}_{i=2}^\infty\stackrel{\mbox{{\rm\tiny b.s}}}{\boxplus}K'_1\otimes\{L'_i\}_{i=2}^\infty\curvearrowleft\;$,
\item [(4)]
$c_i=p_{_{K'_i}}x_{k_i}p_{_{L'_i}}$ for every $i\geq2$.
\end{itemize}
\end{corollary}

\section{Two standard results  in quasi-Banach spaces}
The following is a  generalization of a  classical result for Banach spaces (see e.g. \cite[p.47]{LT}). The proof follows from  a similar approach with only minor modifications.
\begin{fact}\label{qbssd}
Let $X$ be a quasi-Banach space and $\{X_n\}_{n=1}^\infty$ be a Schauder decomposition of $X$.
Define partial-sum maps $S_n:X\to X$ by
\[
S_nx=\sum_{k=1}^nx_k,
\]
where $x=\sum_{k=1}^\infty x_k$ and $x_k\in X_k$ for each $k$.
Then
\[
\sup_{n\geq1}\Vert S_n\Vert<\infty.
\]
\end{fact}

\begin{proof}
Assume that the modulus of concavity of the quasi-norm of $X$ is not greater than $ \kappa$. Let us consider a new quasi-norm on $X$ denoted by the formula
\[
\vert\vert\vert x\vert\vert\vert=\sup_{n\geq1}\Vert S_nx\Vert,
\]
where the finiteness of $\sup_{n\geq1}\Vert S_nx\Vert$ is guaranteed by $S_nx\to x$ as $n\to\infty$. 
Therefore, we have $\norm{x}\le\kappa\vert\vert\vert x \vert\vert\vert$ for any $x\in X$. By the open mapping theorem\cite[2.12 Corollaries (b)]{Rudin3}, we only need to show $(X,\left\vert\left\vert\left\vert\cdot\right\vert\right\vert\right\vert)$ is complete.

Suppose that $(y_m)_{m=1}^\infty$ is a Cauchy sequence in $(X,\left\vert\left\vert\left\vert\cdot\right\vert\right\vert\right\vert)$. 
In particular,  $\{y_m\}_{m=1}^\infty$ converges  to some $y\in X$ in the original quasi-norm $\norm{\cdot}$, and for each fixed $k$ the sequence $\{P_ky_m\}_{m=1}^\infty$ is convergent to some $z_k\in X$ in the original norm\footnote{
Since $\Vert P_ky_m-P_ky_{m'}\Vert\leq\kappa(\Vert S_k(y_m-y_{m'})\Vert+\Vert S_{k-1}(y_m-y_{m'})\Vert)\leq2\kappa\vert\vert\vert y_m-y_{m'}\vert\vert\vert$, it follows that $\{P_ky_m\}_{m=1}^\infty$ is a Cauchy sequence in the original norm. 
}, where $P_k=S_k-S_{k-1}$ and $S_0=0$. For any $l,n\in\mathbb N^+$, we have
\begin{eqnarray*}
\bigg\Vert y-\sum_{k=1}^nz_k\bigg\Vert&\stackrel{\eqref{lim-norm}}{\leq}&\kappa\varlimsup_{m\to\infty}\Vert y_m-S_n y_m\Vert\\
&\leq&\kappa^2\varlimsup_{m\to\infty}\Vert y_m-y_l\Vert+\kappa^3\Vert y_l-S_ny_l\Vert+\kappa^3\varlimsup_{m\to\infty}\Vert S_ny_l-S_ny_m\Vert\\
&\leq& \kappa^3\Vert y_l-S_ny_l\Vert+2\kappa^3\varlimsup_{m\to\infty}\vert\vert\vert y_m-y_l\vert\vert\vert.
\end{eqnarray*}
Therefore,  by   $\Vert y_l-S_ny_l\Vert\to0$ as $n\to\infty$, we have 
\[
\varlimsup_{n\to\infty}\bigg\Vert y-\sum_{k=1}^nz_k\bigg\Vert\leq2\kappa^3\varlimsup_{m\to\infty}\vert\vert\vert y_m-y_l\vert\vert\vert.
\]
Moreover, since $ \varlimsup_{l,m\to\infty}\vert\vert\vert y_m-y_l\vert\vert\vert=0$ and $l$ is arbitrarily taken, it follows that $\varlimsup_{n\to\infty}\bigg\Vert y-\sum_{k=1}^nz_k\bigg\Vert=0$.
Thus,  
$\sum_{k=1}^\infty z_k=y$, by the uniqueness of the expansion of $y$ with respect to the decomposition,  we have $S_ny=\sum_{k=1}^nz_k$ for each $n$. Now,
\[
\vert\vert\vert y_m-y\vert\vert\vert=\sup_{n\geq1}\Vert S_ny_m-S_ny\Vert\leq\sup_{n\geq1}\varlimsup_{l\to\infty}\kappa\Vert S_ny_m-S_ny_l\Vert\leq\kappa\varlimsup_{l\to\infty}\vert\vert\vert y_m-y_l\vert\vert\vert.
\]
This implies that   $\lim_{m\to\infty}\vert\vert\vert y_m-y\vert\vert\vert=0$ and $(X,\vert\vert\vert\cdot\vert\vert\vert)$ is complete.
\end{proof}

The following fact is certainly known to experts in Banach space theory, which is   an immediate consequence of\cite[Proposition 2.a.2. and 2.c.5.(i)]{LT}. Indeed, this result also holds for quasi-Banach spaces, and the proof requires minor adjustments  only.
For the sake of completeness, we present the proof below.

\begin{fact}\label{lp}
Let $X_1,\dots,X_n$ be the quasi-Banach spaces. 
If $\ell_p$, $0< p<\infty$, (or $c_0$) is not isomorphic to a  subspace of any  of  $X_1,\dots,X_n$, then $\ell_p$ (or, $c_0$) is not isomorphic to a subspace of $X_1\oplus\cdots\oplus X_n$\footnote{
The direct sum here is equipped with the product topology. 
Since  $n<\infty $, it follows that  
all quasi-norms inducing   the product topology are equivalent. 
 Therefore, we do not  specify  the particular  choice of  a  quasi‑norm.

}.
\end{fact}

\begin{proof}
It suffices to prove  the case of $n=2$. By Theorem \ref{qbs-norm}, without loss of generality, we may assume that the quasi-norm of $X_1\oplus X_2$ is $r$-subadditive. Assume that $T$ is an isomorphic embedding mapping  from $\ell_p$ into $X_1\oplus X_2$. In particular,   there is a positive number $c$ such that
    \[
    \Vert T(a)\Vert \geq c^{1/r }\Vert a\Vert_p\;\;\;\;\;\;\;{\rm for\;any}\;a\in\ell_p.
    \]
Let $P$ be the projection from $X_1\oplus X_2$ onto $X_1$. Since $\ell_p\not\hookrightarrow X_1$, it follows that $PT|_{[e_{k}^{\ell_p}]_{k=n}^\infty}$ is not an isomorphic embedding for each $n$. This shows that there is an increasing sequence $\{n_i\}_{i=1}^\infty$ of positive integers, and a sequence $\{a_i\}_{i=1}^\infty\subset \ell_p$ satisfying $a_i\in [e_k^{\ell_p}]_{k=n_i+1}^{n_{i+1}}$ and $\Vert a_i\Vert_p=1$ for every $i$ such that
\[
\Vert PT(a_i)\Vert^r<\frac{c}{2^{i+1}}.
\]
Indeed, to start the induction pick $n_1=0$, and assume that  $n_1,\dots,n_i$ have been chosen. 
Since $PT|_{[e_{k}^{\ell_p}]_{k=n_i+1}^\infty}$ is not an isomorphic embedding, it follows that there is an operator $a'_i\in[e_{k}^{\ell_p}]_{k=n_i+1}^\infty$ with $\Vert a'_i\Vert_p =1 $ such that $\Vert PT(a'_i)\Vert^r<\frac{c}{2^{i+1}} \frac{1}{2^r }$.
Let $a_i'= \sum_{k=n_i+1}^\infty  a_i'(k)e^{\ell_p} _k$. 
There exists a positive integer $n_{i+1}>n_i$ such that 
$ a_i'':=  \sum_{k=n_i+1}^{n_{i+1}} a_i'(k)e^{\ell_p} _k$ satisfies that $\norm{a_i''}_{p}\ge \frac{1}{2}$ and 
 $\Vert PT(a_i'')\Vert^r<\frac{c}{2^{i+1}}  \frac{1}{2^r }$.
 Letting $a_i:=\frac{a_i''}{\norm{a_i''}_{p}}$, we have $\norm{a_i}_{p}=1$ and $\norm{PT(a_i)}^r <\frac{c}{2^{i+1}}$. 
This completes the induction.

Since $\{a_i\}_{i=1}^\infty\simeq\{e_i^{\ell_p}\}_{i=1}^\infty$, it follows that $\Vert PT|_{[a_i]_{i=1}^\infty}\Vert^r<c/2$. Consequently, we have 
\[
\Vert (I-P)T(a)\Vert^r\geq\Vert T(a)\Vert^r-\Vert PT(a)\Vert^r\geq c\Vert a\Vert_p^r-\frac{c}{2}\Vert a\Vert_p^r\geq\frac{c}{2}\Vert a\Vert_p^r
\]
for all $a\in[a_i]_{i=1}^\infty \cong \ell_p$. This implies that $\ell_p$ is isomorphic to a subspace of $X_2$, contradicting
our hypothesis.
Since  the proofs for $c_0$ and $\ell_p$ are identical, we omit the proof for the case of $c_0$.
\end{proof}

\bibliographystyle{amsalpha}

\begin{thebibliography}{99}


\bibitem{Albiac Kalton}
F.~Albiac, N.~Kalton, \emph{Topics in Banach Space Theory}, Graduate Texts in Mathematics, \textbf{233}, 2nd edn. Springer, New York (2016)

\bibitem{Apostol_lp}
C.~Apostol, \emph{Commutators on $\ell_p$ spaces}, Rev. Roum. Math.
  Appl. \textbf{17} (1972), 1513--1534.

\bibitem{Apostol_c0}
C.~Apostol, \emph{Commutators on $c_{0}$-spaces and on $\ell_1$ -spaces,}  Rev. Roum. Math. Pures Appl.  \textbf{18}  (1973), 1025--1032.

\bibitem{Arazy1}
J.~Arazy, \emph{Some remarks on interpolation theorems and the boundedness of the triangular projection in unitary matrix spaces}, Integral Equations Operator Theory. Vol. I (1978), 453--495.

\bibitem{Arazy2}
J.~Arazy, \emph{On large subspaces of the Schatten $p$-classes}, Compositio Math. \textbf{41} (1980), 297--336.

\bibitem{Arazy3}
J.~Arazy, \emph{A remark on complemented subspaces of unitary matrix spaces}, Proc. Amer. Math. Soc. \textbf{79} (1980), 601--608.

\bibitem{Arazy4}
J.~Arazy, \emph{Basic sequences, embedding, and the unique of the symmetric structure in unitary matrix spaces}, J. Funct. Anal. \textbf{40} (1981), 302--340.

\bibitem{Arazy81b}
J. Arazy, {\it   More on convergence in unitary matrix spaces},  Proc. Amer. Math. Soc. {\bf 83}  (1981), no. 1, 44--48.

\bibitem{Arazy83}
J. Arazy, {\it On stability of unitary matrix spaces},
Proc. Amer. Math. Soc. {\bf 87} (1983), no. 2, 317--321. 

\bibitem{Arazy83b}
J. Arazy, {\it 
Isomorphisms of unitary matrix spaces. Banach space theory and its applications (Bucharest, 1981),}
1--6, Lecture Notes in Math., 991, Springer, Berlin, 1983.


\bibitem{Arazy Friedman}
J.~Arazy, Y.~Friedman, \emph{The isometries of $C_p^{m,n}$ into $C_p$}, Israel J. Math. \textbf{26} (1977), 151--165.

\bibitem{Arazy Lindenstrauss}
J.~Arazy, J.~Lindenstrauss, \emph{Some linear topological properties of the space $C_p$ of operator on Hilbert space}, Compositio Math. \textbf{30} (1975), 151--165.
\bibitem{AHS}
S. Astashkin, J.  Huang, F. Sukochev, {\it Lack of isomorphic embeddings of symmetric function spaces into operator ideals},  J. Funct. Anal. {\bf 280}(5)  (2021),  Paper No. 108895, 34 pp

\bibitem{BKL2020}
K, Beanland, T. Kania, N.J. Laustsen, \emph{Closed ideals of operators on the Tsirelon and Schreier spaces}, J. Funct. Anal. \textbf{279} (2020), 108668, 28 pp.

\bibitem{BS21}
A. 
Bikchentaev, F.  Sukochev, {\it 
Inequalities for the block projection operators}, J. Funct. Anal. {\bf 280} (2021), no. 7, Paper No. 108851, 18 pp. 
\bibitem{Brown Pearcy}
A.~Brown, C.~Pearcy,  \emph{Structure of commutators of operators}, Ann. of Math. \textbf{82} (1965), 112--127.

\bibitem{BVL}
A. Bukhvalov, A. Veksler, G. Lozanovskii, {\it Banach lattices---some Banach aspects of their theory,}
Russian Math. Surveys, {\bf 34}(2) (1979), 159--212. 


\bibitem{Chen Johnson Bentuo Zheng}
D.~Chen, W.~B.~Johnson, B. Zheng, \emph{Commutators on $(\sum\ell_q)_p$}, Studia Math. \textbf{206}(2) (2011), 284--290.

\bibitem{CY}
L.~Cheng, Z. Yu, \emph{Representation of commutators on Schatten  $p$-classes}, J. Noncommut. Geom. \textbf{18} (2024), no. 2, 741--770.
\bibitem{CKS92}
V. Chilin, A. Krygin, F. Sukochev, {\it 
Extreme points of convex
fully symmetric sets of measurable operators,}
Integral Equations Operator Theory 15 (1992),
no. 2, 186--226. 


 
 \bibitem{DPS}
P. Dodds, B. de Pagter, F. Sukochev,
{\it Noncommutative Integration and Operator Theory, }
Progress in Mathematics, volume 349, Birkh\"auser, Cham, Switzerland, 2023. 


\bibitem{Dosev_l1}
D.~Dosev, \emph{Commutators on $\ell_1$}, J.   Funct. Anal. \textbf{256} (2009), 3490--3509.


\bibitem{Dosev Johnson}
D.~Dosev, W.~B.~Johnson, \emph{Commutators on $\ell_\infty$}, Bull. London Math. Soc. \textbf{42} (2010), 155--169.

\bibitem{Dosev Johnson Schechtman}
D.~Dosev, W.~B.~Johnson, G.~Schechtman, \emph{Commutators on $L_p$, $1\leq p<\infty$}, J. Amer. Math. Soc. \textbf{26} (2013), 101--127.


\bibitem{FGM1960}
I.A. Feldman I.C. Gohberg, A.S. Markus, \emph{Normally solvable operators and ideals associated with them}, Bul. Akad. \v{S}tiince RSS Moldoven. \textbf{10(76)} (1960), 51--70.



\bibitem{Friedman}
Y. Friedman, {\it  Subspace of $LC(H)$ and $C_p$}, Proc. Amer. Math. Soc. {\bf 53}(1) (1975),  117–-122.

\bibitem{GHSY}
C. Gao, J. Huang, F. Sukochev, Z. Yu,
{\it On isomorphic embeddings of noncommutative Lorentz spaces}, in preparation. 

\bibitem{Gohberg Krein}
I.~C.~Gohberg, M.~G.~Krein, \emph{Theory and applications of Volterra operators in Hilbert spaces (translated from Russian)}, Amer. Math. Soc. Translation \textbf{24}.


\bibitem{HST}
F. Hernandez, E. Semenov, P. Tradacete,
{\it Rearrangement invariant spaces with Kato property,}
Funct. Approx. Comment. Math.   50 (2014), no. 2, 215--232.

\bibitem{Hiai}
F. Hiai, {\it Log-majorizations and norm inequalities for exponential operators,}
Linear operators (Warsaw, 1994), Banach Center Publ., vol. 38, Polish Acad. Sci. Inst. Math., Warsaw,1997, pp. 119--181.

\bibitem{Holub}
J. R. Holub,  {\it On subspaces of norm ideals},
Bull. Amer. Math. Soc. {\bf 79}  (1973),  446--448.

\bibitem{HLS}
J. Huang, G. Levitina, F. Sukochev, {\it 
Completeness of symmetric $\delta$-normed spaces of $\tau$-measurable operators}, 
Studia Math. 237 (2017), no. 3, 201--219. 

\bibitem{HSZ25a}
J. Huang, F. Sukochev, D. Zanin,
{\it Isomorphisms between symmetric spaces over infinite and finite von Neumann algebras,}
submitted manuscript. 

\bibitem{HSZ25}
J. Huang, F. Sukochev, D. Zanin, {\it Lack of isomorphic embedding from Lorentz function spaces into Lorentz operator ideals}, to appear. 

\bibitem{HSS}
J. Huang, O. Sadovskaya, F. Sukochev, {\it On Arazy's problem concerning isomorphic embeddings of ideals of compact operators}, Adv. Math. {\bf 406}  (2022), Paper No. 108530, 21 pp. 

\bibitem{Johnson Maurey}
W.B. Johnson, B. Maurey, G. Schechtman, L. Tzafriri,  \emph{Symmetric structures in Banach spaces}, Mem. Amer. Math. Soc. \textbf{19} (1979), 239--264.

\bibitem{Kalton98}
N. J. Kalton, 
{\it 
Spectral characterization of sums of commutators.} I, 
J. Reine Angew. Math.504 (1998), 115--125,

\bibitem{KPR}
N.J. Kalton, N. T. Peck, J.W. Roberts
\textit{An F-Space Sampler}, London Math. Soc. Lecture Notes, vol. 89, Cambridge Univ. Press, 1984.

\bibitem{KS}
N. Kalton, F. Sukochev,
\textit{Symmetric norms and spaces of operators},
J. Reine Angew. Math.  621 (2008), 81--121.

\bibitem{Kania Laustsen1}
T.~Kania, N.~Laustsen, \emph{Uniqueness of the maximal ideal of the Banach algebra of bounded operators on $C([0,\omega_1])$}, J. Funct. Anal. \textbf{11} (2012), 4831--4850.

\bibitem{Kania Laustsen2}
T.~Kania, N.~Laustsen, \emph{Uniqueness of the maximal ideal of operators on the $\ell_p$-sum of $\ell_\infty^n$ $(n\in\mathbb{N})$ for $1<p<\infty$}, Math. Proc. Cambridge Philos. Soc. \textbf{160} (2016), 413--421.

\bibitem{Kato}
T. Kato, {\it 
Perturbation theory for nullity deficiency and other quantities of
linear operators,}
J. Analyse Math. 6 (1958), 273--322. 



\bibitem{KwapienPelczynski}
S. Kwapien, A. Pe{\l}czy\'nski, \emph{The main triangle projection in matrix spaces and its applications}, Studia Math. \textbf{34} (1970), 43--68.




\bibitem{KR07}
A. Kaminska, Y. Raynaud, {\it  Copies of $\ell^p$ and $c_0$ in general quasi-normed Orlicz--Lorentz sequence spaces,} in: Function Spaces, in: Contemp. Math., vol. 435, 2007, pp. 207--227.


\bibitem{Leung1}
D.~Leung, \emph{Ideals of operators on $(\oplus\ell^\infty(n))_{\ell_1}$}, Proc. Amer. Math. Soc. \textbf{143} (2015), 3047--3053.

\bibitem{Leung2}
D.~Leung, \emph{Maximal ideals in some spaces of bounded linear operators}, Proc. Edinb. Math. Soc. \textbf{61} (2018), 251--264.

\bibitem{LPSZ}
G. Levitina, A. Pietsch, F. Sukochev, D. Zanin, \emph{Completeness of quasi-normed operator ideals generated by s-numbers}, Indag. Math. \textbf{25} (2014), no. 1, 49--58.

\bibitem{Lin Sari Zheng}
P.-K. Lin, B.~Sari, B. Zheng, \emph{Norm closed ideals in the algebra of bounded linear operators on Orlicz sequence spaces}, Proc. Amer. Math. Soc. \textbf{142} (2014), 1669--1680.


\bibitem{LT}
J.~Lindenstrauss, L.~Tzafriri, \emph{Classical Banach spaces. I.} Ergebnisse der Mathematik und ihrer Grenzgebiete, Vol.\textbf{ 92}
 Springer-Verlag, Berlin-New York, (1977) xiii+188 pp



 \bibitem{LSZ}
S.~Lord, F.~Sukochev, D.~Zanin,
{\it Singular traces: Theory and applications,}
  De Gruyter Studies in Mathematics,
 46. De Gruyter, Berlin, 2013.


\bibitem{McCarthy}
C.~A.~McCarthy, \emph{$c_p$}, Israel J. Math. \textbf{5} (1967), 249--271.

\bibitem{NP}
A. Nekvinda, D. Pe\v{s}a,
{\it On the Properties of Quasi-Banach Function Spaces},
 J. Geom. Anal. 34 (2024), no. 8, Paper No. 231, 29 pp.


\bibitem{Pietsch}
A. Pietsch, {\it 
History of Banach spaces and linear operators},
Birkh\"auser Boston, Inc., Boston, MA, 2007. 

\bibitem{PR}
M. Popov, B. Randrianantoanina, {\it Narrow operators on function spaces and vector lattices}, De Gruyter Studies in Mathematics, \textbf{45}. Walter de Gruyter \& Co., Berlin, 2013.

\bibitem{Rudin3}
W. Rudin, {\it Functional Analysis}, Second Edition [M]. New York: The McGraw-Hill Companies, Inc, 1991.

\bibitem{Schatten}
R. Schatten, 
{\it 
Norm ideals of completely continuous operators},
Ergebn. Math. Grenzgeb. N. F. 27, Springer-Verlag, Berlin 1960. 


\bibitem{Shoda}
K. Shoda, \emph{Some theorems on matrices}, Jpn. J. Math. \textbf{13(3)} (1937), 361--366.

\bibitem{Simon}
B. Simon, {\it 
Trace ideals and their applications}. Second edition. Mathematical Surveys and Monographs, 120. American Mathematical Society, Providence, RI, 2005.

\bibitem{S14}
F. Sukochev, {\it 
Completeness of quasi-normed symmetric operator spaces},
Indag. Math. (N.S.) 25 (2014), no. 2, 376--388.


\bibitem{SZ}
F. Sukochev, D. Zanin,
{\it Symmetric Banach sequence spaces  respect  Weyl submajorization,}
Proc. Amer. Math. Soc. 151(7) (2023), 2907--2917. 

\bibitem{Tarbard}
M.~Tarbard, \emph{Hereditarily indecomposable, separable $\mathcal{L}_\infty$ spaces with $\ell_1$ dual having few operators, but not very few operators}, J. London Math. Soc. \textbf{85} (2012), 737--769.

\bibitem{Whitley}
R.J. Whitley, {\it 
Strictly singular operators and their conjugates,}
Trans. Amer. Math. Soc. 113 (1964), 252--261.

\bibitem{Wintner}
A.~Wintner, \emph{The unboundedness of quantum-mechanical matrices}, Phys. Rev
  \textbf{71} (1947), 738--739.



\bibitem{Bentuo Zheng}
B. Zheng, \emph{Commutators on $(\sum\ell_q)_{\ell_1}$}, J. Math. Anal. Appl. \textbf{413} (2014), 155--169.

\end{thebibliography}

\end{document}